\documentclass[12pt]{amsart}\usepackage{a4wide,enumerate,color,graphicx}
\usepackage{amsmath}
\usepackage{scrextend} % for labeling environment
\usepackage{amssymb}
\usepackage{amsfonts}
\usepackage{amsthm}
\usepackage{enumerate}
\usepackage{mathrsfs}
\usepackage{mhchem}     % chemical equations with \ce{}

\newtheorem{lemma}{Lemma}[section]
\newtheorem{prop}[lemma]{Proposition}
\newtheorem{theo}[lemma]{Theorem}

\newtheorem{rem}[lemma]{Remark}
\newtheorem{coro}[lemma]{Corollary}

\DeclareMathOperator{\divv }{div}

\DeclareMathAlphabet{\mathdutchcal}{U}{dutchcal}{m}{n}
\newcommand{\calp}{\ensuremath\mathdutchcal{p}}
\newcommand{\calq}{\ensuremath\mathdutchcal{q}}
\newcommand{\cals}{\ensuremath\mathdutchcal{s}}

\newcommand{\calf}{\ensuremath\mathdutchcal{f}}
\newcommand{\calg}{\ensuremath\mathdutchcal{g}}

\newcommand{\calw}{\ensuremath\mathdutchcal{w}}
\newcommand{\calu}{\ensuremath\mathdutchcal{u}}

\newcommand{\calx}{\ensuremath\mathdutchcal{x}}
\newcommand{\calv}{\ensuremath\mathdutchcal{v}}

\newcommand{\calr}{\ensuremath\mathdutchcal{r}}
\allowdisplaybreaks

\begin{document}

\title[Incompressible limit for a fluid mixture]{Incompressible limit for a fluid mixture}

	% Math. Subject Classif.
%\pacs[2008]{82.45.Gj 82.45.Mp 82.60.Lf}				% ggf. Physics Astronomy Classif.

\author[P.-E.~Druet]{Pierre-Etienne Druet}
\address{Weierstrass Institute, Mohrenstr. 39, 10117 Berlin, Germany}
\email{pierre-etienne.druet@wias-berlin.de}

\date{\today}

\begin{abstract}
In this paper we discuss the incompressible limit for multicomponent fluids in the isothermal ideal case. Both a direct limit-passage in the equation of state and the low Mach-number limit in rescaled PDEs are investigated. Using the relative energy inequality, we obtain convergence results for the densities and the velocity-field under the condition that the incompressible model possesses a sufficiently smooth solution, which is granted at least for a short time. Moreover, in comparison to single-component flows, uniform estimates and the convergence of the pressure are needed in the multicomponent case because the incompressible velocity field is not divergence-free. We show that certain constellations of the mobility tensor allow to control gradients of the entropic variables and yield the convergence of the pressure in L1.  
\end{abstract}
\maketitle
%\vspace{-1cm}

\setcounter{tocdepth}{2}
\tableofcontents

\section{Introduction}

The equation of state of a pure fluid usually expresses the density $\varrho = \hat{\varrho}(T,p)$ or the specific volume $1/\varrho = \hat{\calv}(T,p)$ by constitutive functions $\hat{\varrho}$, $\hat{\calv}$ of temperature and pressure. Primarily, incompressibility is a property of this relation defined via $\partial_p\hat{\varrho} = 0$ or $\partial_p \hat{\calv} = 0$. Here ''zero'' means that the pressure variations occurring in the empirical physical system are too small to allow for substantial changes of the mass density. 
For the isothermal viscous case, the balance equations are the compressible Navier-Stokes equations which, in the pressure-velocity formulation, read
\begin{align}\begin{split}\label{I1}
\partial_p\hat{\rho}(T,p) \, (\partial_t p + v \cdot \nabla p) = & - \hat{\rho}(T,p) \, \divv v \, ,\\
\hat{\rho}(T,p) \,(\partial_t v + (v \cdot \nabla) v) +\nabla p = & \divv \mathbb{S} + \hat{\varrho}(T,p) \, b \, ,
\end{split}
\end{align}
and $\partial_p\hat{\varrho} \rightarrow 0$ yields the incompressible Navier-Stokes equations with constant density for the main variables $p, \, (v_1,v_2,v_3)$. Due to the lack of estimates on the pressure-field, it is not possible to study this limit with rigorous mathematical methods by the present state of the art.

Very often too, ''incompressible'' is taken as synonymous for dynamic volume conservation $\divv v = 0$, which in fact is a consequence of $\partial_p\hat{\rho} = 0$ and of the conservation of mass. However, the condition $\divv v = 0$ does not imply that $\varrho$ is constant, and it has also no impact on the equation of state of the fluid which is a material property\footnote{Solutions to the so-called density-dependent (incompressible) Navier-Stokes equations (see \cite{lionsfilsI}, Ch.\ 2) satisfy $\divv v = 0$ and $\partial_t\varrho + v\cdot \nabla \varrho=0$ independently. In this problem, the number of variables increases from the four standard variables $(p, \, v_1,v_2,v_3)$ to five: $(\varrho, \, p, \, v_1,v_2,v_3)$. It seems unclear whether this model can be reached as some asymptotic limit of the Navier-Stokes equations}. Finally, incompressibility is related to the Mach-number ${\rm Ma} = \sqrt{\varrho}|v|/\sqrt{p}$ or ${\rm Ma} = |v|/\sqrt{\partial_{\rho}p}$. In applications, a flow is called a low Mach-number flow if, after rescaling of position, time, and all variables and data occurring in the Navier-Stokes equations, it obeys
\begin{align}\begin{split}\label{I2}
 \partial_t \varrho + \divv (\varrho \, v) = & 0\, ,\\
 \varrho \, (\partial_t v+(v \cdot \nabla) v) + \frac{1}{\epsilon^2} \, \nabla p = & \divv \mathbb{S} + \varrho \, b \, ,
 \end{split}
\end{align}
with a global Mach-number ${\rm Ma} = \epsilon \ll 1$. Unlike \eqref{I1}, the limit $\epsilon \rightarrow 0$ for \eqref{I2} can be rigorously studied within the framework of weak solutions, which was first established in the pioneering work \cite{lionsmasmoudi}.
The density and the velocity field converge as $\epsilon \rightarrow 0$ and one recovers a weak solution to the incompressible Navier-Stokes equations with constant density. No uniform bounds are available for the pressure, but it can suitably be eliminated since the incompressible velocity field is solenoidal.
Typically, the transformed variables $\varrho$ and $p$ in \eqref{I2} obey an effective equation of state like
%\begin{align*}
$ \varrho  = (\epsilon^2 \, p+1))^{\frac{1}{\gamma}} =: \hat{\varrho}^{\epsilon}(p)$. Here $\gamma >1$ is some constant and, evidently, we have $\partial_p \hat{\varrho}^{\epsilon} \rightarrow 0$ for $\epsilon \rightarrow 0$. This shows that the low Mach-number limit in rescaled PDEs is profoundly related to the limit \eqref{I1}.

The motivation for the present paper is providing a first rigorous study of the incompressible limit for a fluid mixture of $N > 1$ components. For simplicity, the fluid mixture is assumed homogeneous, Newtonian and ideal. In particular, it is volume-additive and therefore it obeys the equation of state 
\begin{align*}
 \sum_{i=1}^N \rho_i \, \hat{\calv}_i(T, \, p) = 1 \, ,
\end{align*}
where $\rho_{1}, \ldots,\rho_N$ are the partial mass densities of the components and $\hat{\calv}_i(T,p) = 1/\hat{\rho}_i(T,p)$ denotes the specific volumes of the $i^{\rm th}$component \emph{as pure substance}. For instance, if each of the components is itself an incoompressible fluid, the density of the pure substance is constant near the reference temperature $T = T^{\rm R}$ and the reference pressure $p = p^{\rm R}$ of the system, hence
\begin{align}\label{I3}
\sum_{i=1}^N \rho_i \, \hat{\calv}_i(T, \, p) = \sum_{i=1}^N \rho_i \, \hat{\calv}_i(T^{\rm R}, \, p^{\rm R}) = \sum_{i=1}^N \frac{\rho_i}{\hat{\rho}_i^{\rm R}} = 1 \, .
\end{align}
We see that, in the very common case that the mixed fluids have different densities, the mass density $\varrho := \sum_{i=1}^N \rho_i$ of the mixture is in general not constant. Hence there are profound differences between the single-component and the multicomponent case concerning incompressibility. In a mixture which is incompressible in the sense of the equation of state, $\divv v = 0$ does not follow from the continuity equation. For this reason, certain authors distinguish between incompressibility and quasi-incompressibility: \cite{lowe}, \cite{feima16}.

The definition of incompressibility as $\partial_p \hat{\varrho} = 0$ has an important impact on the thermodynamic modelling of fluids. In particular, for the non-isothermal context, the question whether incompressibility forbids thermal expansion is widely studied: \cite{mueller}, \cite{bechtel}, \cite{gouin}, \cite{bothedreyer} for the single-component case. For the multicomponent case, asymptotic free energies have been studied in \cite{bothedreyerdruet}, exhibiting further interesting consequences of the independence of the equation of state on pressure.

In the present paper, we shall perform the limit in the PDEs in the spirit of \cite{lionsmasmoudi} and show the convergence to the incompressible or quasi-incompressible mixture model for the isothermal case. Let us recall that the results of Lions and Masmoudi were later refined and extended to the non-isothermal case, giving a justification for the Boussinesq approximation in a rigorous framework: see a.o.\ \cite{feinov07}, \cite{feinov09}, \cite{feinov13}, and the book \cite{feinovbook}. Likewise, extension of our results to the multicomponent non-isothermal case is certainly possible for the non-isothermal ideal model exposed in \cite{drumaxmix}. Interestingly, it is not to be expected that the Boussinesq approximation is valid in the multicomponent case for a dense mixture: See \cite{bothedreyerdruet}. Thus, the non-isothermal limit raises acute questions and should be studied with priority too. However, the notion of global weak solution to the energy or entropy equation has to be relaxed to variational sub-solution, as done for the single-component case in \cite{feinov07}. This is associated with yet more technicalities, which we would like to avoid in this first investigation. So we prefer to postpone sketching the full picture to a forthcoming paper. Another major challenge, that we do not tackle here, concerns dealing with the incompressible limit of \emph{reactive mixtures}. Mathematically, chemical reactions are difficult to be handled in a weak solution context, due among others to the exponential dependence of the reaction rates on pressure. More essentially, chemical reactions strongly affect the volume of a mixture: We refer to \cite{bothedreyerdruet} for more thoughts on this topic.

From the viewpoint of the mathematical method, our results are mainly based on the relative entropy inequality, here a relative {\it energy} inequality since the system is assumed isothermal. This method was applied also in the context of numerical approximation, see \cite{fisherentropy}, and it nowadays the main tool to study comparison issues in the context of PDEs with entropy structure. In essence, we follow the road-map of the book \cite{feinovbook} (see also \cite{feinovws12}, \cite{feinov13} and \cite{feijinnov}) to study singular thermodynamic limits in fluid mechanics. However, there are important differences to the single-component case. The most significant one is the necessity to discuss pressure bounds. For mixtures, the incompressible velocity field is not solenoidal, the terms involving the pressure in the relative energy functional do not vanish, and must be proved to converge with independent methods.

Incompressible fluid mixtures in the sense of constant density have been studied in several papers, among others: see \cite{bothesoga}, \cite{pekarsamohyl} for examples in the modelling direction, and \cite{chenjuengel}, \cite{mariontemam} for mathematical analysis. For an ideal mixture, constant density occurs in the incompressible case either if all mixed substances possess the same density, or if $N-1$ components are dilute in a dominant fluid. From the viewpoint of the present investigation, these are very particular cases. The incompressible model with \eqref{I3} is considered in \cite{mills} for a binary mixture and, with the special case of perfect gases in which $\hat{\calv}_i(T,p) = RT/(p \, M_i)$, in the section 3.3 of \cite{giovan} devoted to multicomponent low Mach-number flows. Let us emphasise that in this latter reference, the focus is not on the mathematical investigation of the asymptotics. In the paper \cite{majdasethian}, a low Mach-number limit for binary gas mixtures was derived in the context of combustion. Due to this particular context, the resulting models are less general than the ones in \cite{giovan}. Some mathematical well-posedness analysis of imcompressible mixture models without constant density was performed in few papers, among them \cite{feima16}, \cite{druetmixtureincompweak} for weak solutions, and \cite{bothedruetincompress} for strong solutions. The low Mach-number asymptotics for binary reactive flows in a rigorous spirit was dealt with in \cite{feipet10}, under rather strong restrictions on the admissible constitutive equations and with a phenomenological Fickian diffusion law. Assuming a dilute scaling, that is, the fraction of one species is of the same order as the Mach-number, the authors recover the Boussinesq approximation for a mixture. For a full characterization of the limit in the non-isothermal case, showing in particular that the Boussinesq--approximation is valid only for such dilute mixtures, we refer to \cite{bothedreyerdruet}.
For further references on the low Mach-number limit of multicomponent flows, we refer to \cite{giovan}, Ch.\ 3.3 and \cite{feipet10}.

Our road map in the present paper: In the sections \ref{PDES} and \ref{IKK}, we set up the multicomponent flow models for an isothermal ideal fluid mixture and, in particular, for the incompressible case. Then, in the section \ref{MATH1} we expose the mathematical assumptions on the data of the models, and the convergence statements for the multicomponent counterparts of \eqref{I1}, \eqref{I2} (see the Theorems \ref{heuriheura}, \ref{rigolo} and \ref{weaktoweak}). The proof are provided in the sections \ref{heuri} and \ref{RIGOR}.
The paper possesses a relatively long appendix, where we have collected some prerequisites for the proofs.

\section{PDE description of isothermal multicomponent fluids}\label{PDES}

We consider an isothermal liquid consisting of $N > 1$ different chemical substances that mix homogeneously, for instance a mixture of $N$ liquids. The thermodynamic state of the fluid is described by the vector of the partial mass densities $\rho = (\rho_1, \ldots,\rho_N)$ of the substances. Recall that the partial mass density $\rho_i$ is the mass of the substance ${\rm A}_i$ per unit volume of the mixture. The mass density of the fluid is defined as $\varrho = \sum_{i=1}^N \rho_i$ and is also called the total mass density.

The evolution of the partial mass densities $\rho_1, \ldots, \rho_N$ and of the velocity field $v = (v_1, \, v_2, \, v_3)$ is described by the following system of partial differential equations:
\begin{align}
\label{mass} \partial_t \rho_i + \divv(\rho_i \, v + J^i) =&  0 \quad \text{ for } \quad i=1, \ldots, N \, ,\\
\label{momentum} \partial_t (\varrho \, v) + \divv( \varrho \, v \otimes v - \mathbb{S} ) + \nabla p =&  \varrho \, b \, .
\end{align}
In addition to the main variables, we encounter here the diffusion fluxes $J^1, \ldots, J^N$, the viscous stress tensor $\mathbb{S}$, the thermodynamic pressure $p$, and the gravitational acceleration $b$. 

For the diffusion fluxes, we consider the Fick--Onsager or Maxwell--Stefan closure equations, while we assume that the viscous stress tensor is Newtonian:
\begin{align}
\label{DIFFUSFLUX} J^i = & - \sum_{j=1}^N M_{ij} \, \nabla \frac{\mu_j}{T} \quad \text{ for } \quad i = 1,\ldots,N \, , \\
\label{STRESS} \mathbb{S} = & 2 \, \eta \, (\nabla v)_{\text{sym}} + \lambda \, \divv v \, \mathbb{I}  \, .
\end{align}
Here $\mu_1, \ldots,\mu_N$ are the chemical potentials. In this paper, we restrict our considerations to so called ideal mixtures for which, by definition, 
\begin{align}\label{muideal}
\mu_i = g_i(T, \, p) + \frac{RT}{M_i} \, \ln x_i \, ,
\end{align}
with the gas constant $R$ and, for each component ${\rm A}_i$, the Gibbs free enthalpy $g_i$ which is a function of temperature and pressure, and the constant molar mass $M_i >0$. Denoting by $\hat{\upsilon}_i(T,p)$ the specific volume, and by $\hat{\rho}_i(T,p)$ the mass density of the $i^{\rm th}$ component at temperature $T$ and pressure $p$, we have $$\partial_pg_i(T,p) = \hat{\upsilon}_i(T,p) = 1/\hat{\rho}_i(T,p) \, .$$ Moreover, $-T \, \partial^2_Tg_i =  c_p^i$ is the heat capacity at constant pressure, while the heat capacity at constant volume of the pure substance is $-T \, [\partial^2_T g_i - (\partial^2_{p,T} g_i)^2/\partial^2_p g_i] = c_{\upsilon}^i$.
In particular, it follows that $p \mapsto g_i(T,p)$ is concave. 
We refer for instance to \cite{drumaxmix}, Section 7 for a construction of a simple thermodynamically consistent example.

In \eqref{muideal}, the mole fraction $x_i$ is related to the main variables via
\begin{align}\label{molefraction}
 x_i = \hat{x}_i(\rho) = \frac{\rho_i}{M_i \, \sum_{j=1}^N (\rho_j/M_j)} \quad \text{ implying that } \quad  \sum_{i=1}^N x_i = 1 \, . 
 \end{align}
The pressure $p$ is related to the main variables via the equation of state
\begin{align}\label{EOS}
 \sum_{i=1}^N \rho_i \, \partial_pg_i(T, \, p) = 1 \, ,
\end{align}
which expresses the volume-additivity of ideal mixtures. This equation  defines the pressure implicitly as a function of $T$ and $\rho$. Throughout the paper, we denote this function by $\hat{p}(T, \, \rho)$ or also, ignoring the temperature which is only a parameter, by $\hat{p}(\rho)$ 

With the constitutive relations \eqref{DIFFUSFLUX}, \eqref{STRESS}, \eqref{muideal} and \eqref{molefraction}, \eqref{EOS}, the PDEs \eqref{mass} and \eqref{momentum} constitute a closed system for the variables $\rho_1, \ldots, \rho_N$ and $v_1, \, v_2, \, v_3$.

For the mathematical and numerical analysis, and for the study of stability issues concerning isothermal systems, the Helmholtz free energy arises as the natural thermodynamic potential. Here it is assumed that the Helmholtz free energy of the system possesses the density
\begin{align}\label{FE}
\varrho \psi = f(T, \, \rho_1, \ldots, \rho_N) \, 
\end{align} 
with a certain constitutive function $f$ of the temperature $T$ - a constant in the present context - and of the partial mass densities. The choice of $f$ consistent with \eqref{muideal}, \eqref{EOS} is 
\begin{align}\label{FEhere}
 f(T, \, \rho) = \sum_{i=1}^N \rho_i \, g_i\big(T, \, \hat{p}(T,\rho)\big) - \hat{p}(T,\rho) + RT \, \sum_{i=1}^N \frac{\rho_{i}}{M_i} \, \ln \hat{x}_i(\rho) \, .
\end{align}
The chemical potentials and the pressure then obey
\begin{align}\label{CHEMPOT}
 \mu_{j} = & \partial_{\rho_j} f(T, \, \rho_1, \ldots,\rho_N) \quad \text{ for } \quad j=1,\ldots,N \, ,\\
\label{GIBBSDUHEM} p = & -f(T, \rho_1, \ldots,\rho_N) + \sum_{i=1}^N \rho_i \, \mu_i \, , \quad \text{-- Gibbs--Duhem equation.}
\end{align}
The thermodynamic diffusivities $M_{ij} = M_{ij}(T, \rho_1, \ldots,\rho_N)$ in \eqref{DIFFUSFLUX} constitute a symmetric positive semidefinite matrix subject to
\begin{align}\label{MCONSTRAINT}
\sum_{i=1}^N M_{ij}(T,\rho_1,\ldots,\rho_N) = 0 \quad \text{ for all } \quad j = 1,\ldots, N \text{ and all } \rho_1, \ldots, \rho_N > 0 \, .
\end{align}
Note that the theory of irreversible processes originally postulates fluxes of the form
\begin{align}\label{DIFFUSPRIME}
 J^i = -\sum_{j=1}^{N-1} \widetilde{M}_{ij}(T,\rho) \, \nabla (\mu_j-\mu_N) \quad \text{ for } \quad i = 1,\ldots,N-1 \, \quad \text{ and } \quad  J^{N} := - \sum_{i=1}^{N-1} J^i \, ,
\end{align}
and the $N-1 \times N-1$ matrix $\{\widetilde{M}_{ij}\}$ is actually called the Onsager--operator. Singling out a particular dominant species (the ''solvent'') with index $N$ might be meaningful in many applications. However we prefer, in the present paper, using the form \eqref{DIFFUSFLUX} which, as shown for instance in \cite{bothedruetMS}, is completely equivalent. 
For simplicity, the viscosity coefficients in \eqref{STRESS} are assumed constant, and are required to satisfy
\begin{align*}
\eta > 0, \quad \lambda+\frac{2}{3} \, \eta \geq 0 \, . 
\end{align*}

{\bf Incompressible fluid mixture.} The volume-additive fluid mixtures are subject to the equation of state \eqref{EOS}, which we can also express as
\begin{align}\label{EOSVA}
 \upsilon = \sum_{i=1}^N M_i \, x_i \, \partial_pg_i(T, \, p) =: \hat{\upsilon}(T,\, p,\, x_1, \ldots,x_N) \, ,
\end{align}
where $\upsilon$ denotes the molar volume of the mixture, and $\hat{\upsilon}(T,p,x)$ is its constitutive representation as a function of temperature, pressure and the mole fractions. Here we use $x = (x_1,\ldots,x_N)$ as a useful abbreviation for the vector of mole fractions.

Following \cite{bothedreyerdruet}, we call a multicomponent fluid incompressible if its molar volume\footnote{Due to the relation $1/\varrho = \upsilon/(\sum_i M_i \,x_i)$, the definition of incompressibility is not affected by the choice of the molar or specific volume.} is independent on pressure, hence
\begin{align}\label{DEFOFINCOMP}
\partial_p \hat{\upsilon}(T, \, p, \, x_1, \ldots,x_N) = 0 \, \quad \Longleftrightarrow \quad \sum_{i=1}^N M_i \, x_i \, \partial_{p}^2g_i(T, \, p) = 0\, .
\end{align}
According to this notion\footnote{The fluids characterised by \eqref{DEFOFINCOMP}, which we here call incompressible, were also called \emph{quasi-incompressible} elsewhere: See \cite{lowe}, \cite{feima16}. In dynamic situations, fluid flows subjects to \eqref{DEFOFINCOMP} can conserve mass without satisfying $\divv v = 0$.}, and recalling that each $g_i$ is concave in $p$, mixtures characterised by the volume-additivity \eqref{EOSVA} are thus incompressible under the following condition: For all $i = 1, \ldots, N$, either the $i^{\rm th}$ component is incompressible -- this means that $\partial_p \hat{\upsilon}^i = \partial^2_pg_i = 0$ --, or this constituent is dilute in the mixture -- this means that $x_i \ll 1$. However, since we want to allow here not only for dilute solutions, but also for mixtures of fluids where no species needs being dilute, the only possibility to fulfill \eqref{DEFOFINCOMP} is to require $\partial^2_pg_i = 0$ for $i = 1,\ldots,N$. Hence, each constituent is itself incompressible and its Gibbs energy $g_i$ is affine in $p$.

The definition \eqref{DEFOFINCOMP} affects the thermodynamic states and the constitutive equations for the chemical potentials in a special way.
We introduce the specific volumes at reference temperature $T^{\rm R} = T$ and pressure $p^{\rm R}$ via
\begin{align}\label{defbarupsiloni}
\calv_i = \partial_pg_i(T^{\rm R}, \, p^{\rm R}) = \frac{1}{\hat{\rho}_i(T^{\rm R}, \, p^{\rm R})} \quad \text{ for } \quad i = 1,\ldots,N \, .
\end{align}
Then, in view of \eqref{EOS} the mass densities in an incompressible fluid are subject to the side-condition
\begin{align}\label{CONSTRAINT} \sum_{i=1}^N \rho_i \, \calv_i = & 1 \quad \Big( \text{ equiv.\ with } \quad\varrho = \frac{\sum_{i=1}^N M_i \, x_i}{\sum_{i=1}^N M_i \, x_i \, \calv_i}\Big) \, .
\end{align}
Unlike \eqref{EOS}, it is to note that this constraint cannot be satisfied by appropriate choice of $p$. Hence, in the incompressible case, the pressure remains an independent variable in the PDEs. In other words, the main variables are $T$ (here constant), $\rho_1, \ldots,\rho_N$ and $p$. 

Due to the fact that $g_i$ is affine in $p$ with slope $\calv_i$, the chemical potentials are given by
\begin{align}\label{CHEMPOTincom}
 \mu_i := p \, \calv_i + \frac{RT}{M_i} \, \ln x_i \quad \text{ for } \quad i = 1, \ldots, N \, . 
\end{align}
Using a singular Helmholtz free energy function, it is also possible to generalise the expressions \eqref{FE}, \eqref{CHEMPOT}, \eqref{GIBBSDUHEM} to the incompressible case. We define
 \begin{align} \label{limitFE} f^{\infty}(T, \, \rho) := \begin{cases} 
RT \, \sum_{i=1}^N \frac{\rho_i}{M_i} \, \ln \hat{x}_i(\rho) & \text{ for } \sum_{i=1}^N \rho_i \, \calv_i = 1 \, ,\\
+\infty & \text{ otherwise. }
\end{cases}
\end{align}
It can be shown (see \cite{druetmixtureincompweak}, \cite{bothedruetincompress}, \cite{bothedreyerdruet}) that \eqref{CHEMPOTincom} is equivalent to
\begin{align*}%\label{CHEMPOT2}
\mu \in \partial_{\rho} f^{\infty}(T,\rho) \, , \quad p = - f^{\infty}(T,\rho) + \sum_{i=1}^N \rho_i \, \mu_i \, ,
\end{align*}
where $\partial_{\rho}$ is the subdifferential. This provides a direct generalisation of \eqref{CHEMPOT}, \eqref{GIBBSDUHEM}.

In the incompressible case, the variables are $(T,p,\rho)$, and it is necessary to state the pressure-dependence of the phenomenological coefficients explicitly. Hence, the diffusion fluxes are written as
\begin{align}
\label{DIFFUSFLUXimcomp} J^i = & - \sum_{j=1}^N M_{ij}(T,\,p, \, \rho_1,\ldots,\rho_N) \, \nabla \frac{\mu_j}{T} \quad \text{ for } \quad i = 1,\ldots,N \, ,
\end{align}
However, in this paper, we shall restrict for simplicity to the particular case that there is no substantial pressure-dependence. Then the $M_{ij}$'s are regular functions depending only on $\rho_1,\ldots,\rho_N$ for both the compressible and incompressible case. 
Another important remark is that, throughout this paper, we shall assume that there are at least two different (reference) specific volumes $\calv_i \neq \calv_j$ which, in vector notations, means that $\calv$ is not parallel to $1^N$. The case $\calv = \calv_0 \, 1^N$ with $\calv_0 > 0$ - which means that all components of the mixture possess the same density at reference conditions - is very particular. The incompressibility condition reduces to a constant mass density, and we observe a complete de-coupling of the systems \eqref{mass}, \eqref{momentum}. The momentum equations reduce to the incompressible Navier-Stokes equations for a viscous fluid. In fact, this limit must be studied with the same methods as the single-component case. This means that a convergence result for the pressure is not to be expected by present state of the art.

\section{Incompressibility: $\partial_p\upsilon = 0$ vs.\ ${\rm Ma} = 0$}\label{IKK} Consider in \eqref{muideal} a sequence $g^{m} = g^m(T,p)$ with $m \in \mathbb{N}$. The function $- \partial_pg^m_i \, \partial^2_pg^m_i$ is called (isothermal) compressibility of the substance ${\rm A}_i$. Finite compressibility means that $\partial^2_pg^m < 0$. Assume that for $m \rightarrow \infty$ the function $g^{m}_i$ converge in a suitable sense to a $g^{\infty}_i = g^{\infty}_i(T,p)$ which is affine in $p$.
In the paper \cite{bothedreyerdruet}, we have proved in a very general setting that the indexed Helmholtz potentials $f^m$ of the form \eqref{FEhere} Gamma-converge to a singular limit $f^{\infty}$ with the structure \eqref{limitFE}.

The main question of the present paper concerns the behaviour not of the thermodynamic structures, but of solutions to the PDEs \eqref{mass}, \eqref{momentum} where the constitutive model is indexed by $m$. If $(\rho^{m}, \, v^{m})$ are solutions to the PDE system \eqref{mass}, \eqref{momentum} with constitutive relations relying on \eqref{muideal} do these solutions converge as $m \rightarrow +\infty$, to a solution to the incompressible model with \eqref{CONSTRAINT} \eqref{CHEMPOTincom} and \eqref{DIFFUSFLUXimcomp}?

Now, we must distinguish between the incompressible limit according to the physical definition \eqref{DEFOFINCOMP} and the low Mach-number limit in rescaled PDEs.\\

{\bf Stable incompressible fluid phase.}  The physical statement of incompressibility presupposes that the temperature, pressure and composition of the fluid remain in a range where the compressibility function of the mixture $-\partial_p\upsilon/\upsilon = -\sum_i \partial^2_pg_i(p) \, \rho_i$ is very small compared to some empirical measure. Hence it is possible to replace the free enthalpies $g_i$'s by affine functions. 
 
 In order to perform this first variant of the limit in the isothermal context, we assume that
 \begin{align}\label{firstvariat}
 g_i^m(T, \, p) \rightarrow \calv_i \, p \quad \text{ for } m \rightarrow \infty, \quad \text{ for all } p \in ]p_1, \, p_2[
 \end{align}
 where $0 < p_1 < p_2 \leq +\infty$ is a ''reasonable range'' of pressures for which the incompressible phase remains stable.

 From the viewpoint of rigorous mathematical asymptotics, the limit based on \eqref{firstvariat} is very difficult to perform on the original PDE system. To the present state of knowledge, no methods are available for proving that solutions to the PDEs \eqref{mass}, \eqref{momentum} possess a globally in time bounded pressure, or that the densities remain strictly positive, let alone for proving that the solutions remains confined to the real domain of the state space where thermodynamic stability makes sense. However, it is possible to study rigorously a second variant of the limit, for certain scalings of the PDEs \eqref{mass}, \eqref{momentum}.\\

{\bf Low Mach-number limit in rescaled PDEs.} The complete scaling procedure is described in the Appendix, Section \ref{RESCALE}. With a reference time $t^{\rm R}$ and a reference length $L^{\rm R}$, we introduce the normalised domain $Q^{\rm R} := \{(\bar{x}, \, \bar{t}) \, : \, \bar{x} = x/L^{\rm R} \, , \, \bar{t} = t/t^{\rm R}, \, (x,t) \in \Omega \times ]0,\, \bar{\tau}[\}$. Rescaling of the PDEs \eqref{mass}, \eqref{momentum} yields
\begin{align}\label{massrescfin}
 \partial_{\bar{t}} \bar{\rho}_i + \overline{\divv} \Big(\bar{\rho}_i \, \bar{v} -  \sum_{j=1}^N \bar{M}_{ij} \, \bar{\nabla} \bar{\mu}_j\Big) =& 0 \, ,\\\label{momentumrescfin}
\bar{\varrho} \, ( \partial_{\bar{t}} \bar{v}  + \bar{v} \cdot \bar{\nabla} \bar{v}) +\overline{\divv} \, \overline{\mathbb{S}}(\bar{\nabla} \bar{v}) + \bar{\nabla} \bar{p}_{\Delta} =& -  \bar{\varrho} \, e^3 \, .
\end{align}
%where $\text{Fr} = \sqrt{L^{\text{R}}/(g \, (t^{\text{R}})^2)}$ is the Froude number. 
The bars on functions, fields and differential operators denote a renormalised/dimensionless quantity. The relative pressure $\bar{p}_{\Delta}$ is obtained from the thermodynamic pressure by two steps: The pressure is first normalised, then the deviations from reference pressure are rescaled by the Mach-number squared, to obtain that 
\begin{align*}%\label{PPATTENTION}
 \bar{p}_{\Delta}(\bar{x}, \, \bar{t}) =  \frac{1}{{\rm Ma}^2} \, \Big(\frac{p(L^{\rm R} \, \bar{x}, \, t^{\rm R} \, \bar{t})}{p^{\rm R}}-1\Big) \, , \quad \text{ where } \quad  {\rm Ma} = \frac{v^{\rm R}}{c^{\rm R}} \, ,
 \end{align*}
 with the mean pressure $p^{\rm R}$, the mean modulus of velocity of the fluid $v^{\rm R}$ and the mean variation of pressure over density at fixed temperature and composition $c^{\rm R}$, which possesses the same magnitude as the speed of sound in the medium.\footnote{In fact the rescaling procedure uses the special choice $c^{\rm R} = p^{\rm R}/\varrho^{\rm R}$, where $\varrho^{\rm R}$ is the average density.} 
%Note that we here use for the pressure the reference value $p^0 := \varrho^{\rm R} \, (c^{\rm R})^2$, with $\varrho^{\rm R}$ being the average mass density. 
We use the notation $\bar{p}_{\Delta}$ to hint at the fact that pressure \emph{variations} are rescaled rather than of the pressure itself.

Introducing $m := {\rm Ma}^{-2}$, the constitutive choices for \eqref{massrescfin}, \eqref{momentumrescfin} are given by 
\begin{align}\label{MUIRESC}
& \bar{\mu}_i = \hat{\mu}^{m}_i(\bar{p}_{\Delta},\bar{x}_i) :=  \bar{g}_i^m(\bar{p}_{\Delta}) + \frac{1}{\bar{M}_i} \, \ln \bar{x}_i \, ,\\
\text{ with } \quad & \bar{g}_i^m(\bar{p}_{\Delta}) =  m \, \Big(\bar{g}_i\big(1+\frac{\bar{p}_{\Delta}}{m}\big) - \bar{g}_i(1)\Big) \, ,\label{ginormalised}
\end{align}
with rescaled masses $\bar{M}_i = M_i(v^{\rm R})^2/(RT)$, and with the rescaled free energy functions 
\begin{align}\label{bargipi}
\bar{g}_i(\cdot) = \frac{g_i(p^{\rm R} \, \cdot)}{g^{\rm R}} \quad \text{ where } \quad g^{\rm R} := \frac{\varrho^{\rm R}}{p^{\rm R}} \, .
%\quad \varrho^{\rm R} = \text{averaged mass density} \,.
\end{align}
Within these definitions, it is also to note that a rescaled equation of state for the pressure is valid in the form $\bar{p}_{\Delta} = \hat{\pi}^m(\bar{\rho}) \, ,$ where $\hat{\pi}^{m}$ is the implicit function defined by the normalised equation of state
\begin{align}\label{EOSnormalised}
 \sum_{i=1}^N (\bar{g}^{m}_i)^{\prime}(\pi) \, \bar{\rho}_i = 1 \quad \Longleftrightarrow \quad \pi = \hat{\pi}^{m}(\bar{\rho}_1,\ldots,\bar{\rho}_N) \, .
\end{align}
Next, we compare
the rescaled PDEs \eqref{massrescfin}, \eqref{momentumrescfin} and constitutive equations \eqref{MUIRESC} and \eqref{EOSnormalised} with the original \eqref{mass}, \eqref{momentum}, \eqref{muideal}, \eqref{EOS}, and see that the structures are identical, up to the difference that the pressure is replaced by the ''normalised pressure variations'' $\bar{p}_{\Delta}$. In the formal limit of $m \rightarrow \infty$, the equation \eqref{ginormalised} implies that
\begin{align}\label{CONVFORMALE}
 \bar{g}_i^m(\bar{p}_{\Delta}) = \int_{0}^1 \bar{g}_i^{\prime}\Big(1+ \frac{\theta}{m} \, \bar{p}_{\Delta} \Big) \, d\theta \, \bar{p}_{\Delta} \longrightarrow  \bar{g}_i^{\prime}(1) \, \bar{p}_{\Delta} = \varrho^{\rm R} \, g_i^{\prime}(p^{\rm R}) \, \bar{p}_{\Delta} = \varrho^{\rm R} \,\calv_i \, \, \bar{p}_{\Delta} \, ,
\end{align}
with the reference mass density $\varrho^{\rm R}$. This shows that the (formal) limit of the functions $\bar{g}^m_i$ for $m\rightarrow \infty$ is the affine function $\bar{g}^{\infty}_i(\bar{p}_{\Delta}) = \bar{\calv}_i \, \bar{p}_{\Delta}$ with $\bar{\calv}_i := \varrho^{\rm R} \, \calv_i$. Thus, the expected limit of the rescaled system is formally identical with the incompressible system, with the important difference that, now, the constitutive equations make sense independently of the sign or the boundedness of $\bar{p}_{\Delta}$.
As already shown in the single-component case by the pioneering work \cite{lionsmasmoudi}, the rescaling approach allows to make sense of the incompressibile limit in a rigorous mathematical framework. To study stability issues on the rescaled system, note  that we can introduce a rescaled free energy function via
\begin{align}\label{FEm}
 \bar{f}^m(\bar{\rho}) := \sum_{i=1}^N \bar{g}_i^m(\hat{\pi}^m(\bar{\rho})) \, \bar{\rho}_i - \hat{\pi}^m(\bar{\rho}) + \sum_{i=1}^N \frac{\bar{\rho}_i}{\bar{M}_i} \, \ln \frac{\bar{\rho}_i/\bar{M_i}}{\sum_{j=1}^N  (\bar{\rho}_j/\bar{M_j})}\, ,
\end{align}
and then can verify that $\partial_{\bar{\rho}_i} \bar{f}^m(\bar{\rho}) = \bar{\mu}_i =  \hat{\mu}^{m}_i(\hat{\pi}^m(\bar{\rho}), \, \hat{x}_i(\bar{\rho}))$.

\section{Notations, assumptions and statement of the results}\label{MATH1}
\addtocontents{toc}{\protect\setcounter{tocdepth}{1}}

{\bf Notations:} We define
$ \mathbb{R}^N_+ := \{\rho \in \mathbb{R}^N \, : \, \rho_1,\ldots,\rho_N > 0\}$ and
\begin{align*}
\overline{\mathbb{R}^N_{+}} := \{\rho \in \mathbb{R}^N_+ \, : \, \rho_1,\ldots,\rho_N \geq 0\} \, .
\end{align*}
For $x \in \mathbb{R}^N$, and $1 \leq \calp < +\infty$, the $\calp-$norm of $x$ is $|x|_\calp = (\sum_{i=1}^N |x_i|^\cap)^{1/\calp}$ and $|x|_{\infty} = \max_{i=1,\ldots,N} |x_i|$. We moreover denote $\max x = \max_{i=1,\ldots,N} x_i$ and $\min x = \min_{i=1,\ldots,N} x_i$.

We let $\mathcal{P}: \, \mathbb{R}^{N} \rightarrow \{1^N\}^{\perp}$ (orthogonal complement of the vector $1^N = (1, \ldots,1) \in \mathbb{R}^N$) denote the orthogonal projection
\begin{align*}
 \mathcal{P} \xi := \xi - \frac{1}{N} \, \sum_{i=1}^N \xi_i \, 1^N \quad \text{ for } \quad \xi \in \mathbb{R}^N \, .
\end{align*}
For $\rho \in \mathbb{R}^N_+$ the mass and mole fractions are given by
\begin{align}\begin{split}\label{LESFRACS}
 & y_i = \hat{y}_i(\rho) = \frac{\rho_{i}}{\sum_{j=1}^N \rho_j}, \quad x_i = \hat{x}_i(\rho) = \frac{\rho_{i}}{M_i \, \sum_{j=1}^N (\rho_j/M_j)} \, , \\
& \text{ with the relationship }\quad  y_i = \frac{M_i}{\sum_{j=1}^N M_j \, x_j} \, x_i \, .
\end{split}
\end{align}
We will make use of the function
\begin{align}\label{KFUKK}
 k(\rho) = RT \, \sum_{i=1}^N \frac{\rho_i}{M_i} \, \ln \hat{x}_i(\rho) \, , 
\end{align}
and of its rescaled variant $\bar{k}(\bar{\rho}) = \sum_{i=1}^N \frac{\bar{\rho}_i}{\bar{M}_i} \, \ln \hat{x}_i(\bar{\rho})$. These functions are positively homogeneous. In particular, $Dk(\rho) \cdot \rho = k(\rho)$ and $D^2k(\rho) \, \rho = 0$. In order to study the behaviour of the function $\hat{\pi}^m$ introduced in \eqref{EOSnormalised}, we introduce a function $\hat{\bar{p}}(\bar{\rho}) = \widehat{p}(\bar{\rho})$ via
\begin{align}\label{EOSnormalised2}
 \sum_{i=1}^N \bar{g}_i^{\prime}(\pi) \, \bar{\rho}_i = 1 \quad \Longleftrightarrow \quad \pi = \widehat{p}(\bar{\rho}_1,\ldots,\bar{\rho}_N) \, ,
\end{align}
where $\bar{g}$ are the normalised functions of \eqref{bargipi}. Then, \eqref{EOSnormalised} and \eqref{EOSnormalised2} characterise
\begin{align}\label{piiswidehatp}
 \hat{\pi}^m(\bar{\rho}) = m \, (\widehat{p}(\bar{\rho})-1) \, .
\end{align}
Moreover, in the context of mathematical proofs, we shall denote the reference pressure by $p^0$ ($=p^{\rm R}$).

We assume throughout the paper that $\Omega \subset \mathbb{R}^3$ is a bounded domain with Lipschitz boundary. We denote by $\bar{\tau} >0$ the final time, and $Q = Q_{\bar{\tau}} = \Omega \cup ]0,\bar{\tau}[$ is the space-time cylinder in $\mathbb{R}^4$. For a Lebesgue measurable subset $E$ of $\Omega$ (of $Q$), we denote by $|E|$ its three-dim.\ (four-dim.) Lebesgue-measure, and by $E^{\rm c}$ the complement $\Omega \setminus E$ (the complement $Q \setminus E$).

Several usual function spaces shall occur. For $1\leq \calp \leq \infty$ the Lebesgue-spaces $L^\calp(\Omega)$ and $L^\calp(0,\bar{\tau})$ are well-known. For
$1\leq \calp,\calq < +\infty$, the Lebesgue space $L^{\calq,\calp}(Q)$ consists of all integrable functions $\calu: \, Q \rightarrow \mathbb{R}$ such that $\int_{0}^{\bar{\tau}} (\int_{\Omega} |\calu(x,\tau)|^{\calq}dx)^{\calp/\calq} \, d\tau$ is finite and, for $\calp = +\infty$, such that $\text{ess\,sup}_{0 < t < \bar{\tau}}\int_{\Omega} |\calu(x,\tau)|^{\calq}dx < +\infty$. We also use the evolution-space variants $L^{\calp}(0;\bar{\tau}; \, L^{\calq}(\Omega))$ of these spaces. For the norms, we adopt the either the usual full denotation or the following abbreviations: The norm on $L^\calp(\Omega)$ is denoted $|\cdot|_{\calp}$, while the norm on $L^{\calq,\calp}(Q)$ is denoted $\|\cdot\|_{\calp,\calq}$ and $\|\cdot\|_{\calp}$ for $\calp = \calq$. 

Moreover we encounter the Sobolev spaces $W^{1,\calq}(\Omega)$, $W^1_{\calp}(Q)$, $W^{1,0}_{\calp}(Q) = L^{\calp}(0,\bar{\tau}; \, W^{1,\calp}(\Omega))$ and $W^{2,1}_{\calp}(Q) =   W^1_{\calp}(0,\bar{\tau}; \, L^{\calp}(\Omega))\cap L^{\calp}(0,\bar{\tau}; \, W^{2,\calp}(\Omega))$. All definitions are well-known from standard monographs. The norm are denoted by $|\cdot|_{W^{1,\calp}}$, $\|\cdot\|_{L^{\calq}W^{1,\calq}}$, etc.\ For an integrable function $\calu: \, Q\rightarrow \mathbb{R}$, we denote by $(\calu)_M\in L^1(0,\bar{\tau})$ the function $(\calu)_M := \frac{1}{|\Omega|} \, \int_{\Omega} \calu(x,\cdot) \, dx$.

% Due to the large number of symbols needed, we cannot avoid everywhere double denotations. 
% 

\subsection{Assumptions for the free energy functions}

We adopt the viewpoint that the thermodynamic pressure is in essence positive, and simplify the discussion by assuming that the functions $g_1(T, \cdot), \ldots,g_N(T,\cdot)$ occurring in \eqref{muideal} are defined on the whole positive real line. Motivated by requirements of thermodynamic stability, we assume that $(g_1, \ldots,g_N)$ is subject to: 
\begin{labeling}{(A44)}
\item[(A1)] $g_i \in C^2(]0, \, +\infty[)$;
\item[(A2)] $g_i^{\prime}(p) > 0$ and $ g_i^{\prime\prime}(p) < 0$ for all $p > 0$; 
\item[(A3)] $\lim_{p \rightarrow 0} g^{\prime}_i(p) = + \infty$ and $\lim_{p \rightarrow +\infty} g^{\prime}_i(p) = 0$;
\end{labeling}
The conditions (A2) express for each constituent the requirement that: first its density $\hat{\rho}_i(p) = 1/g_i^{\prime}(p)$ is positive (first condition), and second that its volume $g_i^{\prime}(p)$ is a strictly decreasing function on pressure (second condition). The first of the conditions in (A3) means that the lower-pressure threshold corresponds to infinite volume, and the second one that infinite pressure leads to volume zero. 

For the weak solution analysis, we shall have to moreover require specific growth behaviour for small and large arguments:
\begin{labeling}{(A44)}
\item[(A4)] There are positive constants $\bar{c}_1 \leq \bar{c}_2$ and $\bar{s} > p^0$ such that
\begin{align*}
0< \bar{c}_1 \leq g_i^{\prime}(p) \, p \leq \bar{c}_2\quad \text{ for all }\quad 0 < p < \bar{s} \, .
\end{align*}
\item[(A5)] There are $\alpha_1, \ldots, \alpha_N \geq \beta > 1$ such that
\begin{align*}
 \alpha_i \, p \, g_i^{\prime}(p) \geq g_i(p) \geq \beta \, p \, g_i^{\prime}(p) \quad \text{ for all }\quad p \geq \bar{s} \, .
\end{align*}
\item[(A6)] There is $\bar{c}_3 > 0$ such that $p \, |g^{\prime\prime}_i(p)| \leq \bar{c}_3 \, g_i^{\prime}(p)$ for all $p>0$, $i = 1,\ldots,N$. 
\end{labeling}
For small and moderate arguments, (A4) implies that $g_i$ grows like $ \ln p $ and, for large arguments, (A5) implies that $g_i$ grows faster than $ p^{1/\alpha_i}$ with $\alpha_i > 1$.
The Helmholtz free energy is defined by \eqref{FEhere}. Under the assumptions (A1), (A2) it can be shown that $f = f^m$ is a strictly convex function on $\mathbb{R}^N_+$, which express the thermodynamic stability.
Under the assumptions (A4), (A5) there are $c_0,  \, c_1 > 0$ depending only on the constants occuring in (A4), (A5) such that
\begin{align}\label{FEGROWTH}
 f(\rho) \geq c_0 \, |\rho|^{\gamma} - c_1 \quad \text{ with } \quad \gamma := \max_{i= 1,\ldots,N} \alpha_i/(\max_{i= 1,\ldots,N} \alpha_i-1) \, .
\end{align}
Using (A4), (A5) and $g^{\prime}(\hat{p}(\rho)) \cdot \rho = 1$ we can moreover show that
\begin{align}\label{la-bas}
\bar{c}_1 \, \varrho &\leq \hat{p}^m(\rho) \leq \bar{c}_2 \,\varrho \, &  \text{ for } \quad 0 < \hat{p}(\rho)  < \bar{s} \, , \\
\label{plarge12}
\left(\frac{\beta\, \min g^{\prime}(\bar{s})}{\max \alpha} \right)^{\gamma} \, \bar{s} \, \varrho^{\gamma} \leq&  \hat{p}(\rho) \leq \left(\frac{\max\alpha\, \max g^{\prime}(\bar{s})}{\beta} \right)^{\frac{\beta}{\beta-1}} \, \bar{s} \, \varrho^{\frac{\beta}{\beta-1}}&  \text{ for }\quad  \hat{p}(\rho) \geq  \bar{s} \, .
\end{align}
These statements and other growth properties of the free energy functions independently on the Mach-number are established in the appendix, Proposition \ref{UNIFF1}.

Next considering the rescaling procedure, we
let $m = 1,2,3,\ldots$ and define $\bar{g}^m$ according to \eqref{ginormalised}, with $\bar{g}$ from \eqref{bargipi}. We see that $\bar{g}^m$ is defined in the interval $]-m, \, +\infty[$. Moreover,
\begin{align*}
( \bar{g}_i^m)^{\prime}(\pi)= \partial_p\bar{g}_i\Big(1 + \frac{\pi}{m}\Big) \quad \text{ for all } \quad \pi \in ]-m, \, + \infty[ \, .
\end{align*}
Thus, if the underlying functions $\bar{g}_1 , \ldots, \bar{g}_N$ satisfy the asssumptions (A1)-(A3), then 
\begin{labeling}{(A'44)}
\item[($\bar{\rm A}$1)] $\bar{g}_i^m \in C^2(]-m, \, +\infty[)$;
\item[($\bar{\rm A}$2)] $(\bar{g}_i^m)^{\prime}(\pi) > 0$ and $ (\bar{g}_i^m)^{\prime\prime}(\pi) < 0$ for all $\pi> -m$; 
\item[($\bar{\rm A}$3)] $\lim_{\pi \rightarrow -m} (\bar{g}^m_i)^{\prime}(\pi) = + \infty$ and $\lim_{\pi \rightarrow +\infty} (\bar{g}^m_i)^{\prime}(\pi) = 0$;
\end{labeling}
For the rescaled model, the free energy function is discussed in Appendix, Prop.\ \ref{UNIFF2}.

\subsection{Convergence assumptions}

In order to study the incompressible limit, we will consider two situations.

In the first situation, we assume in \eqref{muideal} that the functions $g_i$ are indexed by a large parameter, hence $g_i = g_i^m$ with $m \in \mathbb{N}$. Moreover, for all $i = 1,\ldots,N$ and $m = 1,2,\ldots$ the function $g^m_i$ is subject to the assumptions (A1)--(A3). 
For the incompressible limit in its physical sense, we define the affine functions $g^{\infty}_i(p) = \calv_i \, p$ and, for a certain range $0<p_1 < p_2 < +\infty$ (containing the reference value $p^0$) in which thermodynamic stability is valid, we assume that
\begin{align}\label{CONVER1}
 g_i^m \longrightarrow g^{\infty} \quad \text{ in } \quad C^2([p_1,p_2]) \, .
\end{align}

\vspace{0.2cm}

In the second situation, we consider for each $m \in \mathbb{N}$ the rescaled system \eqref{massrescfin}, \eqref{momentumrescfin} with constitutive equations \eqref{MUIRESC}, where functions $\bar{g}_1, \ldots,\bar{g}_N$ are fixed and satisfy the assumptions (A). We obtain the effective functions $\bar{g}^m_i$ applying the rescaling step \eqref{ginormalised}, \eqref{bargipi}.
Then, the function $\bar{g}^m_i$ satisfies (${\rm \bar{A}}$1),  ($\rm \bar{A}$2), ($\rm \bar{A}$3). We easily show that (cf.\ \eqref{CONVFORMALE})
\begin{align}\label{CONVER2}
 \bar{g}_i^m \longrightarrow \bar{g}^{\infty} \quad \text{ in } \quad C^2(K) \quad \text{ for all compact sets } \quad  K \subset \mathbb{R}  \, ,
\end{align}
with the affine limit $\bar{g}^{\infty}(\pi) = \bar{\calv}_i \, \pi$ for $\pi \in \mathbb{R}$. Here, $\bar{\calv}$ is the rescaled vector of specific volumes $\bar{\calv} = \varrho^{\rm R} \, \calv$.
%\begin{align}\label{Barcalv}
%(\bar{\calv}_1, \ldots, \bar{\calv}_N) = \varrho^{\rm R} \, (\calv_1, \ldots,\calv_N) \, .                                                                                                                                                                                                                                                                                                                          \end{align}

\subsection{Assumptions for the mobility tensor}

For $\{M_{ij}\}$ (or the rescaled $\{\bar{M}_{ij}\}$), we adopt the following assumptions (B): 
\begin{labeling}{(B44)}
\item[(B1)] $ M_{ij} = M_{ij}(T,\cdot) \in C^{0,1}(\overline{\mathbb{R}^N_{+}})$ for all $ i,j=1,\ldots,N$;
\item[(B2)] For all $\rho_1,\ldots,\rho_N > 0$, $M_{ij}(\rho) = M_{ji}(\rho)$ for all $i\neq j$ and $\sum_{j=1}^N M_{ij}(\rho) = 0$;
\item[(B3)] There is a positive function $\lambda_0 \in C(\mathbb{R}^N_+)$ such that $\sum_{i,j=1}^N M_{ij}(\rho) \xi_j \, \xi_i \geq \lambda_0(\rho) \, |\mathcal{P}\xi|^2$ for all $\rho \in \mathbb{R}^N_+, \, \xi \in \mathbb{R}^N$;
\end{labeling}
These assumptions are natural and they also include the Maxwell--Stefan (short: M--S) choice of the mobility tensor. Exemplarily, if all M--S interaction coefficients are equal, then it is well--known that
\begin{align}\label{maxstefbase}
 M_{ij}(\rho) = d \, \rho_i \, (\delta^i_j - \rho_j) \, ,
\end{align}
which might serve as a toy example. The general condition for a system with Maxwell-Stefan diffusion consists, according to the paper \cite{bothedruetMS}, in the property
\begin{align}\label{maxstefreg}
d_0(p) \, \mathdutchcal{P}^{\sf T} \, R \, \mathdutchcal{P} \leq M(\rho) \leq d_1(p) \,  \mathdutchcal{P}^{\sf T} \, R \, \mathdutchcal{P} \, ,
\end{align}
where $0 < d_0 \leq d_1$ are functions of pressure (and temperature), $R := \text{diag}(\rho_1,\ldots,\rho_N)$ and $\mathdutchcal{P} := \mathbb{I} - 1^N \otimes \hat{y}(\rho)$. Obviously, the toy model \eqref{maxstefbase} occurs for $d_0 = d = d_1$ implying that $M = \mathdutchcal{P}^{\sf T} \, R \, \mathdutchcal{P}$.

In the context of weak solutions in which, to the present best knowledge the densities do not need remaining uniformly positive, we shall need to reinforce the assumption (B3). This will be stated explicitly stated in the theorems. 

As a last remark concerning the mobilities and the other phenomenological coefficients, 
note that it is possible, and even more precise, to conceptualise $M_{ij}$ as a function of the temperature $T$, the pressure $p$ and the mole fractions $x_1, \ldots,x_N$ instead of the main variables as done in (B1). See in this respect also \eqref{DIFFUSFLUXimcomp}. This question has an important impact on the discussions on incompressibility (See the remark \ref{lesprecisions}, \eqref{cetteremarquela}).

\subsection{Solution concept for the PDEs}\label{SolConc}

In order to investigate the low Mach--number asymptotics while simplifying the technical mathematical discussions as much as possible, we shall blend out possible external influences on the physical system. We consider insulating boundary conditions for the fluxes and no velocity slip:\begin{align}\label{NOPENET}
J^i \cdot \nu(x) = 0 \text{ for } i = 1,\ldots,N, \qquad v = 0\, , \quad \text{ on } \partial \Omega \times ]0,\bar{\tau}[ \, .
\end{align}
We consider only the relaxation of the system starting from certain (non-equilibrium) initial data
\begin{align}\label{INIT}
\rho(x, \, 0) = \rho^{0,m}(x), \quad v(x, \, 0) = v^0(x) \quad \text{ for } \quad x \in \Omega \, .
\end{align}
The case $v(x,0) = v^{0,m}(x)$ with indexed initial data for the velocity can be easily treated too.

For $m \in \mathbb{N}$, we call (IBVP$^m$) the initial-boundary-value-problem \eqref{mass}, \eqref{momentum} with constitutive equations \eqref{DIFFUSFLUX}, \eqref{STRESS}, \eqref{muideal}, and initial and boundary conditions \eqref{NOPENET}, \eqref{INIT}. Hereby, we assume in \eqref{muideal} that $g_i = g_i^m$ with $m \in \mathbb{N}$ where, for all $i = 1,\ldots,N$ and $m = 1,2,\ldots$ the function $g^m_i$ is subject to the assumptions (A).

For $m \in \mathbb{N}$, we might consider the rescaled IBVP \eqref{massrescfin}, \eqref{momentumrescfin} with constitutive equations, \eqref{MUIRESC}. In this case, the initial and boundary conditions \eqref{NOPENET}, \eqref{INIT} are also rescaled. Whenever we need to separately refer the rescaled problem, we shall denote it by ($\overline{\text{ IBVP}}^m$).

We call $(\rho, \, v) = (\rho^m, \, v^m)$ a weak solution to (IBVP$^m$) if the initial-boundary-value-problem is satisfied in the usual sense of distributions and, moreover, the following energy inequality
\begin{align}\begin{split}\label{DISSIP}
 & \int_{\Omega} \Big( \frac{\varrho(x,t)}{2} \, |v(x,t)|^2 + f^m(\rho(x,t))\Big) \, dx + \int_0^t \int_{\Omega} \mathbb{S}(\nabla v) \, : \, \nabla v +\zeta^{\rm Diff} \, dxd\tau \\
 & \qquad \leq  \int_{\Omega} \Big( \frac{\varrho_0^m(x)}{2} \, |v^0(x)|^2 + f^m(\rho^{0,m}(x))\Big) \, dx + \int_0^t\int_{\Omega} \varrho \, b \cdot v \, dxd\tau
 \end{split}
\end{align}
is valid for all $0 \leq t \leq \bar{\tau}$, where $\zeta^{\rm Diff}$ is the entropy dissipation of diffusion, and we moreover defined $\varrho_0^m(x) := \sum_{i=1}^N \rho^{0,m}_i(x)$.
Depending on the regularity of the weak solution and the assumption on the data, we will have different representations of $\zeta^{\rm Diff}$. If the solution were smooth and the densities $\rho_1,\ldots,\rho_N$ uniformly positive, then
\begin{align*}%\label{zetadiffbase}
\zeta^{\rm Diff} =  - J \, : \, \nabla \mu = \sum_{i,j=1}^N M_{ij}(\rho) \, \nabla \mu_i \cdot \nabla \mu_j \geq 0 \, .
\end{align*}
Let us next specify minimal requirements in order that all integrals involved in the definition of a weak solution are making sense. 
%, we can identify $-J \, : \, \nabla \mu = M(\rho) \nabla \mu \, : \, \nabla \mu \geq 0$. 
Using \eqref{FEGROWTH} with $\zeta^{\rm Diff} \geq 0$, the minimal regularity of weak solutions can be read off from \eqref{DISSIP}:
\begin{align}\begin{split}\label{NC}
 \rho \in L^{\gamma,\infty}(Q_{\bar{\tau}}; \, \mathbb{R}^N) \quad \text{ and } \quad 
 v \in  L^2(0,\bar{\tau}; W^{1,2}_0(\Omega; \, \mathbb{R}^3)) \text{ with } \sqrt{\varrho} \, v \in L^{2,\infty}(Q_{\bar{\tau}}; \, \mathbb{R}^3) \, ,
 \end{split}
\end{align}
with $\gamma > 1$ from \eqref{FEGROWTH}. Recall moreover that $\rho = (\rho_1, \ldots,\rho_N)$ is non-negative.

However, the properties \eqref{NC} -- which reflect the typical regularity for compressible Navier--Stokes equations -- do not allow to make sense of diffusion terms: At first, the chemical potentials $\mu_i = \hat{\mu}_i(\rho)$ obey \eqref{muideal} and are not defined on sets where a mole fraction $\hat{x}_i(\rho)$ vanishes. At second we need an information on gradients. 
%If we restrict to the possibilities offered by \eqref{DISSIP}, this information can only be gained from the contribution $-J : \, \nabla \mu$.

Two different attempts to introduce appropriate ''diffusive variables'', allowing to define global-in-time weak solutions, were presented in \cite{dredrugagu20}, \cite{druetmaxstef}. In both cases, the condition (B3) for the mobility tensor needs being reinforced. Since these ideas are of technical nature, we shall recall them in the appendix of the paper, in Section \ref{ShortSurvey}, for readers interested in this kind of discussions. 
In the present introductory context, let us simplify the technical problems by considering the case of uniformly positive weak solutions, those satisfying
\begin{align}\label{UNIFORMPOS}
 \inf_{i=1,\ldots,N} \rho_i(x, \, t) \geq s_0 > 0 \quad  \text{ for almost all } \quad (x,t) \in Q_{\bar{\tau}} \, .
\end{align}
Then, exploiting (B3), we get
\begin{align*}
- \sum_{i=1}^N J^i \cdot \nabla \mu_i = \sum_{i,j=1}^N M_{ij}(\rho) \, \nabla \mu_i \cdot \nabla \mu_j \geq \Big(\inf_{\rho_1,\ldots,\rho_N \geq s_0}\lambda_0(\rho)\Big) \, \,\,|\mathcal{P}\nabla \mu|^2 \, ,
\end{align*}
and, in addition to \eqref{NC}, the inequality \eqref{DISSIP} motivates the regularity
\begin{align}\label{NC+}
\nabla \mathcal{P} \hat{\mu}(\rho) \in L^2(Q_{\bar{\tau}}; \, \mathbb{R}^N) \quad \text{ and }\quad \sum_{i,j=1}^N M_{ij}(\rho) \, \nabla \hat{\mu}_i(\rho) \cdot \nabla \hat{\mu}_j(\rho) \in L^1(Q_{\bar{\tau}})  \, . 
\end{align}
As a consequence of (B1), the entries of the mobility matrix satisfy
\begin{align}\label{b1prime}
 |M_{ij}(\rho)| \leq |M_{ij}(0)| + \sup_{r\in\mathbb{R}_N^+} |\partial_r M_{ij}(r)| \, |\rho| \leq \bar{\lambda} \, (1+|\rho|) \, .
\end{align}
Using the Cauchy-Schwarz inequality, we then show that
\begin{align*}
 |J^i| \leq &  M_{ii}^{\frac{1}{2}}(\rho) \,  \Big(\sum_{i,j=1}^N M_{ij}(\rho) \, \nabla \hat{\mu}_i(\rho) \cdot \nabla \hat{\mu}_j(\rho)\Big)^{\frac{1}{2}}\\
 \leq & \bar{\lambda}^{\frac{1}{2}} \, \sqrt{1+|\rho|} \, \Big(\sum_{i,j=1}^N M_{ij}(\rho) \, \nabla \hat{\mu}_i(\rho) \cdot \nabla \hat{\mu}_j(\rho)\Big)^{\frac{1}{2}} \, ,
\end{align*}
and with the help of \eqref{NC}, \eqref{NC+} and H\"older's inequality, we can see that the diffusion fluxes satisfy a bound in $L^{2\gamma/(1+\gamma),2}(Q_{\bar{\tau}}; \, \mathbb{R}^{N\times 3})$. The regularity of positive weak solutions to (IBVP$^m$) hence consists of the conditions \eqref{NC}, \eqref{NC+}. 

Note that, in order to fully exploit the multicomponent character, it would be more precise to state the regularity of $\rho$ in the Orlicz class generated by the free energy function. The $\gamma-$growth in \eqref{NC} is usual in the analytical context of single-component compressible Navier--Stokes equations but it is only the worst case scenario for mixtures.
\\

We shall call (IBVP$^{\infty}$) the initial-boundary-value-problem \eqref{mass}, \eqref{momentum}, \eqref{NOPENET} with \eqref{STRESS} for the Newtonian stress, and incompressible constitutive equations \eqref{CHEMPOTincom}, \eqref{DIFFUSFLUXimcomp} for the fluxes. The initial condition \eqref{INIT} is replaced by
\begin{align*}%\label{INITincom}
 \rho(x, \, 0) = \rho^{0,\infty}(x) \quad \text{ for } \quad x \in \Omega \, ,
\end{align*}
where $\rho^{0,\infty}$ is an incompressible initial state subject to \eqref{CONSTRAINT}.

In this paper, only strong solutions $(\rho^{\infty}, \, p^{\infty}, \, v^{\infty})$ to (IBVP$^{\infty}$) shall be considered. A particularity to note is that, in the incompressible model, the thermodynamic diffusivity $M_{ij}$ and the viscosity coefficient $\eta$ and $\lambda$ are independent on pressure, and the boundary conditions also do not involve the pressure. As a consequence, the pressure field of a solution to (IBVP$^{\infty}$) is determined only up to an arbitrary function of time. In order to rule out this multiplicity, we adopt a special way to prescribe the mean-value via
\begin{align}\label{CONVEN}
\int_{\Omega} p^{\infty}(x, \, t) \, dx \overset{!}{=} - RT \, \sum_{i=1}^N \, \frac{\eta_i}{M_i} \, \int_{\Omega} \, \ln \hat{x}_i(\rho^{\infty}(x,t)) \, dx \, ,
\end{align}
with real numbers $\eta_1, \ldots,\eta_N$ satisfying $\sum_{i=1}^N \eta_i \, \calv_i = 1$ and, if the vectors $\calv$ and $1^N$ are not parallel, also $\sum_{i=1}^N \eta_i = 0$. This condition implies that a certain linear combination of the chemical potentials has mean-value zero over $\Omega$, which turned out convenient for the analysis of (IBVP$^{\infty}$)in \cite{druetmixtureincompweak}, \cite{bothedruetincompress}.
Overall, concerning the resolvability of the incompressible model (IBVP$^{\infty}$), we adopt the following list (C) of assumptions:
\begin{labeling}{(C44)}
\item[(C)] There exists a sufficiently smooth solution $(\rho^{\infty}, \, p^{\infty}, \, v^{\infty})$ with strictly positive densities to (IBVP$^{\infty}$) on some interval $]0,\,\bar{\tau}[$, moreover satisfying \eqref{CONVEN}, and such that 
\begin{gather*}
 \rho^{\infty}_1, \ldots,\rho^{\infty}_N, \, \,  p^{\infty}  \in W^{1}_{\infty}(Q_{\bar{\tau}}), \qquad v^{\infty} \in W^{2,1}_{\calp}(Q_{\bar{\tau}}; \, \mathbb{R}^3) \quad \text{ with }\quad \calp \geq 6 \, , \\[0.1cm]
  \min_{(x,t) \in \overline{Q_{\bar{\tau}}}, \, i = 1,\ldots,N} \rho^{\infty}_i(x,t) \geq r_0 > 0 \, .
\end{gather*}
\end{labeling}
We remark that, owing to the Theorem \ref{EXISSTRONGINFTY}, the property (C) can be verified locally in time if the data are smooth enough.

\subsection{Convergence results}

We shall at first state, in Theorem \ref{heuriheura}, a convergence result for the scenario \eqref{CONVER1}. Here the approximate solutions are subject to additional \emph{a priori} assumptions, which express the physical consistency of the pressure-field.
Due to this, the result of Theorem \ref{heuriheura} can not be applied to \emph{available} weak solutions, and its value is mainly heuristic. However, it has the advantage to allow us introducing the main estimate in a more comfortable way widely free of technicalities. In order to simplify the proof, we shall adopt the normalisation assumptions
\begin{align}\label{NORMIN}
g^m(p^0) =0, \quad \partial_pg^m(p^0) = \calv \quad \text{ for all } m \in \mathbb{N} \, ,
\end{align}
In general, $g^m_1(p^0), \ldots,g^m_N(p^0)$ are constants related to the entropy and the enthalpy of the species under reference thermodynamic conditions $(T,p^0)$, while $\partial_pg^m(p^0)$ are the specific volumes under reference conditons. Thus, \eqref{NORMIN}$_2$ is quite natural, while the reader can easily verify by himself that the Theorem remains valid for general $g^m(p^0)$ such that $\sup_m |g^m(p^0)| < +\infty$.
%weak solutions in the case of a uniformly positive mobility tensor.
\begin{theo}\label{heuriheura}
Suppose that for all $m \in \mathbb{N}$, the functions $g_1^m,\ldots,g_N^m$ satisfy the growth and regularity assumptions ${\rm (A1)-(A3)}$ and the normation assumptions \eqref{NORMIN},
%\begin{align*}%\label{NORMIN}
% g^m(p^0) =0, \quad \partial_pg^m(p^0) = \calv \quad \text{ for all } m \in \mathbb{N} \, ,
%\end{align*}
that $\{M_{ij}\}$ satisfies ${\rm (B)}$, and that ${\rm (C)}$ is valid for \emph{(IBVP$^{\infty}$)}. Suppose that, for every $m > 1$, the vector $(\rho^{m}, \, v^{m})$ is a weak solution with the regularity \eqref{NC}, \eqref{NC+} to \emph{(IBVP$^m$)}, and that $g^m \rightarrow g^{\infty}$ for $m \rightarrow \infty$ as in \eqref{CONVER1} with two fixed thresholds $0< p_1 \leq  p^0 \leq  p_2 < +\infty$.
%with $\gamma > 3/2$ and \eqref{NC+} to the compressible problem $(P_m)$. 
Assume moreover that the pressure $p^m = \hat{p}^m(\rho^m)$ is confined to the interval of convergence of $\{g^m\}$, that is,
\begin{align}\label{TurevesII}
 p_1 \leq p^m(x,t) \leq p_2 \quad \text{ for all } m \in \mathbb{N}, \text{ for almost all } (x,t) \in Q_{\bar{\tau}} \, .
\end{align}
%\end{enumerate}
Suppose that $v^0 \in W^{2-2/\calp}_\calp(\Omega; \, \mathbb{R}^3)$ and that $\rho^{0,m} \in L^{\infty}(\Omega; \, \mathbb{R}^N)$ satisfies $\min_i \rho^{0,m}_i \geq s_0 > 0$ and the following conditions:
\begin{enumerate}[(a)]
\item\label{E0bounded} The initial energy is uniformly bounded, that is
\begin{align*}%\label{E0bounded}
 \sup_{m\in\mathbb{N}}\int_{\Omega} \Big( \frac{\varrho_0^m(x)}{2} \, |v^0(x)|^2 + f^m(\rho^{0,m}(x))\Big) \, dx < + \infty \, ,
\end{align*}
\item With the constants $p_1,p_2$ of  \eqref{TurevesII}: $p_1 \leq \hat{p}^m(\rho^{0,m}(x)) \leq p_2$ for all $x \in \Omega$, $m \in \mathbb{N}$;
\item\label{E0converges} $\rho^{0,m} \rightarrow \rho^{\infty}(\cdot,0)$ in $L^1(\Omega; \, \mathbb{R}^N)$ as $m \rightarrow +\infty$.
\end{enumerate}
Then, $\rho^{m} \rightarrow \rho^{\infty}$ in $  L^1(Q_{\bar{\tau}}; \, \mathbb{R}^N)$, and $v^{m} \rightarrow v^{\infty}$ in $L^1(Q_{\bar{\tau}}; \, \mathbb{R}^3)$. Moreover, if $\calv$ is not parallel to $1^N$, then with $\eta_1,\ldots,\eta_N$ chosen according to \eqref{CONVEN}, there is a sequence $\bar{\zeta}^m$ of functions of time such that $\sup_m \|\bar{\zeta}^m\|_{L^{\infty}(0,\bar{\tau})} < +\infty$ and $p^m + \bar{\zeta}^m\rightarrow p^{\infty}$ in $L^1(Q_{\bar{\tau}})$. 
\end{theo}
\begin{rem}
If $\calv$ and $1^N$ are parallel, then the weak convergence $p^m \rightharpoonup p^{\infty}$ in $L^2(Q_{\bar{\tau}})$ is a consequence of \eqref{momentum} and the strong convergence of $\{\rho^m\}$ and $\{v^m\}$, but we are not able to show a strong convergence result.
\end{rem}

At second, in Theorem \ref{rigolo}, we state an entirely rigorous convergence statement valid for the rescaled problems ($\overline{\text{ IBVP}}^m$) under the convergence \eqref{CONVER2} of the free energy functions.
In order to obtain pressure bounds we have to reinforce the assumption (B3) as follows:
\begin{labeling}{(C44)}
\item[(B3$^{\prime}$)] The function $\lambda_0$ in (B3) is bounded away from zero. In other words, there is $\lambda_0 > 0$ such that $\sum_{i,j=1}^N \bar{M}_{ij}(\rho) \, \xi_i \, \xi_j \geq \lambda_0 \, |\mathcal{P} \xi|^2$ for all $\rho \in \overline{\mathbb{R}^N_{+}}$, $\xi \in \mathbb{R}^N$.
\end{labeling}
As shown in \cite{dredrugagu20} (see Appendix, Th.\ \ref{EXIWEAK}), (B3$^\prime$) allows to construct weak solutions in the natural class \eqref{NC}, \eqref{NC+}. For this result, in addition to (A1)-(A6), we adopt a simplifying technical assumption concerning the functions $\bar{g}_1, \ldots, \bar{g}_N$:
\begin{labeling}{(C44)}
\item[(A$^{\prime}$)] There is a permutation $\{k_1, \ldots,k_N\}$ of $\{1,\ldots,N\}$ such that $$\bar{g}^{\prime}_{k_1}(s) < \ldots < \bar{g}^{\prime}_{k_N}(s)\quad \text{ for all } \quad  s > 0 \, ,$$
\end{labeling}
which means that the mass densities of the species are strictly ordered, with the ordering independent on pressure. Since $\bar{\calv}_i = \bar{g}^{\prime}_i(1)$, (A$^{\prime}$) implies in particular that the species under reference conditions have all different specific volumes. This is the interesting case in practice.
\begin{theo}\label{rigolo}
For the free energy functions $\bar{g}$, we assume ${\rm (A)}$ with $\gamma \geq 9/5$ and ${\rm (A^{\prime})}$. Assume ${\rm (B1)}$, ${\rm (B2)}$, and ${\rm (B3^{\prime})}$ for the matrix $\{M_{ij}\}$, and ${\rm (C)}$ for the solution $(\rho^{\infty},p^{\infty},v^{\infty})$ of the incompressible model. 
we let $\calw_0^m := \frac{1}{|\Omega|} \, \int_{\Omega} \varrho^{0,m}(x) \, dx$, and we assume that $\bar{\varrho}_{\min} < \inf_m \calw_0^m$ and $\sup_m \calw_0^m < \bar{\varrho}_{\max}$. 
Suppose that the initial data satisfy \eqref{E0bounded}, \eqref{E0converges} as in Theorem \eqref{heuriheura} and moreover
\begin{align*}
\limsup_{m \rightarrow \infty} m \, \|\widehat{p}(\bar{\rho}^{0,m}) - 1\|_{L^2(\Omega^{\rm R})}^2 = 0 \, ,
\end{align*}
where $\widehat{p}$ is the function defined in the normalised equation of state \eqref{EOSnormalised2}.
%implicietely by the equation $\bar{g}(\widehat{p}(r)) \cdot r = 1$. 
Suppose that, for every $m > 0$, the vector $(\bar{\rho}^{m}, \, \bar{v}^{m})$ is a weak solution with the regularity \eqref{NC}, \eqref{NC+} to the rescaled compressible problem \emph{($\overline{\text{ IBVP}}^m$)}. 
Then, $\bar{\rho}^{m} \rightarrow \rho^{\infty}$ in $  L^1(Q^{\rm R}; \, \mathbb{R}^N)$, and $\bar{v}^{m} \rightarrow v^{\infty}$ in $L^1(Q^{\rm R}; \, \mathbb{R}^3)$. Moreover, there is a sequence $\{\bar{\zeta}^m\}$ of functions of time such that $\sup_m \|\bar{\zeta}^m\|_{L^1(0,\bar{\tau})} < +\infty$ and such that $\bar{p}^m_{\Delta} + \bar{\zeta}^m \longrightarrow p^{\infty}$ in $L^1(Q^{\rm R})$.
\end{theo}
\begin{rem}\label{lesprecisions}
\begin{enumerate}[(i)]
\item Weak solutions to \emph{($\overline{\text{ IBVP}}^m$)} exist globally (see Th.\ \ref{EXIWEAK}). Moreover under the conditions of Theorem \ref{EXISSTRONGINFTY}, the problem \emph{(IBVP$^{\infty}$)} possesses a unique classical local solution on an interval $]0,t^*[$ so that assumption $\rm (C)$ is valid on this interval;
\item In the case of Maxwell-Stefan diffusion (cf.\ \eqref{maxstefbase}, \eqref{maxstefreg}), the condition $\rm (B3^{\prime})$ is known to be violated. In this case, we cannot expect the class \eqref{NC+} and must find alternative ways of making sense of the diffusion flux. A result in this direction is to find in the paper \cite{druetmaxstef}. This case is yet more technical and cannot be treated with available methods, as a uniform bound for the pressure is missing even in $L^1$;
\item In both theorems, the functions $\bar{\zeta}^m(t)$ are needed in order to match the correct mean--value of the limit pressure in accordance with the condition \eqref{CONVEN};
\item \label{cetteremarquela} It is fundamental for this result that the phenomenological coefficients $M_{ij}$ are independent on pressure, that is, they depend only on $\rho_1, \ldots,\rho_N$ with the smoothness in $\rm (B1)$. We expect being able to treat the case $M_{ij} = M_{ij}^m(p, \, x_1, \ldots,x_N)$ only if, among other restrictions, $\partial_pM^m_{ij}$ has compact support and if $\partial_p M^m_{ij}$ converges to zero uniformly for $m \rightarrow \infty$.
\end{enumerate}
\end{rem}
Finally, let us remark that it is also possible in this context to remove the assumption (C) entirely, and prove the convergence of weak solutions to weak solutions. However, the pressure-field of weak solutions to the incompressible model is affected by a defect measure as explained in the Appendix, Section \ref{ShortSurvey}.
\begin{theo}\label{weaktoweak}
Assume ${\rm (A)}$ with $\gamma \geq 9/5$, ${\rm (A^{\prime})}$ and ${\rm (B1)}$, ${\rm (B2)}$, and ${\rm (B3^{\prime})}$.
Suppose that the initial data satisfy \eqref{E0bounded}.
Suppose that, for every $m > 0$, the vector $(\bar{\rho}^{m}, \, \bar{v}^{m})$ is a weak solution to the rescaled compressible problem \emph{($\overline{\text{ IBVP}}^m$)}. 
Then, there is a weak solution $(\rho^{\infty}, \, p^{\infty}+d\kappa, \, v^{\infty})$ $\text{with defect measure}$ to \emph{( IBVP$^{\infty}$)}, and $\bar{\rho}^{m_k} \rightarrow \rho^{\infty}$ in $  L^1(Q^{\rm R}; \, \mathbb{R}^N)$, and $\bar{v}^{m_k} \rightarrow v^{\infty}$ in $L^1(Q^{\rm R}; \, \mathbb{R}^3)$ for a subsequence $\{m_k\}_{k\in\mathbb{N}}$. 
Moreover, there is a sequence $\{\bar{\zeta}^m\}$ of functions of time such that $\sup_m \|\bar{\zeta}^m\|_{L^1(0,\bar{\tau})} < +\infty$ and such that $\bar{p}_{\Delta}^m + \bar{\zeta}^m \overset{{\rm ACP}}{\longrightarrow} p^{\infty}$. 
\end{theo}
The remainder of the paper is devoted to proving these theorems. We begin with the physical limit of Theorem \ref{heuriheura} which, being partly formal, is simpler. For the sake of clarity, we will moreover restrict in the main text to proving the theorem under the additional asssumption that the approximate solutions have uniformly positive densities as in \eqref{UNIFORMPOS}. We show in the Appendix, Section \ref{SStormII}, how to remove the positivity assumption, but this is associated with rather technical discussions.
In the Section \ref{RIGOR}, we show the complete result of Theorem \ref{rigolo}.

\addtocontents{toc}{\protect\setcounter{tocdepth}{2}}
\section{Convergence for solutions with physically consistent pressure-field}\label{heuri}

Let $f^m = f^m(T,\cdot)$ be the free energy function \eqref{FEhere} with $g = g^m$. If the latter satisfies the assumptions (A), it is shown in Lemma \ref{UNIFF1} that $f^m$ is a co--finite function of Legendre type on $\mathbb{R}^N_+$ satisfying the condition \eqref{FEGROWTH}. With $\hat{p}^m(\rho)$ defined as the unique root to the equation \eqref{EOS}, and $k(\rho)$ being the abbreviation \eqref{KFUKK}, the following identities are valid:
\begin{align}
\label{gradient}\hat{\mu}_i^m(\rho) = & \partial_{\rho_i} f^m(\rho) = g_i^m(\hat{p}^m(\rho)) + \frac{RT}{M_i} \, \ln \hat{x}_i(\rho) = g_i^m(\hat{p}(\rho)) + \partial_{\rho_i}k(\rho) \,,\\% \quad \forall\, i=1,\ldots,N \, ,\\ 
\label{Hessian} \partial^2_{\rho_i,\rho_j}f^m(\rho) = & - \frac{\partial_pg_i^m(\hat{p}^m(\rho)) \, \partial_p g_j^m(\hat{p}^m(\rho)) }{\sum_{k=1}^N \rho_k \, \partial^2_pg_k^m(\hat{p}^m(\rho))} + \frac{RT}{M_iM_j \, \hat{n}(\rho)} \, \left(\frac{\delta^i_j}{\hat{x}_i(\rho)} -1\right) \, ,%\quad \forall i,j=1,\ldots,N \, ,
 \end{align}
 where we abbreviated $\hat{n}(\rho) = \sum_{i=1}^N (\rho_i/M_i)$. 
We assume that $(\rho, \, v) := (\rho_1, \ldots,\rho_N, \, v_1, \, v_2, \, v_3)$ is a weak solution vector with the regularity \eqref{NC} for (IBVP$^m$). We let $(r, \, u)$ with $r: \, \overline{Q_{\bar{\tau}}} \mapsto \mathbb{R}^N_+$ and $u: \, \overline{Q_{\bar{\tau}}} \mapsto \mathbb{R}^3$ be smooth vector fields. For $t \geq 0$ we define the relative energy functional
\begin{align*}%\label{REFunkt}
\Big(\mathcal{E}^m(\rho, \, v \, | \, r, \, u)\Big)(t) := \int_{\Omega} \Big(\frac{ \varrho}{2} \, |v-u|^2 + f^m(\rho) - f^m(r) - \hat{\mu}^m(r) \cdot (\rho - r)\Big) \, dx \, .
\end{align*}
In this formula, we use the abbreviation $\hat{\mu}^m(r) := Df^m(r) = \nabla_r f^m(r)$. 
The regularity of weak solutions guarantees that $\mathcal{E}^m(t)$ is finite for all $t$. An equivalent expression is
\begin{align}\label{in1}
 & \Big(\mathcal{E}^m(\rho, \, v \, | \, r, \, u)\Big)(t) \nonumber\\
 & =  \int_{\Omega} \Big( \varrho \, \frac{|v|^2}{2} + f^m(\rho)\Big) \, dx + \int_{\Omega}\Big( \varrho \, \frac{|u|^2}{2} - \varrho \, u\cdot v - f^m(r) - \hat{\mu}^m(r) \cdot (\rho-r) \Big) \, dx\, .
\end{align}
We moreover introduce the (relative) viscous dissipation functional via
\begin{align*}%\label{RelVisc}
 \mathcal{D}^{\rm Visc}(t) = & \Big(\mathcal{D}^{\rm visc}(v \, |  \, u)\Big)(t) := \int_{\Omega} \mathbb{S}(\nabla (v-u)) \, :\, \nabla (v-u) \, dx \,.
 \end{align*}
The section is divided into two parts. In the first subsection, we state the relative energy inequality, and in the second, we provide the stability estimate and the proof of Theorem \ref{heuriheura}.

\subsection{Relative energy inequality and stability estimate}

We here restrict to state the natural form of the energy inequality for positive weak solutions.
Meanwhile, relative energy/entropy inequalities are used extensively in the literature, so that the proof ideas are sufficiently well-known. We refer to \cite{feinovbook} for systematical developments. The proof is therefore to be found in the appendix, Section \ref{SStorm}, together with more refined versions of the inequality valid for non-necessarily positive weak solutions.
\begin{prop}\label{calculheuri}
Let $(\rho^m, \, v^m)$ be a weak solution to \emph{(IBVP$^m$)} satisfying \eqref{NC}, \eqref{UNIFORMPOS} and \eqref{NC+}. Let $(\rho^{\infty}, \, p^{\infty}, \, v^{\infty})$ satisfy \emph{(IBVP$^{\infty}$)} and $\rm (C)$. We define $\mu^{\infty} = p^{\infty} \, \calv + (RT/M)  \, \ln \hat{x}(\rho^{\infty})$ and $\mathcal{E}^m(t) := (\mathcal{E}^m(\rho^m, \, v^m \, | \, \rho^{\infty}, \, v^{\infty}))(t)$.
Then, for all $t \in ]0,\bar{\tau}[$,
\begin{align*}
& \mathcal{E}^m(t) + \int_{0}^t (\mathcal{D}^{\rm Visc}(v^m \, | \, v^{\infty}))(\tau) d\tau + \int_0^t\int_{\Omega}  M(\rho^m) \, \nabla \mathcal{P}(\mu^m-\mu^{\infty}) \, : \, \nabla \mathcal{P}(\mu^m-\mu^{\infty}) \, dxd\tau\\
\leq & \int_0^t\int_{\Omega} (\partial_t \rho^{\infty} + \divv (\rho^{\infty}\, v^{\infty})) \cdot ( g^m(p^m)- p^m \, \calv) - (\rho^m \cdot \calv - 1) \, (\partial_t p^{\infty} + v^{\infty} \cdot \nabla p^{\infty}) \, dxd\tau \\
& + 
\mathcal{E}^m(0) + \int_{\Omega} p^{\infty}(x,\cdot) \, (\rho^m(x,\cdot) \cdot \calv-1) \, dx \Big|_0^t +  \int_{0}^t \mathcal{R}^m(\tau) \, d\tau \, ,
\end{align*}
with a quadratic remainder term $\mathcal{R}^{m}(t) = \sum_{i=1}^4 \mathcal{R}^{m,i}(t)$ defined by
\begin{align*}
\mathcal{R}^{m,1}(t) := & \int_{\Omega} \Big((\varrho^m - \varrho^{\infty}) \, (\partial_t v^{\infty} +(v^m\cdot \nabla )v^{\infty}) + \varrho^{\infty} \, [(v^m-v^{\infty})\cdot \nabla ]v^{\infty}\Big) \cdot (v^{\infty}-v^m) \, dx\\
\mathcal{R}^{m,2}(t) := & \int_{\Omega} (v^{\infty}-v^m) \, (\rho^m - \rho^{\infty}) \, : \,  (\nabla \mu^{\infty} + 1^N \otimes b)\, dx \\
 \mathcal{R}^{m,3}(t) := & -\int_{\Omega} (M(\rho^m) - M(\rho^{\infty})) \nabla \mu^{\infty} \cdot \nabla (\mathcal{P}\mu^m-\mu^{\infty}) \, dx\\
\mathcal{R}^{m,4}(t) := & \int_{\Omega} (\partial_t \rho^{\infty}+\divv( \rho^{\infty} \, v^{\infty})) \cdot \big(Dk(\rho^m) - Dk(\rho^{\infty}) - D^2k(\rho^{\infty}) \, (\rho^m - \rho^{\infty}) \big)\, dx \, .
\end{align*}
\end{prop}
Next we want to estimate each part of the right-hand side in Prop.\ \ref{calculheuri}, showing either that it tends to zero, or that it is controlled by the relative energy functional.\\

\paragraph{\bf Some elementary bounds.}

For ease of writing, we drop the index $m$ on the state in this section. Since the incompressible state $\rho^{\infty}$ is subject to the constraint $\sum_{i=1}^N \rho^{\infty}_i(x,t) \, \calv_i = 1$, it obviously follows that
\begin{align}\label{rhominmax}
 \varrho_{\min} := \frac{1}{\max_{i=1,\ldots,N} \calv_i} \leq \varrho^{\infty}(x,t) \leq \frac{1}{\min_{i=1,\ldots,N} \calv_i} =: \varrho_{\max} \, \quad \text{ for all } \quad (x,t) \in Q_{\bar{\tau}} \, .
\end{align}
Moreover, with an index $i_0$ such that $\calv_{i_0} = \min \calv$ ($\calv_{i_0} = \max \calv$), we have
\begin{align*}
 \frac{1}{\varrho^{\infty}} - \calv_{i_0} = \sum_{i\neq i_0} \hat{y}_i(\rho^{\infty}) \, (\calv_i-\calv_{i_0}) \, .
\end{align*}
Since in (C) we assume that the strong solution $\rho^{\infty}$ possesses uniformly positive densities, all larger than $r_0>0$ it follows that
\begin{align*}
 \left|\frac{1}{\varrho^{\infty}(x,t)} - \calv_{i_0}\right| \geq \inf_{\calv_i\neq\calv_j} |\calv_i - \calv_j| \, (N-1) \, \frac{r_0}{\varrho_{\max}} \, .
\end{align*}
Thus, we even can find constants $r_{\min} > \varrho_{\min}$ and $r_{\max} < \varrho_{\max}$ such that
\begin{align}\label{rminmax}
 r_{\min} \leq \varrho^{\infty}(x,t) \leq r_{\max} \quad \text{ for all } \quad (x,t) \in Q_{\bar{\tau}} \, .
\end{align}
Note that at each $(x,t) \in Q$, $$f^m(\rho) - f^m(\rho^{\infty}) - \hat{\mu}^m(\rho^{\infty}) (\rho-\rho^{\infty}) = \frac{1}{2} \, D^2f^m(\theta \, \rho + \rho^{\infty})\, (\rho-\rho^{\infty})\cdot(\rho-\rho^{\infty}) \, ,$$
with a $\theta \in [0,1]$.
Invoking the Lemma \ref{UNIFF1} of the appendix, there is $\lambda_1 > 0$ such that
 \begin{align*}
\sum_{i,j = 1}^N D^2_{ij}f^m(\rho) \xi_i \, \xi_j \geq
 \lambda_{1}\, |\xi|^2 \quad \text{ for all } \xi \in \mathbb{R}^N \text{ and all } \rho \, \,\,\text{s.t.\ } \, p_1 \leq \hat{p}^m(\rho) \leq p_2 \, .
\end{align*}
Under the assumptions of Theorem \ref{heuriheura} it follows that
\begin{align}\label{FEPOSDEF}
f^m(\rho) - f^m(\rho^{\infty}) - \hat{\mu}^m(\rho^{\infty}) (\rho-\rho^{\infty}) \geq \frac{\lambda_1}{2} \, |\rho-\rho^{\infty}|^2 \, .
\end{align}
For the proof of Theorem \ref{heuriheura}, we can rely on the boundedness of pressure. In view of $\partial_pg^m(\hat{p}^m(\rho)) \cdot \rho = 1$, we have
\begin{align*}%\label{rhop1p2}
\frac{1}{\max \partial_pg^m(\hat{p}^m(\rho))} \leq \varrho \leq \frac{1}{\min \partial_pg^m(\hat{p}^m(\rho))} \, , 
\end{align*}
and since we assume that $\{g^m\}$ converges in $C^2([p_1,p_2])$ this implies that
\begin{align}\label{rhocompact}
a_0:= \inf_{m \in \mathbb{N}, \, s\in [p_1,p_2]} \frac{1}{\max \partial_pg^m(s)} \leq \varrho \leq \sup_{m \in \mathbb{N}, \, s\in [p_1,p_2]} \frac{1}{\min \partial_pg^m(s)} =: b_0\, . 
\end{align}
We proceed with estimating the remainders in the relative energy inequality of Prop.\ \ref{calculheuri}.

\paragraph{\bf Estimate of $\mathcal{R}^{m,1}$}

By means of H\"older's inequality and \eqref{FEPOSDEF}
\begin{align*}
 \int_{\Omega} (\varrho-\varrho^{\infty}) \, \partial_t v^{\infty} \cdot (v^{\infty}-v) \, dx \leq & \sqrt{N}  \, |\partial_t v^{\infty}|_{L^{3}} \,|v^{\infty} - v|_{L^6(\Omega)} \, \, \left(\int_{\Omega} |\rho - \rho^{\infty}|^2 \, dx\right)^{\frac{1}{2}} \\
 \leq& \sqrt{\frac{2N}{\lambda_1}} \,  |\partial_t v^{\infty}|_{L^{3}} \,|v^{\infty} - v|_{L^6} \,   (\mathcal{E}^m(t))^{\frac{1}{2}} \, ,
\end{align*}
Next we have the Korn inequality and the Sobolev embedding theorem
\begin{align}\label{KKorn}
 \int_{\Omega} \mathbb{S}(\nabla(v^{\infty} - v)) \, : \, \nabla(v^{\infty} - v) \, dx \geq c \,\|v^{\infty} - v\|_{W^{1,2}_0(\Omega)}^2 \geq c^{\prime} \,\|v^{\infty} - v\|_{L^6(\Omega)}^2\, .
\end{align}
Thus, by means of the Young inequality, we can easily prove for $0<\epsilon < 1$ arbitrary that
\begin{align}\label{menschenae}
 \left|\int_{\Omega} (\varrho-\varrho^{\infty}) \, \partial_t v^{\infty} \cdot(v^{\infty} - v) \, dx\right| \leq & \epsilon \, \mathcal{D}^{\rm Visc}(v^m \, | \, v^{\infty})(t) + \frac{C^2}{\epsilon}\,  |\partial_t v^{\infty}|_{L^{3}}^2  \, \mathcal{E}^m(t) \, ,
\end{align}
where $C$ depends on $\Omega$ via $c^{\prime}$ and on $p_1,p_2$ via $\lambda_1$. With $(v \cdot \nabla) v^{\infty}$ in the place of $\partial_t v^{\infty}$, use of $|v \nabla v^{\infty}|_{L^3} \leq |v|_{L^6} \, |\nabla v^{\infty}|_{L^6} \leq c \, |\nabla v|_{L^2}|\nabla v^{\infty}|_{L^6}$ helps finding that
\begin{align}\label{menschenae2}\left|\int_{\Omega} (\varrho - \varrho^{\infty}) \, (v\cdot \nabla )v^{\infty} \cdot(v^{\infty} - v) \, dx\right| \leq &  \epsilon \,\mathcal{D}^{\rm Visc}(v^m \, | \, v^{\infty})(t) \nonumber\\
& + \frac{C^2}{\epsilon} \, |\nabla v^{\infty}|_{L^{6}}^2 \, |\nabla v|_{L^2}^2 \, \mathcal{E}^m(t) \, , 
\end{align}
Since $\varrho(x,t) \geq a_0$ owing to \eqref{rhocompact},
\begin{align}\label{menschenae3}
& \left| \int_{\Omega} \varrho^{\infty} \, \big((v^{\infty}-v) \cdot \nabla\big) v^{\infty} \cdot(v^{\infty} - v) \, dx\right| \leq  \frac{r_{\max}}{\sqrt{a_0} }\, \int_{\Omega} |\nabla v^{\infty}| \,|v^{\infty} - v| \, |\sqrt{\varrho} \,(v^{\infty} - v)|\, dx\nonumber\\
& \quad \leq \frac{r_{\max}}{\sqrt{a_0}} \, |\nabla v^{\infty}|_{L^{3}} \,|v^{\infty} - v|_{L^6} \, (2\,\mathcal{E}^m(t))^{\frac{1}{2}} \leq   \epsilon \, \mathcal{D}^{\rm Visc}(v^m \, | \, v^{\infty})(t) + \frac{C}{\epsilon} \,  |\nabla v^{\infty}|_{L^{3}}^2   \, \mathcal{E}^m(t) \, .
\end{align}
Overall, the first remainder obeys
\begin{align}\label{gouldenae1}
 |\mathcal{R}^{m,1}| \leq \epsilon \, \mathcal{D}^{\rm Visc}(v^m \, | \, v^{\infty})(t) + \frac{C}{ \epsilon} \, \psi^m(t) \, \mathcal{E}^m(t) \, ,
\end{align}
in which $\|\psi^m\|_{L^1(0,\bar{\tau})} \leq \| \partial_tv^{\infty}\|_{L^{3,2}(Q)}^2 + \|\nabla v^{\infty}\|_{L^{6,\infty}(Q)}^2 \, (\sup_m \|\nabla v^m\|_{L^{2}(Q)}^2 + \|\nabla v^{\infty}\|_{L^{2}(Q)}^2)$.

\paragraph{\bf Estimate of $\mathcal{R}^{m,2}$.}

The contribution 
$ \int_{\Omega}(v^{\infty} - v) \, (\rho -\rho^{\infty}) \, : \, (\nabla \mu^{\infty}+1^N \otimes b) \, dx$ to the remainder exhibits the same structure as the first integral of $\mathcal{R}^{m,1}$, where $\nabla \mu^{\infty}$ and $1^N \otimes b$ play the part of $1^N \otimes \partial_t v^{\infty}$. Hence, repeating the arguments used above in this case, we obtain that
\begin{align}\label{gouldenae2}
\left|\mathcal{R}^{m,2}\right| \leq  & \epsilon \,\mathcal{D}^{\rm Visc}(v^m \, : \, v^{\infty})(t) + \frac{C}{\epsilon} \,  (|\nabla \mu^{\infty}|_{L^3}^2+|b|_{L^3}^2) \, \mathcal{E}^m(t) \, ,
\end{align}
In order to estimate the remaining parts, we must discuss additional positivity questions for the densities and the eigenvalues of mobility tensor.

\paragraph{\bf Estimate of $\mathcal{R}^{m,3}$}

At this stage, to first show a way widely free of technicalities, we simply assume that all densities are uniformly positive according to \eqref{UNIFORMPOS}, in which $s_0$ is independent on $m$. We show later how to remove this assumption.
Then, in view of (B3), the mobility matrix $M(\rho)$ possesses $N-1$ strictly positive eigenvalues, and the remaining zero eigenvalue is associated with the vector $1^N$. We call $\lambda_{\min}(\cdot)$ ($\lambda_{\max}(\cdot)$) the smallest strictly positive eigenvalue (the largest eigenvalue) of $M(\cdot)$. Then, using also \eqref{b1prime}, 
\begin{align*}%\label{MGOOD}
 \lambda_{\max}(\rho) \leq \bar{\lambda} \, (1+|\rho|) \, , \qquad 
 \lambda_{\min}(\rho) \geq \inf_{\min r \geq s_0} \lambda_0(r) = \bar{\lambda}_0 > 0 \, ,
\end{align*}
with $\lambda_0$ from (B3) and a positive constant $\bar{\lambda}$. 
If the densities are positive and \eqref{UNIFORMPOS} hold, the chemical potentials $\mu_i = g_{i}(\hat{p}^m(\rho)) + RT/M_i \, \ln \hat{x}_i(\rho)$ are defined everywhere in the domain. We denote $M^{\dagger}(\rho)$ the Moore-Penrose pseudo--inverse of $M(\rho)$. Since $M$ is symmetric, we have $M^{\dagger}M = \mathcal{P} = MM^{\dagger}$. Then, using the Cauchy-Schwarz inequality and the Young inequality
\begin{align*}
& \quad | (M(\rho) - M(\rho^{\infty})) \, \nabla \mu^{\infty} \cdot \nabla (\mu-\mu^{\infty})| = | M^{\dagger}(M(\rho) - M(\rho^{\infty})) \, \nabla \mu^{\infty} \cdot M(\rho)\nabla (\mu-\mu^{\infty})| \\
& \leq \Big(M(\rho) \,M^{\dagger}(M(\rho) - M(\rho^{\infty})) \, \nabla \mu^{\infty} \cdot M^{\dagger}(M(\rho) - M(\rho^{\infty})) \, \nabla \mu^{\infty}\Big)^{\frac{1}{2}}\\
& \qquad \times (M(\rho) \, \nabla (\mu-\mu^{\infty}) \cdot \nabla (\mu-\mu^{\infty}))^{\frac{1}{2}}\\
& \leq \|M^{\dagger}(\rho)\|^{\frac{1}{2}} \, 
\|M(\rho) - M(\rho^{\infty})\|_{\infty}  \, |\nabla \mathcal{P}\mu^{\infty}| \, (M(\rho) \, \nabla (\mu-\mu^{\infty}) \cdot \nabla (\mu-\mu^{\infty}))^{\frac{1}{2}}\\
& \leq \frac{1}{4 \, \epsilon \lambda_{\min}(\rho)}  \, \| M(\rho) - M(\rho^{\infty})\|^2_{\infty} \, |\nabla \mathcal{P}\mu^{\infty}|^2 + \epsilon \, M(\rho) \, \nabla (\mu-\mu^{\infty}) \cdot \nabla (\mu-\mu^{\infty}) \, .
\end{align*}
Since $\rho \mapsto M(\rho)$ is Lipschitz continuous on $\mathbb{R}^N_+$ according to (B1), calling the Lipschitz constant $\|\partial M\|_{L^{\infty}}$ we have $\| M(\rho) - M(\rho^{\infty})\| \leq \|\partial M\|_{L^{\infty}} \, |\rho - \rho^{\infty}|$. Therefore, using the equivalence of all norms on $\mathbb{R}^N$,
\begin{align}\label{menschenae4}
 & | (M(\rho) - M(\rho^{\infty})) \, \nabla \mu^{\infty} \cdot \nabla (\mu-\mu^{\infty})|\nonumber\\
&  \leq  \frac{c_N \, \|\partial M\|^2_{\infty}}{4\,\epsilon \, \bar{\lambda}_0}  \, |\rho - \rho^{\infty}|^2_{\infty} \, |\nabla\mathcal{P} \mu^{\infty}|^2 + \epsilon \, M(\rho) \, \nabla (\mu-\mu^{\infty}) \cdot \nabla (\mu-\mu^{\infty}) \, .
\end{align}
Overall, the estimation of $\mathcal{R}^{m,3}$ yields
\begin{align}\label{gouldenae3}
 & \int_{\Omega} |(M(\rho) - M(\rho^{\infty})) \, \nabla \mu^{\infty} \cdot \nabla (\mu - \mu^{\infty})| \, dx \nonumber\\
 & \leq  \epsilon \, \int_{\Omega} M(\rho) \, \nabla (\mu - \mu^{\infty}) \cdot \nabla (\mu - \mu^{\infty}) \, dx + \frac{C}{\epsilon} \, |\nabla \mathcal{P} \mu^{\infty}|_{L^{\infty}}^2 \, \mathcal{E}^m(t)\, ,
\end{align}
where we remark that $C$ essentially depends on the constant $s_0$ of \eqref{UNIFORMPOS} and the constants $p_1,p_2$ of Theorem \ref{heuriheura}.

\paragraph{\bf Estimate of $\mathcal{R}^{m,4}$} 

Recall that $k(\rho)$ is given by \eqref{KFUKK}.
Thus $Dk(\rho) = (RT/M) \, \ln \hat{x}(\rho)$ make sense only if all densities are strictly positive. Here we show a straightforward estimate relying again on \eqref{UNIFORMPOS}. Later we will provide an alternative to treat this remainder for weak solutions in general. For positive densities, we easily find that
\begin{align*}
& | Dk(\rho) - Dk(\rho^{\infty}) - D^2k(\rho^{\infty}) (\rho -\rho^{\infty})| \leq \sup_{\theta \in ]0,1[} |D^2k(\theta \, \rho + (1-\theta) \, \rho^{\infty})|_{\infty} \, |\rho-\rho^{\infty}|^2 \, .
\end{align*}
We recall that $D^2_{ij}k(r) = (RT/M_iM_j \, \hat{n}(r)) \, (\delta^i_j/\hat{x}_i(r) -1)$, allowing to show that
\begin{align*}
& |D^2_{ij}k(\theta \, \rho + (1-\theta) \, \rho^{\infty})| \leq  \frac{RT \, \max M}{(\min M)^2 \, \min \{\rho_i,\rho^{\infty}_i\}} \leq \frac{RT\max M}{(\min M)^2\min \{s_0,r_0\}}\, ,
\end{align*}
where we assume $\min \rho \geq s_0$ and $\min \rho^{\infty} \geq r_0$. 
It follows that
\begin{align}\label{gouldenae4}
 & |\mathcal{R}^{m,4}| \leq c\, |\partial_t \rho^{\infty} + \divv(\rho^{\infty} \, v^{\infty})|_{L^{\infty}} \, \mathcal{E}^m(t) \, . 
\end{align}

\subsection{The convergence argument}\label{convarg}

We collect the estimates \eqref{gouldenae1}, \eqref{gouldenae2}, \eqref{gouldenae3} and \eqref{gouldenae4}. Writing now again $\rho^m$, $v^m$ etc., for all $0 < t \leq \bar{\tau}$, they yield
\begin{align}\label{jelacite}
 & \mathcal{E}^m(t) \leq  \mathcal{E}^m(0) +  \int_{\Omega} p^{\infty}(x,\cdot) \, (\rho^m(x,\cdot) \cdot \calv - 1)\, dx\Big|^t_0 \nonumber\\
 & + \int_{0}^t \int_{\Omega} (\partial_t \rho^{\infty} + \divv (\rho^{\infty}v^{\infty})) \cdot (g^m(p^m) - \calv \, p^m) - (\rho^m \cdot \calv-1) \, (\partial_t p^{\infty} + v^{\infty} \cdot\nabla p^{\infty}) dxd\tau\nonumber\\
&  - (1-2\epsilon)\, \int_{0}^t \int_{\Omega} \mathbb{S}(\nabla (v-v^{\infty})) \, : \, \nabla (v-v^{\infty}) + M(\rho) \nabla (\mu - \mu^{\infty}) : \, \nabla (\mu - \mu^{\infty}) \, dx\nonumber\\
 & + \frac{C}{\epsilon} \, \int_{0}^t \psi^m(\tau) \, \mathcal{E}^m(\tau) \, d\tau \, .
\end{align}
Here the function $\psi^m$ obeys
\begin{align}\label{jelaciteII}
  \|\psi^m\|_{L^1(0,\bar{\tau})} \leq & \| \partial_tv^{\infty}\|_{L^{3,2}}^2 + \|\nabla v^{\infty}\|_{L^{6,\infty}}^2 \, (\sup_m\|\nabla v^m\|_{L^{2}}^2 + \|\nabla v^{\infty}\|_{L^{2}}^2) +  \|\nabla \mu^{\infty}\|_{L^{3,2}}^2 \nonumber \\
&+  \|b\|_{L^{3,2}}^2+ \|\nabla \mathcal{P} \mu^{\infty}\|_{L^{\infty,2}}^2  +  \|\partial_t\rho^{\infty} +\divv(\rho^{\infty} \, v^{\infty})\|_{L^{\infty,1}} \, .
\end{align}
From the Gronwall Lemma it follows that $\mathcal{E}^m(t) \leq |A^m(t)| \, \exp(\sup_m \|\psi^m\|_{L^1(0,\bar{\tau})})$, in which
\begin{align*}
& \quad  A^m(t) := \mathcal{E}^m(0)+ \int_{\Omega} p^{\infty}(x,\cdot) \, (\rho^m(x,\cdot) \cdot \calv-1) \, dx\Big|^t_0 \\
& +  \int_{0}^t \int_{\Omega} (\partial_t \rho^{\infty} + \divv (\rho^{\infty}v^{\infty})) \cdot (g^m(p^m) - \calv \, p^m) - (\rho^m \cdot \calv-1) \, (\partial_t p^{\infty} + v^{\infty} \cdot\nabla p^{\infty}) dxd\tau\, .
\end{align*}
In order to prove that $A^m$ tends to zero, we first show that the initial relative energy $\mathcal{E}^m(0)$ tends to zero. Theorem \ref{heuriheura} assumes that $v^m(0) = v^0 = v^{\infty}(0)$. Moreover $\hat{p}^m(\rho^{0,\infty}) = p^0$ and $g^m(p^0) = 0$ imply that $f^m(\rho^{0,\infty}) = -p^0 + k(\rho^{0,\infty})$ and that $Df^m(\rho^{0,\infty}) = Dk(\rho^{0,\infty})$, hence
\begin{align}\label{InitialE}
 \mathcal{E}^m(0) =  &\int_{\Omega} f^m(\rho^{0,m}(x)) - f^m(\rho^{0,\infty}(x)) - \hat{\mu}^m(\rho^{0,\infty}(x)) \cdot (\rho^{0,m}(x) - \rho^{0,\infty}(x)) \, dx \nonumber\\ =  & \int_{\Omega} g^m( \hat{p}^m(\rho^{0,m})) \cdot \rho^{0,m} -   (\hat{p}^m(\rho^{0,m})-p^0) \, dx \\
 & + \int_{\Omega} k(\rho^{0,m}) - k(\rho^{0,\infty}) - Dk(\rho^{0,\infty}) \, (\rho^{0,m}-\rho^{0,\infty}) \, dx \, . \nonumber
\end{align}
Since $k$ is fixed and, by assumption, $\rho^{0,\infty}$ is uniformly positive, the second member in \eqref{InitialE}, right-hand side is easily shown to converge to zero, since we assume that $\rho^{0,m} \rightarrow \rho^{0,\infty}$ in $L^1(\Omega)$. As to the first member, we define $p^{0,m} := \hat{p}^m(\rho^{0,m})$ satisfying $\partial_{p}g^{m}(p^{0,m}) \cdot \rho^{0,m} = 1$. The assumption $g^m(p^0) = 0$ can be used to show that
\begin{align*}
 & g^m( \hat{p}^{0,m}) \cdot \rho^{0,m} -   (p^{0,m}-p^0) = (g^m(p^{0,m}) - g^m(p^0) - \partial_pg^{m}(p^{0,m}) \, (p^{0,m}-p^0)) \cdot \rho^{0,m} \\
 & \quad =  \int_0^1 \partial_pg^{m}(p^0 + \lambda \, p^{0,m}) - \partial_pg^{m}(p^{0,m}) \,d\lambda \cdot\rho^{0,m} \, (p^{0,m}-p^0)  \, .
\end{align*}
Since in Theorem \ref{heuriheura}, we assume that $p_1 \leq p^{0,m} \leq p_2$, we easily show that
\begin{align*}
\int_{\Omega} g^m(p^{0,m}) \cdot \rho^{0,m} -   (p^{0,m}-p^0) \, dx\leq 2 \, \sup_m |\varrho^{0,m}|_{L^{1}}\,(p_2-p_1) \, \sup_{s \in [p_1,p_2]} |\partial_pg^m(s)-\calv| \, ,
\end{align*}
which tends to zero since $\partial_pg^m \rightarrow \calv$ uniformly on $[p_1, \, p_2]$. Next, use of $\partial_pg_i^m(p^0) = \calv_i$ and $\sum_{i=1}^N \rho_i^m \, \partial_pg^m_{i}(p^m) = 1$ yields
\begin{align*}
 1 - \calv \cdot \rho^m = \sum_{i=1}^N (\partial_pg_i^{m}(p^m) - \calv_i) \, \rho^m_i \quad \text{ and } \quad
 | 1 - \calv \cdot \rho^m | \leq |\partial_pg^{m}(p^m) - \calv|_{\infty} \, \varrho^m \, .
\end{align*}
Since $p_1 \leq p^m(x,t) \leq p_2$, we have
% \begin{align*}%\label{zusatz2}
%  p_1 \leq \inf_{(x,t) \in Q, \,  m \in \mathbb{N}} p^m(x,t), \quad \sup_{(x,t) \in Q, \, m\in \mathbb{N}} p^m(x,t) \leq p_2 \, .
% \end{align*}
$| 1 - \calv \cdot \rho^m | \leq \sup_{s\in[p_1,p_2]} |\partial_pg^m(s) - \calv| \, b_0$ with $b_0$ from \eqref{rhocompact}. Thus $\calv \cdot \rho^m \longrightarrow 1$ strongly in $L^{\infty}(Q_{\bar{\tau}})$ which implies that
\begin{align*}
\sup_{0<t<\bar{\tau}} \int_{\Omega} |p^{\infty}(x,t)| \, |\rho^m(x,t)\cdot \calv-1|\, dx \leq \|p^{\infty}\|_{L^{1}(Q_{\bar{\tau}})} \, \|\rho^m\cdot \calv-1\|_{L^{\infty}} \, ,\\
\int_{0}^{\bar{\tau}}\int_{\Omega} |\rho^m \cdot \calv-1| \, |\partial_t p^{\infty}  + v^{\infty} \cdot \nabla p^{\infty}| \, dxd\tau \leq \|\partial_tp^{\infty} + v^{\infty}\cdot\nabla p^{\infty}\|_{L^{1}} \, \|\rho^m \cdot \calv-1\|_{L^{\infty}} \, ,
\end{align*}
and both terms tend to zero. Similarly, the assumptions of Theorem \ref{heuriheura} allow to show that 
\begin{align}\label{htozero}
 |g^m(p^m) - \calv \, p^m| \leq \sup_{p_1 \leq s \leq p_2} |g^m(s) - \calv \, s| \rightarrow 0 \, 
\end{align}
and thus also $\int_{0}^t \int_{\Omega} |\partial_t\rho^{\infty} + \divv (\rho^{\infty}v^{\infty})| \, |g^m(p^m) - \calv \, p^m| \, dxd\tau$ tends to zero.

Overall $\limsup_{m\rightarrow \infty} \sup_{0<t<\bar{\tau}} |A^m(t)| = 0$, hence $\limsup_{m\rightarrow \infty} \sup_{0<t<\bar{\tau}} \mathcal{E}^m(t) = 0$. Using \eqref{jelacite} with $t = \bar{\tau}$ also
\begin{align}\label{SandXitozero}
\limsup_{m\rightarrow \infty}  \int_0^{\bar{\tau}} \mathcal{D}^{\rm Visc}(v^m\,|\,v^{\infty})(\tau) \, d\tau + \int_{Q_{\bar{\tau}}}M(\rho^m) \nabla (\mu^m - \mu^{\infty}) : \, \nabla (\mu^m - \mu^{\infty}) \, dxd\tau = 0 \, .
\end{align}
Let us verify the convergence claims in
Theorem \ref{heuriheura}. Since $\sup_t \mathcal{E}^m(t) \rightarrow 0$, we easily show that $\rho^m \rightarrow \rho^{\infty}$ in $L^{1,\infty}(Q_{\bar{\tau}})$, while \eqref{SandXitozero} imply that $v^m \rightarrow v^{\infty}$ in $L^2W^{1,2}$. We next turn proving a convergence result for $\{p^m\}$. Under the assumption \eqref{UNIFORMPOS}, we can introduce 
\begin{align}\label{MuM}
\mu^m =& \frac{RT}{M} \, \ln \hat{x}(\rho^m) + g^m(p^m) = \frac{RT}{M} \, \ln \hat{x}(\rho^m) + h^m +\calv \, p^m \, ,
 \end{align}
with $h^m := g^m(p^m) - \calv \, p^m$, which we have shown in \eqref{htozero} to converge to zero in $L^{\infty}(Q)$.
Let $\eta\in \mathbb{R}^N$ be the vector of \eqref{CONVEN} satisfying $\eta \cdot \calv = 1$ and, if $\calv \not\parallel 1^N$, also $\eta \cdot 1^N = 0$. Multiplication in \eqref{MuM} with $\eta$ yields $ \eta\cdot\mu^m = \eta \cdot ((RT/M) \, \ln \hat{x}(\rho^m) + h^m) + p^m $.
We subtract the mean-value $(\eta\cdot\mu^m)_M = \int_{\Omega} (\eta\cdot\mu^m)_M \, dx/|\Omega|$ on both sides of this identity, we define $\bar{\zeta}^m :=  -(\eta\cdot\mu^m)_M$, $p^m_* :=  p^m +\bar{\zeta}^m$, and $\zeta^m := \eta\cdot\mu^m - (\eta\cdot\mu^m)_M$, and we get
\begin{align}\label{grouldei}
p^m_* = \zeta^m - \eta \cdot \Big(\frac{RT}{M} \, \ln \hat{x}(\rho^m) + h^m\Big) \, .
\end{align}
Note that $\{\zeta^m\}$ possesses zero mean-value over $\Omega$ for all $t$. Since $\eta$ is perpendicular to $1^N$, we can estimate $|\nabla (\eta \cdot (\mu-\mu^{\infty}))| \leq |\eta| \, |\nabla \mathcal{P}(\mu-\mu^{\infty})|$. Employing \eqref{SandXitozero}, we show that $\nabla \zeta^m \rightarrow \eta\cdot\nabla \mu^{\infty}$ in $L^2(Q)$. Hence 
$\zeta^m \rightarrow \eta \cdot \mu^{\infty} - (\eta\cdot \mu^{\infty})_M$ in $L^2(0,\bar{\tau}; \, W^{1,2}(\Omega))$, and \eqref{grouldei} shows that
\begin{align*}%\label{tildepconvpart}
 p^m_* \longrightarrow \eta\cdot\mu^{\infty} - \eta\cdot \frac{RT}{M} \, \ln \hat{x}(\rho^{\infty}) - (\eta\cdot\mu^{\infty})_M= p^{\infty} - (\eta\cdot\mu^{\infty})_M  \quad \text{ in } L^2(Q) \, .
\end{align*}
Since $(\eta \cdot \mu^{\infty})_M = 0$ due to (C), the latter shows that $p^m_* \rightarrow p^{\infty}$ in $L^2(Q)$.

The proof of Theorem \ref{heuriheura} is now complete under the additional uniform positivity assumption \eqref{UNIFORMPOS}. At some technical cost, this condition can be removed as explained in the Appendix, Section \ref{SStorm}.\\

In the second part of the paper, we investigate the convergence of weak solutions under relaxing  the \emph{a priori} condition \eqref{TurevesII} on pressure control. 

\section{Rigorous convergence for rescaled weak solutions}\label{RIGOR}

In this section we prove the convergence to the incompressible limit for certain rescaled weak solutions to the problems ($\overline{\text{ IBVP}}^m$). At the difference of the previous section, it is possible to show that these solutions exist globally in time.
The main results available on the local- and global-in-time resolvability of the problems (IBVP), (IBVP$^{\infty}$) are recalled in the appendix, Section \ref{ShortSurvey}. Readers might be interested to consult this short survey before starting with the present paragraph which is more technical than the previous one.

Here we begin by showing, in the section \ref{PressureUniform}, that if we reinforce the assumptions on the mobility tensor to (B3$^\prime$), the sequence of ''pressures'' $\{\bar{p}^m_{\Delta}\}$ associated with the rescaled weak solutions satisfy a bound in $L^1$ independently on $m$. Then, a modified form of the relative energy inequality is shown to be valid in Section \ref{RelIntIneqWeak}. It allows to prove the convergence of weak solutions under the assumptions \eqref{CONVER2} and (C).

\subsection{Technical estimates for the pressure-field}\label{PressureUniform}

We commence by stating the existence of weak solutions together with some natural bounds.

\subsubsection{Existence of weak solutions}\label{ENTROPICVAR}

We let $\xi^1,\ldots,\xi^{N-1}, \xi^N$ with $\xi^N := (1,1, \ldots,1) = 1^N$ be a basis of $\mathbb{R}^N$, and let $\eta^1,\ldots, \eta^{N}$ denote the dual basis. Then, the vectors $\eta^1, \ldots,\eta^{N-1}$ constitute a basis of the orthogonal complement of the vector $1^N$ in $\mathbb{R}^N$. A rectangular ''projection'' matrix $\Pi \in \mathbb{R}^{N \times N-1}$ with the vectors $\xi^1, \ldots, \xi^{N-1}$ as columns is further introduced via
\begin{align}\label{Pi}
\Pi := [\xi^1,\ldots,\xi^{N-1}] \, .
\end{align}
One introduces the new variables (entropic variables, diffusive variables)
\begin{align}\label{qplouc}
q_i := \sum_{j=1}^N \eta^i_j \, \mu_j \quad \text{ for } \quad i = 1,\ldots,N-1 \, ,
\end{align}
and reformulate the problem (IBVP) using the change of variables
\begin{align}\label{MV1}
(\rho_1, \ldots,\rho_N) \quad \longleftrightarrow \quad (\varrho, \, q_1, \ldots,q_{N-1}) \, .
\end{align}
The proof that there is a bijection between these variables uses the definition of the chemical potentials \eqref{CHEMPOT}. If the free energy function $\rho \mapsto f(\rho)$ is a so-called \emph{co-finite function of Legendre type} on $\mathbb{R}^N_+$, which is the key assumption in the investigations \cite{bothedruet}, \cite{dredrugagu20}, the equations \eqref{CHEMPOT} can be inverted by means of the convex conjugate function $$f^*(w) := \sup_{r \in \mathbb{R}^N_+} \{w \cdot r - f(r)\}  \quad \text{ for } \quad w \in \mathbb{R}^N\, ,  $$ and \eqref{CHEMPOT} is equivalent with
 \begin{align}\label{inverted}
  \rho_i = \partial_{w_i} f^*(\mu_1, \ldots,\mu_N) \quad \text{ for } \quad i = 1,\ldots,N \, . 
 \end{align}
 Using the definition \eqref{qplouc} of $q$, we next express the vector $\mu$ in new coordinates via $$\mu = \sum_{i=1}^{N-1} q_i \, \xi^i + \mu \cdot \eta^N \, 1^N = \Pi \,q +\mu\cdot\eta^N \, 1^N \, .$$ Summing up over $i = 1,\ldots,N$ in \eqref{inverted} implies that $ \varrho = 1^N \cdot \partial_{w} f^*( \Pi \, q + \mu \cdot \eta^N \, 1^N)$. This defines implicitly $\mu \cdot \eta^N = \mathscr{M}(\varrho, \, q_1, \ldots, q_{N-1})$
 %\begin{align*}%\label{scriptM}
 %\mu \cdot \eta^N = \mathscr{M}(\varrho, \, q_1, \ldots, q_{N-1}) \, ,
 %\end{align*}
 with a function $\mathscr{M}$ of $\varrho$ and $q$. Now, all thermodynamic quantities can be introduced as functions of the variables $\varrho$ and $q_1, \ldots, q_{N-1}$, among others
\begin{align}\label{newpress}
 p = f^*(\mu) = f^*\Big(\Pi q + \mathscr{M}(\varrho,q) \, 1^N \Big) =: P(\varrho, \, q) \quad \text{ pressure, }\\
 \label{newpartdens} \rho = \partial_w f^*\Big(\Pi q + \mathscr{M}(\varrho, \,q) \, 1^N \Big) =: \mathscr{R}(\varrho, \, q) \quad \text{ densities.}
\end{align}
For the ideal model \eqref{muideal}, we have the equivalent representation:
\begin{align}\begin{split}\label{lesidentes}
 p =& \hat{p}(\rho_1, \ldots, \rho_N) = P(\varrho, \, q) \, , \\
 \mu = g(p) + \frac{RT}{M} \, \ln x =& g(\hat{p}(\rho)) + \frac{RT}{M} \, \ln \hat{x}(\rho) =  \Pi q + \mathscr{M}(\varrho, q) \, 1^N \, ,
 \end{split}
\end{align}
where $\hat{p}$ is the function introduced in \eqref{EOS} and $\hat{x}(\rho)$ is defined in \eqref{LESFRACS}.\\

Let us next consider rescaled problems ($\overline{\text{ IBVP}}^m$). 
In this case, the change of variables is performed using the free energy function $\bar{f}^m$ of \eqref{FEm}, with the data $\bar{g}_1^m, \ldots, \bar{g}_N^m$ from \eqref{ginormalised}.

We switch from the main variables to the entropic variables \eqref{MV1} as just explained. After this procedure, the rescaled pressure $\bar{p}_{\Delta}$ associated with a solution $(\bar{\varrho}, \, \bar{q})$ to ($\overline{\text{ IBVP}}^m$) obeys a representation
\begin{align}\label{RescaledbigP}
 \bar{p}_{\Delta} = P_m(\bar{\varrho}, \, \bar{q}_1, \ldots,\bar{q}_{N-1}) = \hat{\pi}^m(\bar{\rho}) \, .
\end{align}
We recover the full vector of chemical potentials as $\bar{\mu} = \Pi \bar{q} + \bar{\mathscr{M}}^m(\bar{\varrho}, \, \bar{q}) \, 1^N$, the mass densities as $\bar{\rho} = (\nabla_{\bar{\rho}} \bar{f}^m)^{-1}(\bar{\mu}) = \nabla_{\bar{\mu}} (\bar{f}^m)^*(\bar{\mu})$ and, by construction, it also follows that
\begin{align}\label{rescaledpot}
 \bar{\mu}_i = \hat{\mu}^{m}_i\big(\bar{p}_{\Delta}, \, \bar{x}\big) = \bar{g}_i^m(\bar{p}_{\Delta}) + \frac{1}{\bar{M}_i} \, \ln \bar{x}_i \, .
\end{align}
Applying the existence result of Appendix, Prop.\ \ref{EXIRESC}, we see that the rescaled variables $(\bar{\varrho}^m, \, \bar{q}^m, \, \bar{v}^m)$ provide a weak solution to ($\overline{\text{ IBVP}}^m$). Under the assumption (B3$^{\prime}$), we also obtain the following uniform bounds:
\begin{align}\begin{split}\label{lesbounds}
 \sup_{m \in \mathbb{N}} \|\bar{\varrho}_m\|_{L^{\gamma,\infty}(Q^{\rm R})} < + \infty, \quad \sup_{m\in\mathbb{N}} \|\bar{\nabla} \bar{v}^m\|_{L^2(Q^{\rm R})} < +\infty \, , \\
  \sup_{m \in \mathbb{N}} \|\sqrt{\bar{\varrho}_m} \, \bar{v}^m\|_{L^{2,\infty}(Q^{\rm R})} < + \infty, \quad 
  \sup_{m \in \mathbb{N}} \|\bar{\nabla} \bar{q}^m\|_{L^2(Q^{\rm R})} < +\infty \, , \\
  \sup_{m \in \mathbb{N}} \sum_{i=1}^{N-2} \|\bar{q}^{m}_i\|_{L^{2}(Q^{\rm R})} < +\infty, \qquad   \sup_{m \in \mathbb{N}} \frac{1}{\sqrt{m}} \, \|\bar{q}^m\|_{L^{2}(Q^{\rm R})} < +\infty \, .
 \end{split}
\end{align}
Moreover, Lemma \ref{Bogovencorelui} shows that
\begin{align}\label{ilstendent1}
 \limsup_{m\rightarrow \infty} \|\rho^m \cdot \bar{\calv} - 1\|_{L^{1,\infty}(Q^{\rm R})} = 0 \, .
\end{align}
Now, the crucial task is to obtain bounds on the sequence of rescaled pressure variations $\{\bar{p}^m_{\Delta}\}$. Due to the fact that we rescale with the large parameter $m$, this bound has to be discussed independently.
In order to simplify the discussions from the viewpoint of notations, we shall from now drop the bars on the state variables and differential operators - while keeping it on the constitutive functions. Hence we write $\varrho$ and $q$ instead of $\bar{\varrho}$ and $\bar{q}$, etc.\ For similar reasons, we redefine $\Omega = \Omega^{\rm R}$ and $Q = Q^{\rm R}$, $\bar{\tau} = \bar{\tau}/t^{\rm R}$, etc.\
Whenever we assume that the functions $\bar{g}_1, \ldots, \bar{g}^N$ satisfy the assumptions (A), we use $p^0 = 1$ therein. We shall write $\pi^m$ instead of $p^m_{\Delta}$.

\subsubsection{A uniform property of $P_m$.}

Several properties of the change of variables \eqref{MV1} and the fields $P$ and $\mathscr{R}$ of \eqref{newpress}, \eqref{newpartdens} were investigated in \cite{dredrugagu20}, but the estimates established there are not independent of the parameter $m$ and cannot be used directly for the incompressible limit. It is our aim to next we prove some uniform estimates for the present context. 

Throughout this section we restrict to a special choice of $\{\xi^1, \ldots, \xi^N\}$ where, in addition to $\xi^N = 1^N$, we require that $\xi^{N-1} = \bar{\calv}$. This is possible if $\bar{\calv}$ is not parallel to $1^N$ which is the common case.
We let $\{\eta^1, \ldots, \eta^N\}$ be the dual basis. Then $\eta^{N-1} \cdot \bar{\calv} = 1$. As exhibited in \cite{druetmixtureincompweak}, \cite{bothedruetincompress} the component $q_{N-1}$ plays a very special part in the analysis of the incompressible model. For a  vector in $X \in \mathbb{R}^{N-1}$, we hence shall employ the notation $X =(X^{\prime}, \, X_{N-1})$ with $X^{\prime} = (X_1,\ldots,X_{N-2})$.

For fixed functions $\bar{g}_1, \ldots, \bar{g}_N$ satisfying (A1)-(A3), recall that \eqref{EOSnormalised2}, \eqref{piiswidehatp} allows to introduce a function $\widehat{p}$ such that $\sum_{i=1}^N\bar{g}_i^{\prime}(\widehat{p}(\rho)) \, \rho_i = 1$. This implies, in particular, that
\begin{align*}
 \min \bar{g}^{\prime}(\widehat{p}(\rho)) \leq \frac{1}{\varrho} \leq \max \bar{g}^{\prime}(\widehat{p}(\rho)) \quad \text{ and } \quad [\max \bar{g}^{\prime}]^{-1}(1/\varrho) \leq \widehat{p}(\rho) \leq  [\min \bar{g}^{\prime}]^{-1}(1/\varrho) \, .
\end{align*}
For states such that $0 < a \leq \varrho \leq b < +\infty$, we therefore have
\begin{align}\label{gabiskrank3}
\widehat{p}_0(1/b):= [\max \bar{g}^{\prime}]^{-1}(1/b) \leq \widehat{p}(\rho) \leq  [\min \bar{g}^{\prime}]^{-1}(1/a) =: \widehat{p}_1(1/a) \, .
\end{align}
\begin{lemma}\label{pressandg}
Assume that $ \bar{g}$ satisfies ${\rm (A1)-(A3)}$ and ${\rm (A^{\prime})}$, and define $\bar{g}^m \in C^2(]-m, \, +\infty[)$ via \eqref{ginormalised}. 
Then the function $P_m$ of \eqref{RescaledbigP} belong to $C^1(]0,+\infty[ \times  \mathbb{R}^{N-1})$ and it is strictly increasing in the first variable. In addition, there are constants $C_1,C_2 >0$ independent on $m$ such that
\begin{align*}
|P_m(\varrho, \, q)| \leq & C_1  \, \Big(\ln \frac{1}{\min\{\bar{\varrho}_{\max} -\varrho, \, \varrho - \bar{\varrho}_{\min}  \}} +|q|\Big) \, ,\\
|P_m(\varrho, \, q) - q_{N-1}| \leq &  C_2 \, \Big(\Big[\ln \frac{1}{\min\{\bar{\varrho}_{\max} -\varrho, \, \varrho - \bar{\varrho}_{\min}  \}}\Big]^2 +|q^{\prime}|+ \frac{|q|^2}{m}\Big) \, ,
\end{align*}
for all $q \in \mathbb{R}^{N-1}$ and for all $\varrho$ subject to $\bar{\varrho}_{\min} = 1/ \max \bar{\calv} < \varrho < 1/\min \bar{\calv} = \bar{\varrho}_{\max}$. 
\end{lemma}
\begin{proof}
 The fact that for all $q \in \mathbb{R}^{N-1}$, the function $\varrho \mapsto P_m(\varrho, \, q)$ is strictly increasing on $\mathbb{R}_+$, was proved in \cite{dredrugagu20}. We next prove in two steps the bound for $|P_m|$.

 Given $(\varrho, \, q) \in \mathbb{R}_+ \times \mathbb{R}^{N-1}$, we recall \eqref{lesidentes} and define $\mu := \Pi q + \bar{\mathscr{M}}^m(\varrho, \, q)$ and $\rho = (\nabla_{\bar{\rho}} \bar{f}^m)^{-1}(\mu)$. With $x = \hat{x}(\rho)$ and $\pi = P_m(\varrho, \, q) = \hat{\pi}^{m}(\rho)$, we then get $\mu = \bar{g}^m(\pi) + 1/\bar{M} \, \ln x$, according to \eqref{rescaledpot}.
%\begin{align}\label{rescaledpot+}
%\mu_i = \bar{g}_i^m(\pi) + \frac{1}{M_i} \, \ln x_i \, .
%\end{align}
Recall that $\sum_{i=1}^N\bar{g}_i^{\prime}(1+\pi/m) \, \rho_i = 1$ characterises $ 1+\pi/m = \widehat{p}(\rho)$, where $\widehat{p}$ is the pressure-function introduced in \eqref{piiswidehatp}. Therefore, \eqref{gabiskrank3} implies that $\widehat{p}_0(\min \bar{\calv}) \leq \widehat{p}(\rho) \leq  \widehat{p}_1(\max \bar{\calv})$ whenever $1/\max \bar{\calv} \leq |\rho|_1 \leq 1/\min \bar{\calv}$. Thus, under this restriction, we find that $\widehat{p}_0 \leq  1 + \pi/m \leq \widehat{p}_1$. We let $k_1$ ($k_N$) be an index such that $\bar{g}^{\prime}_{k_1}= \min  \bar{g}^{\prime}$ ($\bar{g}^{\prime}_{k_N} = \max  \bar{g}^{\prime}$) according to the assumption (A$^{\prime}$), allowing to define a constant
\begin{align}\label{D0}
 d_0 := \min\Big\{\inf_{i\neq k_1,  \, 1 \leq s \leq \widehat{p}_1} \bar{g}_{i}^{\prime}(s)-\bar{g}_{k_1}^{\prime}(s), \, \inf_{i\neq k_N, \, \widehat{p}_0 \leq s \leq 1} \bar{g}_{k_N}^{\prime}(s)-\bar{g}_{i}^{\prime}(s)\Big\} > 0\, .
\end{align}
{\bf Upper bound for $P_m$:} For all $i \neq k_1$ the definition of $\mu$ yields
$\bar{g}^m_i(\pi) - \bar{g}^m_{k_1}(\pi) = 1/M_{k_1} \, \ln x_{k_1} - 1/M_i \, \ln x_i + \mu \cdot (e^i - e^{k_1})$, implying that
\begin{align}\label{imple}
 \bar{g}^m_i(\pi) - \bar{g}^m_{k_1}(\pi) \leq - \frac{1}{M_i} \, \ln x_i + \sqrt{2} \, |\mathcal{P}\mu| \, .
\end{align}
Let $y_i = \hat{y}_i(\rho)$ (cf.\ \eqref{LESFRACS}). Since $\pi = \hat{\pi}^m(\rho)$, the definition \eqref{EOSnormalised} of the rescaled pressure yields
\begin{align*}
 \sum_{i=1}^N y_i \, (\partial_{\pi}\bar{g}^m_{i}(\pi) - \partial_{\pi}\bar{g}_{k_1}^{m}(\pi)) = \frac{1}{\varrho} - \partial_{\pi}\bar{g}_{k_1}^{m}(\pi) \, .
\end{align*}
In view of (A$^{\prime}$) and since $ \partial_{\pi} \bar{g}^m_i(\pi) = \bar{g}^{\prime}_i(1+\pi/m)$, we can rely on the fact that $\partial_{\pi}\bar{g}_{k_1}^{m}(\pi) = \min \partial_{\pi}\bar{g}^{m}(\pi)$. 
We fix an index $i_1$ such that $y_{i_1} = \max_{i\neq k_1} y_i$, the latter implies that
\begin{align*}
 \frac{1}{\varrho} - \partial_{\pi}\bar{g}_{k_1}^{m}(\pi) \leq (N-1) \, y_{i_1} \, (\partial_{\pi}\bar{g}_{k_N}^{m}(\pi) - \partial_{\pi}\bar{g}_{k_1}^{m}(\pi)) \, .
\end{align*}
Since $y_{i_1} \leq \frac{\max M}{\min M}  \, x_{i_1}$ we get
\begin{align*}
 x_{i_1} \geq \frac{\min M}{\max M\, (N-1) \, \varrho} \, \frac{1 -\varrho \, \partial_{\pi}\bar{g}_{k_1}^{m}(\pi)}{\partial_{\pi}\bar{g}_{k_N}^{m}(\pi) - \partial_\pi\bar{g}_{k_1}^{m}(\pi)} \, .
\end{align*}
We choose $i = i_1$ in \eqref{imple}, and we now obtain that
\begin{align}\label{QQ}
\bar{g}^m_{i_1}(\pi) - \bar{g}^m_{k_1}(\pi) \leq & - \frac{1}{M_{i_1}} \, \ln \Big(\frac{\min M}{\max M \, (N-1) \, \varrho}  \, \frac{1 - \varrho\, \partial_\pi \bar{g}_{k_1}^{m}(\pi)}{\partial_\pi \bar{g}_{k_N}^{m}(\pi) - \partial_\pi \bar{g}_{k_1}^{m}(\pi)} \Big) + \sqrt{2} \, |\mathcal{P} \mu| 
=: & S\, .
\end{align}
Since $\bar{g}_i^m(1) = 0$ for $i=1,\ldots,N$ due to \eqref{ginormalised},
\begin{align*}
 \bar{g}^m_{i}(\pi) =& \bar{g}^m_i(\pi) - \bar{g}^m_i(0) = \int_{0}^1 \partial_{\pi}\bar{g}_i^m(\theta \, \pi) \, d\theta \, \pi = \int_{0}^1 \bar{g}_i^{\prime}\Big(1+\frac{\theta}{m} \, \pi \Big) \, d\theta \, \pi \, ,
\end{align*}
and it follows that $\bar{g}^m_{i_1}(\pi) - \bar{g}^m_{k_1}(\pi) = \int_{0}^1 (\bar{g}_{i_1}^{\prime}-\bar{g}_{k_1}^{\prime})(1+\theta\, \pi/m) \, d\theta \, \pi$.
%\end{align*}
For $\pi \geq 0$ (equivalent to $\widehat{p} \geq 1$) we see that $\bar{g}^m_{i_1}(\pi) - \bar{g}^m_{k_1}(\pi) \geq d_0  \, \pi$ with the constant of \eqref{D0}. We combine the latter with \eqref{QQ} and, using also Young's inequality, we obtain that $\pi \leq  S/d_0$. 

Now let us estimate the quantity $S$ of \eqref{QQ}.
If we assume that $1/\max \bar{\calv} < \varrho < 1/\min \bar{\calv}$, then with
\begin{align*}
 \epsilon_0 := \frac{1}{\varrho} \, \frac{1-\varrho \, \min \bar{\calv}}{\max \bar{\calv} - \min \bar{\calv}}\, ,
\end{align*}
we have $1/\varrho = \min  \bar{\calv} \, (1-\epsilon_0) + \max \bar{\calv} \, \epsilon_0$. If $\pi \geq 0$, then $ \bar{g}_k^{\prime}(1+ \pi/m) \leq \bar{g}_k^{\prime}(1) = \bar{\calv}_k$, thus
\begin{align*}
\varrho \leq \frac{1}{\partial_\pi \bar{g}_{k_1}^{m}(\pi) \, (1-\epsilon_0) + \partial_\pi \bar{g}_{k_N}^{m}(\pi) \, \epsilon_0}  \,\quad 
\text{ yielding also } \quad \varrho \, \frac{\partial_\pi \bar{g}_{k_N}^{m}(\pi) - \partial_\pi \bar{g}_{k_1}^{m}(\pi)}{1-\varrho \, \partial_\pi \bar{g}_{k_1}^{m}(\pi)} \leq \frac{1}{\epsilon_0} \, .
\end{align*}
This allows to bound
\begin{align*}
S \leq  \frac{1}{\min M} \, \ln\Big( \frac{ \max M\, (N-1)}{\min M \, \epsilon_0} \Big) 
 + \sqrt{2} \, |\mathcal{P} \mu| \, .
\end{align*}

{\bf Lower bound for $P_m$:}
For the lower bound we argue similarly. In the same way as \eqref{imple}, we prove that $\bar{g}^m_i(\pi) - \bar{g}^m_{k_N}(\pi) \leq - 1/M_i \, \ln x_i + \sqrt{2} \, |\mathcal{P}\mu|$.
%\begin{align}\label{imple2}
% \bar{g}^m_i(\pi) - \bar{g}^m_{k_N}(\pi) \leq - \frac{1}{M_i} \, \ln x_i + \sqrt{2} \, |\mathcal{P}\mu| \, .
%\end{align}
Again, the definition of the pressure implies that
\begin{align*}
 \sum_{i\neq k_N} y_i \, (\partial_{\pi} \bar{g}_{i}^{m}(\pi)-\partial_\pi \bar{g}_{k_N}^{m}(\pi)) = \frac{1}{\varrho} - \partial_\pi \bar{g}_{k_N}^{m}(\pi) \, ,
\end{align*}
and, adapting the definition of $i_1$ so that $y_{i_1} = \max_{i\neq k_N} y_i$ we obtain that
\begin{align*}
 y_{i_1} \geq \frac{1}{(N-1) \, \varrho} \, \frac{\partial_\pi \bar{g}_{k_N}^{m}\, \varrho-1}{\partial_\pi \bar{g}_{k_N}^{m} - \partial_\pi \bar{g}_{k_1}^{m}} \, . 
\end{align*}
It follows that
\begin{align*}
\bar{g}^m_{i_1}(\pi) - \bar{g}^m_{k_N}(\pi) \leq - \frac{1}{M_{i_1}} \, \ln \Big(\frac{\min M}{\max M \, (N-1) \, \varrho}  \, \frac{\varrho\, \partial_\pi \bar{g}_{k_N}^{m}(\pi) - 1}{\partial_\pi \bar{g}_{k_N}^{m}(\pi) - \partial_\pi \bar{g}_{k_1}^{m}(\pi)} \Big) + \sqrt{2} \, |\mathcal{P} \mu| \, .
\end{align*}
Now if $\pi \leq 0$ ($\widehat{p} \leq 1$), then 
\begin{align*}
 \bar{g}^m_{i_1}(\pi) - \bar{g}^m_{k_N}(\pi) =&  \int_{0}^1 (\bar{g}_{i_1}^{\prime}-g_{k_N}^{\prime})\Big(1+\frac{\theta}{m} \, \pi \Big) \, d\theta \, \pi
 \geq d_0 \, |\pi| \, .
\end{align*}
We finish as for the upper bound.
% We choose $\epsilon_0$ such that $\varrho = 1/ (\max \bar{\calv} \, (1-\epsilon_0) + \min \bar{\calv}\, \epsilon_0)$ and, in particular, we have $\varrho \geq 1/ (\partial_\pi \bar{g}_{k_N}^{m}(\pi) \, (1-\epsilon_0) + \partial_\pi \bar{g}_{k_1}^{m}(\pi) \, \epsilon_0)$ because $\partial_\pi \bar{g}_k^{m}(\pi) \geq \bar{\calv}_k$ if $\pi \leq 0$. Thus
% \begin{align*}
% \frac{1}{\varrho}  \, \frac{\varrho \, \partial_\pi \bar{g}_{k_N}^{m}(\pi)-1}{\partial_\pi \bar{g}_{k_N}^{m}(\pi) - \partial_\pi \bar{g}_{k_1}^{m}(\pi)} \geq \epsilon_0 \, ,
% \end{align*}
% and it follows that 
% \begin{align*}
%  \pi \geq - \frac{1}{d_0 \, \min M} \, \ln\Big( \frac{ \max M\, (N-1)}{\min M \, \epsilon_0} \Big) 
%  - d_0^{-1} \, \sqrt{2} \, |\mathcal{P} \mu| \, .
% \end{align*}
Next we discuss the combination $P_m(\varrho,q) - q_{N-1} =: \tilde{P}_m(\varrho,q)$. 

{\bf Bounds for $\tilde{P}_m$:}
Starting from \eqref{rescaledpot}, we also get
 \begin{align}\begin{split}\label{rescaledpot1}
  \mu_i - \bar{\calv}_i \, q_{N-1} = & \frac{1}{M_i} \, \ln x_i  +\bar{g}_i^m(\pi) - \bar{\calv}_i \, \pi + \bar{\calv}_i \, \tilde{\pi} \, .
\end{split}
 \end{align}
For all $i \neq k_1$, \eqref{rescaledpot1} yields
\begin{align*}
(\bar{\calv}_i - \bar{\calv}_{k_1}) \, \tilde{\pi} = & \bar{g}^m_{k_1}(\pi) - \bar{\calv}_{k_1} \, \pi - (\bar{g}^m_{i}(\pi) - \bar{\calv}_{i} \, \pi)  + 1/M_{k_1} \, \ln x_{k_1} - 1/M_i \, \ln x_i \\
& + (\mu - \bar{\calv} \, q_{N-1}) \cdot (e^i - e^{k_1}) \, . 
\end{align*}
Note that $\ln x_{k_1} \leq 0$ and that, due to the concavity of $\bar{g}$,
\begin{align}\label{gabiskrank}
\bar{g}^m(\pi) - \bar{\calv} \, \pi = m \, \Big(\bar{g}\Big(1+\frac{\pi}{m}\Big) - \bar{g}(1) - \bar{g}^{\prime}(1) \, \frac{\pi}{m}\Big)   \leq 0 \, .
\end{align}
Since $(\mu - \bar{\calv} \, q_{N-1}) \cdot (e^i - e^{k_1}) = \Pi^{\prime}q \cdot (e^i - e^{k_1})$, with $\Pi^{\prime} q = \sum_{k=1}^{N-2} q_k \, \xi^k = \Pi (q^{\prime},0)$,
\begin{align}\label{impleprime}
(\bar{\calv}_i - \bar{\calv}_{k_1}) \, \tilde{\pi} \leq - (\bar{g}^m_{i}(\pi) - \bar{\calv}_{i} \, \pi) - \frac{1}{M_i} \, \ln x_i + \sqrt{2} \, |\Pi^{\prime} q| \, .
\end{align}
We choose $i = i_1$ in \eqref{impleprime}, and we now obtain that
\begin{align}\label{QQprime}
(\bar{\calv}_{i_1} - \bar{\calv}_{k_1}) \, \tilde{\pi} \leq & - (\bar{g}^m_{i_1}(\pi) - \bar{\calv}_{i_1} \, \pi) - \frac{1}{M_{i_1}} \, \ln \Big(\frac{\min M}{\max M \, (N-1) \, \varrho}  \, \frac{1 - \varrho\, \partial_\pi \bar{g}_{k_1}^{m}(\pi)}{\partial_\pi \bar{g}_{k_N}^{m}(\pi) - \partial_\pi \bar{g}_{k_1}^{m}(\pi)} \Big)\nonumber \\
& + \sqrt{2} \, |\Pi^{\prime} q|  \, .
\end{align}
The term with the logarithm on the right-hand side of \eqref{QQprime} can be estimate as in the case of \eqref{QQ}. Recalling the identity \eqref{gabiskrank} we also have
\begin{align}\label{gabiskrank2}
 \bar{g}^m(\pi) - \bar{\calv} \, \pi = \int_0^1\int_0^{\theta} \bar{g}^{\prime\prime}\Big(1+\lambda \, \frac{\pi}{m}\Big) \, d\lambda d\theta \, \frac{\pi^2}{m} \, .
\end{align}
Thus, exploiting the inequalities \eqref{gabiskrank3}, we find that $|\bar{g}^m(\pi) - \bar{\calv} \, \pi| \leq \sup_{\widehat{p}_0 < s < \widehat{p}_1} |\bar{g}^{\prime\prime}(s)| \, (\pi)^2/m$.
% \begin{align*}%\label{ui}
% |\bar{g}^m(\pi) - \bar{\calv} \, \pi| \leq \sup_{\widehat{p}_0 < s < \widehat{p}_1} |\bar{g}^{\prime\prime}(s)| \, \frac{(\pi)^2}{m}  \, .
% \end{align*}
Since $\pi = P_m(\varrho, \, q)$, it can be estimated using the first claim. We obtain that
\begin{align*}
 \pi^2 \leq C^2_1 \, \Big(\ln \frac{1}{\min\{\bar{\varrho}_{\max} - \varrho, \, \varrho -\bar{\varrho}_{\min}\}} + |q|\Big)^2 \, .
\end{align*}
Overall, \eqref{QQprime} implies that
\begin{align*}
d_0 \,  \tilde{\pi} \leq  & C^2_1 \, \frac{\sup_{\widehat{p}_0 < s < \widehat{p}_1} |\bar{g}^{\prime\prime}(s)|}{m}  \, \Big(\ln \frac{1}{\min\{\bar{\varrho}_{\max} - \varrho, \, \varrho -\bar{\varrho}_{\min}\}} + |q|\Big)^2  \\ 
 &+ C \, \Big(\ln \frac{1}{\min\{\bar{\varrho}_{\max} - \varrho, \, \varrho -\bar{\varrho}_{\min}\}} + |\Pi^{\prime} q|\Big)  \, .
\end{align*}
We prove the lower bound for $\tilde{\pi}$ similarly.
\end{proof}

\subsubsection{Controlling $\{\pi^m\}$ for type-I weak solutions}

In this section, we begin with proving a uniform $L^1-$bound for the sequence of pressures. It turns out easier to obtain a bound for the functions $\tilde{\pi}^m := \pi^m - q^m_{N-1}$ as for $\pi^m$ directly. Note that $\tilde{\pi}^m = \tilde{P}_m(\varrho, \, q) = P_m(\varrho, \, q) - q^m_{N-1}$ corresponds to the nonlinear part $P_{\infty}$ of the pressure function (cf.\ \eqref{newpartdensincomp} that occurs in the analysis of the incompressible model: \cite{druetmixtureincompweak}, \cite{bothedruetincompress}.

Then, we are able to show that the crucial quantity $|\bar{g}^m(\pi^m) - \bar{\calv} \, \pi^m|$ in the relative energy inequality tends to zero on certain subsets.
We begin the proof of these statements with a useful preliminary.
\begin{lemma}\label{Meslemma}
Assume that $(\rho^m, \, v^m)$ is a weak solution to \emph{($\overline{\text{ IBVP}}^m$)}. For $0< a < b < +\infty$ and $t \in ]0,\bar{\tau}[$ we let $\Omega_{a,b}(t) :=\{x \, : \, a \leq \varrho^m(x,t) \leq b\}$. If $ a < \bar{\varrho}_{\min} = 1/\max\bar{\calv}$ and $b>\bar{\varrho}_{\max} = 1/\min\bar{\calv}$ then 
\begin{align*}
\inf_{0<t<\bar{\tau}} |\Omega_{a,b}(t)| > |\Omega| -  \frac{\|\rho^m \cdot \bar{\calv}-1\|_{L^{1,\infty}}}{\min\{b/\bar{\varrho}_{\max}-1, \, 1- a/\bar{\varrho}_{\min}\}} \quad \text{ for all } \quad m \in \mathbb{N} \, . 
\end{align*}
\end{lemma}
\begin{proof}
For the choices $0 < a < \bar{\varrho}_{\min}$ and $b > \bar{\varrho}_{\max}$, we can show that
\begin{align*} 
 \rho^m \cdot \bar{\calv}  \geq \min \bar{\calv} \, \varrho^m \geq \frac{b}{\bar{\varrho}_{\max}} > 1 \quad \text{ for } \quad \varrho^m \geq b \, , \\
 \rho^m \cdot \bar{\calv}  \leq \max \bar{\calv} \, \varrho^m \leq \frac{a}{\bar{\varrho}_{\min}} < 1 \quad \text{ for } \quad \varrho^m \leq a \, .
\end{align*}
Therefore, $| \rho^m \cdot \bar{\calv} - 1| \geq k_1 := \min\{b/\bar{\varrho}_{\max}-1, \, 1- a/\bar{\varrho}_{\min}\}$ on $\Omega_{a,b}^{\rm c}(t)$. This implies that
\begin{align*}
\sup_{0 < t < \bar{\tau}} |\Omega_{a,b}^{\rm c}(t)| \leq k_1^{-1} \, \|\rho^m \cdot \bar{\calv} - 1\|_{L^{1,\infty}} \, .
\end{align*}
\end{proof}
The pressure bound in $L^1$ for incompressible systems is expected (see \cite{feima16}, \cite{druetmixtureincompweak}). Here we show that it is valid uniformly for the approximating sequences.
\begin{lemma}\label{L1bound}
We assume that $\bar{g}$ satisfies all assumptions of Lemma \ref{pressandg}.
Assume that $(\rho^m, \, v^m)$ is a weak solution to \emph{($\overline{\text{ IBVP}}^m$)}, where the mobility tensor satisfies ${\rm (B1)}$, ${\rm (B2)}$ and ${\rm (B3^\prime})$. We let $\calw_0^m := \frac{1}{|\Omega|} \, \int_{\Omega} \varrho^{0,m}(x) \, dx$, and we assume that $\bar{\varrho}_{\min} < \inf_m \calw_0^m$ and $\sup_m \calw_0^m < \bar{\varrho}_{\max}$. Then, the sequence $\{\tilde{\pi}^m\}$ is uniformly bounded in $L^1(Q)$.
\end{lemma}
\begin{proof}
For $s >0$ and a fixed $b > \max\{\bar{\varrho}_{\max}, \, \sup_{m} |\varrho^{0,m}|_{L^{\infty}}\}$ we let \begin{align*}
L(s) := \begin{cases} 2b - \frac{b^2}{s} & \text{ if } s \geq b\\
          s & \text{ for } 0 \leq s < b
          \end{cases}
          \end{align*}
          We can verify that $L \in C^1(\overline{\mathbb{R}_+})$ and $L(s) - s \, L^{\prime}(s) = 2b\, (1-b/s) \, \chi_{[b, \, +\infty[}(s)$.
The assumptions imply that the total mass density $\varrho^m$ is a renormalised solution to the continuity equation. Then the following identities are valid in the sense of distributions:
\begin{align}\begin{split}\label{Guldes}
\partial_t L(\varrho^m) + \divv(L(\varrho^m) \, v^m) = (L(\varrho^m) - \varrho^m \, L^{\prime}(\varrho^m)) \, \divv v^m \, \quad \text{ over } \quad Q_{\bar{\tau}} \, , \\
 \partial_t \Big(L(\varrho^m) \Big)_M = \frac{1}{|\Omega|} \, \int_{\Omega} (L(\varrho^m) - \varrho^m \, L^{\prime}(\varrho^m)) \, \divv v^m \, dx \quad \text{ over } \quad (0, \, \bar{\tau})\, .
 \end{split}
\end{align}
Since $L(\varrho^m(0)) = L(\varrho^{0,m}) = \varrho^{0,m}$, the latter and Lemma \ref{Meslemma} imply that
\begin{align*}
 \Big|\Big(L(\varrho^m) \Big)_M(t) - \calw_0^m\Big|  = & \frac{1}{|\Omega|} \,  \Big| \int_{0}^t \int_{\Omega} (L(\varrho^m) - \varrho^m \, L^{\prime}(\varrho^m)) \, \divv v^m \, dx d\tau\Big|\\
 \leq & \frac{2b}{|\Omega|} \,  \|\divv v^m\|_{L^2} \, \Big(\int_0^t |\{x \, : \, \varrho^m(x,\tau) \geq b\}| \, d\tau\Big)^{\frac{1}{2}}\\
 \leq & \frac{2\,b \, t^{\frac{1}{2}}\,  \sup_m\|\divv v^m\|_{L^2}}{|\Omega| \, \min\{b/\bar{\varrho}_{\max}-1, \, 1-a/\bar{\varrho}_{\min}\}^{\frac{1}{2}}}  \, \|\rho^m \cdot \bar{\calv} - 1\|_{L^{1,\infty}}^{\frac{1}{2}} \, .
\end{align*}
%where we used Lemma \ref{Meslemma}.
We let $\alpha_0 := \min\{\bar{\varrho}_{\max} - \sup_m\calw_0^m, \, \inf \calw_0^m-\bar{\varrho}_{\min}\}$. In view of \eqref{ilstendent1} and \eqref{lesbounds} we can choose $m_1 \in \mathbb{N}$ such that
\begin{align*}
\frac{2\,b \, t^{\frac{1}{2}}\,  \sup_m\|\divv v^m\|_{L^2} }{|\Omega| \, \min\{b/\bar{\varrho}_{\max}-1, \, 1-a/\bar{\varrho}_{\min}\}^{\frac{1}{2}}} \, \|\rho^m \cdot \bar{\calv} - 1\|_{L^{1,\infty}}^{\frac{1}{2}} \leq \frac{\alpha_0}{2} \, \text{ for all } m \geq m_1 \, ,
\end{align*}
with the consequence that
\begin{align*}
\calw_0^m - \frac{\alpha_0}{2} \leq \Big(L(\varrho^m) \Big)_M(t) \leq \calw_0^m \quad \text{ for all } \quad t \in ]0,\bar{\tau}[, \, m \geq m_1 \, . 
\end{align*}
With the abbreviaton $\sigma_m(t) := \Big(L(\varrho^m) \Big)_M(t)$ the latter implies that
\begin{align}\label{ResRes}
 \min\{\bar{\varrho}_{\max} - \sigma_m(t), \, \sigma_{m}(t) - \bar{\varrho}_{\min}\} \geq \frac{\alpha_0}{2} \quad \text{ for all } \quad  0 < t < \bar{\tau}, \, m \geq m_1 \, .
\end{align}
We apply the result of Lemma \ref{BOGS} with $\calg = L(\varrho^m) - (L(\varrho^m))_M$ and, recalling \eqref{Guldes}, with $\calf := - L(\varrho^m) \, v^m -\nabla \phi^m$ where $\phi^m$ is a solution to the Neumann problem for
\begin{align*}
 -\Delta \phi^m =  (L(\varrho^m) - \varrho^m \, L^{\prime}(\varrho^m)) \, \divv v^m - \frac{1}{|\Omega|} \, \int_{\Omega} (L(\varrho^m) - \varrho^m \, L^{\prime}(\varrho^m)) \, \divv v^m \, dx \, .
\end{align*}
We easily show that $\|\calf\|_{L^{6,2}(Q_{\bar{\tau}})} \leq c_b \, \sup_m \|v^m\|_{L^2W^{1,2}}$. Hence, for $1 \leq \calp < +\infty$ arbitrary, we can find a solution to $\divv \eta^m = \calg$ such that
\begin{align}\label{etambounds}
 \sup_m \|\eta^m\|_{L^{\infty}(0,\bar{\tau}; \, W^{1,\calp}(\Omega))} \leq c_{\calp,b} \, , \quad \sup_m \|\partial_t \eta^m\|_{L^{6,2}(Q_{\bar{\tau}})} \leq c_b \, \sup_m \|v^m\|_{L^2(0,\bar{\tau}; \, W^{1,2}(\Omega))} \, .
\end{align}
We insert $\eta^m$ in the momentum balance equations. We obtain that
\begin{align*}
 & \int_{Q_t} \pi^m \, \divv \eta^m \, dxd\tau = \int_{\Omega} \varrho^m \, v^m \cdot \eta^m \, dx\Big|^t_0 \\
 & \quad 
+ \int_{Q_t} \{ - \varrho^m \, v^m \, (\partial_t \eta^m + (v^m \cdot \nabla) \eta^m) + \mathbb{S}(\nabla v^m) \, : \, \nabla \eta^m + \varrho^m \, \eta^m_3\} \, dxd\tau
 \, . 
\end{align*}
Using the uniform bounds \eqref{lesbounds} and \eqref{etambounds}, we can show that all integrals on the right-hand side are uniformly bounded. We want to spare the details here. Similar calculations will be performed in the proof of Prop.\ \ref{pressla}. Hence $\sup_m |\int_{Q_t} \pi^m \, \divv \eta^m \, dxd\tau | < +\infty$. Moreover
\begin{align*}
 \int_{Q_t} \tilde{\pi}^m \, \divv \eta^m \, dxd\tau = \int_{Q_t} (\pi^m - q^m_{N-1}) \, \divv \eta^m \, dxd\tau = \int_{Q_t} \pi^m \, \divv \eta^m + \nabla q^{m}_{N-1} \cdot \eta^m \, dxd\tau \, ,
\end{align*}
implies together with $\sup_m \|\nabla q^m\|_{L^2} < +\infty$ that also $\sup_m |\int_{Q_t} \tilde{\pi}^m \, \divv \eta^m \, dxd\tau | < +\infty$.
Using the property that $0< \sigma_m(t) < \bar{\varrho}_{\max}$ implies that $\sigma_m(t) = L(\sigma_m(t))$, we next see that
\begin{align*}
&  \int_{Q_t} \tilde{\pi}^m \, \divv \eta^m \, dxd\tau   = \int_{Q_t} \tilde{P}_m(\varrho^m, \, q^m) \, (L(\varrho^m) -L(\sigma_m)) \, dxd\tau\\
=& \int_{Q_t} |\tilde{P}_m(\varrho^m, \, q^m)-\tilde{P}_m(\sigma_m, \, q^m)| \, |L(\varrho^m) - \sigma_m| + \tilde{P}_m(\sigma_m, \, q^m) \, (L(\varrho^m) - \sigma_m) \, dxd\tau \, ,
\end{align*}
where we used that $s \mapsto \tilde{P}_m(s,\cdot)$ is nondecreasing. Hence, it follows that
\begin{align*}
 \int_{Q_t} \tilde{\pi}^m \, \divv \eta^m \, dxd\tau   \geq \int_{Q_t} |\tilde{P}_m(\varrho^m, \, q^m)| \, |L(\varrho^m) - \sigma_m| - 2\,  |\tilde{P}_m(\sigma_m, \, q^m)| \, |L(\varrho^m) -\sigma_m| \, dxd\tau \, .
\end{align*}
Due to \eqref{ResRes} and Lemma \ref{pressandg},  $|\tilde{P}_m(\sigma_m, \, q^m)| \leq C_2 \, (\ln^2 (2/\alpha_0) + |(q^{m})^{\prime}| + |q^m|^2/m)$. Hence, it follows that
\begin{align*}
 \int_{Q_t} \tilde{\pi}^m \, \divv \eta^m \, dxd\tau   \geq & \int_{Q_t} |\tilde{P}_m(\varrho^m, \, q^m)| \, |L(\varrho^m) - \sigma_m| \, dxd\tau\\
 &- c\,  \sup_m (1+ \|(q^{m})^{\prime}\|_{L^1}+\frac{1}{m} \|q^m\|_{L^2}^2)\, .
\end{align*}
We let $\Omega^*(t) := \{x \, :\, L(\varrho^m(x,t)) - \sigma_m(t)| \geq \alpha_0/4\}$. Then we have shown that
\begin{align*}
 \int_{Q_t} \tilde{\pi}^m \, \divv \eta^m \, dxd\tau   \geq \frac{\alpha_0}{4} \, \int_{0}^t \int_{\Omega^*(\tau)} |\tilde{P}_m(\varrho^m, \, q^m)| \, dxd\tau- C\, ,
 \end{align*}
implying that $\sup_m \int_0^t\int_{\Omega^*(\tau)} |\tilde{P}_m(\varrho^m, \, q^m)| dxd\tau < +\infty$. If, otherwise, $
|L(\varrho^m) - \sigma_m| < \alpha_0/4$, then \eqref{ResRes} implies that $|L(\varrho^m) - \calw_0^m| \leq 3\alpha_0/4$. Hence $L(\varrho^m) \in ]\bar{\varrho}_{\min}, \, \bar{\varrho}_{\max}[$. This implies that $\bar{\varrho}_{\min} < \varrho^m < \bar{\varrho}_{\max}$ and that $\min\{\bar{\varrho}_{\max}-\varrho^m, \, \varrho^m - \bar{\varrho}_{\min}\} \geq \alpha_0/4$. Then, Lemma \ref{pressandg} yields
\begin{align*}
 |\tilde{P}_m(\varrho^m, \, q^m)| \leq c \, \Big(\ln \frac{4}{\alpha_0} + |(q^{m})^{\prime}| + \frac{1}{m} \, |q^m|^2\Big) \, \quad \text{ in } \quad \Omega\setminus \Omega^*(t) \, .
\end{align*}
We integrate this inequality over $\Omega\setminus \Omega^*(t) $ and the claim follows using \eqref{lesbounds}.
\end{proof}
The strategy of so-called Bogowski estimates is well-known. We here apply the Th.\ 11.17 of \cite{feinovbook}. References to original literature are to find therein.
\begin{lemma}\label{BOGS} 
Assume that $\Omega$ is a bounded domain of class $\mathscr{C}^{0,1}$. Let $1 < \calp < +\infty$, $1 \leq \calr \leq +\infty$ and $\calg \in L^{\calp,\calr}(Q_{\bar{\tau}})$ satisfy $(g)_{M}(t) \equiv 0$ for almost all $0 < t < \bar{\tau}$. Moreover, we assume that $\partial_t\calg = \divv \calf$ in the sense of distributions, with a field $\calf \in L^{\calq,\cals}(Q; \, \mathbb{R}^3)$ for some $1 < \calq < +\infty$, $1 \leq s \leq \infty$. Then, the problem $\divv \eta = \calg$ in $Q_{\bar{\tau}}$ such that $\eta = 0$ on $\partial \Omega \times (0,\bar{\tau})$ possesses a solution such that, with a constant $c$ depending only on $\Omega$ and the involved exponents, the following estimates are valid: $\|\eta\|_{L^{\calr}W^{1,\calp}} \leq c \, \|\calg\|_{L^{\calp,\calr}}$ and $\|\partial_t \eta\|_{L^{\calq,\cals}} \leq c \, \|\calf\|_{L^{\calq,\cals}}$.
%\begin{align*}%\label{bogonullbounds}
% \|\eta\|_{L^{\calr}W^{1,\calp}} \leq c \, \|\calg\|_{L^{\calp,\calr}} \, , \quad \|\partial_t \eta\|_{L^{\calq,\cals}} \leq c \, \|\calf\|_{L^{\calq,\cals}} \, .
%\end{align*}
\end{lemma}

\subsubsection{Preliminaries for the convergence theorem}

In order to prove convergence for terms involving the pressure, we shall need a little bit more than the $L^1-$bound. In particular, we must show in which sense we can expect that $\bar{g}^m(\pi^m) - \bar{\calv} \, \pi^m$ tends to zero, and we must establish that the $L^1-$norm of the pressure tends to zero over sets where the mass density is small or large. These points require further technical work. We begin with a useful preliminary.
\begin{lemma}\label{Philemma}
We adopt the assumptions of Lemma \ref{pressandg} and of Lemma \ref{Meslemma}.
For $(x,t) \in Q$, we define $\Phi^m(x,t) := - \rho^m(x,t) \cdot \int_0^1 \bar{g}^{\prime\prime}\Big(1+\theta\,\pi^m(x,t)/m\Big) \, d\theta$. Then $\Phi^m$ is nonnegative, $\{\Phi^m\}$ is bounded in $L^{\infty}$ and the following properties are valid:
\begin{enumerate}[(i)]
\item\label{prop1} Letting $ 0< \widehat{p}_0 < \widehat{p}_1$ be as in \eqref{gabiskrank3}, for all $0 < a < b < +\infty$ all $0 < t < \bar{\tau}$
\begin{align*}
 \Phi^m(x,t) \geq a \, \inf\{|\bar{g}_i^{\prime\prime}(s)| \, : \,  \widehat{p}_0(1/b) < s < \widehat{p}_1(1/a), \,  i = 1,\ldots,N\} \quad \text{ over } \quad \Omega_{a,b}(t) \, ,
\end{align*}
\item\label{prop2} There is $m_1 \in \mathbb{N}$ such that $\inf_{m \geq m_1, \, 0 < t < \bar{\tau}} \int_{\Omega} \Phi^m(x,t) \, dx > 0$.
\end{enumerate}
\end{lemma}
\begin{proof}
We directly obtain that $\Phi^m$ is positive due to the concavity of $\bar{g}$. 

Ad \eqref{prop1}. Due to \eqref{piiswidehatp},  $ 1+\pi^m/m = \widehat{p}(\rho^m)$ and \eqref{gabiskrank3} implies that $\widehat{p}_0(1/b) \leq  1 +\theta \, \pi^m/m \leq \widehat{p}_1(1/a)$ in $\Omega_{a,b}(t)$ for all $\theta \in [0,1]$. By the definition of $\widehat{p}$ we then verify over $\Omega_{a,b}(t)$ that
\begin{align*}
 \Phi^m = & - \bar{\rho}^m \cdot \int_0^1 \bar{g}^{\prime\prime}\Big(1+\theta \, (\widehat{p}(\rho^m)-1)\Big) \, d\theta 
\geq   \varrho^m \, \min_{i=1,\ldots,N, \, \theta \in [0,1]} \Big|\bar{g}_i^{\prime\prime}\Big(1+\theta \, (\widehat{p}(\rho^m)-1)\Big)\Big| \, ,
% \geq a \, \min_{i=1,\ldots,N, \, \widehat{p}_{0}(1/b) \leq s \leq \widehat{p}_1(1/a)} |\bar{g}_i^{\prime\prime}(s)| \, .
\end{align*}
and \eqref{prop1} follows.

Ad \eqref{prop2}. 
%Consider in particular $0 < a_1 < \bar{\varrho}_{\min}$ and $\bar{\varrho}_{\max} < b_1 < +\infty$ arbitrary, fi. 
Due Lemma \ref{Meslemma}, we can choose $m_1 \in \mathbb{N}$ such that $\inf_{0<t<\bar{\tau}} |\Omega_{\bar{\varrho}_{\min}/2,2\bar{\varrho}_{\max}}(t)| > |\Omega|/2$, for all $m \geq m_1$, and it follows from \eqref{prop1} that
\begin{align*}
 \int_{\Omega} \Phi^m(x,t) \, dx 
 \geq \frac{\bar{\varrho}_{\min}}{4} \, |\Omega| \, \min_{\widehat{p}_0(\max\bar{\calv}/2) \leq s \leq \widehat{p}_1(2\max\bar{\calv})} \min_{i=1,\ldots,N} |\bar{g}_i^{\prime\prime}(s)| \, .
\end{align*}
This proves \eqref{prop2}. It remains to show the $L^{\infty}$-bound. Employing the growth conditions (A6), we get
\begin{align*}
-\rho^m \cdot \bar{g}^{\prime\prime}\Big(1+\theta \, (\widehat{p}(\rho^m)-1)\Big) \leq \bar{c}_3 \, \rho^m \cdot \frac{\bar{g}^{\prime}\Big(1+\theta \, (\widehat{p}(\rho^m)-1)\Big)}{1+\theta \, (\widehat{p}(\rho^m)-1)} \, .
\end{align*}
Let $\bar{s}$ be the constant occurring in (A5). If $\widehat{p}(\rho^m) > \bar{s}$ we estimate $\bar{g}^{\prime}(1+\theta \, (\widehat{p}(\rho^m)-1)) \leq \bar{g}^{\prime}(1) = \bar{\calv}$, and we obtain that
\begin{align*}
 \Phi^m \leq \bar{c}_3\, \rho^m \cdot \bar{\calv} \, \int_0^1 \frac{1}{1+\theta \, (\widehat{p}(\rho^m)-1)} \, d\theta =  \bar{c}_3\, \rho^m \cdot \bar{\calv} \, \frac{\ln \widehat{p}(\rho^m)}{\widehat{p}(\rho^m)-1} \, .
\end{align*}
Due to \eqref{plarge12} we have $\varrho^m \leq c \, \widehat{p}^{1/\gamma}(\rho^m)$, thus $ |\Phi^m| \leq c \, \sup_{s > \bar{s}} s^{\frac{1}{\gamma}} \, \frac{\ln s}{s-1}$. If $\widehat{p}(\rho^m) \leq \bar{s}$, then (A4) allows to estimate $\bar{g}^{\prime}(\widehat{p}) \leq\bar{c}_2/\widehat{p}$ and thus
\begin{align*}
 \Phi^m \leq & \bar{c}_3\, \rho^m \cdot \int_{0}^1  \frac{\bar{g}^{\prime}\Big(1+\theta \, (\widehat{p}(\rho^m)-1)\Big)}{1+\theta \, (\widehat{p}(\rho^m)-1)} d\theta \\
 \leq&  \bar{c}_3 \, \bar{c}_{2} \, \varrho^m \, \int_{0}^1  \frac{1}{(1+\theta \, (\widehat{p}(\rho^m)-1))^2} d\theta \leq c \, \frac{\varrho^m}{ \widehat{p}(\rho^m)} \, .
\end{align*}
Since $\varrho^m \leq \bar{c}_1^{-1} \, \widehat{p}(\rho^m)$ due to \eqref{la-bas}, the claim follows.
\end{proof}
\begin{coro}\label{l1boundandtozero}.
Assumptions of Lemma \ref{pressandg}, \ref{Philemma}. If $\limsup_{m\rightarrow \infty} \sqrt{m} \,| \int_{\Omega} \rho^{0,m} \cdot\bar{\calv} - 1 \, dx| = 0$, then $\lim_{m\rightarrow \infty} \frac{1}{\sqrt{m}} \, \|q^m\|_{L^2(Q)} = 0$ and  for all $0 < a_0 < b_0 < +\infty$
 \begin{align*}
  \int_{0}^{\bar{\tau}}\int_{\Omega_{a_0,b_0}(\tau)} |\bar{g}^m(\pi^m) - \pi^m \, \bar{\calv})|\, dxd\tau \rightarrow 0  \quad \text{ for } \quad m \rightarrow +\infty\, .
 \end{align*}
\end{coro}
\begin{proof}
Due to \eqref{lesbounds}, $\{q^{m}_i\}$ is bounded in $L^2(Q)$ for $i = 1,\ldots,N-2$. Hence $\{q^{m}_i/\sqrt{m}\}$ obviously tends to zero in $L^2(Q)$. In order to show that also $\{q^m_{N-1}/\sqrt{m}\}$ converges to zero in $L^2$, we at first exploit the equations \eqref{massrescfin}. Multiplying with the vector $\bar{\calv}$ and integrating over $\Omega$, we get $\partial_t (\rho^m \cdot \bar{\calv})_M = 0$. Hence, for all $0 \leq t \leq \bar{\tau}$, 
\begin{align}\label{tillilee}
 m \, \int_{\Omega} \rho^m(x,t) \cdot \bar{\calv} - 1 \, dx = m \,  \int_{\Omega^{\rm R}} \rho^{0,m}(x) \cdot \bar{\calv} - 1 \, dx \, .
\end{align}
With the function $\Phi^m$ of Lemma \ref{Philemma}, we further note that
\begin{align}\label{Phiidentity}
\rho^m \cdot \bar{\calv} - 1 = \rho^m \cdot \Big(\bar{g}^{\prime}(1) - \bar{g}^{\prime}\Big(1+\frac{\pi^m}{m}\Big)\Big) 
= \Phi^m \, \frac{\pi^m}{m} \, .
\end{align}
Now, the identity \eqref{tillilee} implies that $\int_{\Omega} \Phi^m \, \pi^m \, dx = m \, \int_{\Omega} \rho^{0,m}(x) \cdot \bar{\calv} - 1 \, dx$, which we use to prove that
\begin{align*}
 (q_{N-1}^m)_M(t) \, \int_{\Omega} \Phi^m \, dx = m \, \int_{\Omega} \rho^{0,m} \cdot \bar{\calv} - 1 \, dx - \int_{\Omega} \Phi^m \, \Big(\tilde{\pi}^m + q_{N-1}^m - (q_{N-1}^m)_M\Big)  \, dx\, .
\end{align*}
With $k_0 := \inf_{m\geq m_1, 0<t<\bar{\tau}} \int_{\Omega}\Phi^m(x,t) \,dx>0$, this allows to bound
\begin{align*}%\label{john}
k_0 \,  |(q_{N-1}^m)_M(t) | \leq   m \, \left|\int_{\Omega} \bar{\rho}^{0,m} \cdot \bar{\calv} - 1 \, dx\right| + \sup_m \|\Phi^m\|_{L^{\infty}} \, \int_{\Omega} |\tilde{\pi}^m| + | q_{N-1}^m - (q_{N-1}^m)_M| \, dx
\end{align*}
Invoking Lemma \ref{L1bound}, $\{\tilde{\pi}^m\}$ is bounded in $L^1(Q)$ while, owing to the Poincar\'e inequality and \eqref{lesbounds}, $\{q^m - (q^m)_M\}$ is bounded in $L^2(Q)$. Hence
\begin{align}\label{john}
 \limsup_{m \rightarrow \infty} \frac{1}{\sqrt{m}} \, \|(q_{N-1}^m)_M\|_{L^1(0,\bar{\tau})} = 0 \quad \text{ and } \quad  \limsup_{m \rightarrow \infty} \frac{1}{\sqrt{m}} \, \|q_{N-1}^m\|_{L^1(Q)} = 0 \, .
\end{align}
To show the strong convergence to zero in $L^2$, we use the inequality (cf.\ Appendix, Lemma \ref{uamamoto}, \eqref{claim1})
\begin{align*}%\label{jama}
\frac{1}{m} \, \max_{k=1,\ldots,N-1} |(q^m_k)_M(t)|^2 \leq C_1\, \left(\int_{\Omega} (\bar{f}^m(\bar{\rho}) + 1)^{\frac{1}{2}} \, dx\right)^2 + \frac{C_2}{m} \, |\nabla \bar{q}^m|_{L^1(\Omega)}^2 \, ,
\end{align*}
valid with $C_1,C_2$ independent on $m$. Up to constants, the right-hand is estimated by the sum of $\int_{\Omega} |\bar{f}^m| + 1\, dx$ which is uniformly bounded in $L^{\infty}(0,\bar{\tau})$ and of $m^{-1} \, |\nabla \bar{q}^m|_{L^1(\Omega)}^2$, which converges to zero in $L^1(0,\bar{\tau})$. Thus, $\{m^{-1} \, |(q^m_{N-1})_M(t)|^2\}$ is equi-integrable over $(0,\bar{\tau})$. Since \eqref{john} implies that it converges pointwise to zero, we then easily prove the convergence to zero in $L^2(0,\bar{\tau})$, for instance using Egoroff's theorem.

Next we want to prove the second convergence result. By means of \eqref{gabiskrank}, \eqref{gabiskrank2}, there is $\lambda \in [0,1]$ such that $\bar{g}_i^m(\pi^m) - \bar{\calv}_i \, \pi^m = \frac{1}{2} \,  \bar{g}^{\prime\prime}(1+\lambda \, \pi^m/m) \, (\pi^m)^2/m$.
Using the property \eqref{gabiskrank3}, we find that $\widehat{p}_0(1/b_0) \leq  1 + \pi^m/m \leq \widehat{p}_1(1/a_0)$ in $\Omega_{a_0,b_0}(t)$, implying that 
\begin{align*}
|\bar{g}_i^m(\pi^m) - \bar{\calv}_i \, \pi^m| \leq \frac{1}{2} \, \sup_{\widehat{p}_0 < s < \widehat{p}_1} |\bar{g}^{\prime\prime}(s)| \, \frac{(\pi^m)^2}{m} 
\end{align*}
over $\Omega_{a_0,b_0}(t)$. Due to Lemma \ref{pressandg}, we have $|\pi^m| \leq c \, (1+|q^m|)$ on $\Omega_{a_0,b_0}(t)$. Thus, $m \, |\bar{g}_i^m(\pi^m) - \bar{\calv}_i \, \pi^m| \, \chi_{\Omega_{a_0,b_0}(t)} \leq c\, (1+|q^m|^2)$ and the convergence to zero in $L^1$ follows from the first claim.
\end{proof}
The discussion of the energy inequality and of the stability estimate is more technical in the case of weak solutions We next set up some further preliminaries.

For $0\leq t \leq \bar{\tau}$, and $0< a_0 < b_0 < +\infty$, we let
\begin{align} \begin{split}\label{Omegazwish}\Omega_{b_0,+}(t) = \{x \in \Omega \, : \, \varrho(x,t) \geq b_0\}\quad & \text{ and } \quad  \Omega_{a_0,-}(t) = \{x \in \Omega \, : \, \varrho(x,t) \leq a_0\} \, ,\\
\Omega_{a_0,b_0}(t) = \{x \in \Omega \, & : \, a_0 \leq \varrho(x,t) \leq b_0\}\, .
\end{split}
\end{align}
The occurrence of these sets is typical for the discussion of singular limits in fluid dynamics, see \cite{feinovbook}. 
For the multicomponent case, we define another essential set in which all densities are strictly positive 
\begin{align}\label{Omegainf}
B_{s_0,+}(t) = B_{s_0,+}(t; \, a_0,b_0) := \{x \in \Omega_{a_0,b_0}(t) \, : \, \min_{i=1,\ldots,N} \rho_i(x,t) \geq s_0\} \quad \text{ for } s_ 0 > 0 \, .
\end{align}
\begin{lemma}\label{erhobelowlem}
We let $(\rho^{\infty},p^{\infty},v^{\infty})$ be subject to (C) and $(\rho, \, v)$ satisfy \eqref{NC}.
With $r_{\min} := \inf_{Q} \varrho^{\infty}$ and $r_{\max} := \sup_Q \varrho^{\infty}$ (cf.\ \eqref{rminmax}) and $0 < \delta_0$ arbitrary, we define $a_0 := r_{\min}-\delta_0$ and $b_0 := r_{\max}+\delta_0$. Moreover we let $r_0 := \inf_{(x,t) \in Q_{\bar{\tau}}, \, i = 1,\ldots,N} \rho^{\infty}_i(x,t)$ and $0 < s_0 < r_0$. With $\mathcal{E}^m(t) = \mathcal{E}^m(\rho,v \, |\, \rho^{\infty},v^{\infty})(t)$ there are $C, \, C^{\prime} > 0$ such that, for all $m \in \mathbb{N}$ and all $0 \leq t \leq \bar{\tau}$,
 \begin{align*}
  \mathcal{E}^m(t) \geq C \, \Big(|\Omega_{a_0,-}(t)| + \int_{\Omega_{b_0,+}(t)} |\rho|^{\gamma} \, dx \Big) \quad \text{ and } \quad  \mathcal{E}^m(t) \geq C^{\prime} \, |\Omega_{a_0,b_0}(\tau) \setminus B_{s_0,+}(\tau)| \, .
\end{align*} 
\end{lemma}
\begin{proof}
Since $|\rho - \rho^{\infty}| \geq |\varrho - \varrho^{\infty}|/\sqrt{N}$, the bounds \eqref{rminmax} for $\varrho^{\infty}$ imply that
\begin{align}\label{varrhobounds}
\begin{cases}
|\rho(x,t) - \rho^{\infty}(x,t)| \geq  \frac{\delta_0}{\sqrt{N}} & \text{ for } \varrho(x,t) \leq r_{\min}-\delta_0 \, ,\\
 |\rho(x,t) - \rho^{\infty}(x,t)| \geq\frac{\delta_0}{\sqrt{N} \, (r_{\max}+\delta_0)} \, \varrho(x,t)  & \text{ for } \varrho(x,t) \geq r_{\max}+\delta_0\, .
 \end{cases}
\end{align}
The Lemma \ref{UNIFF2} allows for the coercivity
\begin{align}\label{enplus}
 \bar{f}^m(\rho) - \bar{f}^m(\rho^{\infty}) - D\bar{f}^m(\rho^{\infty}) (\rho-\rho^{\infty}) \geq& \frac{\lambda_1}{2} \, \frac{|\rho-\rho^{\infty}|^2}{\omega(\varrho, \, \varrho^{\infty})} \,,
 \end{align}
 where $ \omega(\varrho,\,\varrho^{\infty}) :=  \max\{\varrho, \, \varrho^{\infty}\} \, \big(1+ \max\{\varrho, \, \varrho^{\infty}\}\big)^{\theta_0}$ and \eqref{varrhobounds} then shows that
\begin{align}\label{Erhobelow1bar}
 \mathcal{E}^m(t) \geq \frac{\lambda_1}{2r_{\max}(1+r_{\max})^{\theta_0}} \, \frac{\delta_0^2}{N} \, |\Omega_{a_0,-}(t)| + \frac{\lambda_1}{2} \, \frac{\delta_0^2}{N \, b_0^2(1+1/b_0)^{\theta_0}} \, \int_{\Omega_{b_0,+}(t)} [\varrho(x,t)]^{1-\theta_0} \, dx \, .
 \end{align}
By means of Lemma \ref{UNIFF2}, we can also infer for $\varrho \geq R_1$ that
 \begin{align*}
 \bar{f}^m(\rho) - \bar{f}^m(\rho^{\infty}) - D\bar{f}^m(\rho^{\infty}) \cdot (\rho - \rho^{\infty}) \geq c_0 \, |\rho|^{\gamma} - \bar{f}^m(\rho^{\infty})- D\bar{f}^m(\rho^{\infty}) \cdot (\rho - \rho^{\infty})  \, .
\end{align*}
Using Young's inequality and the bounds established in Lemma \ref{UNIFF2} for $\bar{f}^m(\rho^{\infty})$ and $D\bar{f}^m(\rho^{\infty})$, we easily can prove with $r_0 = \min_{i, (x,t) \in Q} \rho^{\infty}_i(x,t)$ that
\begin{align*}
   |D\bar{f}^m(\rho^{\infty}) \cdot (\rho - \rho^{\infty})| \leq & \frac{1}{\min M} \, \Big|\ln \frac{r_0 \, \min M \, \min \bar{\calv}}{\max M}\Big| \, |\rho-\rho^{\infty}| \\
  \leq & \frac{c_0}{2} \, |\rho|^{\gamma} + C(\gamma,T,M,r_0,r_{\max}) \, ,
  \end{align*}
which implies that 
%\begin{align*}
$ \bar{f}^m(\rho) - \bar{f}^m(\rho^{\infty}) - D\bar{f}^m(\rho^{\infty}) \cdot (\rho - \rho^{\infty}) \geq \frac{c_0}{2} \, |\rho|^{\gamma} -   c^{\prime}_1  $. With $b_1 :=\max\{R_1,b_0\}$ we obtain that 
\begin{align}\label{provisorium}
\mathcal{E}^m(t) \geq \frac{c_0}{2} \, \int_{\Omega_{b_1,+}(t)} |\rho|^{\gamma} \, dx - c_1^{\prime} \, |\Omega_{b_1,+}(t)| \geq \frac{c_0}{2} \, \int_{\Omega_{b_0,+}(t)} |\rho|^{\gamma} \, dx - \Big(c_1^{\prime}+\frac{c_0b_1^{\gamma}}{2}\Big) \, |\Omega_{b_0,+}(t)| \, .
\end{align}
We distinguish two cases. At first, if $\theta_0 \leq 1$, then \eqref{Erhobelow1bar} shows directly that $|\Omega_{b_0,+}(t)| \leq c \, \mathcal{E}^m(t)$, and \eqref{provisorium} implies that $\int_{\Omega_{b_0,+}(t)} |\rho|^{\gamma} \, dx \leq  C \, \mathcal{E}^m(t)$. At second, if $\theta_0 > 1$, then with $b_2 > b_0$ arbitrary
\begin{align*}
|\Omega_{b_0,+}(t)|  = |\Omega_{b_0,b_2}(t)| + |\Omega_{b_2,+}(t)| \leq b_2^{\theta_0-1} \, \int_{\Omega_{b_0,b_2}(t)} \varrho^{1-\theta_0} \, dx + \frac{1}{b_2^{\gamma}} \, \int_{\Omega_{b_2,+}(t)} \varrho^{\gamma} \, dx\\
\leq b_2^{\theta_0-1} \,\int_{\Omega_{b_0,+}(t)} \varrho^{1-\theta_0} \, dx + \frac{1}{b_2^{\gamma}} \, \int_{\Omega_{b_0,+}(t)} \varrho^{\gamma} \, dx \, .
\end{align*}
We combine the latter with \eqref{Erhobelow1bar} and it follows that
\begin{align*}
 |\Omega_{b_0,+}(t)|  \leq b_2^{\theta_0-1} \, \, \frac{2N \, b_0^2(1+1/b_0)^{\theta_0}}{\lambda_1\delta_0^2} \, \mathcal{E}^m(t) + \frac{1}{b_2^{\gamma}} \, \int_{\Omega_{b_0,+}(t)} \varrho^{\gamma} \, dx \, .
\end{align*}
Thus, \eqref{provisorium} helps showing that
\begin{align*}
 \mathcal{E}^m(t) \geq & \frac{c_0}{2} \, \int_{\Omega_{b_0,+}(t)} |\rho|^{\gamma} \, dx - \Big(c_1^{\prime}+\frac{c_0b_1^{\gamma}}{2}\Big) \, \frac{2N \, b_2^{\theta_0-1} \,b_0^2(1+1/b_0)^{\theta_0}}{\lambda_1\delta_0^2}  \, \mathcal{E}^m(t) \\
 & - \frac{1}{b_2^{\gamma}} \,\Big(c_1^{\prime}+\frac{c_0b_1^{\gamma}}{2}\Big)\, \int_{\Omega_{b_0,+}(t)} \varrho^{\gamma}\, dx \, .
\end{align*}
Choosing $b_2$ such that $N^{\frac{\gamma}{2}}(c_1^{\prime}+\frac{c_0b_1^{\gamma}}{2})/b_2^{\gamma}\leq c_0/4$, it follows that $\int_{\Omega_{b_0,+}(t)} |\rho|^{\gamma} \, dx \leq  C \, \mathcal{E}^m(t)$. 

In order to prove the second inequality, we observe that, if $\min_i \rho_i \leq s_0 < \min_i \rho_i^{\infty} = r_0$, then $|\rho-\rho^{\infty}| \geq r_0-s_0 > 0$. Hence, 
 \begin{align*}
 (r_0-s_0)^2 \leq b_0(1+b_0)^{\theta_0} \, \frac{|\rho-\rho^{\infty}|^2}{\omega(\varrho, \, \varrho^{\infty})} \quad \text{ in } \quad \Omega_{a_0,b_0}(\tau) \setminus B_{s_0,+}(\tau) \, ,
 \end{align*}
 and integration yields the claim.
\end{proof}
These bounds allow to prove the following useful estimate. 
\begin{lemma}\label{oneterm}
Let $F: \, Q_{\bar{\tau}} \rightarrow \mathbb{R}^{N\times 3}$ be some given field, and assume that $(\rho^{\infty},p^{\infty},v^{\infty})$ satisfies the conditions $\rm (C)$. Then, there is a constant $C$ depending only on $\rho^{\infty}$ such that, for all $m \geq 1$ and for almost all $0 < t < \bar{\tau}$,
 \begin{align*}
\left|\int_{\Omega}  F(\cdot,t) \, : \, 
 (\rho-\rho^{\infty}) \otimes  (v-v^{\infty}) \, dx\right| \leq & C \, |F(\cdot,t)|_{L^{3}} \, |\nabla(v^{\infty} - v)|_{L^2} \, (\mathcal{E}^m(t))^{\frac{1}{2}} \, .
 \end{align*}
\end{lemma}
\begin{proof}
Consider $\Omega_{b_0,+}^{\rm c}(t) = \{x \, : \,\varrho(x,t) \leq b_0\}$. By means of H\"older's inequality 
\begin{align*}
& \int_{\Omega_{b_0,+}^{\rm c}(t)} F\,:\,(\rho-\rho^{\infty}) \otimes  (v^{\infty}-v) \, dx = \int_{\Omega_{b_0,+}^{\rm c}(t)}\omega(\varrho,\varrho^{\infty})^{\frac{1}{2}} \, F\,:\,\frac{\rho - \rho^{\infty}}{\omega(\varrho,\varrho^{\infty})^{\frac{1}{2}}} \otimes  (v^{\infty} - v) \, dx\\
& \leq \sqrt{b_0} \, (1+b_0)^{\frac{\theta_0}{2}}  \, |F|_{L^{3}} \,|v^{\infty} - v|_{L^6(\Omega)} \, \, \left(\int_{\Omega_{b_0,+}^{\rm c}(t)} \frac{|\rho - \rho^{\infty}|^2}{\omega(\varrho,\varrho^{\infty})} \, dx\right)^{\frac{1}{2}} \\
& \leq \sqrt{2b_0} \,  (1+b_0)^{\frac{\theta_0}{2}} \, |F|_{L^{3}} \,|v^{\infty} - v|_{L^6} \,   (\mathcal{E}^m(t))^{\frac{1}{2}} \, ,
\end{align*}
For the integral over $\Omega_{b_0,+}(t)$, with the help of Hoelder's inequality and the fact that $\varrho \geq b_0 \geq r_{\max} \geq \varrho^{\infty}$ we argue that
\begin{align*}
& \left|\int_{\Omega_{b_0,+}(t)}  
F\, :\, (\rho-\rho^{\infty}) \otimes (v^{\infty} - v) \, dx\right| \\
& \leq \left(\int_{\Omega_{b_0,+}(t)}  |\rho-\rho^{\infty}|^{\gamma} \, dx \right)^{\frac{1}{\gamma}} \, |v^{\infty} - v|_{L^6} \, |F|_{L^3} \, |\Omega_{b_0,+}(t)|^{\frac{1}{2} - \frac{1}{\gamma}}\\
& \leq \sqrt{N} \, \Big(1+ \frac{r_{\max}}{b_0}\Big) \left(\int_{\Omega_{b_0,+}(t)}  |\varrho|^{\gamma} \, dx \right)^{\frac{1}{\gamma}} \, |v^{\infty} - v|_{L^6} \, |F|_{L^3} \, |\Omega_{b_0,+}(t)|^{\frac{1}{2} - \frac{1}{\gamma}} 
\, .
 \end{align*}
Then, invoking Lemma \eqref{erhobelowlem}, the Poincar\'e inequality and the Sobolev embedding theorem, we can show that 
 \begin{align*}
&  \left|\int_{\Omega_{b_0,+}(t)}  F \, : \, 
 (\rho-\rho^{\infty}) \otimes (v^{\infty} - v) \, dx\right| \leq C \, |F|_{L^3} \, (\mathcal{E}^m(t))^{\frac{1}{\gamma}} \, |\nabla(v^{\infty} - v)|_{L^2}  \, (\mathcal{E}^m(t))^{\frac{1}{2} - \frac{1}{\gamma}} \, .
 \end{align*}
\end{proof}
Next we prove the main result in this section, a complement to the $L^1-$bound for $\{\tilde{\pi}^m\}$, which is necessary to discuss the relative energy inequality in the weak solution case.
\begin{prop}\label{pressla}
Let $\bar{g}$ satisfy all assumptions of Lemma \ref{pressandg}, and assume that $(\rho^m, \, v^m)$ is a weak solution to \emph{($\overline{\text{ IBVP}}^m$)}, where the mobility tensor satisfies ${\rm (B1)}$, ${\rm (B2)}$ and ${\rm (B3^\prime})$. We let $(\rho^{\infty}, \, p^{\infty},\, v^{\infty})$ be a strong solution to \emph{(IBVP$^{\infty}$)} subject to  $\rm (C)$. We define: $\mu^{\infty} = p^{\infty} \, \bar{\calv} + (1/\bar{M})\, \ln \hat{x}(\rho^{\infty})$ and for all $1 \leq k \leq N-1$, $q^{\infty}_k = \eta^k \cdot \mu^{\infty}$, and $\tilde{p}^{\infty} := p^{\infty} - q^{\infty}_{N-1}$. With $r_{\min}$ and $r_{\max}$ denoting the lower and upper bound of $\varrho^{\infty}$ over $Q$, we let $ 0 < \delta_0 < \min\{\bar{\varrho}_{\max}-r_{\max}, \, r_{\min}-\bar{\varrho}_{\min}\}$, $a_0:= r_{\min}-\delta_0$ and $b_0 := r_{\max}+\delta_0$. 
Then, there is a positive number $C = C(\delta_0)$ depending on the data such that, for all $\epsilon >0$ and $0\leq t \leq \bar{\tau}$ and all $m \in \mathbb{N}$
\begin{align*}
& \int_0^t\int_{\Omega^{\rm c}_{a_0,b_0}(\tau)}  |\tilde{\pi}^m| \, dxd\tau \leq \epsilon \, \int_{Q_t} |\nabla (v^m-v^{\infty})|^2 + |\nabla (q^m-q^{\infty})|^2 \, dxd\tau\\
& \qquad + \frac{C}{\epsilon} \, \int_0^t \psi^m(\tau) \,  \mathcal{E}^m(\tau) \, d\tau + \frac{C}{m} \, \int_{Q_t} |q^m|^{2} \, dxd\tau \, ,\\
\text{ with } & \,\, \psi^m(t) = 1 + |v^{\infty}(\cdot,t)|_{L^{\infty}}^2+ |\tilde{p}^{\infty}(\cdot,t)|_{L^{\infty}} + |(q^{\infty})^{\prime}(\cdot,t)|_{L^{\infty}} + | (q^{m})^{\prime}(\cdot,t)|_{L^1} + |\tilde{\pi}^m(\cdot,t)|_{L^1} \, .
\end{align*}
\end{prop}
\begin{proof}
Throughout the proof, we re-denote $\rho = \rho^m$, $\mu = \mu^m$, etc.

We choose a nonincreasing, nonnegative function $L: \, \mathbb{R}_+ \rightarrow \mathbb{R}$ such that $L(s) = 1$ for all $0 \leq s \leq r_{\min} - 2\delta_0$ and $L(s) = 0$ for all $s \geq r_{\min}-\delta_0$, and such that that $|L^{\prime}(s)| \leq 4 \, \delta_0^{-1}$ for all $s \in \mathbb{R}_+$.

Consider the auxiliary function $\calg(x,t) := L(\varrho(x,t))$ defined in $Q$. With $1\leq  \calp < + \infty$ arbitrary, and with $\chi$ denoting the characteristic function of the set $\Omega_{a_0,-}(t)$, 
use of $|L| \leq \chi$ yields
\begin{align}\label{normofg}
 |\calg(\cdot,t)|_{L^\calp(\Omega)} = |L(\varrho(\cdot,t))|_{L^{\calp}(\Omega)} \leq |\{x \, : \, \varrho(x,t) < a_0\}|^{\frac{1}{\calp}} 
 \leq c(\delta_0) \, (\mathcal{E}^m(t))^{\frac{1}{\calp}} \, ,
\end{align}
where we employed Lemma \ref{erhobelowlem} (cf.\ \eqref{Erhobelow1bar}). Moreover, since $|L^{\prime}(\varrho)| \leq 4\delta_0^{-1} \, \chi$ and $|L(\varrho)| \leq  \chi$,
\begin{align}\begin{split}\label{jemensf}
 |(L^{\prime}(\varrho) \, \varrho - L(\varrho)) \, \divv v |_{L^2(\Omega)} \leq & c(\delta_0) \, |\chi \, \divv v|_{L^2(\Omega)} \\
 \leq &  |\chi \, \divv( v - v^{\infty}) |_{L^2(\Omega)} + |\divv v^{\infty}|_{L^{\infty}(\Omega)} \,|\Omega_{a_0,-}(t)|^{\frac{1}{2}} \\
\leq & |\divv (v-v^{\infty})|_{L^2(\Omega)}  + c(\delta_0) \, |\divv v^{\infty}|_{L^{\infty}(\Omega)} \, (\mathcal{E}^m(t))^{\frac{1}{2}}\, .
\end{split}
\end{align}
For all $1 < \cals \leq 6$ and for $\cals^{\prime}:=\cals/(\cals-1)$, $W^{1,\cals^{\prime}}(\Omega)$ embedds into $L^{\frac{3\cals}{(2\cals-3)^+}}(\Omega) \subseteq L^2(\Omega)$. We call $c_S(\cals)$ the embedding constant. This with the help of \eqref{Guldes} shows that 
\begin{align*}
|\partial_t \calg|_{[W^{1,\cals^{\prime}}(\Omega)]^*} \leq & |\calg \, v|_{L^{\cals}(\Omega)} + c_S(\cals) \, |(L^{\prime}(\varrho) \, \varrho - L(\varrho)) \, \divv v |_{L^2(\Omega)} \\
\leq & |\calg \, (v-v^{\infty})|_{L^{\cals}(\Omega)} + |v^{\infty}|_{L^{\infty}} \, |\calg|_{L^{\cals}(\Omega)} + c_S(\cals) \, |(L^{\prime}(\varrho) \, \varrho - L(\varrho)) \, \divv v |_{L^2(\Omega)} \, .
\end{align*}
We invoke \eqref{normofg} and \eqref{jemensf}, and we obtain that
\begin{align*}%\label{gtdeux}
  |\partial_t \calg|_{[W^{1,\cals^{\prime}}(\Omega)]^*}  \leq c \, \big(|v-v^{\infty}|_{W^{1,2}(\Omega)} + |v^{\infty}|_{L^{\infty}} \,(\mathcal{E}^m(t))^{\frac{1}{\cals}}+ |\divv v^{\infty}|_{L^{\infty}} \, (\mathcal{E}^m(t))^{\frac{1}{2}}\big) \, .
\end{align*}
 Let $\divv \eta = \calg(x,t) - (\calg)_{M}(t)$ according to Lemma \ref{BOGS}. Then, for all $1 \leq \calp < +\infty$ and $1< \cals \leq 6$, Lemma \ref{BOGS} implies that there are constants $C_b$ depending on $\calp$ and the Bogovski operator such that 
\begin{align}\begin{split}\label{lesbogos}
&  |\eta(\cdot,t)|_{W^{1,\calp}(\Omega)} \leq  C_b \, |\calg(\cdot,t)|_{L^\calp(\Omega)} \leq c(\delta_0) \, (\mathcal{E}^m(t))^{\frac{1}{\calp}}\, ,\\
  & |\partial_t \eta(\cdot,t)|_{L^{\cals}(\Omega)} \leq C_b \, |\partial_t(\calg-(\calg)_M)|_{[W^{1,\cals^{\prime}}(\Omega)]^*)} \\
 & \leq  c \, \big(|(v-v^{\infty})(\cdot,t)|_{W^{1,2}(\Omega)} +\big(|v^{\infty}|_{L^{\infty}}(\mathcal{E}^m(t))^{\frac{1}{\cals}} + |\divv v^{\infty}|_{L^{\infty}}(\mathcal{E}^m(t))^{\frac{1}{2}}\big)\big) \, .
\end{split}
 \end{align}
Inserting $\eta$ into the weak formulations of ($\overline{\text{ IBVP}}^m$) and (IBVP$^{\infty}$) yields
\begin{align*}
& \qquad \int_{Q_t}\tilde{\pi} \, (\calg-(\calg)_M) \, dxd\tau =  \int_{\Omega} \varrho v \cdot \eta \,dx\Big|_0^t \\
& + \int_{Q_t} \{\eta\cdot \nabla q_{N-1} + ( \mathbb{S}(\nabla v) - \varrho \, v \otimes v) : \, \nabla \eta - \varrho \, v \, \partial_t \eta + \varrho \, \eta_3\} \, dxd\tau \, ,\\
 & \qquad \int_{Q_t} \tilde{p}^{\infty} \, (\calg-(\calg)_M) \, dxd\tau = \int_{\Omega} \varrho^{\infty}\, v^{\infty} \cdot \eta \,dx\Big|_0^t \\
& + \int_{Q_t} \{\eta \cdot \nabla q^{\infty}_{N-1}+( \mathbb{S}(\nabla v^{\infty}) - \varrho^{\infty} \, v^{\infty} \otimes v^{\infty}) : \, \nabla \eta - \varrho^{\infty} \, v^{\infty} \, \partial_t \eta + \varrho^{\infty} \, \eta_3\} \, dxd\tau\, .
 \end{align*}
We subtract both identities and we proceed with some estimates. Use of \eqref{lesbogos} and Young's inequality yields
\begin{align}\label{PC0}
\left| \int_{\Omega} \eta \cdot \nabla (q_{N-1} - q_{N-1}^{\infty}) \, dx\right| \leq & |\eta|_{L^2} \, |\nabla (q-q^{\infty})|_{L^2} \leq c \, (\mathcal{E}^m(t))^{\frac{1}{2}} \,|\nabla (q-q^{\infty})|_{L^2} \nonumber \\
\leq & \epsilon \, |\nabla (q-q^{\infty})|_{L^2}^2 + \frac{c^2}{\epsilon} \, \mathcal{E}^m(t) \, .
\end{align}
Similarly, invoking \eqref{lesbogos} and Poincar\'{e}'s inequality,
\begin{align}\label{PC1}
\left|\int_{\Omega} \mathbb{S}(\nabla (v-v^{\infty})) : \, \nabla \eta \, dx\right| \leq & c \, |v^{\infty}-v|_{W^{1,2}(\Omega)} \, |\nabla \eta|_{L^2(\Omega)} \leq c \, |\nabla(v^{\infty}-v)|_{L^{2}(\Omega)} \, (\mathcal{E}^m(t))^{\frac{1}{2}}  \nonumber \\
 \leq & \epsilon \, |\nabla (v^{\infty}-v)|_{L^{2}(\Omega)}^2 + \frac{c^2}{4\epsilon} \, \mathcal{E}^m(t)  \, .
\end{align}
We next consider the acceleration terms, where we first expand as follows:
\begin{align*}
 & \varrho \, v \otimes v - \varrho^{\infty} \, v^{\infty} \otimes v^{\infty} = 
 \varrho \, (v-v^{\infty}) \otimes (v-v^{\infty})
 \\
 & \quad + 
 (\varrho - \varrho^{\infty}) \, \Big(  (v-v^{\infty}) \otimes v^{\infty} + v^{\infty} \otimes (v-v^{\infty}) + v^{\infty}\otimes v^{\infty}\Big)
 \\
 & \quad + \varrho^{\infty}\, \Big(  (v-v^{\infty}) \otimes v^{\infty} + v^{\infty} \otimes (v-v^{\infty})\Big) \, .
\end{align*}
Now we derive a series of estimates. At first, with $s = 6\gamma/(2\gamma-3)$
\begin{align*}
& \quad \quad \left| \int_{\Omega} \varrho \, (v-v^{\infty}) \otimes (v-v^{\infty}) \, : \, \nabla \eta \, dx\right| \\
& \leq \Big|\varrho^{\frac{1}{2}}\Big|_{L^{2\gamma}} \,|\sqrt{\varrho} \, (v-v^{\infty})|_{L^2} \, |v-v^{\infty}|_{L^6(\Omega)} \, |\nabla \eta|_{L^s} \\
& \leq \|\varrho\|_{L^{\gamma,\infty}}^{\frac{1}{2}} \, (2\mathcal{E}^m(t))^{\frac{1}{2}} \, c_S \, |v-v^{\infty}|_{W^{1,2}} \, \|\nabla \eta\|_{L^{s,\infty}}
\leq C \, (\mathcal{E}^m(t))^{\frac{1}{2}} \, |\nabla (v-v^{\infty})|_{L^{2}} \, .
\end{align*}
Second, using Lemma \ref{oneterm} with $F = 1^N \otimes (v^{\infty} \cdot \nabla) \eta$,
\begin{align*}
 \left| \int_{\Omega} (\varrho - \varrho^{\infty}) \, (v-v^{\infty}) \otimes v^{\infty} \, : \, \nabla \eta \, dx\right| 
 & \leq  C \,   (\mathcal{E}^m(t))^{\frac{1}{2}}  \, |v-v^{\infty}|_{W^{1,2}} \, |v^{\infty}|_{L^{\infty}} \, \|\nabla \eta\|_{L^{3,\infty}}\\
& \leq C \, |v^{\infty}|_{L^{\infty}} \, (\mathcal{E}^m(t))^{\frac{1}{2}} \, |\nabla(v-v^{\infty})|_{L^{2}} \, .
\end{align*}
Next we decompose $ \Omega = \Omega_{a_0,b_0}(t) \cup \Omega_{a_0,b_0}^{\rm c}$, where $a_0 = r_{\min}-\delta_0$ and $b_0 := r_{\max} + \delta_0$ according to \eqref{Omegazwish}. On the set $\Omega_{a_0,b_0}(t)$, we have $\varrho \leq b_0$. Invoking also \eqref{lesbogos} and \eqref{enplus},
\begin{align*}
\left| \int_{\Omega_{a_0,b_0}(t)} (\varrho -\varrho_{\infty}) \, v^{\infty} \otimes v^{\infty} \, : \, \nabla \eta \, dx\right|
& \leq \sqrt{b_0} (1+b_0)^{\frac{\theta_0}{2}} \, |v^{\infty}|_{L^{\infty}}^2\, \Big|\frac{\varrho - \varrho^{\infty}}{\omega^{\frac{1}{2}}(\varrho,\varrho^{\infty})}\Big|_{L^{2}}  \, |\nabla \eta|_{L^2} \\
& \leq C \, |v^{\infty}|_{L^{\infty}}^2 \,   (\mathcal{E}^m(t))^{\frac{1}{2}} \, (\mathcal{E}^m(t))^{\frac{1}{2}} \, .
\end{align*}
Choosing $\calp = \gamma^{\prime}$ in \eqref{lesbogos}, use of Lemma \ref{erhobelowlem} implies that
\begin{align*}
 \left| \int_{\Omega_{a_0,b_0}^{\rm c}(t)} (\varrho - \varrho^{\infty}) \, v^{\infty} \otimes v^{\infty} \, : \, \nabla \eta \, dx\right| \leq & |v^{\infty}|_{L^{\infty}}^2 \, |(\varrho -\varrho^{\infty})\, \chi_{\Omega_{a_0,b_0}^{\rm c}(t)}|_{L^{\gamma}} \, |\nabla \eta|_{L^{\gamma^{\prime}}} \\
 \leq & c \,|v^{\infty}|_{L^{\infty}}^2 \, (\mathcal{E}^m(t))^{\frac{1}{\gamma}+\frac{1}{\gamma^{\prime}}} \, ,
\end{align*}
and by similar means $|\int_{\Omega} \varrho^{\infty} \, (v-v^{\infty}) \otimes v^{\infty}\, : \, \nabla \eta \, dx| \leq r_{\max} \, |v^{\infty}|_{L^3} \, |v-v^{\infty}|_{L^6} \, |\nabla \eta|_{L^2}$.
%\begin{align*}
%&  \left|\int_{\Omega} \varrho^{\infty} \, (v-v^{\infty}) \otimes v^{\infty}\, : \, \nabla \eta \, dx \right| \leq r_{\max} \, |v^{\infty}|_{L^3} \, |v-v^{\infty}|_{L^6} \, |\nabla \eta|_{L^2} \,.
%\end{align*}
Overall the acceleration term is estimated as
\begin{align}\label{PC2}
\left|\int_{\Omega} (\varrho \, v\otimes v - \varrho^{\infty} \, v^{\infty}\otimes v^{\infty}) \, : \, \nabla \eta \, dx\right| \leq  \epsilon \, |\nabla (v^{\infty}-v)|_{L^{2}(\Omega)}^2 + \frac{C^2}{\epsilon} \, |v^{\infty}|_{L^{\infty}}^2 \, \mathcal{E}^m(t)\,.
\end{align}
Next we consider the terms with time derivatives, and we decompose
\begin{align*}
 \varrho \, v - \varrho^{\infty} \, v^{\infty} = \varrho  \, (v-v^{\infty}) + (\varrho - \varrho^{\infty}) \, v^{\infty}  \, .
\end{align*}
With $|v^{\infty}|_{L^{\infty}_{\divv}} := |v^{\infty}|_{L^{\infty}} + |\divv v^{\infty}|_{L^{\infty}}$, use of \eqref{lesbogos} implies that
\begin{align*}
& \quad \left| \int_{\Omega_{b_0,+}^{\rm c}(t)} \varrho \, (v-v^{\infty}) \cdot \partial_t \eta \, dx\right| \leq \sqrt{b_0} \,  \, |\sqrt{\varrho} \, (v-v^{\infty})|_{L^2} \, |\partial_t \eta|_{L^2}\\
& \leq \sqrt{2b_0} \, (\mathcal{E}^m(t))^{\frac{1}{2}}(t) \, \, |\partial_t\eta|_{L^2} \leq c \, \sqrt{b_0} \, (\mathcal{E}^m(t))^{\frac{1}{2}} \,  \, (|\nabla (v-v^{\infty})|_{L^2} + |v^{\infty}|_{L^{\infty}_{\divv}} \, (\mathcal{E}^m(t))^{\frac{1}{2}})\\
& \leq \epsilon \, |\nabla (v-v^{\infty})|_{L^2}^2 +c\, \sqrt{b_0} \, \Big(\frac{c \, \sqrt{b_0}}{4\epsilon}+| v^{\infty}|_{L_{\divv}^{\infty}}\Big)  \, \mathcal{E}^m(t) \, ,\\
& \quad \left| \int_{\Omega_{b_0,+}^{\rm c}(t)} (\varrho - \varrho^{\infty}) \, v^{\infty} \cdot \partial_t \eta \, dx\right| \leq \sqrt{b_0}(1+b_0)^{\frac{\theta_0}{2}} \, \Big|\frac{\varrho - \varrho^{\infty}}{\omega^{\frac{1}{2}}(\varrho,\varrho^{\infty})}\Big|_{L^{2}} \, |v^{\infty}|_{L^{\infty}} \,  |\partial_t \eta|_{L^{2}}\\
& \leq c \, |v^{\infty}|_{L^{\infty}} \, (\mathcal{E}^m(t))^{\frac{1}{2}}  \, (|\nabla(v^{\infty}-v)|_{L^2} + | v^{\infty}|_{L_{\divv}^{\infty}} \,  (\mathcal{E}^m(t))^{\frac{1}{2}}) \, .
\end{align*}
Moreover, invoking \eqref{lesbogos} and Lemma \ref{erhobelowlem} again, with $\bar{\mathcal{E}} := \sup_{m} \sup_{0<t<\bar{\tau}} \mathcal{E}^m(t) < +\infty$,
\begin{align*}
& \quad \left| \int_{\Omega_{b_0,+}(t)} \varrho \, (v-v^{\infty}) \cdot \partial_t \eta \, dx\right| \leq \Big|\sqrt{\varrho}\, \chi_{\Omega_{b_0,+}(t)}\Big|_{L^{2\gamma}} \, |\sqrt{\varrho} \, (v^{\infty}-v)|_{L^2} \, |\partial_t \eta|_{L^{\frac{2\gamma}{\gamma-1}}}\\
& \leq c \, (\mathcal{E}^m(t))^{\frac{1}{2\gamma}} \, (\mathcal{E}^m(t))^{\frac{1}{2}} \, \big(|\nabla(v^{\infty}-v)|_{L^2} +(|v^{\infty}|_{L^{\infty}}\, (\mathcal{E}^m(t))^{\frac{\gamma-1}{2\gamma}} + |\divv v^{\infty}|_{L^{\infty}} \, (\mathcal{E}^m(t))^{\frac{1}{2}})\big) \\
& \leq \epsilon\, |\nabla (v^{\infty}-v)|_{L^2}^2 \, +c\, \Big(\frac{c\bar{\mathcal{E}}^{\frac{1}{\gamma}}}{4\epsilon} + |v^{\infty}|_{L^{\infty}}+\bar{\mathcal{E}}^{\frac{1}{2\gamma}} \, |\divv v^{\infty}|_{L^{\infty}} \Big) \, \mathcal{E}^m(t) \, .
\end{align*}
Here we used $2\gamma/(\gamma-1) \leq 6$ to employ \eqref{lesbogos} with $\cals = 2\gamma/(\gamma-1)$. In order to estimate the next contribution, we define: If $\gamma \leq 2$, then $\cals_1 := \gamma/(\gamma-1)$ and $\cals_2 = +\infty$; If $\gamma > 2$, then $\cals_1 = 2$ and $\cals_2 := 2\gamma/(\gamma-2)$. We have $1/\gamma+1/s_1+1/s_2 = 1$ and $1/\gamma+1/s_2 \geq 1/2$. By means of H\"older's inequality and \eqref{lesbogos} it follows that
\begin{align*}
& \quad \left| \int_{\Omega_{b_0,+}(t)} (\varrho - \varrho^{\infty}) \, v^{\infty} \cdot \partial_t \eta \, dx\right| \leq |(\varrho - \varrho^{\infty}) \, \chi_{\Omega_{b_0,+}(t)}|_{L^{\gamma}} \, |\Omega_{b_0,+}(t)|^{\frac{1}{s_2}} \, |v^{\infty}|_{L^{\infty}}\, |\partial_t \eta|_{L^{s_1}}\\
& \leq c \, (\mathcal{E}^m(t))^{\frac{1}{\gamma}+\frac{1}{s_2}} \, |v^{\infty}|_{L^{\infty}} \, \big(|\nabla(v^{\infty}-v)|_{L^2} +  (  |\divv v^{\infty}|_{L^{\infty}} \,  (\mathcal{E}^m(t))^{\frac{1}{2}} +|v^{\infty}|_{L^{\infty}} \,(\mathcal{E}^m(t))^{\frac{1}{s_1}})\big) \,\\
& \leq c \, |v^{\infty}|_{L^{\infty}} \, \big[\bar{\mathcal{E}}\Big]^{\frac{1}{\gamma}+\frac{1}{s_2}-\frac{1}{2}} \,  (\mathcal{E}^m(t))^{\frac{1}{2}} \, \Big(|\nabla(v^{\infty}-v)|_{L^2} + |\divv v^{\infty}|_{L^{\infty}} \, (\mathcal{E}^m(t))^{\frac{1}{2}}\Big) \\
& \quad + c \, |v^{\infty}|_{L^{\infty}}^2 \, \mathcal{E}^m(t) \, .
\end{align*}
% where we use that $\sup_m \sup_{0< t < \bar{\tau}} \mathcal{E}^m(t) < +\infty$. 
Hence, the integral with time derivatives obeys
\begin{align}\label{PC3}
\left|\int_{\Omega} (\varrho \, v - \varrho^{\infty}\, v^{\infty}) \, : \, \partial_t \eta \, dx\right| \leq  \epsilon \, |\nabla (v^{\infty}-v)|_{L^{2}(\Omega)}^2 + \frac{C^2}{\epsilon} \, |v^{\infty}|_{L^{\infty}_{\divv}}^2  \, \mathcal{E}^m(t) \, .
\end{align}
The term $|\int_{\Omega} (\varrho - \varrho^{\infty}) \, \eta_3 \, dx|$ is easily treated by similar means.
At last we consider the intergals involving the pressure. At first
\begin{align}\label{PC4}
\left| \int_{\Omega} \tilde{p}^{\infty} \, (\calg - (\calg)_M) \, dx\right| \leq 2\, |\tilde{p}^{\infty}|_{L^{\infty}} \, |\calg|_{L^1} \leq 2\, |\tilde{p}^{\infty}|_{L^{\infty}} \, |\Omega_{a_0,-}(t)| \leq c \, \mathcal{E}^m(t) \, .
\end{align}
Moreover, using that $(\calg)_M$ depends only on time, the bound \eqref{normofg} implies that
\begin{align}\label{PC5}
 \left| \int_{\Omega} \tilde{\pi}^{m} \,  (\calg)_M \, dx\right| \leq |\calg|_{L^1} \, |\tilde{\pi}^m|_{L^1} \leq c \,  |\tilde{\pi}^m|_{L^1} \, \mathcal{E}^m(t) \, ,
\end{align}
Hence, with the help of \eqref{PC0}, \eqref{PC1}, \eqref{PC2}, \eqref{PC3}, \eqref{PC4}, \eqref{PC5},
\begin{align}\begin{split}\label{ttshow}
\left| \int_0^t\int_{\Omega} \tilde{\pi} \, L(\varrho) \, dxd\tau\right| \leq & \epsilon \, (\|\nabla (v^{\infty}-v)\|_{L^2(Q_t)}^2 + (\|\nabla (q^{\infty}-q)\|_{L^2(Q_t)}^2)\\
& + \frac{C}{\epsilon} \,  \int_0^{t} (|v^{\infty}|_{L^{\infty}}^2 +|\nabla v^{\infty}|_{L^{\infty}} + |\tilde{p}^{\infty}|_{L^{\infty}}|+ |\tilde{\pi}^m|_{L^1}) \, \mathcal{E}^m(\tau) \, d\tau \, .
\end{split}
\end{align}
On the support of $L(\varrho)$, we know that $\varrho(x,t) \leq r_{\min} - \delta_0$. Hence $\varrho(x,t) - \varrho^{\infty}(x,t) \leq 0$. Since $\tilde{P}_m$ is increasing in the first argument (cf.\ Lemma \ref{pressandg}), we get $( \tilde{\pi} - \tilde{P}_m(\varrho^{\infty}, \, q)) \, L(\varrho) \leq 0$. We thus can write
\begin{align}\label{pp11}
  \tilde{\pi} \, L(\varrho) = & ( \tilde{\pi} - \tilde{P}_m(\varrho^{\infty}, \, q)) \, L(\varrho) + \tilde{P}_m(\varrho^{\infty}, \, q) \, L(\varrho) \nonumber\\
 =&  - |\tilde{\pi} - \tilde{P}_m(\varrho^{\infty}, \, q)| \, L(\varrho) + \tilde{P}_m(\varrho^{\infty}, \, q) \, L(\varrho) \, , \nonumber\\[0.1cm]
 \text{ implying that} \qquad
 - \tilde{\pi} \, L(\varrho) \geq & | \tilde{\pi} | \, L(\varrho) - 2 \, |\tilde{P}_m(\varrho^{\infty}, \, q)| \, L(\varrho) \, .
\end{align}
Since $\min\{\bar{\varrho}_{\max}-\varrho^{\infty}, \, \varrho^{\infty} - \bar{\varrho}_{\min}\} \geq \delta_0$, the properties of $\tilde{P}_m$ proved in Lemma \ref{pressandg} show that
$ |\tilde{P}_m(\varrho^{\infty},\, q)| \leq c \, (\ln (1/\delta_0)+ |q^{\prime}|  + |q|^2/m)$.
With the help of \eqref{pp11} we see that
\begin{align*}
  - \tilde{\pi} \, L(\varrho) \geq & |\tilde{ \pi}| \, L(\varrho) - C(\delta_0) \, (1 + |(q-q^{\infty})^{\prime}| + |(q^{\infty})^{\prime}|) \, L(\varrho) \\
  & -C(\delta_0) \, \frac{1}{m} \,  (1 + |q|)^{2} \, L(\varrho) \, .
\end{align*}
Then, invoking Young's inequality again, we conclude $0 < \epsilon < 1$ arbitrary that
\begin{align*}
- \tilde{\pi} \, L(\varrho) \geq & | \tilde{\pi}| \, L(\varrho) - \epsilon \, |q^{\prime}-(q^{\infty})^{\prime} - (q-q^{\infty})^{\prime}_M)|^2 \, L(\varrho) \\
  & - C(\delta_0) \, \Big(\frac{1}{\epsilon}  +  |(q-q^{\infty})^{\prime}_M| +|(q^{\infty})^{\prime}|+ \frac{|q|^2}{m}\Big) \, L(\varrho) \, .
\end{align*}
Since $|\Omega_{a_0,-}(t)| \leq c \mathcal{E}^m(t)$, we bound
\begin{align*}
& \int_{\Omega}  (\epsilon^{-1} +  |(q-q^{\infty})^{\prime}_M| + |( q^{\infty})^{\prime}|) \, L(\varrho) \, dx \leq c \, \Big(\epsilon^{-1} +  |(q-q^{\infty})^{\prime}|_{L^1} + |(q^{\infty})^{\prime}|_{L^{\infty}}\Big) \, \mathcal{E}^m(t) \, .
\end{align*}
We use also the Poincar\'{e} inequality for functions with mean value zero and, with $0 < \epsilon$ arbitrary, we conclude that
\begin{align*}
& \left| \int_{\Omega} \tilde{\pi} \, L(\varrho) \, dx\right| \geq \int_{\Omega} |\tilde{\pi}| \, L(\varrho) \, dx  - c_{\Omega} \, \epsilon \, \int_{\Omega} |\nabla (q-q^{\infty})|^2 \, dx \\
&- c \, ((4\epsilon)^{-1} +  |(q-q^{\infty})^{\prime}|_{L^1} + |(q^{\infty})^{\prime}|_{L^{\infty}}) \, \mathcal{E}^m(t)- \frac{C}{m}  \, \int_{\Omega} |q|^{2} \, dx\, ,
\end{align*}
which we integrate in time and combine with \eqref{ttshow}. Without essential modifications, we can obtain the same result with a function $L: \, \mathbb{R}_+ \rightarrow \mathbb{R}$ such that $L(s) = 1$ for all $ s \geq r_{\max} +2\delta_0 $ and $L(s) = 0$ for all $s \leq r_{\max}+\delta_0$. The support of $L(\varrho)$ is then contained in $\Omega_{b_0,+}(t)$.
%and such that that $|L^{\prime}(s)| \leq 4 \, \delta_0^{-1}$ for all $s \in \mathbb{R}_+$.
This proves the claim.
\end{proof}
Finally, we state two auxiliary estimates that shall also be of use in the proof of the convergence theorem.
\begin{lemma}\label{better}
We adopt the same assumptions as in Lemma \ref{pressandg} and Coro.\ \ref{l1boundandtozero}.
Let further $a_0$ and $b_0$ be defined as in \eqref{Omegazwish} and $0 < s_0 < r_0= \inf_{(x,t) \in Q_{\bar{\tau}}, \, i = 1,\ldots,N} \rho^{\infty}_i(x,t) $.
 Suppose that $\eta \in L^{\infty}(Q_{\bar{\tau}}; \, \{1^N\}^{\perp})$. Then, there are functions $\{o_m\}$ such that $o_m \rightarrow 0$ in $L^1(Q_{\bar{\tau}})$ and such that, for all $0< t < \bar{\tau}$ and $0<\epsilon<1$, 
 \begin{align*}
&  \int_0^t\int_{\Omega_{a_0,b_0}(\tau) \setminus B_{s_0,+}(\tau)} \Big| \sum_{i-1}^N \eta_i \, \frac{1}{\bar{M}_i} \, \ln \hat{x}_i(\rho^m) \Big| \, dxd\tau \leq \epsilon \, \int_{Q_t} |\nabla (q^m-q^{\infty})|^2 \, dxd\tau\\
  & \qquad + C \, \int_0^t \Big(\epsilon^{-1} \, \|\eta\|_{L^{\infty}}^2 + (|(q^{\infty})^{\prime}|_{L^{\infty}} + |(q^{m})^{\prime}|_{L^1}) \, \|\eta\|_{L^{\infty}} \Big) \, \mathcal{E}^m(\tau) \, d\tau + \|o_m\|_{L^1(Q_t)} \, ,
 \end{align*}
 where $B_{s_0,+}(t)$ is defined in \eqref{Omegainf} and $C$ is a constant depending on $a_0$, $b_0$ and $s_0-r_0$.
%\end{enumerate}
\end{lemma}
\begin{proof}
In the sets $\Omega_{a_0,b_0}(t)$ the density $\varrho$ is strictly positive. Hence, all densities $\rho_1, \ldots,\rho_N$ are positive almost everywhere for weak solutions of type-I (cf.\ Th.\ \ref{EXIWEAK}). Since $\pi^m$ is finite, the chemical potentials $\mu_i^m = \bar{g}_i^m(\pi^m) + (1/\bar{M}_i)\, \ln \hat{x}_i(\rho^m)$ are finite almost everywhere. If $\eta$ attains values in the orthogonal complement of $1^N$, we get
\begin{align*}
\sum_{i=1}^N \eta_i \, \frac{1}{\bar{M}_i} \, \ln \hat{x}_i(\rho^m) = & \eta \cdot (\mu^m - \bar{g}^m(\pi^m)) = \eta \cdot (\Pi q^m -   \bar{g}^m(\pi^m))\\
= & \eta \cdot \Big(\Pi^{\prime} q^m -   \big(\bar{g}^m(\pi^m) - \pi^m \, \bar{\calv}\big) - \tilde{\pi}^m \, \bar{\calv}\Big) \, ,
\end{align*}
where we used the identities $\tilde{\pi}^m = \pi^m - q_{N-1}^m$ and $\Pi q^m - q^m_{N-1} \, \bar{\calv} = \Pi^{\prime} q^m$.
Hence, over a set $\Omega_{a_0,b_0}(t)$ we obtain that
\begin{align*}
\Big|\sum_{i=1}^N \eta_i \, \frac{1}{\bar{M}_i} \, \ln \hat{x}_i(\rho^m) \Big|& \leq C \, |\eta| \, ( |(q^{m})^{\prime}| +  |\tilde{\pi}^m|
+ |\bar{g}^m(\pi^m) - \pi^m \, \bar{\calv}|) 
\, ,
\end{align*}
with a constant $C$ that depends only on $|\Pi|$ and $|\bar{\calv}|$. In order to estimate $|\tilde{\pi}^m|$, we rely on Lemma \ref{pressandg}, yielding
\begin{align*}
\Big|\sum_{i=1}^N \eta_i \, \frac{1}{\bar{M}_i} \, \ln \hat{x}_i(\rho^m) \Big| & \leq C \,  |\eta| \, \Big(1+ |(q^{ m})^{\prime}| + m^{-1} \, |q^m|^{2} + |\bar{g}^m(\pi^m) - \pi^m \, \bar{\calv}|\Big) \\
& \leq C \,  |\eta| \, (  1 + |(q^{ m}-q^{\infty})^{\prime}| + |(q^{\infty})^{\prime}| + o_m) \, ,\\
\text{ with }  o_m(x,t) :=& m^{-1} \, |q^{m}(x,t)|^{2}
 + \chi_{\Omega_{a_0,b_0}(t)} \, |\bar{g}^m(\pi^m(x,t)) - \pi^m(x,t) \, \bar{\calv}| \, .
 \end{align*}
 To see that $o_m \rightarrow 0$ in $L^1(Q_{\bar{\tau}})$, we invoke \eqref{lesbounds} and the Corollary \ref{l1boundandtozero}.
With the help of Young's inequality, we obtain that
\begin{align*}
 \Big|\sum_{i=1}^N \eta_i \, \frac{1}{\bar{M}_i} \, \ln \hat{x}_i(\rho^m) \Big| \leq & \epsilon \, |q^m-q^{\infty} - (q^m-q^{\infty})_M|^2 + \frac{C^2}{\epsilon} \, |\eta|^2\\
 & 
 + C \, |\eta| \, (1+|(q^{\infty})^{\prime}| + |((q^{m})^{\prime}-(q^{\infty})^{\prime})_M| + o_m) \, .
\end{align*}
We integrate over $\Omega_{a_0,b_0}(\tau) \setminus B_{s_0,+}(\tau)$ for $\tau \in ]0,t[$. Observe that
\begin{align*}
 & \int_{\Omega_{a_0,b_0}(\tau) \setminus B_{s_0,+}(\tau)} |\eta| \, (1+|(q^{\infty})^{\prime}|+ |((q^{m})^{\prime}-(q^{\infty})^{\prime})_M|) \, dx \\
 \leq & |\eta(\cdot,\tau)|_{L^{\infty}} \, (1+|(q^{\infty})^{\prime}(\cdot,\tau)|_{L^{\infty}} + |(q^{m})^{\prime}(\cdot,\tau)-(q^{\infty})^{\prime}(\cdot,\tau)|_{L^1}) \, |\Omega_{a_0,b_0}(\tau) \setminus B_{s_0,+}(\tau)| \, .
 \end{align*}
 Invoking Lemma \ref{erhobelowlem} it follows that
 \begin{align*}
& \int_{\Omega_{a_0,b_0}(\tau) \setminus B_{s_0,+}(\tau)} |\eta| \, (1+|(q^{\infty})^{\prime}|+ |((q^{m})^{\prime}-(q^{\infty})^{\prime})_M|) \, dx
 \\
 \leq &  c \,  |\eta(\cdot,\tau)|_{L^{\infty}} \, (1+|(q^{\infty})^{\prime}(\cdot,\tau)|_{L^{\infty}}+ |(q^{m}(\cdot,\tau)-q^{\infty}(\cdot,\tau))^{\prime}|_{L^1}) \, \mathcal{E}^m(\tau)  \, .
\end{align*}
Finally, we invoke the Poincar\'e inequality to control $\int_{Q_t} |q-q^{\infty}-(q-q^{\infty})_M|^2$ using the gradient. Integrating in time, we are done.
\end{proof}

\subsection{The convergence proof for weak solutions}\label{RelIntIneqWeak}

\addtocontents{toc}{\protect\setcounter{tocdepth}{3}}

As the mass densities do not need being bounded from below as assumed in Prop.\ \ref{calculheuri}, the two last integrals in the expression of the remainder $\mathcal{R}^m(t)$ are not finite.
In order to obtain an identity similar to Prop.\ \ref{calculheuri} in the weak solution context, we must provide a substitute for the vector of chemical potentials. The techniques to employ would be different for the two types of weak solutions distinguished in Section \ref{ShortSurvey}.
Here we however discuss only the case of a uniformly positive Onsager operator (cf.\ (B3$^{\prime}$), weak solutions of type-I). Under the assumption (B3$^\prime$), the variable $q$ belongs to $L^2(0,\bar{\tau}; \, W^{1,2}(\Omega;\mathbb{R}^{N-1}))$ and the identity $J = - M(\rho) \, \Pi \, \nabla q$.
% \begin{align}\label{PimuFO}
%  J = - M(\rho) \, \Pi \, \nabla q  \, ,
% \end{align}
holds over the domain. Let $0<\delta_0 < 1$, $a_0 := r_{\min} - \delta_0$ and $b_0 := r_{\max} + \delta_0$. The sets $\Omega_{a_0,b_0}(t) $ and $ B_{s_0,+}(t) $ are
 as in \eqref{Omegazwish}, \eqref{Omegainf}. Following the appendix, Section \ref{SStorm}, we can derive the following variant of Prop.\ \ref{calculheuri}.
\begin{prop}\label{calculFO}
Let $(\rho, \, v)$ satisfy \emph{($\overline{\text{ IBVP}}^m$)} and $(\rho^{\infty}, \, p^{\infty}, \, v^{\infty})$ satisfy \emph{(IBVP$^{\infty}$)}. Then
\begin{align*} 
& \mathcal{E}^m(t) + \int_0^t \mathcal{D}^{\rm Visc}(v^m \, |\, v^{\infty})(\tau) \, d\tau +  \int_{0}^t\int_{\Omega}  \Pi^{\sf T} M(\rho) \Pi \nabla (q-q^{\infty}) \, : \, \nabla (q-q^{\infty}) \, dxd\tau \\
& \leq  \int_0^t\int_{\Omega_{a_0,b_0}(\tau)}  (\partial_t v^{\infty} + \divv (\rho^{\infty}\, v^{\infty})) \cdot (\bar{g}^m(\pi)- \pi \, \bar{\calv}) - (\rho \cdot \bar{\calv} - 1) \, (\partial_t p^{\infty} + v^{\infty} \cdot \nabla p^{\infty})dxd\tau \\
& \quad  
+  \int_{\Omega} p^{\infty}(x,\cdot) \, (\rho^m(x,\cdot) \cdot \bar{\calv} - 1)\, dx\Big|^t_0 + \mathcal{E}^m(0)  +  \int_{0}^t (\mathcal{R}^m(\tau)+\mathdutchcal{E}^m(\tau)) \, d\tau \, ,
\end{align*}
with the remainder $\mathcal{R}^m$ and the additional error $\mathdutchcal{E}^m$. We have $\mathcal{R}^m = \sum_{i=1}^4 \mathcal{R}^{m,i}$, where $\mathcal{R}^{m,i}$ is as in Proposition \ref{calculheuri} for $i=1,2$ and
\begin{align*}
& \mathcal{R}^{m,3}(t) :=  - \int_{\Omega} (M(\rho) - M(\rho^{\infty})) \nabla \mu^{\infty} \cdot (\Pi\nabla q-\mu^{\infty}) \, dx\\
& \mathcal{R}^{m,4}(t) :=   \int_{B_{s_0,+}(t)} (\partial_t \rho^{\infty}+\divv( \rho^{\infty} \, v^{\infty})) \cdot \big(D\bar{k}(\rho) - D\bar{k}(\rho^{\infty}) - D^2\bar{k}(\rho^{\infty}) \, (\rho - \rho^{\infty}) \big)\, dx \, , 
\end{align*}
and, with $\eta^{\infty} := \partial_t \rho^{\infty}+\divv( \rho^{\infty} \, v^{\infty})$,
\begin{align*}
& 
\mathdutchcal{E}^m(t) :=  \int_{\Omega_{a_0,b_0}(t) \setminus B_{s_0,+}(t)} \eta^{\infty} \cdot \big(D\bar{k}(\rho) - D\bar{k}(\rho^{\infty}) - D^2\bar{k}(\rho^{\infty}) \, (\rho - \rho^{\infty}) \big) \, dx  \\
& + \int_{\Omega_{a_0,b_0}^{\rm c}(t)}\eta^{\infty} \cdot  \big(\Pi^{\prime} (q - q^{\infty}) - D^2\bar{k}(\rho^{\infty})\, (\rho-\rho^{\infty})\big) + (\tilde{p}^{\infty}-\tilde{\pi}) \, \divv v^{\infty} \, dx \, 
 \, .
\end{align*}
\end{prop}
Next we want to prove convergence. The arguments are however more technical in this case than in Section \ref{heuri}, since the pressure is not bounded.
We at first consider the remainder $\mathcal{R}^m$, and we notice that it essentially possesses the same expression as in Prop.\ \ref{calculheuri}. For the estimate of $\mathcal{R}^{m,1}$, employing the Lemma \ref{oneterm} with $F = 1^N \otimes \partial_tv^{\infty}$, we obtain that
\begin{align*}
& \Big|\int_{\Omega} (\varrho-\varrho^{\infty}) \, \partial_t v^{\infty} \cdot (v^{\infty}-v) \, dx\Big| \leq C \, |\partial_t v^{\infty}|_{L^{3}} \,|\nabla (v^{\infty} - v)|_{L^2} \,   (\mathcal{E}^m(t))^{\frac{1}{2}} \, .
\end{align*}
Invoking \eqref{KKorn} and the Young inequality, we obtain the estimate \eqref{menschenae}, and \eqref{menschenae2} is also proved as above.
The integral $| \int_{\Omega_{a_0,-}^{\rm c}(t)} \varrho^{\infty} \, \big((v^{\infty}-v) \cdot \nabla\big) v^{\infty} \cdot(v^{\infty} - v) \, dx|$ can be controlled exactly as in \eqref{menschenae3}. Moreover, use of \eqref{Erhobelow1bar} shows that
\begin{align*}
& \left| \int_{\Omega_{a_0,-}(t)} \varrho^{\infty} \, \big((v^{\infty}-v) \cdot \nabla\big) v^{\infty} \cdot(v^{\infty} - v) \, dx \right| \leq  r_{\max} \, |\nabla v^{\infty}|_{L^{6}} \,|v^{\infty} - v|_{L^6}^2 \, |\Omega_{a_0,-}(t)|^{\frac{1}{2}} \\
& \leq r_{\max} \, |\nabla v^{\infty}|_{L^{6}} \, (|v^{\infty}|_{W^{1,2}}+|v|_{W^{1,2}}) \,|v^{\infty} - v|_{W^{1,2}} \, |\Omega_{a_0,-}(t)|^{\frac{1}{2}}\\
& \leq \epsilon \, \int_{\Omega} \mathbb{S}(\nabla(v^{\infty} - v)) \, : \, \nabla(v^{\infty} - v) \, dx + \frac{C}{\epsilon} \, |\nabla v^{\infty}|_{L^{6}}^2 \, (|v^{\infty}|_{W^{1,2}}+|v|_{W^{1,2}})^2 \, \mathcal{E}^m(t) \, .
\end{align*}
Overall, the first remainder $\mathcal{R}^{m,1}$ still obeys \eqref{gouldenae1}, with $\epsilon > 0$ arbitrary, the same function $\psi(t)$, but a possibly different constant $C$. As to $\mathcal{R}^{m,2}$, since \eqref{gouldenae1} holds true, the contributions  
$ \int_{\Omega}(v^{\infty} - v) \, (\rho -\rho^{\infty}) \, : \, \nabla \mu^{\infty} \, dx$ and $ \int_{\Omega}(v^{\infty} - v) \, b \, (\varrho -\varrho^{\infty})  \, dx$ are estimated as above, yielding \eqref{gouldenae2}.

For the estimate of $\mathcal{R}^{m,3}$, the argument must be somewhat updated.
Consider the symmetric matrix $\widetilde{M}(\rho) = \Pi^{\sf T} M(\rho) \, \Pi$ with $\Pi$ from \eqref{Pi}.
For $w \in \mathbb{R}^{N-1}$ and $k = 1,\ldots,N-1$, $\eta^k \cdot \mathcal{P} \Pi w = \eta^k \cdot \Pi w = w_k$ due to the choice of $\Pi$ and of $\{\eta^1,\ldots,\eta^{N-1}\}$ being the dual basis. Hence, with $c_{\eta} := \sup_{k=1,\ldots,N-1} |\eta^k|_{\infty}$ we see that $|w| \leq  c_{\eta} \, |\mathcal{P} \Pi w|$. We employ (B3$^{\prime}$) to see that 
\begin{align}\label{tildempos}
 \Pi^{\sf T} M(\rho) \, \Pi \, w \cdot w \geq \frac{\lambda_0}{c_{\eta}} \, |w|^2 \quad \text{ for all } \quad w \in \mathbb{R}^{N-1} 
\end{align}
and $\widetilde{M}(\rho)$ is uniformly positive definite and invertible.
We have $$(M(\rho) - M(\rho^{\infty})) \, \nabla \mu^{\infty} \, : \, (\Pi \nabla q - \nabla \mu^{\infty}) = (\widetilde{M}(\rho) - \widetilde{M}(\rho^{\infty})) \, \nabla q^{\infty} \, : \, \nabla (q - q^{\infty}) \, .$$
Arguing first as in Section \ref{calculheuri}, we rely on the smallest positive eigenvalue of $M(\rho)$ and show that (cp.\ \eqref{menschenae4})
\begin{align*}%\label{menschenae4pr}
 & | (\widetilde{M}(\rho) - \widetilde{M}(\rho^{\infty})) \, \nabla q^{\infty} \cdot \nabla (q-q^{\infty})|\nonumber\\
&  \leq  \frac{c \, \|\partial M\|^2_{\infty}}{4\,\epsilon \, \lambda_0}  \, |\rho - \rho^{\infty}|^2_{\infty} \, |\nabla q^{\infty}|^2 + \epsilon \, \widetilde{M}(\rho) \, \nabla (q-q^{\infty}) \cdot \nabla (q-q^{\infty}) \, \\
& \leq\frac{\omega_0 \, \|\partial M\|^2_{\infty}}{4 \, \epsilon \, \lambda_{0}} \, |\nabla q^{\infty}|^2 \, \frac{|\rho-\rho^{\infty}|^2}{\omega(\varrho, \, \varrho^{\infty})} 
 + \epsilon \, \widetilde{M}(\rho) \, \nabla (q-q^{\infty}) \cdot \nabla (q-q^{\infty}) \, , \nonumber
\end{align*}
where we used the fact that $\max\{\varrho, \, \varrho^{\infty}\} \leq b_0$, hence $\omega(\varrho,\varrho^{\infty}) \leq b_0(1+b_0)^{\theta_0} =: \omega_0 $ in $\Omega_{b_0,-}(t)$. In order to also treat the integral over $\Omega_{b_0,+}(t)$ we now observe that $\|\widetilde{M}^{-1}(\rho)\|_2 \leq c_{\eta}/\lambda_0$. Hence, with $\bar{\lambda}$ from \eqref{b1prime} we can also bound
\begin{align*}
& |(\widetilde{M}(\rho) - \widetilde{M}(\rho^{\infty})) \, \nabla q^{\infty} \cdot \nabla (q - q^{\infty})| = | ({\rm I} - \widetilde{M}^{-1}(\rho) \widetilde{M}(\rho^{\infty})) \, \nabla q^{\infty} \cdot \widetilde{M}(\rho) \nabla (q - q^{\infty})|\\
& \leq \Big(\widetilde{M}(\rho) \,({\rm I} - \widetilde{M}^{-1}(\rho) \widetilde{M}(\rho^{\infty})) \, \nabla q^{\infty} \cdot({\rm I} - \widetilde{M}^{-1}(\rho) \widetilde{M}(\rho^{\infty})) \nabla q^{\infty}\Big)^{\frac{1}{2}} \\
& \qquad \times  \, \Big(\widetilde{M}(\rho) \, \nabla (q - q^{\infty}) \cdot \nabla (q - q^{\infty})\Big)^{\frac{1}{2}}\\
& \leq  \|\widetilde{M}(\rho)\|^{\frac{1}{2}} \, \Big(1+c_{\eta}\frac{\|\widetilde{M}(\rho^{\infty})\|}{\lambda_0}\Big) \, |\nabla q^{\infty}| \,  (\widetilde{M}(\rho) \, \nabla (q - q^{\infty}) \cdot \nabla (q - q^{\infty}))^{\frac{1}{2}}\\
& \leq |\Pi|\, \sqrt{\bar{\lambda}} \, \Big(1+c_{\eta} \, |\Pi|^2 \, (1+r_{\max})\frac{\bar{\lambda}}{\lambda_0}\Big)\,  \sqrt{1+|\rho|} \, |\nabla q^{\infty}| \,  (\widetilde{M}(\rho) \, \nabla (q - q^{\infty}) \cdot \nabla (q - q^{\infty}))^{\frac{1}{2}} \, .
\end{align*}
Hence, invoking also \eqref{Erhobelow1bar} 
\begin{align*}
& \int_{ \Omega_{b_0,+}(t)} |(\widetilde{M}(\rho) - \widetilde{M}(\rho^{\infty})) \, \nabla q^{\infty} \cdot \nabla (q - q^{\infty})| \, dx \\ & \leq C\, \int_{\Omega_{b_0,+}(t)} \sqrt{1+|\rho|}   \, |\nabla q^{\infty}| \,  (\widetilde{M}(\rho) \, \nabla (q - q^{\infty}) \cdot \nabla (q - q^{\infty}))^{\frac{1}{2}} \, dx\\
& \leq C \, |\nabla q^{\infty}|_{L^{\infty}} \, \Big(\int_{\Omega} \widetilde{M}(\rho) \, \nabla (q- q^{\infty}) \cdot \nabla (q - q^{\infty}) \, dx\Big)^{\frac{1}{2}} \, \Big(\int_{\Omega_{b_0,+}(t)} 1+|\rho| \, dx\Big)^{\frac{1}{2}}\\
& \leq \epsilon \, \int_{\Omega} \widetilde{M}(\rho) \, \nabla (q- q^{\infty}) \cdot \nabla (q - q^{\infty}) \, dx + \frac{C^{\prime}}{4\epsilon} \, |\nabla q^{\infty}|_{L^{\infty}}^2 \, \mathcal{E}^m(t) \, ,
\end{align*}
Overall, the estimation of $\mathcal{R}^{m,3}$ yields \eqref{gouldenae3} again, possibly with a different constant.

For estimating $\mathcal{R}^{m,4}$, we exploit that  $D\bar{k}(\rho) - D\bar{k}(\rho^{\infty}) - D^2\bar{k}(\rho^{\infty}) \cdot (\rho-\rho^{\infty})$ is integrated on $B_{s_0,+}(t)$ in which $\inf_{(x,t), i=1,\ldots,N} \rho_i(x,t) \geq s_0 > 0$ and $\varrho(x,t) \leq b_0$. Thus, the same procedure as in Section \ref{calculheuri} can be applied for estimating this integral, to the result \eqref{gouldenae4}.

Overall, the estimate \eqref{jelacite} is also valid and together with the inequality of Prop.\ \ref{calculFO}, we get
\begin{align*}
& \mathcal{E}^m(t) + (1-2\epsilon)\,\int_0^t \mathcal{D}^{\rm Visc}(v^m\,|\, v^{\infty})(\tau) + \int_{\Omega}  \widetilde{M}(\rho^m) \nabla (q^m-q^{\infty}) : \, \nabla (q^m-q^{\infty}) \, dx\, d\tau \\
& \leq \int_{0}^t \int_{\Omega_{a_0,b_0}(\tau)}\partial_t\rho^{\infty} + \divv (\rho^{\infty} \, v^{\infty}) \cdot (\bar{g}^m(\pi) - \bar{\calv} \, \pi^m) - (\rho^m \cdot \bar{\calv}-1) \, (\partial_t p^{\infty} + v^{\infty} \cdot\nabla p^{\infty})dxd\tau\\
% & \quad - \int_0^t \int_{\Omega} (\rho^m \cdot \bar{\calv}-1) \, (\partial_t p^{\infty} + v^{\infty} \cdot\nabla p^{\infty}) \,  dxd\tau
& \quad + \mathcal{E}^m(0) +  \int_{\Omega} p^{\infty}(x,\cdot) \, (\rho^m(x,\cdot) \cdot \bar{\calv} - 1)\, dx\Big|^t_0  
+ \frac{C}{\epsilon} \, \int_{0}^t \psi^m(\tau) \, \mathcal{E}^m(\tau) \, d\tau + \int_0^t \mathdutchcal{E}^m(\tau) \, d\tau \,.
\end{align*}
Here $\widetilde{M}(\cdot) = \Pi^{\sf T} M(\cdot)\Pi$, and the function $\psi^m$ obeys \eqref{jelaciteII}, so that $\{\psi^m\}$ is uniformly bounded in $L^1(0,\bar{\tau})$.

It remains to show how to control the error term $\mathdutchcal{E}^m$. In the estimates, we abbreviate $\eta^{\infty} := \partial_t \rho^{\infty} + \divv (\rho^{\infty} \, v^{\infty})$, which is a $L^{\infty}-$field with values in the orthogonal complement of $1^N$. We next consider successively every contribution to $\mathdutchcal{E}^m$. At first, invoking Young's inequality, and noting that Lemma \ref{erhobelowlem} implies that $|\Omega_{a_0,b_0}^{\rm c}(t)| \leq c \, \mathcal{E}^m(t)$, we see that
\begin{align*}
& \left| \int_0^t\int_{\Omega_{a_0,b_0}^{\rm c}(\tau)} \eta^{\infty} \cdot  \Pi^{\prime} (q - q^{\infty}) \, dxd\tau\right| 
\leq  \epsilon \, \int_0^t\int_{\Omega} |q-q^{\infty} - (q-q^{\infty})_M|^2 dxd\tau\\
 + & \int_0^t (\frac{1}{4\epsilon} \, |\eta^{\infty}|_{L^{\infty}}^2 + |q^{\prime}-(q^{\infty})^{\prime}|_{L^1} \, |\eta^{\infty}|_{L^{\infty}}) \, |\Omega_{a_0,b_0}^{\rm c}(\tau)| \, d\tau  \\
\leq & \epsilon^{\prime}  \, \int_{Q_t} |\nabla (q-q^{\infty})|^2 dx +c\,  \int_0^t (\frac{1}{4\epsilon^\prime} \, |\eta^{\infty}|_{L^{\infty}}^2 + |q^{\prime}-(q^{\infty})^{\prime}|_{L^1} \, |\eta^{\infty}|_{L^{\infty}}) \, \mathcal{E}^m(\tau)  \, d\tau \, .
\end{align*}
Second, we can use Lemma \ref{erhobelowlem} again to show that
\begin{align*}
\left|\int_{0}^t \int_{\Omega_{a_0,b_0}^{\rm c}(t)} \eta^{\infty} \cdot D^2\bar{k}(\rho^{\infty})\, (\rho-\rho^{\infty})  \, dxd\tau\right| \leq & c(r_0) \, \int_{0}^t |\eta^{\infty}|_{L^{\infty}} \, \int_{\Omega_{a_0,b_0}^{\rm c}(\tau)} |\rho-\rho^{\infty}| \, dxd\tau \\
\leq & c \, \int_{0}^t |\eta^{\infty}|_{L^{\infty}} \, \mathcal{E}^m(\tau) \, dxd\tau  \, .
\end{align*}
For all $0 < \epsilon < 1$, the Proposition \ref{pressla} implies that
\begin{align*}
& \left|  \int_{0}^t \int_{\Omega_{a_0,b_0}^{\rm c}(\tau)}  \tilde{\pi} \, \divv v^{\infty} \, dxd\tau \right| \leq \|\divv v^{\infty}\|_{L^{\infty}} \, \int_0^t \int_{\Omega_{a_0,b_0}^{\rm c}(\tau)}  |\tilde{\pi}|\, dxd\tau\\
&  \leq \|\divv v^{\infty}\|_{L^{\infty}} \, \epsilon \, \int_{Q_t} |\nabla (q^m-q^{\infty})|^2 + |\nabla v^{m}-v^{\infty}|^2 \, dxd\tau \\
& \quad +   \|\divv v^{\infty}\|_{L^{\infty}} \, C \, \epsilon^{-1} \, \int_0^t \psi^m(\tau) \, \mathcal{E}^m(\tau) \, d\tau  + C \, \|\divv v^{\infty}\|_{L^{\infty}} \, \frac{1}{m} \, \|q^m\|_{L^2(Q_T)}^2\, .
\end{align*}
With \eqref{Erhobelow1bar}, we next see that
\begin{align*}
 \left|  \int_{0}^t \int_{\Omega_{a_0,b_0}^{\rm c}(t)}  \tilde{p}^{\infty} \, \divv v^{\infty} \, dxd\tau \right|
 \leq c \, \int_0^t |\tilde{p}^{\infty}|_{L^{\infty}} \, |\divv v^{\infty}|_{L^{\infty}} \, \mathcal{E}^m(\tau) \, d\tau \, .
\end{align*}
It was shown in Lemma \ref{better} that
\begin{align*}
& \int_0^t \int_{\Omega_{a_0,b_0}(\tau) \setminus B_{s_0,+}(\tau)} \eta^{\infty}\cdot D\bar{k}(\rho) \, dxd\tau \leq \epsilon \, \int_{Q_t} |\nabla (q^m-q^{\infty})|^2 \, dxd\tau\\
  & \qquad +  C \, \int_0^t \Big(\epsilon^{-1} \, \|\eta^{\infty}\|_{L^{\infty}}^2 + (|(q^{\infty})^{\prime}|_{L^{\infty}} + |(q^{m})^{\prime}|_{L^1}) \, \|\eta^{\infty}\|_{L^{\infty}} \Big) \, \mathcal{E}^m(\tau) \, d\tau + \|o_m\|_{L^1(Q_t)} \, .
\end{align*}
Finally, invoking Lemma \ref{erhobelowlem}, we bound
\begin{align*}
& \int_{\Omega_{a_0,b_0}(t) \setminus B_{s_0,+}(t)} \eta^{\infty} \cdot \big(- D\bar{k}(\rho^{\infty}) - D^2\bar{k}(\rho^{\infty}) \, (\rho - \rho^{\infty}) \big) \, dx\\
 & \quad \leq c(r_{0},r_{\max},b_0) \, |\eta^{\infty}|_{L^{\infty}} \, |\Omega_{a_0,b_0}(t) \setminus B_{s_0,+}(t)| \leq c \, |\eta^{\infty}|_{L^{\infty}} \, \mathcal{E}^m(t) \, .
\end{align*}
Thus the error satisfies the bound 
\begin{align*}
& \int_{0}^t \mathdutchcal{E}^m(\tau) \, d\tau \leq  \epsilon \, \int_{Q_t} \widetilde{M}(\rho^m)\nabla (q^m-q^{\infty}) : \, \nabla (q^m-q^{\infty}) \, dx + \frac{C}{\epsilon}\,  \int_{0}^t \tilde{\psi}^m(\tau) \, \mathcal{E}^m(\tau) \, d\tau \, ,\\
& \tilde{\psi}^m(\tau) :=  |\eta^{\infty}|_{L^{\infty}}^2 + (1+|(q^{m})^\prime|_{L^1} +|(q^{\infty})^\prime|_{L^1}) \, |\eta|_{L^{\infty}}) \\
& \phantom{\tilde{\psi}^m(\tau) :=} + \|\divv v^{\infty}\|_{L^{\infty}} \,\psi^m(\tau)  + |\tilde{p}^{\infty}|_{L^{\infty}} \, |\divv v^{\infty}|_{L^{\infty}} \, .
\end{align*}
Hence, with $\tilde{o}_m = C \, (\|o_m\|_{L^1(Q)} + \|\divv v^{\infty}\|_{L^{\infty}} \, m^{-1} \, \|q^m\|^2_{L^2(Q)})$, we see that
\begin{align*}
& \mathcal{E}^m(t) + (1-3\epsilon)\,\int_0^t \Big(\mathcal{D}^{\rm Visc}(v^m\,|\, v^{\infty})(\tau) + \int_{\Omega}  \widetilde{M}(\rho^m) \nabla (q^m-q^{\infty}) : \, \nabla (q^m-q^{\infty}) \, dx\Big)\, d\tau \\
& \leq \frac{C}{\epsilon} \, \int_{0}^t \tilde{\psi}^m(\tau) \, \mathcal{E}^m(\tau) \, d\tau + \tilde{o}_m+\int_{0}^t \int_{\Omega_{a_0,b_0}(t)} (\partial_t\rho^{\infty} + \divv (\rho^{\infty} \, v^{\infty})) \cdot (\bar{g}^m(\pi) - \bar{\calv} \, \pi^m) dxd\tau\\
&  - \int_0^t \int_{\Omega} (\rho^m \cdot \bar{\calv}-1) \, (\partial_t p^{\infty} + v^{\infty} \cdot\nabla p^{\infty}) \,  dxd\tau+ \mathcal{E}^m(0) +  \int_{\Omega} p^{\infty}(x,\cdot) \, (\rho^m(x,\cdot) \cdot \bar{\calv} - 1)\, dx\Big|^t_0 \, .
\end{align*}
Next we want to establish the convergence of the right-hand side. We first consider the initial relative energy $\mathcal{E}^m(0)$. Since $v^{m}(0) = v^0 = v^{\infty}(0)$, 
\begin{align*}
  \mathcal{E}^m(0) 
 = & \int_{\Omega} \bar{g}^m( \hat{\pi}^m(\rho^{0,m})) \cdot \rho^{0,m} -   \hat{\pi}^m(\rho^{0,m}) \, dx \\
 & + \int_{\Omega} \bar{k}(\rho^{0,m}) - \bar{k}(\rho^{0,\infty}) - D\bar{k}(\rho^{0,\infty}) \, (\rho^{0,m}-\rho^{0,\infty}) \, dx \, .
\end{align*}
Since $\rho^{0,\infty}$ and $\rho^{0,m} \rightarrow \rho^{0,\infty}$ in $L^1(\Omega)$, the second term tends to zero.
Recalling that $1+\pi^{0,m}/m = \widehat{p}(\rho^{0,m})$ we can show that
\begin{align*}
& \bar{g}^m( \hat{\pi}^m(\rho^{0,m})) \cdot \rho^{0,m} -   \hat{\pi}^m(\rho^{0,m}) \\
& \quad  = \int_0^1 \bar{g}^{\prime}(1 + \lambda \,[ \widehat{p}(\rho^{0,m})-1]) - \bar{g}^{\prime}(\widehat{p}(\rho^{0,m})) \,d\lambda \cdot\rho^{0,m} \, (\widehat{p}(\rho^{0,m})-1) \, .
\end{align*}
Since in Theorem \ref{rigolo}, we assume that $\rho^{0,m}$ is uniformly positive and bounded, and that $\widehat{p}(\rho^{0,m})-1$ converges to zero in $L^1(\Omega)$, we easily see that $\mathcal{E}^m(0) \rightarrow 0$.

Invoking \eqref{ilstendent1}, we can further show that
\begin{align*}
 \sup_{0 < t < \bar{\tau}} \left|\int_{\Omega} p^{\infty}(x,t) \, (\rho^m(x,t) \cdot \bar{\calv} - 1) \, dx\right| +
\int_{Q_{\bar{\tau}}} |\rho^m \cdot \bar{\calv} -1| \, |\partial_t p^{\infty} + v^{\infty} \cdot \nabla p^{\infty}| \, dxd\tau \rightarrow 0 & \, .
\end{align*}
Due to Cor.\ \ref{l1boundandtozero}, $\int_{0}^{\bar{\tau}} \int_{|\Omega{a_0,b_0}(\tau)} |\partial_t\rho^{\infty} + \divv (\rho^{\infty} \, v^{\infty})| \cdot |\bar{g}^m(\pi) - \bar{\calv} \, \pi^m|dxd\tau$ tends to zero. Hence, using the Gronwall--Lemma as in Section \ref{convarg}, we conclude that
\begin{align}\label{ilstendentverszero}
& \limsup_{m\rightarrow \infty} \sup_{0< t<\bar{\tau}} \mathcal{E}^m(t) = 0 \, \quad \text{ and } \\
\limsup_{m\rightarrow \infty} & \int_{Q_{\bar{\tau}}} \mathbb{S}(\nabla (v^m-v^{\infty})) \, : \, \nabla (v^m-v^{\infty}) + \widetilde{M}(\rho^m) \nabla (q^m-q^{\infty}) \, : \, \nabla (q^m-q^{\infty}) \, dxd\tau = 0 \, .\nonumber
\end{align}
Finally let us verify the convergence claims of Theorem \ref{rigolo}. The convergence $\rho^m \rightarrow \rho^{\infty}$ in $L^1(\Omega)$ for all $t$, and $v^m \rightarrow v^{\infty}$ in $L^2(0,\bar{\tau}; \, W^{1,2}(\Omega;\mathbb{R}^3))$, are direct consequences of \eqref{ilstendentverszero}.

We next turn proving a convergence result for $\{\pi^m\}$.
On the sets $\Omega_{a_0,b_0}(t)$ recall that
\begin{align}\begin{split}\label{istesrighti}
 \sum_{k=1}^{N-1} q^m_k \,  \xi^k +\bar{\mathscr{M}}^m(\varrho^m,q^m) = & \frac{1}{\bar{M}} \, \ln \hat{x}(\rho^m) + \bar{g}^m(\hat{\pi}^m(\rho^m))\\
 = & \frac{1}{\bar{M}} \, \ln \hat{x}(\rho^m) + h^m + \bar{\calv} \,  \hat{\pi}^m(\rho^m) \, ,
\end{split}
 \end{align}
with $h^m := \chi_{\Omega_{a_0,b_0}(t)} \, (\bar{g}^m(\hat{\pi}^m(\rho^m)) - \bar{\calv} \, \hat{\pi}^m(\rho^m))$, which we have shown to converge to zero in $L^1(Q_{\bar{\tau}})$ (cp.\ Cor.\ \ref{l1boundandtozero}).

Let $\eta \in \mathbb{R}^N$ according to \eqref{CONVEN} satisfy $\eta \cdot \bar{\calv} = 1$ and $\eta \cdot 1^N = 0$. Multiplication in \eqref{istesrighti} with $\eta$ yields $  q^m_{N-1} = \eta \cdot \Big(\frac{1}{\bar{M}} \, \ln \hat{x}(\rho^m) + h^m\Big) + \hat{\pi}^m(\rho^m)$.
%\end{align*}
Subtracting the mean-value $(q^{m}_{N-1})_M$ on both sides of this identity yields, for $\pi^m_* :=  \hat{\pi}^m(\rho^m) - (q^{m}_{N-1})_M$, and $\zeta^m := q^m_{N-1} -  (q^{m}_{N-1})_M$,
\begin{align}\label{groulde}
\pi^m_* = \zeta^m - \eta \cdot \Big(\frac{1}{\bar{M}} \, \ln \hat{x}(\rho^m) + h^m\Big) \, .
\end{align}
Since $\{\zeta^m\}$ possesses zero mean-value over $\Omega$, it is bounded in $L^2(0,\bar{\tau}; \, W^{1,2}(\Omega))$. Employing \eqref{ilstendentverszero}, $\zeta^m \rightarrow q^{\infty}_{N-1} - (q^{\infty}_{N-1})_M$ in $L^2(0,\bar{\tau}; \, W^{1,2}(\Omega))$.
%\begin{align}\label{zetamconv}
%  \quad \text{  in } \quad  \, .
%\end{align}
On $B_{s_0,+}(t)$, the functions $\eta \cdot \frac{1}{\bar{M}} \, \ln \hat{x}(\rho^m)$ are uniformly bounded. Using Lemma \ref{better}, we also see that $\int_{0}^T \int_{\Omega_{a_0,b_0} \setminus B_{s_0,+}} |\eta \cdot (\frac{1}{\bar{M}} \, \ln \hat{x}(\rho^m))| \, dxd\tau \rightarrow 0$. Then, \eqref{groulde} shows that $\{\chi_{\Omega_{a_0,b_0}(t)} \, \pi^m_* \}$ is dominated by a strongly convergent sequence in $L^1(Q)$ and converges pointwise, which yields
\begin{align}\label{tildepiconvpart}
\chi_{\Omega_{a_0,b_0}(t)} \, \pi^m_* \longrightarrow q^{\infty}_{N-1} - (q^{\infty}_{N-1})_M - \eta \cdot \frac{1}{\bar{M}} \, \ln \hat{x}(\rho^{\infty}) \quad \text{ strongly in } L^1(Q) \, .
\end{align}
Note that $\pi^m_* = \tilde{\pi}^m + q^m_{N-1} - (q^m_{N-1})_M$, and $ \int_0^{\bar{\tau}} \int_{\Omega_{b_0,+}(\tau) \cup \Omega_{a_0,-}(\tau)} |\tilde{\pi}^m| \, dxd\tau \rightarrow 0$ is a consequence of Prop.\ \ref{pressla}. Since also
\begin{align*}
 \int_{\Omega_{b_0,+}(\tau) \cup \Omega_{a_0,-}(\tau)} |q^m_{N-1} - (q^m_{N-1})_M| \, dx \leq & |q^m_{N-1} - (q^m_{N-1})_M|_{L^2}  \, |\Omega_{b_0,+}(\tau) \cup \Omega_{a_0,-}(\tau)|^{\frac{1}{2}}\\
 \leq & c \, |\nabla q^m|_{L^2} \,  (\mathcal{E}^m(\tau))^{\frac{1}{2}} \, ,
 \end{align*}
we see that $\chi_{\Omega_{a_0,b_0}^{\rm c}(t)} \, \pi^m_* \rightarrow 0$ in $L^1(Q)$, and \eqref{tildepiconvpart} implies that
\begin{align*}%\label{tildepiconvvoll}
\pi^m_* \longrightarrow q^{\infty}_{N-1} - (q^{\infty}_{N-1})_M - \eta \cdot \frac{1}{\bar{M}} \, \ln \hat{x}(\rho^{\infty}) = p^{\infty} \quad \text{ strongly in } L^1(Q) \, .
\end{align*}
% Multiplying the identity $\mu^{\infty} = p^{\infty} \, \bar{\calv} + (1/\bar{M}) \, \ln \hat{x}(\rho^{\infty})$ with $\eta$, we see on the other hand that $$q^{\infty}_{N-1} - (q^{\infty}_{N-1})_M - \eta \cdot \frac{1}{\bar{M}} \, \ln \hat{x}(\rho^{\infty}) = p^{\infty} -(q^{\infty}_{N-1})_M = p^{\infty} - \eta \cdot (\mu^{\infty})_M\, . $$ 
% For the solution to (${\rm IBVP}^{\infty}$) such that $\eta \cdot (\mu^{\infty})_M = 0$, the latter shows that $ \pi^m_* \rightarrow p^{\infty}$ in $L^1(Q)$. The proof of Theorem \ref{rigolo} is now complete.\\

We shall motivate only very shortly the proof of Theorem \ref{weaktoweak} for the convergence of weak solutions to weak solutions to (IBVP$^{\infty}$). It relies on the \emph{a priori} bounds \eqref{lesbounds}, Lemma \ref{L1bound} and the convergence property \eqref{ilstendent1} valid for the sequence $\{\varrho^m, \, q^m, \, v^m\}$. With these bounds at hand, we can pass to the limit as in the sections 8--11 of \cite{druetmixtureincompweak} and obtain the convergence up to subsequences to a weak solution with defect measure for (IBVP$^{\infty}$).

%  \bibliographystyle{alpha}
%   \bibliography{/home/druet/Citations/StaRad}

\begin{thebibliography}{BFRW03}

\bibitem[BD15]{bothedreyer}
D.~Bothe and W.~Dreyer.
\newblock Continuum thermodynamics of chemically reacting fluid mixtures.
\newblock {\em Acta Mech.}, 226:1757--1805, 2015.

\bibitem[BD20]{bothedruetMS}
D.~Bothe and P.-E. Druet.
\newblock On the structure of continuum thermodynamical diffusion fluxes: a
  novel closure scheme and its relation to the {M}axwell-{S}tefan and the
  {F}ick-{O}nsager approach.
\newblock Preprint, 2020.
\newblock Available at: {\footnotesize
  \verb|http://www.wias-berlin.de/preprint/2749/wias_preprints_2749.pdf|} and
  at arXiv:2008.05327 [math-ph].

\bibitem[BD21a]{bothedruet}
D.~Bothe and P.-E. Druet.
\newblock Mass transport in multicomponent compressible fluids: local and
  global well-posedness in classes of strong solutions for general class-one
  models.
\newblock {\em Nonlinear Analysis. T., M. \& A.}, 210:112389, 2021.
\newblock https://doi.org/10.1016/j.na.2021.112389.

\bibitem[BD21b]{bothedruetincompress}
D.~Bothe and P.-E. Druet.
\newblock Well-posedness analysis of multicomponent incompressible flow models.
\newblock {\em J. Evol. Equ.}, 21:4039--4093, 2021.
\newblock https://doi.org/10.1007/s00028-021-00712-3.

\bibitem[BDD]{bothedreyerdruet}
D.~Bothe, W.~Dreyer, and P.-E. Druet.
\newblock Multicomponent incompressible fluids -- {A}n asymptotic study.
\newblock {\em ZAMM}.
\newblock Open access. http://doi.org/10.1002/zamm.202100174.

\bibitem[BFRW03]{bechtel}
S.E. Bechtel, M.G. Forest, F.J. Rooney, and Q.~Wang.
\newblock Thermal expansion models of viscous fluids based on limits of free
  energy.
\newblock {\em Phys. Fluids}, 15:2681--2693, 2003.

\bibitem[BS16]{bothesoga}
D.~Bothe and K.~Soga.
\newblock Thermodynamically consistent modeling for dissolution/growth of
  bubbles in an incompressible solvent.
\newblock In {\em Amann H., Giga Y., Kozono H., Okamoto H., Yamazaki M. (eds)
  Recent Developments of Mathematical Fluid Mechanics}, Advances in
  Mathematical Fluid Mechanics. Birkh\"{a}user, Basel, 2016.

\bibitem[CJ15]{chenjuengel}
X.~Chen and A.~J{\"{u}}ngel.
\newblock Analysis of an incompressible {N}avier-{S}tokes-{M}axwell-{S}tefan
  system.
\newblock {\em Commun. Math. Phys.}, 340:471--497, 2015.

\bibitem[DDGG20]{dredrugagu20}
W.~Dreyer, P.-E. Druet, P.~Gajewski, and C.~Guhlke.
\newblock Existence of weak solutions for improved {N}ernst-{P}lanck-{P}oisson
  models of compressible reacting electrolytes.
\newblock {\em Z. Angew. Math. Phys.}, 71(119):119/1--119/68, 2020.
\newblock https://doi.org/10.1007/s00033-020-01341-5.

\bibitem[Dru]{drusurv}
P.-E. Druet.
\newblock Existence of global solutions for multicomponent flows.
\newblock In preparation.

\bibitem[Dru21a]{druetmixtureincompweak}
P.-E. Druet.
\newblock Global--in--time existence for liquid mixtures subject to a
  generalised incompressibility constraint.
\newblock {\em J.\ Math.\ Analysis and Appl.}, 499:125059, 2021.
\newblock doi:10.1016/j.jmaa.2021.125059.

\bibitem[Dru21b]{druetmaxstef}
P.-E. Druet.
\newblock A theory of generalised solutions for ideal gas mixtures with
  {M}axwell--{S}tefan diffusion.
\newblock {\em Disc. Cont. Dyn. Syst. S}, 14:4035--4067, 2021.
\newblock doi:10.3934/dcdss.2020458.

\bibitem[Dru22]{drumaxmix}
P.-E. Druet.
\newblock Maximal mixed parabolic-hyperbolic regularity for the full equations
  of multicomponent fluid dynamics.
\newblock {\em Nonlinearity}, 2022.
\newblock To appear.

\bibitem[Fis15]{fisherentropy}
J.~Fischer.
\newblock A posteriori modeling error estimates for the assumption of perfect
  incompressibility in the {N}avier-{S}tokes equations.
\newblock {\em SIAM J.\ Numer.\ Anal.}, 53:2178--2205, 2015.

\bibitem[FJN12]{feijinnov}
E.~Feireisl, B.~Ja Jin, and A.~Novotn\'{y}.
\newblock Relative entropies, suitable weak solutions, and weak--strong
  uniqueness for the compressible {N}avier-{S}tokes system.
\newblock {\em J. Math. Fluid Mech.}, 14:717--730, 2012.

\bibitem[FLM16]{feima16}
E.~Feireisl, Y.~Lu, and J.~M\'alek.
\newblock On {PDE} analysis of flows of quasi--incompressible fluids.
\newblock {\em Z. Angew. Math. Mech.}, 96:491--508, 2016.

\bibitem[FN]{feinovbook}
E.~Feireisl and A.~Novotn\`{y}.
\newblock {\em Singular Limits in Thermodynamics of Viscous Fluids}.
\newblock Birk\"auser. Springer International Publishing.

\bibitem[FN07]{feinov07}
E.~Feireisl and A.~Novotn\`{y}.
\newblock On the low {M}ach number limit for the full
  {N}avier-{S}tokes-{F}ourier system.
\newblock {\em Arch. Ration. Mech. Anal.}, 186:77--107, 2007.

\bibitem[FN09]{feinov09}
E.~Feireisl and A.~Novotn\`{y}.
\newblock The {Oberbeck-Boussinesq} approximation as a singular limit of the
  full {Navier-Stokes-Fourier} system.
\newblock {\em J. Math. Fluid Mech.}, 11:274--302, 2009.

\bibitem[FN12]{feinovws12}
E.~Feireisl and A.~Novotn\`{y}.
\newblock Weak-strong uniqueness property for the full
  {N}avier-{S}tokes-{F}ourier system.
\newblock {\em Arch. Ration. Mech. Anal.}, 204:683--706, 2012.

\bibitem[FN13]{feinov13}
E.~Feireisl and A.~Novotn\`{y}.
\newblock Inviscid incompressible limits of the full
  {N}avier-{S}tokes-{F}ourier system.
\newblock {\em Commun. Math. Phys.}, 321:605--628, 2013.

\bibitem[FP10]{feipet10}
E.~Feireisl and H.~Petzeltov\`{a}.
\newblock Low mach number asymptotics for reacting compressible fluid flows.
\newblock {\em Disc. Contin. Dyn. Syst.}, 26:455--480, 2010.

\bibitem[Gio99]{giovan}
V.~Giovangigli.
\newblock {\em Multicomponent Flow Modeling}.
\newblock Birkh\"{a}user, Boston, 1999.

\bibitem[GM13]{giovanmatuII}
V.~Giovangigli and L.~Matuszewski.
\newblock Mathematical modeling of supercritical multicomponent reactive
  fluids.
\newblock {\em Math. Mod. Meth. Appl. Sci.}, 23:2193--2251, 2013.

\bibitem[GMR12]{gouin}
H.~Gouin, A.~Muracchini, and T.~Ruggeri.
\newblock On the {M}\"{u}ller paradox for thermal-incompressible media.
\newblock {\em Continuum Mech. Thermodyn.}, 24:505--513, 2012.

\bibitem[Lio96]{lionsfilsI}
P.-L. Lions.
\newblock {\em Mathematical topics in fluid dynamics. Vol. 1, Incompressible
  models}.
\newblock Oxford Science Publication, Oxford, 1996.

\bibitem[LM98]{lionsmasmoudi}
P.-L.. Lions and N.~Masmoudi.
\newblock Incompressible limit for a viscous compressible fluid.
\newblock {\em J. Math. Pures Appl.}, 77:585--828, 1998.

\bibitem[LT98]{lowe}
J.~Lowengrub and L.~Truskinovsky.
\newblock Quasi--incompressible {C}ahn--{H}illiard fluids and topological
  transitions.
\newblock {\em R. Soc. Lond. Proc. Ser. A Math. Phys. Eng. Sci.},
  454:2617--2654, 1998.

\bibitem[Mil66]{mills}
N.~Mills.
\newblock Incompressible mixtures of {N}ewtonian fluids.
\newblock {\em Int. J. Engng Sci.}, 4:97--112, 1966.

\bibitem[MS85]{majdasethian}
A.~Majda and J.~Sethian.
\newblock The derivation and numerical solution of the equations for zero mach
  number combustion.
\newblock {\em Combustion Science and Technology}, 42:185--205, 1985.

\bibitem[MT15]{mariontemam}
M.~Marion and R.~Temam.
\newblock Global existence for fully nonlinear reaction-diffusion systems
  describing multicomponent reactive flows.
\newblock {\em J. Math. Pures Appl.}, 104:102--138, 2015.

\bibitem[M{\"u}l85]{mueller}
I.~M{\"u}ller.
\newblock {\em Thermodynamics}.
\newblock Pitman, London, 1985.

\bibitem[PS14]{pekarsamohyl}
M.~Peka\v{r} and I.~Samoh\'{y}l.
\newblock {\em The Thermodynamics of Linear Fluids and Fluid Mixtures}.
\newblock Springer International Publishing, Cham, Switzerland, 2014.

\bibitem[PSZ19a]{piashiba19}
T.~Piasecki, Y.~Shibata, and E.~Zatorska.
\newblock On strong dynamics of compressible two-component mixture flow.
\newblock {\em SIAM J. Math. Anal.}, 51:2793--2849, 2019.

\bibitem[PSZ19b]{piashiba19pr}
T.~Piasecki, Y.~Shibata, and E.~Zatorska.
\newblock On the isothermal compressible multi--component mixture flow: The
  local existence and maximal {L}p--{L}q regularity of solutions.
\newblock {\em Nonlinear Analysis. T., M. \& A.}, 189:511--571, 2019.

\end{thebibliography}

\appendix

\addtocontents{toc}{\protect\setcounter{tocdepth}{1}}

\section{Rescaling of the PDEs}\label{RESCALE}

We first introduce in well-known manner average quantities in order to rescale the equations \eqref{momentum}. These are: $p^{\text{av}}$ (pressure), $v^{\text{av}}$ (modulus of velocity), $\varrho^{\rm av}$ (density). We further denote by $\bar{g}$ the constant of gravitational acceleration.

Next we introduce reference quantities: $t^{\rm R} =v^{\rm av}/\bar{g}$ (time), $L^{\rm R} = t^{\rm R} \, v^{\text{av}} = (v^{\text{av}})^2/\bar{g}$ (position), $(c^{\rm R})^2:= p^{\rm av}/\varrho^{\rm av}$ (speed of compression waves), $\eta^{\rm R} := \varrho^{\rm av} \, v^{\rm av} \, L^{\rm R}$ (viscosity), $M^{\rm R} := \varrho^{\rm av} \, T \, t^{\rm R}$ (mobility), $g^{\rm R} := \varrho^{\rm av}/p^{\rm av}$ (free energy).
% \begin{itemize}
% \item Time of reference: ;
% \item Length of reference: $L^{\rm R} = t^{\rm R} \, v^{\text{av}} = (v^{\text{av}})^2/g$.
% \item Mean speed of compression waves: $(c^{\rm R})^2:= p^{\rm av}/\varrho^{\rm av}$.
% \item Reference-value of viscosity: ;
% \item Reference-value of mobility: ;
% \item Reference-value of free energy: $g^{\rm R} := \varrho^{\rm av}/p^{\rm av}$.
% \end{itemize}
These choices of the reference quantities are dictated by the wish that, after rescaling the PDEs, the Mach- number is the only one small parameter. 

From now, all these constants will be denoted with the sup-script ${\rm R}$ independently of whether it is an average or an arbitrary reference-value.
We renormalise time and space via $t := \bar{t} \, t^{\text{R}}$ and $\calx = \bar{\calx} \, L^{\text{R}}$. For a function $f = f(\calx, \, t)$, its renormalisation is defined via $\bar{f}(\bar{\calx}, \, \bar{t}) := f(L^{\text{R}}\, \calx, \, t^{\text{R}}\, \bar{t})/f^{\text{R}}$, where $f^{\text{R}}$ is a reference quantity with the same physical dimension as $f$.
For $b = -\bar{g} \, e^3$, the rescaled equations \eqref{momentum} assume the form
\begin{align}\label{momentumresc}
\bar{\varrho} \, ( \partial_{\bar{t}} \bar{v}  + \bar{v} \cdot \bar{\nabla} \bar{v}) + \overline{\divv} \, \overline{\mathbb{S}}(\bar{\nabla} \bar{v}) + \frac{1}{\text{Ma}^2} \, \bar{\nabla} \bar{p} = - \bar{\varrho} \, e^3 \, , \quad \text{ with } \quad \text{Ma} = \frac{v^{\text{R}}}{c^{\text{R}}} = \frac{\sqrt{\varrho^{\rm R}} \, v^{\rm R}}{\sqrt{p^{\rm R}}} \, ,
%- \frac{1}{{\rm Fr}^2}\, \bar{\varrho} \, e^3 \, ,
\end{align}
where $\text{Ma}$ is called the Mach-number.
We also introduced
 \begin{align*}%\label{barstress}
  \overline{\mathbb{S}}(\bar{\nabla} \bar{v}(\bar{\calx},\bar{t})) = 2 \, \bar{\eta} \, (\bar{\nabla} \bar{v}(\bar{\calx},\bar{t}))_{\text{sym}} + \bar{\lambda} \, \overline{\divv} \bar{v}(\bar{\calx},\bar{t}) \, \mathbb{I}, \quad \text{ with } \quad \bar{\eta} = \eta/\eta^{\rm R}, \, \bar{\lambda} = \lambda /\eta^{\rm R} \, .
 \end{align*}
 In this place there is an important observation concerning the comparison of the multicomponent case with the single-component case. 
 \begin{rem}
The Froude number $\text{Fr}$ is the ratio $\sqrt{\frac{g \, t^{\text{R}}}{v^{\text{R}}}}$. Our choice of the scaling guarantees that $\text{Fr} = 1$. 
In the single-component case, a typical scaling uses a bounded ratio ${\rm Fr}^2/{\rm Ma}$ in \eqref{momentumresc}. This can not be expected to yield a limit in the general multicomponent case. Let formally expand $\bar{p} = p_0 + p_1 \, {\rm Ma} + p_2 \, {\rm Ma}^2  + \ldots$, and choose ${\rm Fr}^2 = {\rm Ma}$. Then, matched asymptotics in \eqref{momentumresc} requires that $\nabla p_1 = -\bar{\varrho} \, e^3$. In general, this equation has no solution $p_1$ for a non constant $\bar{\varrho}$. Now, the mass density of a multicomponent liquid is not approximately constant at low Mach-number. For instance, consider a mixture of incompressibile liquids subject to \eqref{EOS}. Then $\sum_{i=1}^N \rho_i/\rho_i^{\rm R} = 1$, where $\rho_i^{\rm R}$ are the densities of the fluids at reference conditions (cf.\ \eqref{I1}). Hence, the total mass $\varrho$ is subject to the constraint $1/\varrho = 1/\sum_{i=1}^N \rho_i^{\rm R} \, y_i$ with the mass fractions $y_i$. Except for the very singular case of all the $\rho_i^{\rm R}$s being identical, a constant mass density will occur only in the case of mixtures where a single dominant widely dominate, which are called dilute. 
\end{rem}
Next we want to normalise of the mass continuity equations. We define
\begin{align}\label{MUMU}
 \bar{\mu}_j(\bar{\calx}, \, \bar{t}) := &\frac{1}{{\rm Ma}^2} \, \frac{1}{g^{\text{R}}} \,\Big( g_j\Big( p^{\text{R}} \,\bar{p}(\bar{\calx},\bar{t})\Big) -g_j(p^{\rm R})\Big) + \frac{R \, T}{(v^{\text{R}})^2 \, M_j} \, \ln \bar{x}_j(\bar{\calx}, \, \bar{t}) \, .
\end{align}
Note that the fractions are normalised using the refrence-value $x_j^{\rm R} = 1$. Since $p^{\rm R}/g^{\rm R} = \varrho^{\rm R}$,
\begin{align*}
 \bar{\nabla} \frac{\bar{\mu}_j(\bar{\calx},\bar{t})}{T} = & \frac{\varrho^{\text{R}} \, \partial_pg_j\Big(p^{\text{R}}\, \bar{p}(\bar{\calx},\bar{t})\Big)}{{\rm Ma}^2 \, T} \, \bar{\nabla} \bar{p}(\bar{\calx},\bar{t}) +  \frac{R}{(v^{\text{R}})^2 \, M_j} \, \bar{\nabla} \ln \bar{x}_j(\bar{\calx},\bar{t})\\
 =  & \frac{L^{\text{R}}}{(v^{\text{R}})^2 \, T}\, \Big(\partial_pg_j(p(\calx,t) ) \, \nabla p +  \frac{R \, T}{M_j}\,  \nabla \ln x_j(\calx,t)\Big)= \frac{L^{\text{R}}}{(v^{\text{R}})^2} \, \nabla \frac{\mu_j}{T} \, .
 \end{align*}
We infer that
\begin{align*} 
 J^i(\calx,t) = -\sum_{j=1}^N M_{ij} \, \nabla \frac{\mu_j}{T} = -\sum_{j=1}^N \frac{(v^{\text{R}})^2\, M_{ij}}{L^{\text{R}}} \,   \bar{\nabla} \frac{\bar{\mu}_j(\bar{\calx},\bar{t})}{T}  = - \frac{(v^{\text{R}})^2 \, M^{\text{R}}}{L^{\text{R}} \, T} \, \sum_{j=1}^N \bar{M}_{ij} \, \bar{\nabla} \bar{\mu}_j(\bar{\calx},\bar{t}) \, .
\end{align*}
Using the special choice of $M^{\text{R}} = \varrho^{\rm R} \, T \, t^{\rm R}$, we then infer for the normalised partial mass densities $\bar{\rho}_i(\bar{x},\bar{t}) := \rho_i(\bar{x} \, L^{\text{R}}, \bar{t} \, t^{\text{R}})/\varrho^{\text{R}}$ that 
\begin{align}\label{massresc}
 \partial_{\bar{t}} \bar{\rho}_i + \overline{\divv} \Big(\bar{\rho}_i \, \bar{v} - \sum_{j=1}^N \bar{M}_{ij} \, \bar{\nabla} \bar{\mu}_j\Big) = 0 \, ,
\end{align}
in which $\bar{\mu}_j$ obeys \eqref{MUMU}. We let $\bar{p}_{\Delta}(\bar{x}, \, \bar{t}) := (\bar{p}(\bar{x},\bar{t})-1)/{\rm Ma}^2$. For $s \in ]{\rm Ma}^{-2}, \, +\infty[$ -this is the maximal possible range of $\bar{p}_{\Delta}$ for $\bar{p}$ being positive- we define a new function 
\begin{align*}%\label{gnormalised}
g^{\rm Ma}_j(s) := \frac{1}{{\rm Ma}^2}  \, \frac{1}{g^{\rm R}} \, \Big( g_j\big(p^{\text{R}}  \, [1+ {\rm Ma}^2 \, \, s]\big) -g_j(p^{\rm R})\Big) \, . 
\end{align*}
Then, we can rewrite \eqref{MUMU} as $\bar{\mu}_j = g^{\rm Ma}_j(\bar{p}_{\Delta}) + (1/\bar{M}_j) \, \ln \bar{x}_{j}$, in which $\bar{M}_j$ is a normalised molar mass defined via $\bar{M}_j := M_j\, (v^{\text{R}})^2/(RT)$.
%\begin{align}\label{gasgas}
%\bar{M}_j := \frac{R \,  T}{M_j\, (v^{\text{R}})^2} \, .
%\end{align}
We note that $(g^{\rm Ma}_i)^{\prime}(s) =\varrho^{\rm R} \, \partial_pg_i(p^{\text{R}}  \, [1+ {\rm Ma}^2 \, \, s])$, which implies that
\begin{align*}
 & (g^{\rm Ma}_i)^{\prime}(\bar{p}_{\Delta}) = \varrho^{\rm R} \, \partial_pg_i\big(p\big) \quad
 \text{and} \quad  \sum_{i=1}^N  (g^{\rm Ma}_i)^{\prime}(\bar{p}_{\Delta}) \, \bar{\rho}_i = \sum_{i=1}^N  \partial_pg_i(p) \, \rho_i\, .
\end{align*}
Since $p(\calx,t) = \hat{p}(T,\, \rho(\calx,t))$ with the implicit function $\hat{p}(T,\cdot)$ defined by \eqref{EOS}, the latter shows that $\bar{p}_{\Delta}(\bar{\calx},\bar{t}) = \hat{\pi}(\bar{\rho}(\bar{\calx},\bar{t}))$, where now $\hat{\pi}$ is the implicit function defined by the normalised equation of state $\sum_{i=1}^N (g^{\rm Ma}_i)^{\prime}(\pi) \, \bar{\rho}_i = 1$ with solution $\pi = \hat{\pi}(\bar{\rho}_1,\ldots,\bar{\rho}_N)$. Overall, we can write the system \eqref{momentumresc}, \eqref{massresc} as 
\begin{align*}%\label{massrescfin}
 \partial_{\bar{t}} \bar{\rho}_i + \overline{\divv} \Big(\bar{\rho}_i \, \bar{v} - \sum_{j=1}^N \bar{M}_{ij} \, \bar{\nabla} \bar{\mu}_j\Big) =&  0 \, ,\\ %\label{momentumrescfin}
\bar{\varrho} \, ( \partial_{\bar{t}} \bar{v}  + \bar{v} \cdot \bar{\nabla} \bar{v}) + \overline{\divv} \, \overline{\mathbb{S}}(\bar{\nabla} \bar{v}) + \bar{\nabla} \bar{p}_{\Delta} = & -  \bar{\varrho} \, e^3 \, ,
\end{align*}
where the Mach number does not occur explicitly. Indeed, it occurs only at the level of the constitutive equations. For ${\rm Ma}^{-2} =:m$ and $p_{\Delta} =: \pi$, the constitutive model is given by 
\begin{align*}
\bar{\mu}_i = \hat{\mu}^{m}_i(\pi,\bar{x}_i) := & \bar{g}_i^m(\pi) + \frac{1}{\bar{M}_i} \, \ln \bar{x}_i \quad
\text{ with } \quad \bar{g}_i^m(
\pi) =  m \, \Big(\bar{g}_i\big(1+\frac{\pi}{m}\big) - \bar{g}_i(1)\Big) \, ,
\end{align*}
%with $\bar{R}$ defined by \eqref{gasgas} and 
with the normalised functions $\bar{g}_i(s) = g_i(p^{\rm R} \, s)/g^{\rm R}$.
Note also that these definitions allow to show that the chemical potentials are rescaled via
$ \bar{\mu}_i(\bar{\calx}, \, \bar{t}) = (\mu_i(L^{\rm R} \, \bar{\calx}, \, t^{\rm R} \, \bar{t}) - \mu_i^{\rm R})/(v^{\rm R})^2$ with $\mu_i^{\rm R} := g_i(p^{\rm R})$.

\section{Uniform properties of the free energy function}

We let $m \in \mathbb{N}$ and $g_1^m, \ldots, g^m_N$ satisfy all assumptions (A1)--(A3). For $\rho \in \mathbb{R}^N_+$ and $p > 0$, we let $V^m(p, \, \rho) := \sum_{i=1}^N \partial_pg_i^m(p) \, \rho_i$. Exploiting (A3), we verify that $V^m(p,\rho) \rightarrow 0$ for $p \rightarrow +\infty$, while $V^m(p,\rho) \rightarrow +\infty$ for $p \rightarrow 0$. Since (A2) implies that $p \mapsto V^m(p,\cdot)$ is decreasing, the equation $V^m(p,\rho) = 1$ possesses a unique solution $p = \hat{p}^m(\rho)$ for all $\rho \in \mathbb{R}^N_+$. We easily check that
\begin{align}\label{pderivative}
 \partial_{\rho_i} \hat{p}^m(\rho) = - \frac{\partial_pg^m_i(\hat{p}^m(\rho))}{ \partial^2_pg^{m}(\hat{p}^m(\rho)) \cdot \rho} \, . 
\end{align}
For $\rho \in \mathbb{R}^N_+$, we recall that the free energy function is given by (cp.\ \eqref{FE})
\begin{align*}
 f^m(\rho) = \sum_{i=1}^N \rho_i \, g_i^m(\hat{p}^m(\rho)) - \hat{p}^m(\rho) + RT\, \sum_{i=1}^N \frac{\rho_i}{M_i} \, \ln \hat{x}_i(\rho) \, .
\end{align*}
\begin{lemma}\label{UNIFF1}
Assume that $g^m$ is subject to $\rm (A1)-(A3)$. Then, $f^m$ belongs to $C^2(\mathbb{R}^N_+)$ and is a function of Legendre--type on $\mathbb{R}^N_+$. Moreover, the following properties are valid:
\begin{enumerate}[(i)]
 \item If $\rm (A4)$ is valid, then $f^m \in C(\overline{\mathbb{R}^N_+})$;
 \item If both $\rm (A4)$ and $\rm (A5)$ hold, then $f^m$ is co-finite ($ Df^m: \, \mathbb{R}^N_+ \rightarrow \mathbb{R}^N$ is surjective);
 \item\label{casegrowthplow} If $\rm (A4)$ is valid with constants $\bar{c}_1, \, \bar{c}_2$ independent on $m$ then
 \begin{align*}
\bar{c}_1 \, \varrho &\leq \hat{p}^m(\rho) \leq \bar{c}_2 \,\varrho \, &  \text{ for } \quad 0 < \hat{p}^m(\rho)  < \bar{s} \, .
\end{align*}
\item\label{casegrowthp} If $(A5)$ is valid with constants independent on $m$ and $\calv_0 := \liminf_{m \rightarrow \infty} \min \partial_pg^m(\bar{s}) >0$ then 
 \begin{align*}
\left(\frac{\beta\, \calv_0}{\max \alpha} \right)^{\frac{\max\alpha}{\max\alpha-1}} \, \bar{s} \, \varrho^{\frac{\max\alpha}{\max\alpha-1}} \leq&  \hat{p}^m(\rho) \leq \left(\frac{\max\alpha\, \max \calv}{\beta} \right)^{\frac{\beta}{\beta-1}} \, \bar{s} \, \varrho^{\frac{\beta}{\beta-1}}&  \text{ for }\quad  \hat{p}^m(\rho) \geq  \bar{s} \, 
\end{align*}
\item\label{grothoff} Under the assumptions of \eqref{casegrowthplow}, \eqref{casegrowthp}, there are  $c_0, \, c_1 > 0$ such that $ f^m(\rho) \geq c_0 \, |\rho|^{\gamma} - c_1$ for all $\rho \in \mathbb{R}^N_+$ and all $m \in \mathbb{N}$ with $\gamma = \max\alpha/\max\alpha-1$; 
\item \label{caseunif} If $\{g^m\}$ is bounded in $C^2(K)$ for every compact subset $K \in \mathbb{R}_+$, then, there is a positive function $\lambda_1 \in C(]0,+\infty[)$ such that
\begin{align*} 
\lambda_{\min}(D^2_{\rho,\rho}f^m(\rho)) \geq\lambda_1(\hat{p}^m(\rho)) \, ,\quad \text{ for all } \rho \in \mathbb{R}^N_+, \, m =1,2,\ldots
 \end{align*}
\item \label{casea6} If $\rm (A4)-(A6)$ are valid with constants independent on $m$, and if the assumptions \eqref{casegrowthplow}, \eqref{casegrowthp} hold, then there is $\lambda_1 > 0$ such that for all $\rho \in \mathbb{R}^N_+$, $m \geq 1$,
\begin{align*} 
\lambda_{\min}(D^2_{\rho,\rho}f^m(\rho)) \geq \frac{\lambda_1}{\varrho \, (1+\varrho)^{\theta_0}} \, \quad \text{ with } \quad \theta_0 := 2\frac{\beta}{\beta-1} \, \Big(\frac{1}{\beta}-\frac{1}{\max \alpha}\Big) \, .
\end{align*}
\end{enumerate}
% Suppose that $g^m(p^0) = 0$ and $\partial_pg^m(p^0) = \calv$, then for all $\rho\in \mathbb{R}^N_+$ s.t.\ $\sum_{i=1}^N \rho_i \, \calv_i = 1$, 
% \begin{align*}
% |f^m(\rho)| \leq &\frac{1}{ \min \calv} \,  \frac{N}{e} \, \frac{\max M}{\min M}\, , \quad
% |Df^m(\rho)| \leq   - \frac{1}{\min M} \,\Big(\ln \min_i \rho_i + \ln \frac{\min M \, \min \calv}{\max M}\Big) \, .
% \end{align*}
\end{lemma}
\begin{proof}
Owing to \eqref{pderivative} and (A1), $\hat{p}^m \in C^1(\mathbb{R}^N_+)$. The expressions for $Df^m(\rho)$ was already obtained in \eqref{gradient}. It shows that $f^m \in C^2(\mathbb{R}^N_+)$, and the Hessian $D^2f^m(\rho)$ is given in \eqref{Hessian}.
% 
% We compute the gradient of $f^m$ and, using that $V^m(\hat{p}^m(\rho), \, \rho) = 1$, we obtain that
% \begin{align*}%\label{gradient}
%  \partial_{\rho_i}f^m(\rho) = g_i(\hat{p}^m(\rho)) + \frac{1}{M_i} \, \ln \hat{x}_i(\rho) \, .
% \end{align}
% This shows that $f^m \in C^2(\mathbb{R}^N_+)$, and that
% \begin{align*}%\label{Hessian}
%  \partial^2_{\rho_i,\rho_j}f^m(\rho) = - \frac{\partial_pg^m_i(\hat{p}^m(\rho)) \,  \partial_pg^m_j(\hat{p}^m(\rho))}{\sum_{k=1}^N \partial^2_pg^{m}_k(\hat{p}^m(\rho)) \, \rho_k} + \frac{1}{M_iM_j \, \hat{n}(\rho)} \, \Big(\frac{\delta_i^j}{\hat{x}_i(\rho)} - 1\Big) \, .
% \end{align*}
Let $\upsilon^m(\rho) := \partial_pg^m(\hat{p}^m(\rho))$ and $K^m(\rho) := -1/ \sum_{k=1}^N \partial^2_pg^{m}_k(\hat{p}^m(\rho)) \, \rho_k$. Using the function $k$ of \eqref{KFUKK}, $D^2f^m$ is equivalently expressed
\begin{align*}
 D^2f^m(\rho) = K^m(\rho) \, \upsilon^m(\rho) \otimes \upsilon^m(\rho) + D^2k(\rho)  \, .
\end{align*}
In order to show that $f^m$ is strictly convex, we first note that
\begin{align*} 
 \frac{\varrho}{RT} \, M^{\frac{1}{2}} \, Y^{\frac{1}{2}} \, D^2k(\rho) \, Y^{\frac{1}{2}} \, M^{\frac{1}{2}} = \mathbb{I} - \sqrt{x} \otimes \sqrt{x} \, ,
\end{align*}
where $M = \text{diag}(M_1,\ldots,M_N)$, $Y = \text{diag}(y_1,\ldots,y_N)$ and $\sqrt{x} = (\sqrt{x_1},\ldots, \sqrt{x_N})$ with $y = \hat{y}(\rho)$ and $x = \hat{x}(\rho)$. The matrix on the right-hand side is the projection onto the orthogonal complement of $\text{span} \{\sqrt{x}\}$. Hence, for $\eta \in \mathbb{R}^N$ arbitrary
\begin{align*}
 \frac{\varrho}{RT} \,  D^2k(\rho) \eta \cdot \eta \geq |Y^{-\frac{1}{2}} \, M^{-\frac{1}{2}} \, \eta|^2 -  (Y^{-\frac{1}{2}} \, M^{-\frac{1}{2}} \, \eta \cdot \sqrt{x})^2 \, .
\end{align*}
We let $\eta^1 := Y^{-\frac{1}{2}} \, M^{-\frac{1}{2}} \, \eta -  (Y^{-\frac{1}{2}} \, M^{-\frac{1}{2}} \, \eta \cdot \sqrt{x}) \, \sqrt{x}$ be the projection. Then
\begin{align*} 
 \eta - Y^{\frac{1}{2}} \, M^{\frac{1}{2}} \, \eta^1 = (Y^{-\frac{1}{2}} \, M^{-\frac{1}{2}} \, \eta \cdot \sqrt{x}) \, Y^{\frac{1}{2}} \, M^{\frac{1}{2}} \sqrt{x} \, .
 \end{align*}
We multiply the latter with $\upsilon^m(\rho)$. Since $Y^{\frac{1}{2}} \, M^{\frac{1}{2}} \sqrt{x} = \rho/\sqrt{\varrho \, \hat{n}(\rho)}$, and since $\upsilon^m(\rho) \cdot \rho = 1$, it follows that $
 \eta = Y^{\frac{1}{2}} \, M^{\frac{1}{2}} \, \eta^1 + \rho \, (\upsilon^m(\rho) \cdot \eta - \upsilon^m(\rho) \cdot Y^{\frac{1}{2}} \, M^{\frac{1}{2}} \, \eta^1)$. The latter clearly allows to bound
 \begin{align*}
 |\eta| \leq & \sqrt{\max M} \, (1+|\rho| \, |\upsilon^m(\rho)|) \, |\eta^1| + |\rho| \, |\upsilon^m(\rho) \cdot \eta|\\
 \leq & \sqrt{RT \, \max M} \, (1+|\rho| \, |\upsilon^m(\rho)|) \, \sqrt{\frac{\varrho}{RT} \,  D^2k(\rho) \eta \cdot \eta} + \frac{|\rho|}{\sqrt{K^m}}  \, \sqrt{K^m} \, | \upsilon^m(\rho) \cdot \eta|\\
\leq & \max \Big\{\sqrt{\varrho \, RT\, \max M} \, (1+|\rho| \, |\upsilon^m(\rho)|) , \,  \frac{|\rho|}{\sqrt{K^m}}\Big\} \, \sqrt{D^2f^m(\rho) \eta \cdot \eta} \, .
\end{align*}
Recall that $|\upsilon^m(\rho)| \leq N \, \max \partial_pg^m(\hat{p}^m(\rho))$ and that $|\rho| \leq N^{1/2} \, \varrho \leq  N^{1/2} /\min \partial_pg^m(\hat{p}^m(\rho))$. Hence we see that 
\begin{align}\label{EVD2fm}
 D^2f^m(\rho) \eta \cdot \eta \geq&  \min \Big\{\frac{RT}{\varrho \, \max M \, (1+N^{\frac{3}{2}} \, \frac{\max \partial_pg^m(\hat{p}^m(\rho))}{\min\partial_pg^m(\hat{p}^m(\rho))})^2} , \,  \frac{K^m(\rho)}{N \, \varrho^2}\Big\} \, |\eta|^2 \, .
\end{align}
This shows that $f^m$ is strictly convex.\\

In order to prove that $f^m$ is a function of Legendre-type, it remains to prove that $|\nabla f^m(\rho)| \rightarrow +\infty$ as $\rho \rightarrow \rho^0$ with $\rho^0_i = 0$ for some $i$ (essential smoothness). If $\rho \rightarrow \rho^0 \neq 0$, then it is readily seen that $\hat{p}^m(\rho) \rightarrow \hat{p}^m(\rho^0)$ finite due to the definition of $\hat{p}^m$. Hence, recalling \eqref{gradient}, 
\begin{align*}
 |\nabla f^m(\rho)| \geq -\frac{1}{m_i} \, \ln \hat{x}_i(\rho) - \sup_m |g^m(\hat{p}^m(\rho))| \rightarrow + \infty \, \quad \text{ for } \rho \rightarrow \rho^0 \text{ s.t.\ } \rho_i^0 = 0, |\rho^0| > 0\, .
\end{align*}

If $\rho \rightarrow 0$, then $\hat{p}^m(\rho) \rightarrow 0$ and $g_i^m(\hat{p}^m(\rho)) \rightarrow -\infty$ both result of (A3). Then \eqref{gradient} shows that $\partial_{\rho_i} f^m(\rho) \rightarrow -\infty$ for $i = 1,\ldots,N$.\\

We next discuss the growth conditions. Using (A4) and $\partial_pg^m(\hat{p}^m(\rho)) \cdot \rho = 1$ directly yields \eqref{casegrowthplow}.
%that
%\begin{align}\label{proofoflabas}
%\bar{c}_1 \, \varrho \leq \hat{p}^m(\rho) \leq \bar{c}_2 \,\varrho \, \quad \text{ for } \quad 0 < \hat{p}^m(\rho) < p^0 \, .
%\end{align}
For $\hat{p}^m(\rho) \geq \bar{s}$, we can integrate (A5) to show that
\begin{align}\label{plarge12proof1}
\left(\frac{\min g^m(\bar{s})}{\bar{s}^{\frac{1}{\max \alpha}} \, \max\alpha} \, \varrho\right)^{\gamma} \leq \hat{p}^m(\rho) \leq \left(\frac{\max g^m(\bar{s})}{\bar{s}^{\frac{1}{\beta}} \, \beta} \, \varrho\right)^{\frac{\beta}{\beta-1}}  \, 
\end{align}
and since (A5) implies that $\max\alpha \, \bar{s} \, \partial_pg^m(\bar{s}) \geq  g^m(\bar{s}) \geq \beta \, \bar{s} \, \partial_pg^m(\bar{s}) $, we get \eqref{casegrowthp}.
In order to prove the lower bound for $f^m$, we reason as in the Lemma 6.1 of \cite{druetmaxstef}. First, we use the fact that $g^m$ is concave and increasing to show with the help of (A5) that
\begin{align*}
 0 \leq g^m(\hat{p}^m(\rho)) \cdot \rho - \hat{p}^m(\rho) & \leq  g^m(\bar{s}) \cdot \rho \leq \max\alpha \, \partial_pg^m(\bar{s}) \cdot \rho  \,\\
 & \leq \max \alpha \, \partial_pg^m(\hat{p}^m(\rho)) \cdot \rho = \max \alpha \quad & \text{ if }0 < \hat{p}^m(\rho) \leq \bar{s} \, .
\end{align*}
If $\hat{p}^m(\rho) \geq \bar{s}$, then (A5) implies that $g^m(\hat{p}^m(\rho)) \cdot \rho - \hat{p}^m(\rho) > (\beta-1) \, \hat{p}^m(\rho)$, and \eqref{casegrowthp} can be used to show that
\begin{align*}
 g^m(\hat{p}^m(\rho)) \cdot \rho - \hat{p}^m(\rho) \geq (\beta-1) \,  \left(\frac{\beta\, \calv_0}{\sqrt{N}\max \alpha} \right)^{\gamma} \, \bar{s} \, |\rho|^{\gamma} \, \quad \text{ if } \quad \hat{p}^m(\rho) \geq \bar{s} \, .
\end{align*}
Hence, combining the two latter results,
\begin{align*}
 (\beta-1) \,  \left(\frac{\beta\, \min \partial_pg^{m}(\bar{s})}{\sqrt{N}\max \alpha} \right)^{\gamma} \, \bar{s} \, |\rho|^{\gamma} \leq & g^m(\hat{p}^m(\rho)) \cdot \rho - \hat{p}^m(\rho) + (\beta-1) \, \bar{s} \\ 
 = & f^m(\rho) - k(\rho) +  (\beta-1) \, \bar{s} \, .
\end{align*}
It remains to observe that $|k(\rho)| \leq c \, |\rho|$ with $c = N^{3/2}/(e\min M)$. Thus, Young's inequality allows to prove \eqref{grothoff} with 
\begin{align*}
 c_0 = \frac{\beta-1}{2} \,  \left(\frac{\beta\, \calv_0}{\sqrt{N}\max \alpha} \right)^{\gamma} \, \bar{s} \, , \quad c_1 = (\beta-1) \, \bar{s} + c(\gamma)  \, .
\end{align*}

In order to prove that $f^m$ is co-finite, we must show that the equations $\nabla f^m(\rho) = \mu$ are solvable for all $\mu$ in $\mathbb{R}^N$. As shown in \cite{dredrugagu20}, a sufficient condition is that the equation
\begin{align*}
 \sum_{i=1}^N \exp(M_i \, (\mu_i - g_i^m(s))) = 1 \, ,
\end{align*}
possesses a unique solution $s = s(\mu)$ for all $\mu$ in $\mathbb{R}^N$. This is the case if $g_i^m(s)$ converges to $-\infty$ for $s \rightarrow 0$ for at least one $i$, and if $g_i^m(s)$ tend to $+\infty$ for $s \rightarrow \infty$ for every $i$. These two conditions clearly result of (A4), (A5).\\

We next prove the bounds for the smallest eigenvalue of $D^2f^m$. Due to \eqref{EVD2fm}, it is sufficient to bound the quotients $\min \partial_pg^m(\hat{p}^m(\rho))/\max \partial_pg^m(\hat{p}^m(\rho))$ and $K^m(\rho)/\varrho$ from below in order to obtain a bound.

We note that $K^m(\rho)/\varrho \geq 1/(\varrho^2 \, |\partial^2_pg^m(\hat{p}^m(\rho))|_1)$. Thus, for this term, it is sufficient to bound $\varrho^2 \, |\partial^2_pg^m(\hat{p}^m(\rho))|$ from above.

For the bound in \eqref{caseunif}, recall that 
%\begin{align*}
$1/\max \partial_pg^{m}(\hat{p}^m(\rho)) \leq \varrho \leq  1/\min \partial_pg^{m}(\hat{p}^m(\rho))$ is a consequence of \eqref{EOS}.
%\end{align*}
Hence 
$\varrho^2 \, |\partial^2_pg^m(\hat{p}^m(\rho))| \leq |\partial^2_pg^m(s)|/\min^2 \partial_pg^{m}(s)$. The function $\lambda_1(s)$ of \eqref{caseunif} is thus obtained as
\begin{align*}%\label{EVD2fm}
 \lambda_1(s) := \inf_{m} \min \Big\{\frac{RT\, \min \partial_pg^m(s)}{\max M \, (1+N^{\frac{3}{2}} \,\frac{\max \partial_pg^m(s)}{\min \partial_pg^m(s)})^2} , \,  \frac{\min^2 \partial_pg^{m}(s)}{N \, |\partial^2_pg^m(s)|}\Big\} \, .
\end{align*}
For the case \eqref{casea6}, $\max \partial_pg^m(s)/\min \partial_p g^m(s) \leq \bar{c}_2/\bar{c}_1$ for $0 < s \leq \bar{s}$ due to (A4).
%we can compute that
For $s \geq \bar{s}$, we use $\partial_gg^m(\bar{s}) \leq \calv$, $\min\partial_gg^m(\bar{s}) \geq  \calv_0$ and (A5) implies that
\begin{align*}
 \frac{\max \partial_pg^m(s)}{\min \partial_p g^m(s)} \leq \frac{\max \alpha}{\beta} \, \frac{\max g^m(\bar{s})}{\min g^m(\bar{s})} \, \left(\frac{s}{\bar{s}}\right)^{\frac{1}{\beta}-\frac{1}{\max\alpha}}\leq \left(\frac{\max \alpha}{\beta}\right)^2 \, \frac{\max \calv}{\calv_0} \, \left(\frac{s}{\bar{s}}\right)^{\frac{1}{\beta}-\frac{1}{\max\alpha}}  \, .
\end{align*}
In connection with \eqref{casegrowthp}, this implies that $ \frac{\max \partial_pg^m(\hat{p}^m(\rho))}{\min \partial_p g^m(\hat{p}^m(\rho))} \leq c^{\prime} \, \Big(1+\varrho^{ \frac{\beta}{\beta-1} \, \left(\frac{1}{\beta}-\frac{1}{\max \alpha}\right)}\Big)$.
If (A6) holds then 
\begin{align*}
\frac{1}{K^m(\rho)} & = \sum_{k=1}^N |\partial^2_p g^m(\hat{p}^m(\rho))| \, \rho_k = \frac{1}{\hat{p}^m(\rho)} \, \sum_{k=1}^N \hat{p}^m(\rho) \, |\partial^2_p g^m(\hat{p}^m(\rho))| \, \rho_k\\
& \leq \frac{\bar{c}_3}{\hat{p}^m(\rho)} \,\sum_{k=1}^N \partial_p g^m(\hat{p}^m(\rho)) \, \rho_k = \frac{\bar{c}_3}{\hat{p}^m(\rho)} \, .
\end{align*}
Hence, $K^m(\rho) \geq \bar{c}_3^{-1} \, \hat{p}^m(\rho)$.
The inequalities $\hat{p}^m(\rho) \geq \bar{c}_1 \, \varrho$ or $\hat{p}^m(\rho) \geq c \, \varrho^{\gamma}$ for $p$ small or large (cf.\ \eqref{casegrowthplow}, \eqref{casegrowthp}) can be exploited to obtain that $K^m(\rho)/\varrho$ is bounded below as desired.
% 
% Finally, consider $\rho \in S^0$. Then
% \begin{align*}
%  |f^m(\rho)| \leq |g^m(p_0)|_{\infty} \, \varrho +|k(\rho)| \leq & \sup_m | g^m(p_0)| \, \varrho + \frac{N}{e} \, \frac{\max M}{\min M}  \, \varrho \\
%  \leq & \frac{1}{ \min \calv} \, \Big(\sup_m | g^m(p_0)| + \frac{N}{e} \, \frac{\max M}{\min M}\Big) \, ,\\
%  |D f^m(\rho)| \leq & \sup_m | g^m(p_0)| + \frac{1}{\min M} \, \ln \frac{1}{\min \hat{x}(\rho)} \, ,
% \end{align*}
% and all claims are proved.
\end{proof}

To obtain uniform properties for the free energy in the case of the rescaling, we at first recall the definition \eqref{FEm} of $\bar{f}^m$. Let $\bar{g}_1, \ldots, \bar{g}_N$ be the normalised functions of \eqref{bargipi}. If we assume that $\bar{g}$ is subject to the assumptions (A), then we can construct a strictly convex free energy function $\bar{f}$, satisfying all claims of Lemma \ref{UNIFF1}, via
\begin{align*}
& \bar{f}(\bar{\rho}) = \bar{g}(\widehat{p}(\bar{\rho})) \cdot \bar{\rho} - \widehat{p}(\bar{\rho}) + \bar{k}(\bar{\rho}) =: \bar{f}_1(\bar{\rho}) + \bar{k}(\bar{\rho}) \quad
\text{ where } \quad  \sum_{i=1}^N \bar{g}_i^{\prime}(\widehat{p}(\bar{\rho})) \, \bar{\rho}_i = 1 \, .
\end{align*}
Due to the latter definition of $\widehat{p}$, note that $\widehat{p}(\bar{\rho}) = 1$ iff $\bar{\calv} \cdot \bar{\rho} = 1$ with $\bar{\calv} = \bar{g}^{\prime}(1)$. 
%Then $\bar{f}_1(\rho^{\infty}) = \bar{g}(1) \cdot r - 1$.  
For $m \in \mathbb{N}$, we next define $\hat{\pi}^m(\bar{\rho}) := m \, (\widehat{p}(\rho) - 1)$. Then, using that
\begin{align*}
 \bar{f}_1(\bar{\rho}) = \bar{g}(\widehat{p}(\bar{\rho})) \cdot \bar{\rho} - \widehat{p}(\bar{\rho}) = \bar{g}\Big(1+\frac{\hat{\pi}^m(\bar{\rho})}{m}\Big) \cdot \bar{\rho} -1 -\frac{ \hat{\pi}^m(\bar{\rho})}{m} \, ,
\end{align*}
it is readily seen that
\begin{align*}
m \, (  \bar{f}_1(\bar{\rho}) - \bar{g}(1) \cdot \bar{\rho}+1) = m \, \Big(\bar{g}\Big(1+\frac{\hat{\pi}^m(\bar{\rho})}{m}\Big) -\bar{g}(1) \Big) \cdot \bar{\rho} -  \hat{\pi}^m(\bar{\rho}) \, .
\end{align*}
Thus, we also obtain the representation
\begin{align*}
 \bar{f}^m(\bar{\rho}) = m \, (  \bar{f}_1(\bar{\rho}) - \bar{g}(1) \cdot \bar{\rho}+1) + \sum_{i=1}^N \frac{\bar{\rho}_i}{\bar{M}_i} \, \ln \hat{x}_i(\bar{\rho}) \, .
\end{align*}
This allows to verify the growth properties. Since $\bar{f}_1(\bar{\rho}) \geq c_0 \, |\bar{\rho}|^{\gamma} - c_1$ with constants $c_0,c_1$ independent of $m$, we obtain that
\begin{align*}
 \bar{f}^m(\bar{\rho}) \geq & m \, (c_0 \, |\bar{\rho}|^{\gamma} - (c_1-1) - |\bar{g}(1)| \, |\bar{\rho}|) - \frac{N^{3/2} \, \max \bar{M}}{e \, (\min \bar{M})^{3/2}} \, |\bar{\rho}|)
 \geq  m \, \frac{c_0}{2} \, |\bar{\rho}|^{\gamma}\\
 & \quad \text{ for all } \quad |\bar{\rho}| \geq \max\Big\{\Big( \frac{6(c_1-1)}{c_0}\Big)^{\frac{1}{\gamma}}, \, \Big(\frac{6|\bar{g}(1)|}{c_0} \Big)^{\frac{1}{\gamma-1}}, \, \Big(\frac{6N^{3/2} \, \max \bar{M}}{c_0\, e \, (\min \bar{M})^{3/2}} \Big)^{\frac{1}{\gamma-1}})     \Big\}  \, .
\end{align*}
Moreover the Hessian satisfies
\begin{align*}
 D^2_{\bar{\rho},\bar{\rho}}\bar{f}^m(\bar{\rho}) = m \, D^2_{\bar{\rho},\bar{\rho}}\bar{f}_1(\bar{\rho}) + D^2_{\bar{\rho},\bar{\rho}} \bar{k}(\bar{\rho}) \geq D^2_{\bar{\rho},\bar{\rho}}\bar{f}(\bar{\rho}) \, , 
\end{align*}
proving by means of Lemma \ref{UNIFF1} that the smallest eigenvalue $\lambda_1$ can be chosen independently on $m$.\\

Next we prove an additional growth property.
We consider $0 < \bar{\theta} < 1$ and $\bar{\rho}$ subject to $-m < \hat{\pi}^m(\bar{\rho}) \leq -m \, \bar{\theta}$. Since $\hat{\pi}^m > -m$, the identity $\bar{g}^{\prime}(1+\hat{\pi}^m(\bar{\rho})/m) \cdot \bar{\rho} = 1$ is valid and, together with (A4) for $\bar{g}$, it allows to show that $\bar{\varrho} \leq \bar{c}_1^{-1} \, (1+ \hat{\pi}^m(\bar{\rho})/m) \leq \bar{c}_1^{-1} \, (1-\bar{\theta})$, 
and that
\begin{align*}
 \bar{g}^m(\hat{\pi}^m(\bar{\rho})) = & \int_{0}^1 \bar{g}^{\prime}\Big(1+ \theta \, \frac{\hat{\pi}^m(\bar{\rho})}{m}\Big) \, d\theta \, \hat{\pi}^m(\bar{\rho}) \geq \bar{c}_2 \, \int_{0}^1 \frac{1}{1+\theta \, \frac{\hat{\pi}^m(\bar{\rho})}{m}} \, d\theta \, \hat{\pi}^m(\bar{\rho}) \\
 = & \bar{c}_2 \, \frac{m}{\hat{\pi}^m(\bar{\rho})} \, \ln \Big(1+ \theta \, \frac{\hat{\pi}^m(\bar{\rho})}{m}\Big)  \,  \hat{\pi}^m(\bar{\rho}) 
 \geq  \frac{\bar{c}_2}{\bar{\theta}} \, |\ln (1-\bar{\theta})| \, \hat{\pi}^m(\bar{\rho})\, .
\end{align*}
Hence, for all $\bar{\rho}$ subject to $0<\hat{\pi}^m(\bar{\rho}) \leq -m \, \bar{\theta}$, we have 
\begin{align*}
& \bar{g}^m(\hat{\pi}^m(\bar{\rho})) \cdot \bar{\rho} \geq - |\bar{g}^m(\hat{\pi}^m(\bar{\rho}))|_{\infty} \, \bar{\varrho} \geq - \frac{\bar{c}_2}{\bar{c}_1} \, \frac{1-\bar{\theta}}{\bar{\theta}} \, |\ln (1-\bar{\theta})|\\
\text{ implying that } &\\
&  \bar{g}^m(\hat{\pi}^m(\bar{\rho})) \cdot \bar{\rho} - \hat{\pi}^m(\bar{\rho}) \geq -\hat{\pi}^m(\bar{\rho}) \, \Big(1 - \frac{\bar{c}_2}{\bar{c}_1} \, \frac{1-\bar{\theta}}{\bar{\theta}} \, |\ln (1-\bar{\theta})|\Big) \, ,
\end{align*}
We can choose $\bar{\theta}= \bar{\theta}(\bar{c}_1,\bar{c}_2)$ near to one, and we get
\begin{align*}
\bar{f}^m(\bar{\rho}) - \bar{k}(\bar{\rho}) \geq \frac{1}{2} \,  | \hat{\pi}^m(\bar{\rho})| \quad \text{ whenever } \quad - m < \hat{\pi}^m(\bar{\rho}) \leq - m \, \bar{\theta} \, .
\end{align*}
Note that, if $\bar{\rho} \rightarrow 0$ and $\hat{\pi}^m(\bar{\rho}) \rightarrow - m$, and $\bar{f}^m(\bar{\rho}) \rightarrow m$. Hence the latter inequality extends to $\bar{\rho} = 0$. Using these and similar arguments, we can prove the following statement.
\begin{lemma}\label{UNIFF2}
Assume that $\bar{g}$ is subject to the assumptions $\rm (A1)-(A6)$. Then, $\bar{f}^m$ belongs to $C^2(\mathbb{R}^N_+) \cap C(\overline{\mathbb{R}^N_+})$ and is a co-finite function of Legendre--type on $\mathbb{R}^N_+$ such that $\bar{f}^m(0) = m$. There are constant $\bar{c}_0, R_1 > 0$ and $\lambda_1 > 0$ such that, for all $m\in\mathbb{N}$,
\begin{align*}
 \bar{f}^m(\bar{\rho}) \geq \bar{c}_0 \, |\rho|^{\gamma} \text{ for all } |\bar{\rho}| \geq R_1 \, ,
\qquad \lambda_{\min}(D^2_{\bar{\rho},\bar{\rho}}\bar{f}^m(\bar{\rho})) \geq \frac{\lambda_1}{\bar{\varrho} \, (1+\bar{\varrho})^{\theta_0}} \, \text{ for all } \bar{\rho} \in \mathbb{R}^N_+ \, ,
 \end{align*}
with $\gamma = \max \alpha/\max \alpha-1$ and $\theta_0 = 2\beta \, (1/\beta-1/\max\alpha)/(\beta-1)$ ($\beta,\alpha_1,\ldots,\alpha_N$ from $\rm (A5)$). There is $ 0 < \bar{\theta} < 1$ independent on $m$ such that
\begin{align*}
\bar{f}^m(\bar{\rho}) \geq \frac{1}{2} \,  | \hat{\pi}^m(\bar{\rho})| \quad \text{ whenever } \quad \hat{\pi}^m(\bar{\rho}) \leq - m \, \bar{\theta} \, .
\end{align*}
For all $r \in \mathbb{R}^N_+$ subject to $  \sum_{i=1}^N r_i \, \bar{\calv}_i = 1$ with $\bar{\calv} = \bar{g}^{\prime}(1)$, we can bound
\begin{align*}
|\bar{f}^m(r)| \leq &\frac{1}{ \min \bar{\calv}} \, \Big(\frac{N}{e} \, \frac{\max M}{\min M}\Big)\, , \quad
|D \bar{f}^m(r)| \leq  - \frac{1}{\min M} \, \ln \Big(\frac{\min M \, \min \bar{\calv}}{\max M} \, \min r\Big) \, .
\end{align*}
\end{lemma}

\section{Survey of existence results for multicomponent fluids}\label{ShortSurvey}

Both for the compressible problem (IBVP$^m$) (or its rescaled variant ($\overline{\text{ IBVP}}^m$) and for the incompressible model (IBVP$^{\infty}$), the local-in-time existence of strong solutions starting from sufficiently regular initial data is known from the papers \cite{bothedruet}, \cite{bothedruetincompress}. About well-posedness results for the compressible case, let us also mention \cite{giovan}, Ch.\ 9 and the papers \cite{giovanmatuII}, \cite{piashiba19}, \cite{piashiba19pr}, \cite{drumaxmix}. 
Moreover the existence of certain global-in-time weak solutions has been proved in \cite{dredrugagu20} and \cite{druetmaxstef}. 

In the most references, the existence results for compressible models (this means (IBVP$^m$) with $m < + \infty$) rely on a change of variables allowing to exhibit the parabolic-hyperbolic normal form of the PDE system \eqref{mass}, \eqref{momentum}. Indeed, the $N$ partial mass densities $\rho_1, \ldots,\rho_N$, cannot enjoy full parabolic regularity, because, summing up over $i = 1,\ldots,N$ in \eqref{mass} and using the property \eqref{MCONSTRAINT} of $\{M_{ij}\}$, the total mass density $\varrho$ is seen to be subject to the continuity equation $\partial_t\varrho+\divv(\varrho\,v) = 0$.
For this reason, beside the 'hyperbolic variable' $\varrho$, further $N-1$ 'diffusive variables' are needed to allow exhibiting the full parabolic regularity inherent to diffusion.

Canonically, the differences $\mu_1 - \mu_N, \ldots, \mu_{N-1} - \mu_N$ occurring in \eqref{DIFFUSPRIME} are considered. However, in order to study the incompressible limit, we shall introduce a slightly more general way which, still, exploits linear combinations of the chemical potentials. 
Beside this possibility, another more complicated approach of choosing diffusive variables was introduced in \cite{druetmaxstef} in order to treat the weak solution theory of multicomponent systems with Maxwell--Stefan diffusion (cf.\ \eqref{maxstefbase}, \eqref{maxstefreg}). In this case, the linear proportionality between the mobility coefficients and the densities allows for not more than weighted for the gradients of $\mu_i-\mu_N$, and the control degenerates for vanishing densities.

%\subsubsection{Strong solutions}\label{SSC}

\subsection{Strong solutions}

The change of variables \eqref{MV1} of the section \ref{ENTROPICVAR} allows to introduce strong and weak solutions with mixed parabolic-hyperbolic regularity. For an overview of the original ideas, we refer to the Chapter 8 of \cite{giovan} for an overview.
We let $\xi^{1}, \ldots, \xi^{N}$ with $\xi^N = 1^N$ be the basis used to perform the transformation, and $\eta^{1}, \ldots, \eta^N$ be the dual basis.
Next, we can multiply the mass continuity equations \eqref{mass} with the axes $\xi^1, \ldots,\xi^N$. With $R_k(\varrho, \, q) := \xi^k \cdot \mathscr{R}(\varrho,q)$ for $k = 1,\ldots,N-1$, where $\mathscr{R}$ is the map of \eqref{newpartdens}, we obtain the equivalent system of equations
\begin{align}\begin{split}\label{massnew}
\partial_t R_k(\varrho, \, q) +\divv\Big(R_{k}(\varrho,q) \, v- \sum_{\ell = 1}^{N-1} \widetilde{M}_{k\ell}(\varrho,q)\, \nabla q_{\ell}\Big) =&  0 \quad \text{ for } k = 1,\ldots,N-1 \, ,\\
\partial_t \varrho + \divv (\varrho \, v) =&  0 \, .
\end{split}
\end{align}
Here we have introduced the $N-1 \times N-1$ Onsager matrix
\begin{align*}\begin{split}%\label{newmatrix}
& \widetilde{M}_{k\ell}(\varrho, \, q) := \xi^k \, M(\mathscr{R}(\varrho, \, q)) \, \xi^{\ell} \quad \text{ for } \quad k,\ell = 1,\ldots,N-1 \, ,
\end{split}
\end{align*}
or, equivalently $\widetilde{M}(\varrho, \, q) = \Pi^{\sf T} \, M(\mathscr{R}(\varrho, \, q)) \, \Pi$, where $\Pi$ is from \eqref{Pi}.
Under the condition (B3), the matrix $\widetilde{M}(\cdot)$ can be shown to be symmetric and strictly positive definite on regular states, because we factorised the kernel ${\rm span}(\xi^N)$ of $M(\rho)$ by the projection (cf.\ \eqref{tildempos}).

For the new variables $(\varrho, \, q_1,\ldots,q_{N-1})$ and $v$, the PDE system \eqref{momentum} reads
\begin{align}\label{momentumnew}
 \partial_t(\varrho \, v) + \divv(\varrho \, v\otimes v -\mathbb{S}(\nabla v)) + \nabla P(\varrho, q) = \varrho \, b \, .
\end{align}
The system \eqref{massnew}, \eqref{momentumnew} is a parabolic-hyperbolic system in normal form. Another viewpoint is to see it as a compressible Navier--Stokes system coupled, via the pressure, to a parabolic system of the general form for the variables $q_1,\ldots,q_{N-1}$. In \cite{bothedruet}, adopting this second viewpoint, we performed some well-posedness analysis for this PDE system in the strong solution class
\begin{align*}
\varrho \in W^{1,1}_{\calp,\infty}(Q_{t^*}), \quad  q \in W^{2,1}_\calp(Q_{t^*}; \, \mathbb{R}^{N-1}), \quad v \in W^{2,1}_\calp(Q_{t^*}; \, \mathbb{R}^3) \, ,
\end{align*}
with an index $\calp > 3$ and where $t^*$ is some (small) constant determined by the data. Strictly positive mass densities of class $W^1_\calp(Q_{t^*}; \, \mathbb{R}^N)$ are recovered with the help of the map $\mathscr{R}$.
For the interested readers, the full assumptions necessary for the short-time existence result can be read in \cite{bothedruet}: They amount to requiring (A1)-(A3) for $g_1, \ldots,g_N$ and (B1)-(B3) for $\{M_{ij}\}$, where the smoothness in (A1) has to be increased to $C^3$ and the one in (B1) to $C^2$. The domain $\Omega$ should be assumed of class $\mathscr{C}^2$.

\subsubsection{Incompressible model}

For the incompressible case (IBVP$^{\infty}$), the mathematical structure of the PDE system is exhibited with a similar change of variables. However, we are faced with a singular potential $f^{\infty}$ (cf.\ \eqref{limitFE}) where  $f^{\infty}(\rho) = + \infty$ everywhere outside of the surface $S^0$ of states satisfying \eqref{CONSTRAINT}:
$ S^0 := \{ \rho \in \mathbb{R}^N_+ \, : \, \sum_{i=1}^N \calv_i \, \rho_i = 1\}$, with the vector $\calv$ of \eqref{defbarupsiloni}. Interestingly, the convex conjugate 
$$ (f^{\infty})^*(\mu) := \sup_{\rho \in S^0} \{\mu \cdot \rho - f^{\infty}(\rho)\} \quad \text{ for }\quad \mu \in \mathbb{R}^N \, ,$$
is a smooth convex function, though it is not strictly convex. In fact,  $(f^{\infty})^*$ is affine in the direction of $\calv$ and we have
\begin{align}\label{affine}(f^{\infty})^* (w + s \, \calv) = (f^{\infty})^*(w) + s \quad \text{ for all } w \in \mathbb{R}^N, \, s \in \mathbb{R} \, .
 \end{align}
This property can be exploited by further specialising the choice of the basis $\xi^1, \ldots,\xi^N$ used for the projection. In addition to $\xi^N = 1^N$, we choose $\xi^{N-1} := \calv$. This is possible if, among the species, at least two have a different specific volume. In this way, the vectors $1^N$ and $\calv$ are not parallel.

We again define $q_i = \eta^i \cdot \mu$ for $i = 1,\ldots,N-2$ and, since the combination $\zeta := \mu \cdot \eta^{N-1}$ plays a very special role, we single it out with an independent name. Next we can, as in the compressible case, eliminate $\mu \cdot \eta^N = \mathscr{M}(\varrho, \, q_1,\ldots,q_{N-2})$ using the equation
\begin{align*}
 \varrho = 1^N \cdot \nabla (f^{\infty})^*\Big(\Pi \, q + \mathscr{M} \, 1^N\Big) \, ,
\end{align*}
in which the component $\zeta$ does not occur owing to the property \eqref{affine}.

As shown in \cite{bothedruetincompress}, the change of variables $(\rho_1, \ldots,\rho_N) \leftrightarrow (\varrho, \, q_1, \ldots,q_{N-2}, \, \zeta)$ yields for the pressure and the densities the following expressions
\begin{align}\begin{split}\label{newpartdensincomp}%\label{newpressincomp}
 p = (f^{\infty})^*\Big(\sum_{k=1}^{N-2} q_{k} \, \xi^k + \mathscr{M}(\varrho,q) \, 1^N\Big) + \zeta =: P^{\infty}(\varrho, \, q) + \zeta \quad \text{ pressure,}\\
 \rho = \nabla_{\mu} (f^{\infty})^*\Big(\sum_{k=1}^{N-2} q_{k} \, \xi^k + \mathscr{M}(\varrho,q) \, 1^N\Big) =: \mathscr{R}^{\infty}(\varrho, \, q) \quad \text{ densities} \, .
 \end{split}
\end{align}
The equivalent expression of the PDEs \eqref{mass} for the new variables $\varrho$, $q_1, \ldots,q_{N-2}$ and $\zeta$ is an elliptic--parabolic--hyperbolic system, which reads
\begin{align}\begin{split}\label{massnewincomp}
 & \partial_t R^{\infty}_k(\varrho, q) +\divv\Big( R^{\infty}_k(\varrho, q) \, v - \sum_{\ell = 1}^{N-2}\widetilde{M}_{k\ell}(\varrho,q,\zeta) \, \nabla q_{\ell} - A_k(\varrho,q,\zeta) \,\nabla \zeta\Big) = 0 \\
 & \qquad \text{ for } k = 1,\ldots,N-2 \, , \\
& - \divv\Big(d(\varrho,q,\zeta) \, \nabla \zeta + \sum_{\ell = 1}^{N-2} A_{\ell}(\varrho,q,\zeta) \, \nabla q_{\ell}\Big) = 0 \,, \\
& \partial_t \varrho + \divv (\varrho \, v) = 0 \, .
\end{split}
\end{align}
Here $R_k^{\infty}(\varrho,q) = \mathscr{R}^{\infty}(\varrho,q) \cdot \xi^k$ for $k = 1,\ldots,N-2$, while the diffusion coefficients are given by
\begin{align*}
\widetilde{M}_{k\ell}(\varrho,q,\zeta) := & \xi^k \, M(P^{\infty}(\varrho,q)+\zeta, \, \mathscr{R}(\varrho,q)) \, \xi^{\ell} \, , \quad \text{ for } k,\ell = 1,\ldots,N-2 \, ,\\
A_{\ell}(\varrho,q,\zeta) := & \xi^{\ell} \, M(P^{\infty}(\varrho,q)+\zeta, \, \mathscr{R}(\varrho,q)) \, \xi^{N-1}\, , \quad \text{ for } \quad \ell = 1,\ldots,N-2 \, ,\\
d(\varrho,q,\zeta) := & \xi^{N-1} \, M(P^{\infty}(\varrho,q)+\zeta, \, \mathscr{R}(\varrho,q)) \, \xi^{N-1} \, .
\end{align*}
In the new variables, the equation \eqref{momentum} reads 
\begin{align}\label{momentumnewicomp}
 \partial_t (\varrho v) + \divv(\varrho \, v \otimes v -\mathbb{S}(\nabla v)) + \nabla P^{\infty}(\varrho,q) + \nabla \zeta = \varrho \, b \, .
\end{align}
The system \eqref{massnewincomp}, \eqref{momentumnewicomp} has been analysed in the paper \cite{bothedruetincompress} for the case that the thermodynamic diffusivities $M(p,\rho) = M(\rho)$ satisfy (B1)--(B3) and do not depend on pressure. Then we find a unique strong solution in the class
\begin{align}\label{PHE}
\varrho \in W^{1,1}_{\calp,\infty}(Q_{t^*}), \quad  q \in W^{2,1}_\calp(Q_{t^*}; \, \mathbb{R}^{N-2}), \quad \zeta \in W^{2,0}_\calp(Q_{t^*})\, , \quad v \in W^{2,1}_\calp(Q_{t^*}; \, \mathbb{R}^3) \, ,
\end{align}
where $\calp>3$ and $t^*$ is some (small) constant determined by the data.

We recover strictly positive partial mass densities subject to \eqref{CONSTRAINT} by defining $\rho:= \mathscr{R}^{\infty}(\varrho, \, q)$ with $\mathscr{R}$ from \eqref{newpartdensincomp}. This implies that $\rho \in W^1_\calp(Q_{t^*}; \, \mathbb{R}^N)$.

We note that the regularity required in the assumption (C) for (IBVP$^{\infty}$) is stronger than \eqref{PHE}. Next we formulate an existence theorem for the incompressible model, where the regularity is improved from 'strong' to 'classical'. This shall be useful to prove that (C) is valid and the weak solutions to (IBVP$^m$) converge to the unique regular solution to (IBVP$^{\infty}$) at least for a small time-interval.
\begin{theo}\label{EXISSTRONGINFTY}
Let $\Omega \subset \mathbb{R}^3$ be a domain of class $\mathscr{C}^{2,\alpha}$ with a $0 < \alpha \leq 1$.
 We assume that $g_1,\ldots,g_N$ belong to $C^3(]0,+\infty[)$ and are strictly concave functions satisfying $\rm (A2), (A3)$, and that the entries of $\{M_{ij}\}$ are functions of class $C^{2}(\mathbb{R}^N_+)$ satisfying $\rm (B2)$ and $\rm (B3)$. If the initial data satisfy $v^0 \in C^{2,\alpha}(\overline{\Omega}; \, \mathbb{R}^3)$, $v^0 = 0$ on $\partial \Omega$ and $\rho^0 \in C^{2,\alpha}(\overline{\Omega}; \, \mathbb{R}^N)$, $\inf_{x \in \Omega} \min_{i} \rho^0_i(x) \geq s_0 > 0$, and $\rho^0(x) \cdot \calv \equiv 1$, then the boundary-value-problem \emph{(IBVP$^{\infty}$)} possesses a local solution $(\rho^{\infty}, \, p^{\infty}, \, v^{\infty})$ such that $\rho^{\infty}_1, \ldots,\rho^{\infty}_N \in W^1_{\infty}(Q_{\bar{\tau}})$ and $v^{\infty} \in C^{2+\alpha,1+\alpha/2}(Q_{\bar{\tau}}; \, \mathbb{R}^3)$.
\end{theo}
For a full proof we refer to the upcoming survey \cite{drusurv}.

\subsection{Weak solutions}

While for the theory of strong solutions, the initial data provide, at least for a short time, the strict positivity of the densities, the existence theory for global weak solutions has to face the additional difficulty how to make sense of the model for vanishing species. 

The main mathematical problem concerns the way how to appropriately express diffusion. Indeed, recall that the chemical potentials are defined via $\mu_i = \partial_{\rho_i}f(\rho)$. If $f$ is a function of Legendre type, then it is by definition \emph{essentially smooth}, and this implies the blow-up of the gradient on the boundary of $\mathbb{R}^N_+$:
$$ |\mu| = |\nabla_{\rho}f(\rho)| \rightarrow + \infty \quad \text{ for } \min_{i=1,\ldots,N} \rho_i \rightarrow 0+ \, .$$ 
Hence, in the case that $\rho_i = 0$ on some set of positive measure, several expressions in the PDEs that involve the chemical potentials must be transformed in order to remain meaningful under even extreme behaviour of the system.
In this point, the appropriate weak solution concept is affected by the properties of the matrix $M(\rho)$. In particular, it depends on the behaviour of the positive eigenvalues of $M(\rho)$ for $\min_i \rho_i \rightarrow 0+$. Recall that the crucial property for the definition of weak solutions is the inequality \eqref{DISSIP}. Now, let us represent $\mu = \Pi q + \mathscr{M} \, 1^N$. 
Using \eqref{DIFFUSFLUX} and the constraint \eqref{MCONSTRAINT}, we obtain that
\begin{align}\label{dissnew}
 - \int_{Q_t} J \, : \, \nabla \mu \, dxd\tau = & \sum_{i,j=1}^N \int_{Q_t} M_{ij}(\rho) \, \nabla \mu_i \cdot \nabla \mu_j = \sum_{k,\ell = 1}^{N-1} \int_{Q_t} \xi^k \,  M(\rho) \, \xi^{\ell} \, \nabla q_{k} \cdot \nabla q_{\ell} \, dxd\tau \nonumber\\
 = & \int_{Q_t} \widetilde{M}(\varrho, \, q) \, \nabla q \cdot \nabla q \, dxd\tau\, .
\end{align}
The smallest eigenvalue of $M(\rho)$ is always zero due to (B2). If we assume that the second smallest eigenvalue is uniformly positive, that is, (B3$^{\prime}$) is valid, then the Onsager operator $\widetilde{M}(\varrho, \, q) = \Pi^{\sf T}\, M(\rho) \Pi$ generates a uniformly positive quadratic form, and, with $\lambda_0$ as in (B3$^{\prime}$) and $c_{\eta}$ from \eqref{tildempos}, we shall have
$ - \int_{Q_t} J \, : \, \nabla \mu \, dxd\tau \geq (\lambda_0/c_{\eta}) \, \int_{Q_t} |\nabla q|^2 \, dxd\tau$. Thus it is possible and natural to perform the weak existence theory using the variables $q_1,\ldots,q_{N-1}$, and we obtain $q \in L^2(0,\bar{\tau}; \, W^{1,2}(\Omega; \, \mathbb{R}^{N-1}))$ as in the paper \cite{dredrugagu20}.
The second case occurs if the positive eigenvalues of $M(\rho)$ degenerate for $\min_i \rho_i \rightarrow 0+$. This occurs for instance if the matrix $M(\cdot)$ is constructed from inversion of the Maxwell--Stefan equations (cf.\ \eqref{maxstefbase}, \eqref{maxstefreg}). In this case we cannot expect to control the gradients of $q$ if the densities do not remain strictly positive. In the papers \cite{druetmaxstef}, we tried to find a suitable substitute for the diffusive variables. Here too, reinforcing of the condition (B3) is needed, but it is more complex than (B3$^\prime$).

Next we sketch in more details these two possible scenarios.

\subsubsection{Compressible model: Weak solutions of type-I}

In the paper \cite{dredrugagu20} devoted to electrolytes, we imposed the conditions (B1), (B3$^\prime$). Then there are two constants $0< \lambda_0 \leq \bar{\lambda} < + \infty$ such that
\begin{align}\label{ONSA}
 \lambda_0 \, |\mathcal{P}\eta|^2 \leq M(\rho) \eta \cdot \eta \leq \bar{\lambda} \, (1+|\rho|) \, |\mathcal{P}\eta|^2 \quad \text{ for  all } \eta\in \mathbb{R}^N, \, \rho \in \mathbb{R}^N_+ \, .
\end{align}
This corresponds to a uniformly positive Onsager tensor.
 
Moreover, the free energy function must provide a coercivity estimate $f(\rho) \geq c_0 \, |\rho|^{\gamma} - c_1$ with $\gamma \geq \frac{9}{5}$\footnote{We expect that $\gamma >\frac{3}{2}$ is sufficient as for compressible Navier-Stokes, but this was not shown in \cite{dredrugagu20}}. Then, we obtain the existence of a weak solution vector $(\varrho, \, q, \, v)$ in the class
\begin{align}\begin{split}\label{coilemoi}
 \varrho \in L^{\infty}(0,\bar{\tau}, \, L^{\gamma}(\Omega)),& \quad q \in L^2(0,\bar{\tau}; \, W^{1,2}(\Omega;\,\mathbb{R}^{N-1})),\\
v \in L^2(0,\bar{\tau}; \, W^{1,2}(\Omega; \, \mathbb{R}^3)),& \quad \sqrt{\varrho} \, |v| \in L^{\infty}(0,\bar{\tau}; \, L^2(\Omega; \, \mathbb{R}^3)) \, .
\end{split}
\end{align}
Technically, this result exploits the compactness theory for compressible Navier--Stokes equations in connection with the control on $\nabla q$ in $L^2(Q_{\bar{\tau}})$.

Interestingly, the partial mass densities of the species, which are recovered via the map $\rho_i = \mathscr{R}_i(\varrho, \, q)$ of \eqref{newpartdens}, are almost everywhere strictly positive unless the vacuum $\varrho = 0$ occurs of some set of positive measure. To the present state of knowledge, a vacuum cannot be prevented in the weak theory of compressible Navier--Stokes equations.

For $i = 1,\ldots,N$, the diffusion fluxes are expressed via $J^i = - \sum_{k=1}^{N-1} e^i \, M(\rho) \,  \xi^k \, \nabla q_k$.
This expression makes everywhere sense on the domain provided that $\rho \mapsto M(\rho)$ is continuous on $\overline{\mathbb{R}^N_+}$, in particular up to $\varrho = 0$. The pressure $p = P(\varrho, \, q)$ (cf.\ \eqref{newpress}) associated with the weak solution and the diffusion flux matrix satisfy
\begin{align}\label{coilemoi2}
p \in L^r(Q_{\bar{\tau}}) \text{ with } r = \min\{1+1/\gamma, \, 5/3-1/\gamma\} \, , \quad  J \in L^2(0,\bar{\tau}; \, L^{\frac{2\gamma}{1+\gamma}}(\Omega; \, \mathbb{R}^{N\times 3})) \, .
\end{align}
In the dissipation inequality \eqref{DISSIP}, the diffusion $\zeta^{\rm Diff}$ is re-interpreted via \eqref{dissnew}. In particular, weak solutions satisfy
\begin{align}\label{WEAKDISS}
 \int_{Q_t} \Pi^{\sf T} \, M(\rho) \, \Pi \, \nabla q \cdot \nabla q \, dxd\tau < + \infty \, .
\end{align}
The weak solutions $(\varrho, \, q, \, v)$ satisfy the weak form of the equations \eqref{massnew}, \eqref{momentumnew}, while $\rho = \mathscr{R}(\varrho, \, q)$ satisfies with $v$ and $p = P(\varrho, \, q)$ the weak form of \eqref{mass}, \eqref{momentum}. We resume these claims in the following theorem.
\begin{theo}\label{EXIWEAK}
Assume that $g_1, \ldots,g_N$ are subject to the assumptions ${\rm (A1)-(A5)}$, and that $\{M_{ij}\}$ satisfies ${\rm(B1), \, (B2) \text{ and } (B3^\prime)}$. Moreover,  $v^0 \in L^{2}(\Omega; \, \mathbb{R}^3)$ and $\rho^{0} \in L^{\infty}(\Omega;\mathbb{R}^N)$ satisfies $\inf_{i=1,\ldots,N, \, x \in \Omega} \rho^{0,i}(x) > 0$. Then, the problem \emph{(IBVP$^m$)} possesses at least one weak solution such that \eqref{coilemoi}, \eqref{coilemoi2} and \eqref{WEAKDISS} are valid. Introducing $Q^+(\varrho) := \{(x,t) \, : \, \varrho(x,t) > 0\}$, we can show that $\min_i \rho_i(x,t) > 0$ almost everywhere in $Q^+(\varrho)$, and that  $\Pi q + \mathscr{M}(\varrho, \, q) = RT/M \, \ln \hat{x}(\rho) + g(\hat{p}(\rho)) $ for almost all $(x,t) \in Q^+(\varrho)$.
\end{theo}
Remark that in the paper \cite{dredrugagu20}, the existence has been proved in the more general context of electro-chemistry, where there is an additional equation for the electric field. As a counterpart, only a special choice of $g_1, \ldots, g_N$ was considered: $g_i(p) = g_0(p) \, \calv_i$.
Nevertheless, all approximation arguments carry over to (IBVP$^m$). We refer to the forthcoming survey paper \cite{drusurv} for full proofs.

\subsubsection{Compressible model: Weak solutions of type-II}

The assumption (B3$^{\prime}$) of a uniformly positive Onsager tensor is in particular problematic in the study mixtures with dilute constituents. The second approach on diffusion is a positive Onsager--tensor with singular diffusivities, as it for instance arises from inversion of the Maxwell--Stefan equations (see the example \eqref{maxstefbase}). In this case we do not obtain a control on $\nabla q$ in the fashion of \eqref{dissnew}. 
Presently, it does not seem possible to treat in the framework of weak solutions the Maxwell--Stefan--Navier--Stokes system as it is naturally formulated in the literature. However, in the paper \cite{druetmaxstef}, we developed a theory of weak solutions on the idea that the Maxwell--Stefan diffusion coefficients might exhibit certain singularities at very low (or very large) pressure.

%For general diffusion coefficients, we refer ????.
%EXAMPLE WITH D????

In this paper, instead of the change of variables \eqref{newpartdens}, we represented the vector of mass densities via 
\begin{align}\label{TrafotypeII}
\rho = \mathcal{X}(p, \, w_1, \ldots, w_N) \, ,
\end{align}
where $p$ is the pressure and $w_1, \ldots, w_N$ a \emph{normalised state} subject to the constraint
\begin{align*}
 \hat{p}(T,\, w_1, \ldots, w_N) = p^0  \, ,
\end{align*}
where $p^0$ is a reference pressure, and $\hat{p}(T, \cdot)$ is the representation resulting of the equation of state \eqref{EOS}.
This means that the mass densities are represented by a point $w$ on the isobar $p = p^0$, and departure from this isobar is ''measured'' by the pressure variable. It can next be shown that the quotients $\rho_i/w_i$ are bounded away from zero and from above by certain functions of pressure only. This indicates that, in the range of finite pressures, the variable $w_i$ tends to zero if and only if $\rho_i$ tends to zero.
Hence, for a mobility matrix of Maxwell-Stefan type (cf.\ \eqref{maxstefreg}), we have
\begin{align*}
 d_0(p) \, \mathdutchcal{P}^{\sf T} (W^{-1}R) \, W \, \mathdutchcal{P} \leq M(\rho) \, , 
 %\leq  d_1 \, \mathdutchcal{P}^{\sf T} (W^{-1}R) \, W \, \mathdutchcal{P} \, ,
\end{align*}
with $R = \text{diag}(\rho)$, $W = \text{diag}(w)$ and $\mathdutchcal{P} = \mathbb{I} - 1^N \otimes \hat{y}(\rho)$. Since the entries of the diagonal matrix $W^{-1}R$ are below by  functions on pressure, we get
 \begin{align*}
 \tilde{d}_0(p) \, \mathdutchcal{P}^{\sf T} \, W \, \mathdutchcal{P} \leq M(\rho) \, .
 %\leq  \tilde{d}_1 \, \mathdutchcal{P}^{\sf T} \, W \, \mathdutchcal{P} \, ,
\end{align*}
Under certain conditions formulated in the paper \cite{druetmaxstef} for the thermodynamic diffusivity $M_{ij}$ for extreme pressures, we can \emph{a priori} prove that $\inf_p \tilde{d}_0(p) > 0$, 
%$\sup_p \tilde{d}_1(p) < +\infty$. Hence,
and the diffusion in the energy inequality \eqref{DISSIP} yields $\zeta^{\rm Diff} \geq c \, |\nabla \sqrt{w}|^2$. Since $\hat{p}(w) = p^0$, we have $w \cdot \calv = 1$ and it follows that $w$ is bounded. In this way, the bound on the gradient of the square-root yields also a control on $|\nabla w|^2 \leq c \, \zeta^{\rm Diff}$.
The control on the total mass density $\varrho$ in $L^{\infty}(0,\bar{\tau}; \, L^{\gamma}(\Omega)) $ subject to the continuity equation, combined with the control on $\nabla\sqrt{w}$ and $\nabla w$ in $L^2(Q; \, \mathbb{R}^N)$ allow to perform the technical part of the analysis.

In this approach, the diffusion flux matrix is integrated to the solutions vector. A weak solution vector consists of $(\rho, \, v, \, J)$ with
\begin{gather*}
 \rho \in L^{\infty}(0,\bar{\tau}; \, L^{\gamma}(\Omega; \, \mathbb{R}^N)), \quad v \in L^2(0,\bar{\tau}; \, W^{1,2}(\Omega;\mathbb{R}^3)), \quad \sqrt{\varrho} \, |v| \in L^{\infty}(0,\bar{\tau}; \, L^2(\Omega;\mathbb{R}^3))\, ,\\
 J \in L^2(0,\bar{\tau}; \, L^{\frac{2\gamma}{1+\gamma}}(\Omega; \, \mathbb{R}^{N\times 3})) \, .
\end{gather*}
The weak solution satisfies the distributional formulation of \eqref{mass}, \eqref{momentum} and,
moreover, use of the associated normalised mass densities 
%\begin{align*}
$ w \in L^{\infty}(Q_{\bar{\tau}}; \, \mathbb{R}^N) \cap L^2(0,\bar{\tau}; \, W^{1,2}(\Omega; \, \mathbb{R}^N))$,
%\end{align*}
which are introduced via \eqref{TrafotypeII}, yields for the diffusion fluxes the following identities:
\begin{align*}
 J^i =  - \sum_{j=1}^N M_{ij}(\rho) \, \nabla \mu_j  \quad & \text{ a.e.\ in } Q^+ := \{(x,t) \in Q \, : \,  \inf_i \rho_i(x,t) > 0\}\, ,\\
 J^i =  0 & \quad \text{ in } \{(x,t) \in Q\, : \, \rho_i(x,t) = 0, \, \varrho(x,t) > 0\}\, ,\\
 J^i = - \sum_{j=1}^N (\underbrace{M(\rho) \, D^2f(w)}_{=: \Gamma})_{ij} \, \nabla w_j & \quad  \text{ in } \{(x,t) \in Q \, : \, \varrho(x,t) > 0\} \, . 
\end{align*}
Hence, the relationship between gradients of the state variables and diffusion fluxes is specified everywhere but for the vacuum set. For this type of solutions, we also obtain an energy inequality. Unfortunately, this point was not completely enlighted in the paper \cite{druetmaxstef} and it shall have to be discussed in further publications. Let us note that, in \eqref{DISSIP}, we can identify
\begin{align*}
 -\int_{Q_t} \zeta^{\rm Diff} \, dxd\tau = - \int_{Q_t} K(\rho)\, \frac{J}{\sqrt{w}} \cdot \frac{J}{\sqrt{w}} \, dxd\tau \, ,
\end{align*}
where, for $\rho = \mathcal{X}(p, \, w)$, the matrix $K$ is symmetric and positive definite on the orthogonal complement of ${\rm span}(\rho/\sqrt{w})$. Hence, for weak solutions, we also obtain a control on $J/\sqrt{w}$ in $L^2(Q; \, \mathbb{R}^{N\times 3})$.

%FEHLT SOMEWHERE MAXSTEF SYSTEM\\
%ATTENTION: SPECIAL FREE ENERGU MODELS

\subsubsection{ Weak solutions to the incompressible model}

The existence analysis for weak solutions to the incompressible model was tackled in \cite{druetmixtureincompweak}, but only the approach using the assumption (B3$^\prime$) or \eqref{ONSA} of a uniformly positive Onsager operator, which simplifies the handling of diffusion, was implemented up to now.
Hence the solution vector consists of $(\varrho, \, q_1, \ldots, q_{N-2}, \, \zeta, \, v)$. Due to the volume constraint $\sum_{i=1}^N \rho_i \, \calv_i = 1$, the total mass density is interval $[\varrho_{\min}, \, \varrho_{\max}]$ with the constants of \eqref{rhominmax}. Hence $\varrho \in L^{\infty}(Q_{\bar{\tau}})$ is also strictly positive, and the solution class is
\begin{gather*}
 q_1, \ldots, q_{N-2}, \, \zeta \in L^2(0,\bar{\tau}; \, W^{1,2}(\Omega)), \,\,\, 
 v \in L^2(0,\bar{\tau}; \, W^{1,2}(\Omega; \, \mathbb{R}^3)) \cap L^{\infty}(0,\bar{\tau}; \, L^2(\Omega;\mathbb{R}^3)) \, .
\end{gather*}
The main problem of the analysis concerns the pressure, which satisfies a bound only in $L^1(Q_{\bar{\tau}})$. In the paper \cite{druetmixtureincompweak}, the pressure is affected by a singular measure: $ p = (P^{\infty}(\varrho, \, q) +\zeta) \, d\lambda_4 + d\kappa$, where $\lambda_4$ denotes the four-dimensional Lebesgue--measure, and the regular measure $\kappa$ is concentrated in the a set of Lebesgue--measure zero. It can be shown that at every point $(x,t)$ of the singular set where the Lebesgue value of $\varrho$ exists, one of the threshold $\varrho_{\min}$ or $\varrho_{\max}$ of density is attained.

The partial mass densities are recovered as $\rho_i = \mathscr{R}^{\infty}_i(\varrho, \, q)$ with the map of \eqref{newpartdensincomp}. These are nonnegative $L^{\infty}$ functions. Moreover, $\rho_i$ is even strictly positive almost everywhere. The weak solutions satisfy the weak form of \eqref{mass} and \eqref{momentum}, where the pressure $ p = P(\varrho, \, q) + \zeta$ is supplemented by a singular measure $\kappa$.

\section{{\it A priori} bounds for the rescaled system}

\begin{prop}\label{EXIRESC}
Consider thermodynamic functions $\bar{g}_1, \ldots, \bar{g}_N$ satisfying ${\rm (A)}$ with $\gamma \geq 9/5$.
Assume moreover that $\{\bar{M}_{ij}\}$ is subject to ${\rm (B1)}$, ${\rm (B2)}$ and ${\rm (B3^{\prime})}$. 
Then, there exists a weak solution $(\bar{\rho}^m, \, \bar{v}^m)$ to the problem \emph{($\overline{\text{ IBVP}}^m$)} with initial data $(\bar{\rho}^{0,m}, \, \bar{v}^{0})$. The corresponding variables $(\bar{\varrho}^m, \, \bar{q}^m, \, \bar{v}^m)$ satisfy
\begin{align}\begin{split}\label{DISSIPm}
 & \int_{\Omega^{\rm R}} \Big( \frac{\bar{\varrho}_m(x,t)}{2} \, |\bar{v}^m(x,t)|^2 + \bar{f}^m(\bar{\rho}^m(x,t))\Big) \, dx \\
 & \qquad + \int_0^t \int_{\Omega^{\rm R}} \overline{\mathbb{S}}(\nabla \bar{v}^m) \, : \, \nabla \bar{v}^m + \Pi^{\sf T} \, \bar{M}(\bar{\rho}^m)\Pi \nabla \bar{q}^m\, : \, \nabla \bar{q}^m \, dxd\tau \\
 & \qquad \leq  \int_{\Omega^{\rm R}} \Big( \frac{\varrho_0^m(x)}{2} \, |v^0(x)|^2 + \bar{f}^m(\bar{\rho}^{0,m}(x))\Big) \, dx - \int_0^t\int_{\Omega^{\rm R}} \bar{\varrho}^m \,  \bar{v}^m_3 \, dxd\tau
 \end{split}
\end{align}
and a uniform bound for the norms defined in \eqref{coilemoi} is valid.
\end{prop}
\begin{proof}
Owing to the Lemma \ref{UNIFF2} the rescaled free energies $\bar{f}^m$ satisfy all requirements of Theorem \ref{EXIWEAK}, so we obtain the global existence of a weak solution $(\bar{\varrho}^m, \, \bar{q}^m, \, \bar{v}^m)$ so that \eqref{DISSIPm} is valid. From the latter, we can directly read off the bound
\begin{align*}
\sup_{m \in \mathbb{N} } \|\bar{\rho}^m\|_{L^{\gamma,\infty}(Q^{\rm R})} + \|\sqrt{\bar{\varrho}_m} \, \bar{v}^m\|_{L^{2,\infty}(Q^{\rm R})} + \|\bar{\nabla} \bar{v}^m\|_{L^2(Q^{\rm R})} + \|\bar{\nabla} \bar{q}^m\|_{L^2(Q^{\rm R})} < +\infty \, .
\end{align*}
\end{proof}
It remains to show how to obtain a bound for $\bar{q}^m$, since the bounds derived in the paper \cite{dredrugagu20} for $\bar{q}^m$ in $L^{2,1}$ were not shown to be uniform in $m$. In a first step we prove convergence to an incompressible state and a weighted convergence property for the pressure.
\begin{lemma}\label{Bogovencorelui}
We adopt the assumptions of Prop.\ \ref{EXIRESC} and we let $(\bar{\varrho}^m, \, \bar{q}^m, \, \bar{v}^m)$ be a weak solution to \emph{($\overline{\text{ IBVP}}^m$)}. Then $\|\bar{\rho}^m \cdot \bar{\calv} - 1\|_{L^{1,\infty}(Q^{\rm R})} \rightarrow 0$ for $m \rightarrow \infty$.
\end{lemma}
\begin{proof}
We have $|\bar{g}^m(\pi^m) \cdot \rho^m - \pi^m| =  \bar{f}^m(\rho^m) + |\bar{k}(\rho^m)|$. Since $|\bar{k}(\rho^m)| \leq c \, |\rho^m|$ with $c = N^{3/2}/(e\min \bar{M})$, we can bound $\sup_m \|\bar{g}^m(\pi^m) \cdot \rho^m - \pi^m\|_{L^{1,\infty}}$ using the bounds for $\bar{f}^m$ and $\rho^m$ in $L^{1,\infty}$. Recalling that $1+\pi^m/m = \widehat{p}(\rho^m)$ we see that
\begin{align*}%\begin{split}\label{rbarvweakI}
\bar{g}^m(\pi^m) \cdot \rho^m - \pi^m = & m \, \rho^m  \cdot \Big(\bar{g}\Big(1+\frac{\pi^m}{m}\Big) - \bar{g}(1) - \bar{g}^{\prime}\Big(1+\frac{\pi^m}{m}\Big) \, \frac{\pi^m}{m}\Big)\\
= & m \, \rho \cdot \big(\bar{g}(\widehat{p}(\rho^m)) - \bar{g}(1) - \bar{g}^{\prime}(\widehat{p}(\rho^m)) \, (\widehat{p}(\rho^m)-1)\big)\\
\geq & \varrho^m \, m \, \min_i  \Big\{\bar{g}_i(\widehat{p}(\rho^m)) - \bar{g}_i(1) - \bar{g}^{\prime}_i(\widehat{p}(\rho^m)) \, (\widehat{p}(\rho^m)-1)\Big\} \, ,
%\end{split}
\end{align*}
which shows that
\begin{align}\label{rbarvweakI}
\limsup_{m\rightarrow \infty} m \, \int_{\Omega} \varrho^m \, \min_i  \Big\{\bar{g}_i\big(\widehat{p}(\rho^m)\big) - \bar{g}_i(1) - \bar{g}^{\prime}_i\big(\widehat{p}(\rho^m)\big) \, (\widehat{p}(\rho^m)-1)\Big\} \, dx < + \infty \, .
\end{align}
Due to Lemma \ref{UNIFF2}, there is a fixed $0 < \bar{\theta} < 1$ such that $\pi^m < - m \, \bar{\theta}$ implies that $ |\bar{g}^m(\pi^m) \cdot \rho^m - \pi^m| \geq |\pi^m|/2 \geq \bar{\theta} \, m/2$. Hence, invoking $\sup_m \|\bar{g}^m(\pi^m) \cdot \rho^m - \pi^m\|_{L^{1,\infty}} < +\infty$ again,
\begin{align}\label{rbarvweakII}
\limsup_{m \rightarrow \infty} m \, |\{x \, : \, \pi^m(x,t) < - m \, \bar{\theta} \}| < + \infty \, .
\end{align}
For $a > 0$ arbitrary consider the set $\omega_a^m(t) := \{x \in \Omega \, : \, |\widehat{p}(\rho^m(x,t)) - 1| > a\}$. Then, due to \eqref{rbarvweakII},
\begin{align}\label{rbarweakIII}
 \limsup_{m\rightarrow \infty} m \, \Big| \omega^m_a(t) \cap  \{x \, : \, \pi^m(x,t) < - m \, \bar{\theta}\} \Big| < +\infty \, .
\end{align}
In the complement $\{x \, : \, \pi^m(x,t) < - m \, \bar{\theta}\}^{\rm c}$, we know that $1+\pi^m(x,t)/m \geq 1-\bar{\theta}$ which, due to the definition of $\pi^m$ and $\widehat{p}$, implies that $\varrho^m > 1/\max \bar{g}^{\prime}(1-\bar{\theta}) = a_0 > 0$. Hence, exploiting the strict concavity of $\bar{g}$, we find in $\omega_a^m(t) \cap \{x \, : \, \pi^m(x,t) < - m \, \bar{\theta}\}^{\rm c}$ that
\begin{align*}
 \varrho^m \, \min_i  \Big\{\bar{g}_i(\widehat{p}(\rho^m)) - \bar{g}_i(1) - \bar{g}^{\prime}_i(\widehat{p}(\rho^m)) \, (\widehat{p}(\rho^m)-1) \Big\} \geq a_0 \, c(a) >0\\
\text{ with } \quad  c(a) := \inf_{|s-1| > a} \min_i \{\bar{g}_i(s) - \bar{g}_i(1) - \bar{g}^{\prime}_i(s) \, (s-1)\} \, .
\end{align*}
Now, \eqref{rbarvweakI} yields $\limsup_{m \rightarrow \infty} m \, \Big|\omega_a^m(t) \cap \{x \, : \, \pi^m(x,t) < - m \, \bar{\theta}\}^{\rm c}\Big| < +\infty$, and with \eqref{rbarweakIII},
\begin{align}\begin{split}\label{forphattozero}
& \limsup_{m\rightarrow \infty}  \, m \, |\omega_a^m(t)| <+\infty \quad \text{ for all } \quad a >0, \, t > 0 \, .
\end{split}
\end{align}
Next, we recall \eqref{Phiidentity} and $\pi^m/m = \widehat{p}(\rho^m) - 1$ allow to bound
\begin{align*}
 |1 - \rho^m \cdot \bar{\calv}| \leq |\rho^m| \, \int_0^1 |\bar{g}^{\prime\prime}(1+ \lambda \, (\widehat{p}(\rho^m)-1))| \, d\lambda \, |\widehat{p}(\rho^m)-1| \, .
\end{align*}
If $\hat{\pi}^m(\rho^m) \geq - \bar{\theta} \, m$, we have $\widehat{p}(\rho^m) \geq 1-\bar{\theta}$. Hence, if $\hat{\pi}^m(\rho^m) \geq - \bar{\theta} \, m$ and $\widehat{p}(\rho^m) \leq 2$, we get $| 1 - \rho^m \cdot \bar{\calv} | \leq C_1 \, |\widehat{\bar{p}}(\rho^m)-1|$. For $b >0$, consider now sets $\tilde{\omega}^m_b(t) := \{x \, : \, 
| 1 - \rho^m(x,t) \cdot \bar{\calv} | \geq b\}$. We have just shown that
\begin{align*}
\tilde{\omega}^m_b(t) \cap \{x \, : \, \pi^m \geq -\bar{\theta} \, m\} \cap \{x \, : \, \widehat{p}(\rho^m)\leq 2 \} \subseteq \omega^m_{a_1}(t) \quad \text{ with } \quad a_1 = b/C_1  \, . 
\end{align*}
Then, we can easily observe that $\tilde{\omega}^m_b(t) \subseteq \{x \, : \, \pi^m < -m \, \bar{\theta}\} \cup \omega_{a_1}^m(t) \cup \omega_{1}^m(t)$, 
% \begin{align*}
% |\tilde{\omega}^m_b(t)| \leq  |\{x \, : \, \pi^m \leq -m \, \bar{\theta}\}| + |\omega_{a_1}^m(t)| + |\omega_{1}^m(t)| \, ,
% \end{align*}
and \eqref{rbarvweakII} and \eqref{forphattozero} imply that $\limsup_{m \rightarrow \infty} m \, |\tilde{\omega}_b^m(t)| < +\infty$ for all $b > 0$ and $t > 0$. Thus
%\begin{align}\label{lamasse}
% \limsup_{m \rightarrow \infty} m \, \Big|\tilde{\omega}_b^m(t)\Big| < +\infty \quad \text{ for all } b > 0, \, t > 0\, .
%\end{align}
% Since for all $b>0$, we see by means of \eqref{lamasse} that 
\begin{align*}
\|\rho^m\cdot\bar{\calv} - 1\|_{L^1(\Omega)} \leq & b \, |(\tilde{\omega}_b^m(t))^{\rm c}| + \sup_m \|1-\rho^m\cdot\bar{\calv}\|_{L^{\gamma,\infty}} \, |\tilde{\omega}^m_b(t)|^{1-\frac{1}{\gamma}} \,
\leq  b \, |\Omega| + C \, \Big(\frac{1}{m}\Big)^{ 1-\frac{1}{\gamma}} \, .
\end{align*}
Thus $\limsup_{m\rightarrow \infty} \sup_{0 < t < \bar{\tau}} \|\rho^m\cdot\bar{\calv} - 1\|_{L^1(\Omega)} \leq b \, |\Omega|$. Since $b > 0$ was arbitrary, we can not let $b \rightarrow 0$, proving the strong convergence in $L^{1,\infty}(Q)$.
\end{proof}
\begin{lemma}\label{uamamoto}
We adopt the assumptions of Prop.\ \ref{EXIRESC}, and of Lemma \ref{Bogovencorelui}. Moreover, assume that the initial states satisfy $m \, \int_{\Omega} \bar{\rho}^{0,m} \cdot\calv - 1 \, dx \rightarrow 0$ and that $\calw_i^{0,m} := \int_{\Omega} \rho^{0,m}_i(x) \, dx/|\Omega|$ satisfies $\inf \{\calw_i^{0,m} \, : \, 1 = 1,\ldots,N, \, m \geq 1\} > 0$.
Then 
\begin{enumerate}[(i)]
\item\label{claim1} There are $C_1, \, C_2$ independent on $m$ such that for almost all $0 < t < \bar{\tau}$,
\begin{align*}%\label{yamamoto}
\frac{1}{m} \, \max_{k=1,\ldots,N-1} |(\bar{q}^m_k)_M(t)|^2 \leq C_1\, \left(\int_{\Omega_{\bar{\varrho}_{\min}/2,2\bar{\varrho}_{\max}}(t)} (\bar{f}^m(\bar{\rho}) + 1)^{\frac{1}{2}} \, dx\right)^2 + \frac{C_2}{m} \, |\nabla \bar{q}^m|_{L^1(\Omega)}^2 \, .
\end{align*}
\item \label{claim2} $\sup_{m \in \mathbb{N}} \|(\bar{q}^m)^{\prime}\|_{L^2(Q^{\rm R})} < + \infty$ and $\sup_{m\in \mathbb{N}} \frac{1}{\sqrt{m}}  \, \|\bar{q}^m\|_{L^2(Q^{\rm R})} <+\infty$.
\end{enumerate}
% %Moreover, there are $C_1, \, C_2$ independent on $m$ such that for almost all $0 < t < \bar{\tau}$,
% \begin{align*}%\label{yamamoto}
% \frac{1}{m} \, \max_{k=1,\ldots,N-1} |(\bar{q}^m_k)_M(t)|^2 \leq C_1\, \left(\int_{\Omega_{\bar{\varrho}_{\min}/2,2\bar{\varrho}_{\max}}(t)} (\bar{f}^m(\bar{\rho}) + 1)^{\frac{1}{2}} \, dx\right)^2 + \frac{C_2}{m} \, |\nabla \bar{q}^m|_{L^1(\Omega)}^2 \, .
% \end{align*}
\end{lemma}
\begin{proof}
Let $\Omega^{\rm R}_+(t) := \{\calx \, : \, \bar{\varrho}(\calx,t) > 0\} = \{\calx \, : \, \pi(\calx,t) > - m\}$. For $\calx \in \Omega^{\rm R}_+(t)$, we let 
\begin{align}\label{mudef}
\bar{\mu}_i(\calx,t) := \bar{g}^m_i(\pi(\calx,t)) + \frac{1}{\bar{M}_i} \, \ln \bar{x}_i(\calx,t) \, .
\end{align}
Further we let $i_1 = i_1(\calx,t)$ be the largest index such that $\bar{\mu}_{i_1} = \max_{i=1,\ldots, N} \bar{\mu}_i$. For all $i \neq i_1$, we obtain in $\Omega^{\rm R}_+(t)$ that $ 0 \geq \bar{\mu}_i - \bar{\mu}_{i_1} \geq \bar{g}^m_i(\pi) - \bar{g}^m_{i_1}(\pi)  + \frac{1}{\bar{M}_i} \, \ln \bar{x}_i$, and this implies that 
\begin{align*}
\bar{\rho}_i\, |\bar{\mu}_{i_1} - \bar{\mu}_i| \leq \bar{\rho}_i \, |\bar{g}^m_i(\pi) - \bar{g}^m_{i_1}(\pi)| +  \frac{\bar{\varrho}}{e \, \min M} \, .
\end{align*}
Let $a_1 =\bar{\varrho}_{\min}/2$ and $ b_1 = 2\bar{\varrho}_{\max}$. With $\widehat{p}_0$ and $\widehat{p}_1$ from \eqref{gabiskrank3}, recall that $|\bar{g}^m(\pi)| \leq c\,|\pi|$ over $\Omega_{a_1,b_1}(t)$, with $c := \sup_{\widehat{p}_0(1/b_1) < s < \widehat{p}_1(1/a_1)} |\bar{g}^{\prime}(s)|$. After some elementary steps, we thus obtain that
\begin{align}\label{lachose}
 \bar{\rho}_i \, |\bar{\mu}_{i_1} - \bar{\mu}_i| \leq C(a_1,b_1) \, (|\pi| + 1) \, \quad \text{ over } \quad \Omega_{a_1, b_1}(t) \, .
\end{align}
For $i = 1,\ldots,N$, and $w \in \mathbb{R}^N$, we define $T_i(w) := w_i - \max w$. Then the map $T$ is Lipschitz continuous on $\mathbb{R}^N$, it assumes only negative values, and $T(w+ \lambda \, 1^N) = T(w)$ for all $(w, \, \lambda) \in \mathbb{R}^N \times \mathbb{R}$. We define $q_m^* = T(\Pi \bar{q}^m)$. Using the chain rule for Sobolev functions, we obtain that $\|\nabla q^*\|_{L^2(Q^R)} \leq \|\partial T\|_{L^{\infty}(\mathbb{R}^N)} \, \|\nabla \bar{q}\|_{L^2(Q^{\rm R})}$. In $\Omega_{a_1,b_1}(t)$, \eqref{lachose} implies that
\begin{align*}\begin{split}%\label{lachose2}
 \bar{\rho}_i \, |q^*_i| \leq & C_1 \, (|\pi| + 1) \, ,\\
 \bar{\rho}_i \, |(q^*_i)_M(t)| \leq & C_1 \, (|\pi| + 1) + b_1 \, |q^*_i - (q^*_i)_M(t)| \, . 
\end{split}
 \end{align*}
We integrate this inequality and we obtain that
\begin{align}\label{amamoto}
 |(q^*_i)_M(t)| \, \int_{\Omega_{a_1,b_1}(t)} \bar{\rho}_i \, dx \leq C_1 \, \int_{\Omega_{a_1,b_1}(t)} (|\pi|+1) \, dx +b_1\, |q^*_i - (q^*_i)_M(t)|_{L^1(\Omega)} \, .
\end{align}
Since $a_1 < \bar{\varrho}_{\min}$ and $b_1 > \bar{\varrho}_{\max}$ the Lemma \ref{Meslemma} shows that
\begin{align*}
\sup_{0 < t < \bar{\tau}} |\Omega_{a_1,b_1}^{\rm c}(t)| \leq k_1^{-1} \, \|\bar{\rho}^m \cdot \bar{\calv} - 1\|_{L^{1,\infty}} \, .
\end{align*}
Hence, with $\calw_i^{0,m} := \int_{\Omega} \rho^{0,m}_i(x) \, dx/|\Omega|$, it also follows that
\begin{align*} 
\int_{\Omega_{a_1,b_1}(t)} \bar{\rho}_i \, dx \geq & \int_{\Omega} \bar{\rho}_i \, dx - \|\bar{\rho}_i\|_{L^{\gamma,\infty}} \, |\Omega_{a_1,b_1}^{\rm c}(t)|^{1-\frac{1}{\gamma}}\\
\geq & |\Omega| \, \calw_{i}^{0,m} - k_1^{\frac{1}{\gamma}-1} \, \sup_m\|\bar{\rho}^m\|_{L^{\gamma,\infty}} \, (\|\bar{\rho}^m \cdot \bar{\calv} - 1\|_{L^{1,\infty}})^{1-\frac{1}{\gamma}} \, . 
\end{align*}
For $m$ large enough, we achieve for $i = 1,\ldots,N$ that $\int_{\Omega_{a_1,b_1}(t)} \bar{\rho}_i \, dx \geq |\Omega| \, \inf_{i,m} \calw^{0,m}_{i}/2$. 
Therefore, invoking also the Poincar\'e inequality to estimate $q-(q)_M$ in \eqref{amamoto}, it follows that
\begin{align}\label{yamamoto0}
\frac{1}{2} \, \inf_{i,m} \calw_i^{0,m} \, \max_{i=1,\ldots,N} |(q^*_i)_M(t)| \leq C_1\, \int_{\Omega_{a_1,b_1}(t)} |\pi| \, dx + C_2 \, |\nabla \bar{q}|_{L^1(\Omega)} \, .
\end{align}
We further note that
\begin{align*}
\Big|\int_{0}^1 \bar{g}^{\prime}(1+\theta \, \frac{\pi^m}{m}) - \bar{g}^{\prime}(1+\frac{\pi^m}{m}) \, dt\theta \cdot \bar{\rho}\Big| \, |\pi^m| =   |\bar{g}^m(\pi^m) \cdot \bar{\rho} - \pi^m| = \bar{f}^m(\bar{\rho}) +|\bar{k}(\bar{\rho})| \, ,
\end{align*}
to show that, on $\Omega_{a_1,b_1}(t)$ we can bound 
\begin{align*}
\frac{\bar{\varrho}}{2} \,  \inf_{\widehat{p}_0(1/a_1) \leq s \leq \widehat{p}_1(1/b_1)}\min |\bar{g}^{\prime\prime}(s)|  \,  \frac{\pi^2}{m}  \leq \bar{f}^m(\bar{\rho}) +|\bar{k}(\bar{\rho})| \, .
\end{align*}
Thus, $\pi^2/m \leq c \, (\bar{f}^m(\bar{\rho}) + 1)$. Squaring \eqref{yamamoto0} and dividing by $m$ yields
\begin{align*}%\label{yamamoto}
\frac{1}{m} \, \max_{i=1,\ldots,N} |(q^*_i)_M(t)|^2 \leq \tilde{C}_1\, \left(\int_{\Omega_{a_1,b_1}(t)} (\bar{f}^m(\bar{\rho}) + 1)^{\frac{1}{2}} \, dx\right)^2 + \frac{\tilde{C}_2}{m} \, |\nabla \bar{q}|_{L^1(\Omega)}^2 \, .
\end{align*}
Now, the definition of $T$ implies that $\Pi \bar{q}^m = T(\Pi \bar{q}^m) + (\max \Pi \bar{q}^m) \, 1^N$. Since $\eta^k \cdot 1^N = 0$ for $k = 1,\ldots,N-1$, it follows that $\bar{q}_k^m = \eta^k \cdot T(\Pi \bar{q}^m)$ and \eqref{claim1} follows.
%\begin{align}\label{yamamoto}
%\frac{1}{m} \, \max_{k=1,\ldots,N-1} |(\bar{q}^m_k)_M(t)|^2 \leq \bar{C}_1\, \left(\int_{\Omega_{a_1,b_1}(t)} (\bar{f}^m(\bar{\rho}) + 1)^{\frac{1}{2}} \, dx\right)^2 + \frac{\bar{C}_2}{m} \, |\nabla \bar{q}^m|_{L^1(\Omega)}^2 \, .
%\end{align}

Ad \eqref{claim2}. It is readily seen by means of \eqref{claim1} that $ \sup_{m \in \mathbb{N}} \frac{1}{m} \, \|(\bar{q}^m)_M\|_{L^2(0,\bar{\tau})}^2 < +\infty$. Due to the Poincar\'e inequality, this also means that $\bar{q}^m/\sqrt{m}$ is bounded in $L^2(Q^{\rm R}) $.

To prove the second bound, we start from \eqref{mudef} again, and we introduce
\begin{align*}%\label{mudef2}
\tilde{\mu}_i := \bar{\mu}_i - \bar{\calv}_i \, \pi =   \bar{g}^m_i(\pi) - \bar{\calv}_i \, \pi  + \frac{1}{\bar{M}_i} \, \ln \bar{x}_i \, . 
\end{align*}
We choose an index $i_2$ such that $\tilde{\mu}_{i_2} = \max \tilde{\mu}$. Applying the same steps as just seen, we can obtain over $\Omega_{a_1,b_1}(t)$ that
\begin{align*}
\bar{\rho}_i \, |T_i(\tilde{\mu})| \leq \bar{\rho}_i \, |\bar{g}^m_i(\pi) - \bar{\calv}_i \, \pi - (\bar{g}^m_{i_2}(\pi) - \bar{\calv}_{i_2} \, \pi )| + \frac{b_1}{e \, \min M} \, .
\end{align*}
On $\Omega_{a_1,b_1}(t)$, we can bound $|\bar{g}^m(\pi) - \bar{\calv} \, \pi| \leq C \, (\bar{f}^m(\bar{\rho}) +1)$. It follows that
\begin{align*}
\Big|\Big(T_i(\tilde{\mu})\Big)_M(t)\Big| \leq \tilde{C}_1 \, \int_{\Omega} \bar{f}^m(\bar{\rho}) + 1 \, dx + \tilde{C}_2 \, \Big|T_i(\tilde{\mu}) - \Big(T_i(\tilde{\mu})\Big)_M\Big|_{L^1(\Omega)} \, . 
\end{align*}
We let $q^{**} := T(\Pi^{\prime} \, \bar{q})$ of which we control the gradient. 
Recall that $\tilde{\mu} = \Pi^{\prime} \bar{q} + \mathscr{M}(\bar{\varrho}, \, \bar{q})$ on $\Omega^{\rm R}_+(t)$. Hence $T(\tilde{\mu}) = T(\Pi^{\prime} \, \bar{q}) = q^{**}$ on $\Omega_{a_1,b_1}(t)$. Thus
\begin{align*}
 \|(q^{**})_M\|_{L^2(0,\bar{\tau})} \leq C \, (\|\bar{f}^m(\bar{\rho})\|_{L^{1,2}} + \|\nabla \bar{q}\|_{L^{1,2}}) \, .
\end{align*}
Using that $\Pi^{\prime} \bar{q} = q^{**} + \max \Pi^{\prime} \, \bar{q} \, 1^N$, we finally obtain the bound in $L^2(Q^{\rm R})$ for $\{\Pi^{\prime} \bar{q}^m\}$.
\end{proof}

\section{Proof of the relative energy inequality}\label{SStorm}

In this section we provide the proof of Prop.\ \ref{calculheuri} and Prop.\ \ref{calculFO}, and another variant of the relative energy inequality. Even for weak solutions, the vector fields $\rho$ and $\varrho \, v$ possess certain distributional time derivatives that generate weak continuity in time. For $u = (u_1, \, u_2, \, u_3): \, Q_{\bar{\tau}} \rightarrow \mathbb{R}^3$ and $r = (r_1,\ldots,r_N): \, Q_{\bar{\tau}} \rightarrow \mathbb{R}^N$ sufficiently smooth, we can therefore rely for all $0 \leq t \leq \bar{\tau}$ on the identity
\begin{align}\label{timederiv}
& \int_{\Omega}\Big( \varrho \, \frac{|u|^2}{2} - \varrho \, u\cdot v - f(r) - \hat{\mu}(r) \cdot (\rho-r) \Big) \, dx \Big|_{0}^{t} \nonumber\\
& \qquad = \int_0^t  \frac{d}{dt}\int_{\Omega} \Big( \varrho \, \frac{|u|^2}{2} - \varrho \, u\cdot v - f(r) - \hat{\mu}(r) \cdot (\rho-r) \Big) \, dxd\tau \, .
\end{align}
Here $f = f^m$ or $f = \bar{f}^m$ according to which of the problems (IBVP) or ($\overline{\text{ IBVP}}^m$) is considered.

The time--derivative in the right-hand side is next computed. Using the continuity equation and the Navier--Stokes equations, we can replace the occurrences of $\partial_t \varrho$ with $-\divv (\varrho v)$, and those of $\partial_t (\varrho \, v)$ with $\divv(-\varrho \, v\otimes v + \mathbb{S} - p \, \mathbb{I}) + \varrho \, b$ in the sense of distributions. This procedure yields
\begin{align}\label{timederiv1}
\frac{d}{dt} \int_{\Omega}\Big( \varrho \, \frac{|u|^2}{2} - \varrho \, v\cdot u\Big) \, dx = 
%\int_{\Omega} \varrho \, \partial_t u \cdot(u - v) \, dx
%+ \int_{\Omega} \Big(\partial_t \varrho \, \frac{|u|^2}{2} - \partial_t (\varrho \,v ) \cdot u\Big) \, dx \\
 \int_{\Omega} & \varrho \, \partial_t u \cdot(u - v) + \varrho \, (v \cdot \nabla u) \cdot u \\
 & - \varrho \, (v \cdot  \, \nabla) u  \cdot v + \mathbb{S}(\nabla v) \, : \, \nabla u - p \, \divv u - \varrho \, b \cdot u \, dx\nonumber \, . 
\end{align}
We next want to compute the remaining terms in \eqref{timederiv}.
Note that $\partial_t r \cdot \hat{\mu}(r) = \partial_t f(r)$. The weak formulation of \eqref{mass} tested with $\psi = \hat{\mu}(r)$ yields
\begin{align}\label{timederiv2}
\frac{d}{dt} \int_{\Omega} \rho \cdot \hat{\mu}(r) \, dx - \int_{\Omega} \rho \cdot \partial_t \hat{\mu}(r) \, dx
 = \int_{\Omega} (\rho \, v + J) \, : \, \nabla \hat{\mu}(r) \, dx \, .
\end{align}
Collecting the identities \eqref{DISSIP}, \eqref{in1} and \eqref{timederiv}, \eqref{timederiv1} and \eqref{timederiv2}, we obtain that
\begin{align}\label{in2}
 \mathcal{E}(t) \leq & \mathcal{E}(0)  - \int_{0}^t\int_{\Omega} \mathbb{S}(\nabla v) \, : \,\nabla (v-u) + \zeta^{\rm Diff} + J \, : \, \nabla \hat{\mu}(r) \, dxd\tau\nonumber \\
  & + \int_0^t\int_{\Omega} \varrho \, (\partial_t u + (v\cdot \nabla )u) \cdot(u - v)  -  p \, \divv u + \varrho \, b\cdot (v-u) \\
 & \phantom{+ \int_0^t\int_{\Omega} } -\partial_t \hat{\mu}(r) \cdot (\rho - r) - \rho \, v \, : \, \nabla \hat{\mu}(r) \, dxd\tau \, .\nonumber
\end{align}
In order to use this identity in connection with solutions to the incompressible system, it is useful to re-express
\begin{align*}
& \int_{\Omega} \varrho \, (\partial_t u + (v\cdot \nabla )u) \cdot(u - v) \, dx = \int_{\Omega} |r|_1 \, (\partial_t u + (u \cdot \nabla)u) \cdot(u - v) \, dx + \mathcal{R}^1(t) \, ,\\
& \mathcal{R}^1(t) :=  \int_{\Omega} (\varrho - |r|_1) \, (\partial_t u + (v\cdot \nabla )u) \cdot(u - v) + |r|_1 \, [(v-u)\cdot \nabla ]u \cdot(u - v)  \, dx \, .
%\\
%& \qquad \qquad  + \int_{\Omega} (\varrho - |r|_1) \, (v\cdot \nabla )u \cdot(u - v) + |r|_1 \, [(v-u)\cdot \nabla ]u \cdot(u - v) \, dx \, .
\end{align*}
Here we use the notation $|r|_1:= \sum_{i=1}^N r_i$.
% Hence, with elementary rearrangements of \eqref{in2} we get
% \begin{align}\begin{split}\label{in3}
% \mathcal{E}(t) \leq &  \mathcal{E}(0)  - \int_{0}^t\int_{\Omega} \mathbb{S}(\nabla v) \, : \,\nabla (v-u) + \zeta^{\rm Diff} + J \, : \, \nabla \hat{\mu}(r) \, dxd\tau  \\
% + & \int_0^t\int_{\Omega} |r|_1 \, (\partial_t u + (u\cdot \nabla )u) \cdot(u - v)  - p \, \divv u +\varrho \, b\cdot (v-u) dxd\tau\\
% - & \int_0^t\int_{\Omega}  \partial_t \hat{\mu}(r) \cdot (\rho - r) + \rho \, v \, : \, \nabla \hat{\mu}(r) \, dxd\tau+ \int_{0}^t \mathcal{R}^1(\tau) \, d\tau \, .
% \end{split}
% \end{align}
Next we let  $(\rho^{\infty}, \, p^{\infty}, \, v^{\infty})$ be a solution to the incompressible model. With $r = \rho^{\infty}$, $|r|_1 =: \varrho^{\infty}$ and $u = v^{\infty}$, this implies that $\varrho^{\infty} \, (\partial_t v^{\infty} + (v^{\infty}\cdot \nabla )v^{\infty}) +\nabla p^{\infty} = \divv \mathbb{S}(\nabla v^{\infty}) + \varrho^{\infty} \, b$. From \eqref{in2}, we thus infer
\begin{align*}\begin{split}%\label{in4}
 \mathcal{E}(t) \leq  & \mathcal{E}(0) - \int_{0}^t\int_{\Omega} \mathbb{S}(\nabla (v -v^{\infty})) \, : \,\nabla (v-v^{\infty}) +\zeta^{\rm Diff} + J \, : \, \nabla \hat{\mu}(\rho^{\infty}) \, dxd\tau \\
 + & \int_0^t\int_{\Omega}  - p \, \divv v^{\infty} - \nabla p^{\infty} \cdot(v^{\infty} - v) +(\varrho-\varrho^{\infty}) \, b \cdot (v-v^{\infty})  \, dxd\tau\\
-& \int_0^t\int_{\Omega} \partial_t \hat{\mu}(\rho^{\infty}) \cdot (\rho - \rho^{\infty}) + \rho \, v \, : \, \nabla \hat{\mu}(\rho^{\infty}) \, dxd\tau 
 +  \int_{0}^t \mathcal{R}^1(\tau) \, d\tau \, .
\end{split}
\end{align*}
In the incompressible case, the state is restricted by the constraint \eqref{CONSTRAINT} and, using the condition that $\partial_pg(p^0) = \calv$, we have the identities $1 = \sum_{i=1}^N \rho^{\infty}_i \, \calv_i = \sum_{i=1}^N \rho^{\infty}_i \, \partial_p g_i(p^0)$. It follows that $\hat{p}(\rho^{\infty}) = p^0$. We define the chemical potentials $\mu^{\infty}_i$ for the incompressible system via \eqref{CHEMPOTincom}. Then, under the condition that $g^m(p^0) = 0$ (or $\bar{g}^m(1) = 0$ for the rescaled problem), it follows that $ \mu^{\infty}_i = p^{\infty} \, \calv_i + RT/M_i \, \ln \hat{x}_i(\rho^{\infty}) = p^{\infty} \, \calv_i + \hat{\mu}_i(\rho^{\infty})$.
Hence $\nabla \hat{\mu}(\rho^{\infty}) = \nabla (\mu^{\infty} - \calv \, p^{\infty})$. Moreover, the mass continuity equations \eqref{mass} imply that
\begin{align*}
\int_{0}^t\int_{\Omega} (\rho \, v + J) \, : \, \calv \otimes \nabla p^{\infty} \, dxd\tau = - \int_{0}^t \int_{\Omega} (\rho \cdot \calv- 1) \, \partial_t p^{\infty} \, dxd\tau \\
+ \int_{\Omega} \{(\rho(x,t) \cdot \calv - 1) \, p^{\infty}(x, \, t) - (\rho^0(x) \cdot \calv- 1) \, p^{\infty}(x, \, 0)\} \, dx \, .
\end{align*}
With the abbreviation $\tilde{\mathcal{E}}(t) := \mathcal{E}(t) - \int_{\Omega} (\rho(x,t) \cdot \calv - 1) \, p^{\infty}(x, \, t) \, dx $, we obtain that
\begin{align*}%\label{in5tris}
& \tilde{\mathcal{E}}(t) \leq  \tilde{\mathcal{E}}(0) - \int_{0}^t\int_{\Omega} \mathbb{S}(\nabla (v - v^{\infty})) \, : \,\nabla (v-v^{\infty}) + \zeta^{\rm Diff} +  J \, : \, \nabla \mu^{\infty} \, dxd\tau\nonumber \\
& + \int_0^t\int_{\Omega}   p \, \divv v^{\infty} + \nabla p^{\infty} \cdot(v^{\infty} - v) + (\varrho-\varrho^{\infty}) \, b \cdot( v-v^{\infty}) - (\rho \cdot \calv - 1) \, \partial_t p^{\infty} \nonumber \\
& \phantom{+ \int_0^t\int_{\Omega}} - \partial_t \hat{\mu}(\rho^{\infty}) \cdot (\rho - \rho^{\infty}) + \rho \, v \, : \, \nabla \mu^{\infty} \, dxd\tau  
+ \int_{0}^t \mathcal{R}^1(\tau) \, d\tau \, .
\end{align*}
The next steps consist of transforming the right-hand sides in order to obtain quadratic remainders. At first, we can expand $\rho \, v \,  = (\rho - \rho^{\infty}) \, (v-v^{\infty}) + (\rho - \rho^{\infty}) \, v^{\infty} + \rho^{\infty} \, v$. We verify that
\begin{align*}
 \rho^{\infty} \, v \, : \, \nabla \mu^{\infty} =& \rho^{\infty} \, v \, : \, \nabla \hat{\mu}(\rho^{\infty}) + \rho^{\infty}\cdot \calv \, \, v\cdot \nabla p^{\infty}\\
 =& \rho^{\infty} \, v \, : \, \nabla \Big(g(p^0) + \frac{RT}{M} \, \ln \hat{x}(\rho^{\infty})\Big) + \rho^{\infty}\cdot \calv \, v\cdot \nabla p^{\infty} =  v\cdot \nabla p^{\infty} \, .
\end{align*}
Here we used that $\sum_{i=1}^N\rho^{\infty}_i \, \nabla \hat{x}_i(\rho^{\infty}) = 0 $ and $\sum_{i=1}^{N} \calv_i \, \rho^{\infty}_i =1$. Hence we attain
\begin{align}\label{in6tris}
& \tilde{\mathcal{E}}(t) \leq  \tilde{\mathcal{E}}(0) - \int_{0}^t\int_{\Omega} \mathbb{S}(\nabla (v - v^{\infty})) \, : \,\nabla (v-v^{\infty}) + \zeta^{\rm Diff} + J \, : \, \nabla \mu^{\infty} \, dxd\tau \nonumber\\
& - \int_0^t\int_{\Omega}   p \, \divv v^{\infty} + \nabla p^{\infty} \cdot v^{\infty} +\partial_t \hat{\mu}(\rho^{\infty}) \cdot (\rho - \rho^{\infty}) +(\rho-\rho^{\infty}) \, v^{\infty} \, : \, \nabla \mu^{\infty} \, dxd\tau \nonumber \\
& - \int_{0}^t \int_{\Omega} (\rho \cdot \calv - 1) \, \partial_t p^{\infty} \, dxd\tau + \int_{0}^t (\mathcal{R}^1(\tau)+\mathcal{R}^2(\tau)) \, d\tau \, , \\
\mathcal{R}^2(\tau) & = \int_{\Omega}  (\rho - \rho^{\infty}) \, (v-v^{\infty}) \,:\, \nabla \mu^{\infty} + (\varrho-\varrho^{\infty}) \, b \cdot (v-v^{\infty})\, dx \, .
\nonumber
\end{align}
This is the first basic inequality valid for every type of weak solution. 
Next we want to use the stabilisation properties of diffusion, and there will be a branching in the discussion.

\subsection{Positive solutions and weak solutions of type-I}

In this case we can rely on $\zeta^{\rm Diff} = - J \, : \, \nabla \Pi q = M(\rho) \nabla \Pi q \, : \, \nabla \Pi q$, where the meaning of the variable $q$ is explained in \eqref{qplouc}.
With $ \mathscr{D} := \Pi^{\sf T} M(\rho)\, \Pi \nabla (q-q^{\infty}) \, : \, \nabla (q-q^{\infty}) \geq 0$, it follows that
\begin{align*}
& \zeta^{\rm Diff} + J \, : \, \nabla \mu^{\infty} = \Pi^{\sf T} \, M(\rho)  \Pi \, \nabla (q-q^{\infty}) \, : \, \nabla q \, \\
 =&  \mathscr{D} + \Pi^{\sf T} (M(\rho)-M(\rho^{\infty})) \, \Pi \nabla (q-q^{\infty}) \, : \, \nabla q^{\infty} + M(\rho^{\infty})  \Pi \, \nabla (q-q^{\infty}) \, : \, \nabla \mu^{\infty}\,\\
 =&  \mathscr{D} + \Pi^{\sf T} (M(\rho)-M(\rho^{\infty})) \, \Pi \nabla (q-q^{\infty}) \, : \, \nabla q^{\infty} - J^{\infty} \, : \,   \Pi \, \nabla (q-q^{\infty}) \, .
 \end{align*}
By means also of \eqref{in6tris}, we obtain that
\begin{align}\begin{split}\label{in8tris}
& \tilde{\mathcal{E}}^m(t) \leq  \tilde{\mathcal{E}}^m(0)- \int_{0}^t\int_{\Omega} \mathbb{S}(\nabla v - v^{\infty}) \, : \,\nabla (v-v^{\infty}) + \mathscr{D} -  J^{\infty} \, : \, \nabla(\Pi q- \mu^{\infty})  \, dxd\tau \\
& - \int_0^t\int_{\Omega}  p \, \divv v^{\infty} + \nabla p^{\infty} \cdot v^{\infty} +\partial_t \hat{\mu}(\rho^{\infty}) \cdot (\rho - \rho^{\infty}) + (\rho-\rho^{\infty}) \, v^{\infty} \, : \, \nabla \mu^{\infty} \, dxd\tau \\
& - \int_{0}^t \int_{\Omega} (\rho \cdot \bar{\calv} - 1) \, \partial_t p^{\infty} \, dxd\tau + \sum_{i=1}^3 \int_{0}^t  \mathcal{R}^{i}(\tau) \, d\tau \, ,
\end{split}
\end{align}
in which $ \mathcal{R}^{3}(t) :=- \int_{\Omega} (M(\rho) - M(\rho^{\infty})) \, (\Pi \nabla q- \nabla \mu^{\infty}) \, :\,\nabla \mu^{\infty} \, dx$.

We use the fact that $(\rho^{\infty},p^{\infty},v^{\infty})$ is a strong solution to the incompressible model (IBVP$^{\infty}$) to show that
\begin{align*}
 \int_{\Omega} J^{\infty} \, : \, \nabla (\Pi q - \mu^{\infty}) \, dx = & \int_{\Omega} (\partial_t \rho^{\infty} + \divv (\rho^{\infty} \,  v^{\infty})) \cdot (\Pi q - \mu^{\infty}) \, dx \, .
\end{align*}
For $0< a_0 < b_0 < + \infty$, $0 < t < \bar{\tau}$ we let $\Omega_{a_0,b_0}(t) = \{x \, : \, a_0 \leq \varrho(x,t) \leq b_0\}$. On $\Omega_{a_0,b_0}(t)$, the density is finite and strictly positive. Hence, the entire vector of chemical potentials is finite almost everywhere, and it obeys $\mu = \Pi q + \mathscr{M}(\varrho, \, q) \, 1^N$ (cf.\ \eqref{lesidentes}). All densities are strictly positive on this set, and we can split $\mu =g^m(p) + Dk(\rho)$ with $Dk(\rho) = (RT/M) \, \ln \hat{x}(\rho)$. Now,
due to the continuity equation for (IBVP$^{\infty}$) ($\partial_t\varrho^{\infty} + \divv(\varrho^{\infty} \, v^{\infty}) = 0$), we obtain the identity
\begin{align*}
 (\partial_t \rho^{\infty} + \divv (\rho^{\infty} \, v^{\infty})) \cdot (\Pi q - \mu^{\infty}) = (\partial_t \rho^{\infty} + \divv (\rho^{\infty} \, v^{\infty})) \cdot (\mu - \mu^{\infty}) \quad \text{ in } \quad \Omega_{a_0,b_0}(t) \, .
\end{align*}
We introduce the abbreviation $\eta^{\infty} := \partial_t \rho^{\infty} + \divv (\rho^{\infty} \, v^{\infty})$. It follows that
\begin{align}\label{splitmu}
 \eta^{\infty} \cdot (\mu - \mu^{\infty})
  = \eta^{\infty} \cdot (Dk(\rho) - Dk(\rho^{\infty})) + \eta^{\infty} \cdot g(p)  - \divv v^{\infty}\, p^{\infty} \, .
 \end{align}
Moreover, since $D^2k(\rho^{\infty}) \, \rho^{\infty} = 0$,
 \begin{align}\label{moreover}
  \partial_t \hat{\mu}(\rho^{\infty}) \cdot (\rho-\rho^{\infty}) =&  \partial_t \rho^{\infty} \cdot D^2k(\rho^{\infty}) \, (\rho - \rho^{\infty}) \nonumber \\
   = & \eta^{\infty}\cdot D^2k(\rho^{\infty}) \, (\rho - \rho^{\infty})- \divv (\rho^{\infty} \, v^{\infty}) \, D^2k(\rho^{\infty}) \, (\rho - \rho^{\infty}) \nonumber\\
  =&  \eta^{\infty}\cdot D^2k(\rho^{\infty}) \, (\rho - \rho^{\infty}) - v^{\infty} \cdot \nabla Dk(\rho^{\infty}) \, (\rho - \rho^{\infty}) \, .
 \end{align}
By means of \eqref{splitmu} and \eqref{moreover}, we see that
\begin{align*}
&  \int_{\Omega} J^{\infty} \, : \, \nabla(\Pi  q - \mu^{\infty}) - \partial_t\hat{\mu}^m(\rho^{\infty}) \cdot (\rho-\rho^{\infty}) \, dx = \int_{\Omega} v^{\infty} \cdot \nabla Dk(\rho^{\infty}) \cdot (\rho-\rho^{\infty}) \, dx\\
  & +\int_{\Omega_{a_0,b_0}(t)} \eta^{\infty} \cdot (Dk(\rho) - Dk(\rho^{\infty}) - D^2k(\rho^{\infty}) \, (\rho-\rho^{\infty})) \, dx\\
 & + \int_{\Omega_{a_0,b_0}(t)} \eta^{\infty} \cdot g^m(p) - \divv v^{\infty} \, p^{\infty}\, dx + \int_{\Omega_{a_0,b_0}^{\rm c}(t)} \eta^{\infty} \cdot  \big(\Pi q - \mu^{\infty}  - D^2k(\rho^{\infty}) (\rho-\rho^{\infty})\big) \, dx\, .
\end{align*}
Since $\Pi q = \Pi^{\prime}q + q_{N-1} \, \bar{\calv}$, we also note also that
\begin{align*}
& \int_{\Omega_{a_0,b_0}^{\rm c}(t)} \eta^{\infty} \cdot  \big(\Pi q - \mu^{\infty}  - D^2k(\rho^{\infty}) (\rho-\rho^{\infty})\big) \, dx \\
= & \int_{\Omega_{a_0,b_0}^{\rm c}(t)} \eta^{\infty} \cdot  \big(\Pi^{\prime}( q - q^{\infty})  - D^2k(\rho^{\infty}) (\rho-\rho^{\infty})\big) + \divv v^{\infty} \, (q_{N-1} - q^{\infty}_{N-1}) \, dx \, .
\end{align*}
Splitting $\Omega_{a_0,b_0}(t)$ into $B_{s_0,+}(t)$ and its complement, we introduce
\begin{align*}
& \mathcal{R}^{4}(t) := \int_{B_{s_0,+}(t)} (\partial_t \rho^{\infty} + \divv (\rho^{\infty} v^{\infty})) \cdot (Dk(\rho) - Dk(\rho^{\infty}) - D^2k(\rho^{\infty}) \, (\rho-\rho^{\infty})) \, dx \, ,\\
& \mathdutchcal{E}^{1}(t) :=  \int_{\Omega_{a_0,b_0}^{\rm c}(t)} (\partial_t \rho^{\infty} + \divv (\rho^{\infty} v^{\infty})) \cdot  \big(\Pi^{\prime} (q - q^{\infty}) - D^2k(\rho^{\infty}) (\rho-\rho^{\infty})\big) \, dx \\
& +  \int_{\Omega_{a_0,b_0}(t) \setminus B_{s_0,+}(t)} (\partial_t \rho^{\infty} + \divv (\rho^{\infty} v^{\infty})) \cdot (Dk(\rho) - Dk(\rho^{\infty}) - D^2k(\rho^{\infty}) \, (\rho-\rho^{\infty})) \, dx\, .
\end{align*}
We get
\begin{align*}%\label{in8tris+}
& \tilde{\mathcal{E}}^m(t) \leq  \tilde{\mathcal{E}}^m(0) - \int_{0}^t\int_{\Omega} \mathbb{S}(\nabla (v -  v^{\infty})) \, : \,\nabla (v-v^{\infty}) +  \mathscr{D} \, dxd\tau\nonumber\\
& + \int_0^t\int_{\Omega}\chi_{\Omega_{a_0,b_0}(t)} (\eta^{\infty} \cdot g(p) - \divv v^{\infty} \,p^{\infty})  + \chi_{\Omega_{a_0,b_0}^{\rm c}(t)} \, \divv v^{\infty} \, (q_{N-1} - q^{\infty}_{N-1}) \, dx d\tau \nonumber\\
& + \int_0^t\int_{\Omega} (p^{\infty}- p) \, \divv v^{\infty}  - (\rho-\rho^{\infty}) \, v^{\infty}\, : \, \nabla (\mu^{\infty} - Dk(\rho^{\infty}))- (\rho \cdot \bar{\calv} - 1) \, \partial_t p^{\infty} \, dxd\tau \nonumber \\
& + \int_0^t \mathdutchcal{E}^{1}(\tau) + \sum_{i=1}^4 \mathcal{R}^{i}(\tau) \, d\tau \, .
\end{align*}
We define $p = \tilde{p} + q_{N-1}$ and $p^{\infty} = \tilde{p}^{\infty} + q^{\infty}_{N-1}$ and get
\begin{align*}
& \chi_{\Omega_{a_0,b_0}(\tau)} (\eta^{\infty} \cdot g(p) - \divv v^{\infty} \,p^{\infty}) + \chi_{\Omega_{a_0,b_0}^{\rm c}(\tau)} \, \divv v^{\infty} \, (q_{N-1} - q^{\infty}_{N-1})+ (p^{\infty}- p) \, \divv v^{\infty} \\
& =  \chi_{\Omega_{a_0,b_0}(\tau)} \eta^{\infty} \cdot (g(p) - \calv \, p) \,+ \chi_{\Omega_{a_0,b_0}^{\rm c}(\tau)}\,  (\tilde{p}^{\infty} - \tilde{p}) \, \, \divv v^{\infty} \, .
\end{align*}
We define $\mathdutchcal{E}^{2}(t) := \int_{\Omega_{a_0,b_0}^{\rm c}(t)} (\tilde{p}^{\infty}-\tilde{p}) \, \divv v^{\infty} \, dx$. It remains to observe that 
\begin{align*}
 (\rho-\rho^{\infty}) \, v^{\infty}\, : \, \nabla (\mu^{\infty} - Dk(\rho^{\infty})) = (\rho\cdot \calv-1) \, v^{\infty} \cdot \nabla p^{\infty} \, ,
\end{align*}
and \eqref{in8tris} yields
\begin{align*}%\label{in8trisquadris}
\tilde{\mathcal{E}}^m(t) \leq & \tilde{\mathcal{E}}^m(0) - \int_{0}^t\int_{\Omega} \mathbb{S}(\nabla (v -  v^{\infty})) \, : \,\nabla (v-v^{\infty}) +  \mathscr{D} \, dxd\tau\nonumber\\
& + \int_0^t\int_{\Omega_{a_0,b_0}(\tau)} (\partial_t \rho^{\infty} + \divv (\rho^{\infty} v^{\infty})\cdot (g(p) - \calv \, p) \, dx d\tau\\
& - \int_{0}^t \int_{\Omega} (\rho \cdot \bar{\calv} - 1) \, (\partial_t p^{\infty} + v^{\infty} \cdot \nabla v^{\infty}) \, dxd\tau + \int_0^t \sum_{i=1,2} \mathdutchcal{E}^{i}(\tau) + \sum_{i=1}^4 \mathcal{R}^{i}(\tau) \, d\tau \, .
\nonumber
\end{align*}
This establishes the Proposition \ref{calculFO}. Moreover, suppose that $\inf_{i} \rho_i(x,t) \geq s_0$ for almost all $(x,t) \in Q$. Then $|\Omega \setminus B_{s_0,+}(t)| = 0$ for almost all $0 < t < \bar{\tau}$.  
Letting $a_0 \rightarrow 0+$ and $b_0 \rightarrow +\infty$, we see that $\mathdutchcal{E}^{i}(t) = 0$ for all $t \in [0, \, \bar{\tau}]$. This proves Prop.\ \ref{calculheuri}. 

\subsection{Another form of the relative energy inequality}\label{SStormII}

In order to prove the Theorem \ref{heuriheura} we adopted the simplifying positivity assumption \eqref{UNIFORMPOS}. Here we want to  motivate how this assumption can be removed. 

In the case that the strict positivity of the densities and even the uniform positivity (B3$^{\prime})$ of $\{M_{ij}(\rho)\}$ fail, the proof of a relative energy inequality is more delicate. Nevertheless, we can obtain a result by means of techniques developed for weak solutions of type-II, in particular the re-parametrisation $\rho = \mathcal{X}(p, \, w_1,\ldots,w_N)$ of \eqref{TrafotypeII}, where $w = (w_1,\ldots,w_N)$ is a diffusive variable subject to $\hat{p}^m(w) = p^0$. We have first to recall a few basic properties concerning this transformation.\\

{\bf Some preliminaries:} Let $\rho$ and $w$ be related via \eqref{TrafotypeII}. Then it was shown in \cite{druetmaxstef} that $\mathcal{P}(\hat{\mu}(\rho) - \hat{\mu}(w)) = 0$ is valid. Hence, for all $i \neq j$ we obtain that
\begin{align*}
 \frac{1}{M_i} \, \ln \hat{x}_i(\rho) - \frac{1}{M_j} \, \ln \hat{x}_j(\rho) + g^m_i(\hat{p}^m(\rho)) - g^m_j(\hat{p}^m(\rho)) =  \frac{1}{M_i} \, \ln \hat{x}_i(w) - \frac{1}{M_j} \, \ln \hat{x}_j(w)  \, ,
\end{align*}
where we also used that $g^m(p^0) = 0$. We choose $j = i_1$ with $i_1$ such that $\hat{x}_{i_1}(w) = \max \hat{x}(w) \geq 1/N$. With $p = \hat{p}^m(\rho)$, it follows that
\begin{align*}
 \frac{1}{M_i} \, \ln \hat{x}_i(w) \geq \frac{1}{\min M} \, \ln \frac{1}{N} + \frac{1}{M_i} \ln \hat{x}_i(\rho) + g_i^m(p) - g_{i_1}^m(p) \, ,
\end{align*}
and this implies that
\begin{align*}%\label{quotients0}
 \hat{x}_i(w) \geq (N)^{-\frac{\max M}{\min M}} \, \hat{x}_i(\rho) \, \exp\Big(M_i \, (\min g(p) - \max g(p))\Big) \, .
\end{align*}
Similarly, with $j = i_2$ with $i_2$ such that $\hat{x}_{i_2}(\rho) = \max \hat{x}(\rho) \geq 1/N$, it follows that
\begin{align*}
 \hat{x}_i(\rho) \geq (N)^{-\frac{\max M}{\min M}} \, \hat{x}_i(w) \, \exp\Big(M_i \, (\min g(p) - \max g(p))\Big) \, .
\end{align*}
The quotients $\rho_i/w_i$ are nothing else but $\hat{x}_i(\rho)/ \hat{x}_i(w)$ times $\hat{n}(\rho)/\hat{n}(w)$. Since $\hat{p}^m(w) = p^0$ by definition, we have $\sum_{i} w_i \, \calv_i = 1$. The latter implies that $1/\max_i \{M_i\calv_i\} \leq \hat{n}(w) \leq 1/\min_i \{M_i\calv_i\}$, hence $\rho_i/w_i$ is seen to be bounded below and above via
\begin{align}\label{lesquotients}
 & \frac{\min\{M\calv\}}{\max M} \, \ \, \Big(\frac{1}{N}\Big)^{\frac{\max M}{\min M}}\varrho \, \exp\Big(M_i \, (\min g(p) - \max g(p))\Big)\leq \frac{\rho_i}{w_i} \\
 & \quad  \leq \frac{\max\{M\calv\}}{\min M} \, N^{\frac{\max M}{\min M}} \, \varrho \, \exp\Big(M_i \, (\max g(p) - \min g(p))\Big) \, .\nonumber
\end{align}
This can be used to show that, if $\{M_{ij}\}$ satisfies the typical conditions of a Maxwell-Stefan mobility matrix, the product $M(\rho) \, D^2f^m(w)$ is bounded by a function of $p$ and $\varrho$. To see this, we express $$M(\rho) \, D^2f^m(w) = (M(\rho) \, R^{-1}) \, (R \, W^{-1}) \, W \, D^2f^m(w) \, ,$$ with $R = \text{diag}(\rho)$ and $W = \text{diag}(w)$. If $M(\rho) \, R^{-1}$ is bounded, the claim is obvious. Now, a Maxwell-Stefan mobility tensor satisfies \eqref{maxstefreg}, thus $d_0 \, \mathdutchcal{P} \leq M(\rho) \, R^{-1} \leq d_1 \, \mathdutchcal{P}$.

Assume that $(\rho, \, v)$ is a weak solution of type-II. For $0 < p_1 < p_2 < +\infty$ we define
\begin{align*}
\Omega_{p_1,p_2}(t) := \{x \in \Omega \, : \, p_1 \leq \hat{p}^m(\rho(x,t)) \leq p_2\} \, .
\end{align*}
In this set the density $\varrho$ is bounded and strictly positive: We here can refer to the inequalities \eqref{la-bas} and \eqref{plarge12}. Since $\sqrt{w} \in L^2(0,\bar{\tau}; \, W^{1,2}(\Omega; \, \mathbb{R}^N))$ we obtain a representation
\begin{align}\label{JMS}
 J= -2 \, M(\rho) \, D^2f^m(w) \, W^{\frac{1}{2}} \, \nabla \sqrt{w} \, .
\end{align}
For $s_0 > 0$, define $B_{s_0,+}(t) := \{x \in \Omega_{p_1,p_2}(t) \, :\, \min \rho(x,t) \geq s_0\}$.
In the set $B_{s_0,+}(t)$ all $\rho_i$ are finite and strictly positive. Thus, after some straightforward manipulations using \eqref{lesquotients}, it is found that
\begin{align}\label{wposit}
 \min \hat{x}(w) \geq \epsilon_0 := c \, \inf_{s\in [p_1,p_2]} \, \exp(\max M \, (\min g^m(s) - \max g^m(s)) \quad \text{ on } \quad B_{s_0,+}(t) \, 
\end{align}
where $c = c(N,M,\calv)>0$ is some constant. 
This helps introducing another appropriate substitute for the chemical potentials. For $s \in \mathbb{R}$ and $\epsilon > 0$, we define $s_{(\epsilon)} := \min\{s, \, \epsilon\}$, and $\hat{x}_{(\epsilon)} := ((\hat{x}_1)_{(\epsilon)}, \ldots, (\hat{x}_N)_{(\epsilon)})$. Using the chain rule for Sobolev functions we can verify that the function
\begin{align}\label{ERSATZII}
\tilde{\mu}_i^{\epsilon} := \frac{RT}{M_i} \, \ln (\hat{x}_i(w))_{(\epsilon)}  \, ,
\end{align}
belongs to $L^2(0,\bar{\tau}; \, W^{1,2}(\Omega))$ and $L^{\infty}(Q)$. 
For $0 < \epsilon \leq \epsilon_0$, the bound \eqref{wposit} shows that $\tilde{\mu}_i^{\epsilon} = \hat{\mu}_i(w)$ on $B_{s_0,+}(t)$. From now we denote $\tilde{\mu}_i := \tilde{\mu}_i^{\epsilon_0}$, $\tilde{x}_i = (\hat{x}_{i}(w))_{(\epsilon_0)}$.
We then have
\begin{align}\label{PimuMS}
  \mathcal{P} \mu =  \mathcal{P} \hat{\mu}(w) = \mathcal{P} \tilde{\mu}  \quad \text{ in } \quad B_{s_0,+}(t)\, .
  \end{align}
As a substitute for the stabilising contribution of diffusion (that is $M(\rho) \nabla (\mu-\mu^{\infty}) \: : \: \nabla (\mu-\mu^{\infty})$) we define, in $\Omega_{p_1,p_2}(t)$, 
\begin{align*}
\mathscr{D} := M(\rho) \, (2D^2f^m(w) \, W^{\frac{1}{2}} \, \nabla \sqrt{w}-\nabla \mu^{\infty}) \, : \,   (2D^2f^m(w) \, W^{\frac{1}{2}} \, \nabla \sqrt{w}-\nabla \mu^{\infty}) \geq 0 \, ,
\end{align*}
In $B_{s_0,+}(t)$, we by construction have $\mathscr{D} = M(\rho) \nabla (\tilde{\mu}-\mu^{\infty}) \: : \: \nabla (\tilde{\mu}-\mu^{\infty})$.
With these preliminaries, we can prove a more general version of the relative energy inequality.
\begin{prop}\label{calculMS}
Let $(\rho, \, v)$ be a weak solution to \emph{(IBVP$^m$)} with associated diffusive variables $w$. Let $(\rho^{\infty}, \, p^{\infty}, \, v^{\infty})$ satisfy \emph{ (IBVP$^{\infty}$)}. Then, with $\tilde{\mu}$ according to \eqref{ERSATZII},
\begin{align*}%\begin{split}\label{in6} 
& \mathcal{E}^m(t)+\int_{Q_t} \mathbb{S}(\nabla (v-v^{\infty})) \, : \,\nabla (v-v^{\infty}) + \chi_{\Omega_{p_1,p_2}(\tau)} \, \mathscr{D} \, dxd\tau\\
& \leq \int_{Q_t} \chi_{B_{s_0,+}(\tau)} \, \big(\partial_t \rho^{\infty} + \divv (\rho^{\infty} \, v^{\infty})\big) \cdot ( g^m(p)- p \, \calv)-(\rho \cdot \calv - 1) \, (\partial_t p^{\infty} + v^{\infty} \cdot \nabla p^{\infty}) \, dxd\tau \\
& \quad +\mathcal{E}^m(0) - \int_{\Omega} p^{\infty}(x, \cdot ) \, (\rho^m(x,\cdot) \cdot \calv - 1) \, dx \Big|^t_0  
+  \int_{0}^t \mathcal{R}^m(\tau) + \mathscr{E}^m(\tau)  d\tau\, .
\end{align*}
in which $\mathcal{R}^m = \sum_{i=1}^4 \mathcal{R}^{m,i}$ where $\mathcal{R}^{m,i}$  possesses the same representation as in Prop.\ \ref{calculFO} for $i=1,2,4$ and $\mathcal{R}^{m,3}(t) := 
-\int_{B_{s_0,+}(t)} (M(\rho) - M(\rho^{\infty})) \nabla \mu^{\infty} \cdot \nabla (\tilde{\mu}-\mu^{\infty}) \, dx$.
Moreover, the functional $\mathdutchcal{E}^m(t)$ obeys
\begin{align*}
&  \mathscr{E}^m(t)
   :=  -\int_{\Omega} \chi_{\Omega_{p_1,p_2}^{\rm c}(t)} \,  J \, : \, \nabla \mu^{\infty} - \chi_{\Omega_{p_1,p_2}(t) \setminus B_{s_0,+}(t)} \, (J \, : \, \nabla \mu^{\infty} + M(\rho) \, \nabla \mu^{\infty} \cdot \nabla \mu^{\infty} ) \, dx \\
& + \int_{B_{s_0,+}^{\rm c}(t)}  \Big(J^{\infty} \, : \, \nabla( \tilde{\mu} - \mu^{\infty}) + (\partial_t \rho^{\infty} + \divv (\rho^{\infty} \, v^{\infty})) \cdot  (\tilde{\mu} - \mu^{\infty}  - D^2k(\rho^{\infty}) (\rho-\rho^{\infty})) \\
& \phantom{\int_{B_{s_0,+}^{\rm c}(t)}  \Big(\quad } + (p^{\infty}-p) \, \divv v^{\infty}\Big) \, dx \, .
%\end{split}
\end{align*}
\end{prop}
\begin{proof}
To start with, we as previously obtain  the identity \eqref{in6tris}. In the present context we next write
\begin{align*}
 \zeta^{\rm Diff} + J \, : \, \nabla \mu^{\infty} = & \underbrace{\zeta + 2 \, J \, : \, \nabla \mu^{\infty} + M(\rho) \, \nabla \mu^{\infty} \cdot \nabla \mu^{\infty}}_{=: \mathscr{D}}- J \, : \, \nabla \mu^{\infty} - M(\rho) \, \nabla \mu^{\infty} \cdot \nabla \mu^{\infty} \,.
\end{align*} 
We cannot identify $\mathscr{D}$ everywhere in $\Omega$. However, the non-negativity of $\zeta^{\rm Diff}$ implies that $ \mathscr{D} \geq  2 \, J \, : \, \nabla \mu^{\infty} + M(\rho) \, \nabla \mu^{\infty} \cdot \nabla \mu^{\infty}$. Thus
\begin{align*}
 &  \zeta^{\rm Diff} + J \, : \, \nabla \mu^{\infty} \geq   \chi_{\Omega_{p_1,p_2}(t)} \, (\mathscr{D} - J \, : \, \nabla \mu^{\infty} - M(\rho) \, \nabla \mu^{\infty} \cdot \nabla \mu^{\infty} )
  + \chi_{\Omega_{p_1,p_2}^{\rm c}(t)} \, J \, : \, \nabla \mu^{\infty} \\
  & \quad \geq   \chi_{\Omega_{p_1,p_2}(t)} \, \mathscr{D} - \chi_{B_{s_0,+}(t)} \, (J \, : \, \nabla \mu^{\infty} + M(\rho) \, \nabla \mu^{\infty} \cdot \nabla \mu^{\infty} )
   \\
  & \qquad  + \chi_{\Omega_{p_2,p_1}^{\rm c}(t)} \, J \, : \, \nabla \mu^{\infty} -\chi_{\Omega_{p_1,p_2}(t) \setminus B_{s_0,+}(t)} \, (J \, : \, \nabla \mu^{\infty} + M(\rho) \, \nabla \mu^{\infty} \cdot \nabla \mu^{\infty} ) \, . 
 \end{align*}
% We introduce 
% \begin{align*}
% \mathscr{E}^{m,1}(t) := -\int_{\Omega} \chi_{\Omega_{p_1,p_2}^{\rm c}(t)} \,  J \, : \, \nabla \mu^{\infty} - \chi_{\Omega_{p_1,p_2}(t) \setminus B_{s_0,+}(t)} \, (J \, : \, \nabla \mu^{\infty} + M(\rho) \, \nabla \mu^{\infty} \cdot \nabla \mu^{\infty} ) \, dx \, ,
% \end{align*}
We obtain that
\begin{align*}%\label{in7tris}
 & \tilde{\mathcal{E}}^m(t) + \int_{0}^t\int_{\Omega} \mathbb{S}(\nabla (v-v^{\infty}))  \, : \,\nabla (v-v^{\infty})+\chi_{\Omega_{p_1,p_2}(t)} \, \mathscr{D} \, dxd\tau\nonumber\\
 & \leq \tilde{\mathcal{E}}^m(0) +\int_0^t\int_{\Omega} \chi_{B_{s_0,+}(t)} (J \, : \, \nabla \mu^{\infty} + M(\rho)\nabla \mu^{\infty} \, : \, \nabla \mu^{\infty})  \, dxd\tau \nonumber\\
& + \int_0^t\int_{\Omega} -  p \, \divv v^{\infty} - \nabla p^{\infty} \cdot v^{\infty} -\partial_t \hat{\mu}(\rho^{\infty}) \cdot (\rho - \rho^{\infty}) - (\rho-\rho^{\infty}) \, v^{\infty} \, : \, \nabla \mu^{\infty}\nonumber\\
 & \phantom{\int_0^t\int_{\Omega}}- (\rho \cdot \calv - 1) \, \partial_t p^{\infty} \, dxd\tau + \int_{0}^t \mathscr{E}^{m,1}(\tau) + \mathcal{R}^{m,1}(\tau)+\mathcal{R}^{m,2}(\tau) \, d\tau \, ,\\
 & \mathscr{E}^{m,1}(t) := -\int_{\Omega} \chi_{\Omega_{p_1,p_2}^{\rm c}(t)} \,  J \, : \, \nabla \mu^{\infty} - \chi_{\Omega_{p_1,p_2}(t) \setminus B_{s_0,+}(t)} \, (J \, : \, \nabla \mu^{\infty} + M(\rho) \, \nabla \mu^{\infty} \cdot \nabla \mu^{\infty} ) \, dx 
\end{align*} 
We recall \eqref{ERSATZII} and \eqref{PimuMS} to see that, in the set $B_{s_0,+}(t)$,
\begin{align*}
  J \, : \, \nabla \mu^{\infty} + M(\rho) \nabla \mu^{\infty} \, : \, \nabla \mu^{\infty} 
 %= - (M(\rho)-M(\rho^{\infty})) \, (2D^2f(w) \, W^{\frac{1}{2}} \, \nabla \sqrt{w}-\nabla \mu^{\infty}) \cdot \nabla \mu^{\infty}\\
%  & - M(\rho^{\infty}) \, (2D^2f(w) \, W^{\frac{1}{2}} \, \nabla \sqrt{w}-\nabla \mu^{\infty}) \cdot \nabla \mu^{\infty} \\
 = &  (M(\rho)-M(\rho^{\infty}))\nabla \mu^{\infty}  \,  : \, \nabla (\tilde{\mu}- \mu^{\infty})\\
 & + J^{\infty} \, : \, \nabla (\tilde{\mu}- \mu^{\infty}) \,,
\end{align*}
% where we recall \eqref{ERSATZII} to show that $2D^2f(w) \, W^{\frac{1}{2}} \, \nabla \sqrt{w} = \nabla \tilde{\mu}$ in $B_{s_0,+}(t)$. 
We define $\mathcal{R}^{m,3}(t) =  - \int_{B_{s_0,+}(t)} (M(\rho)-M(\rho^{\infty})) \, \nabla ( \tilde{\mu}-\mu^{\infty}) \cdot \nabla \mu^{\infty} \, dx$ and together with $\mathscr{E}^{m,2}(\tau) :=  \int_{B_{s_0,+}^{\rm c}(t)}\,  J^{\infty} \, : \, \nabla( \tilde{\mu} - \mu^{\infty})  \, dx$, we obtain that
% \begin{align*}
%  \mathcal{R}^{m,3}(t) = & - \int_{B_{s_0,+}(t)} (M(\rho)-M(\rho^{\infty})) \, \nabla ( \tilde{\mu}-\mu^{\infty}) \cdot \nabla \mu^{\infty} \, dx \, ,\\
%  \mathscr{E}^{m,2}(\tau) = & \int_{B_{s_0,+}^{\rm c}(t)}\,  J^{\infty} \, : \, \nabla( \tilde{\mu} - \mu^{\infty})  \, dx \, ,
%  \end{align*}
%  and we obtain that
 \begin{align*}\begin{split}%\label{in8tris++}
 & \tilde{\mathcal{E}}^m(t) + \int_{Q_t} \mathbb{S}(\nabla (v-v^{\infty}))\, : \,\nabla (v-v^{\infty}) +\chi_{\Omega_{p_1,p_2}(\tau)} \, \mathscr{D} \, dxd\tau\leq  \tilde{\mathcal{E}}^m(0) \\
 & + \int_{Q_t} J^{\infty} \, : \, \nabla( \tilde{\mu} - \mu^{\infty})-  p \, \divv v^{\infty} - \nabla p^{\infty} \cdot v^{\infty} -\partial_t \hat{\mu}(\rho^{\infty}) \cdot (\rho - \rho^{\infty}) \, dxd\tau \\
 & - \int_{Q_t} (\rho-\rho^{\infty}) \, v^{\infty} \, : \, \nabla \mu^{\infty}+  (\rho \cdot \calv - 1) \, \partial_t p^{\infty} \, dxd\tau + \int_{0}^t\sum_{i=1,2} \mathscr{E}^{m,i}(\tau) + \sum_{i=1}^3   \mathcal{R}^{m,i}(\tau) \, d\tau \, .
 \end{split}
 \end{align*} 
From now on the steps are essentially the same as for the type--one weak solution. First
\begin{align*}
 \int_{\Omega} J^{\infty} \, : \, \nabla (\tilde{\mu} - \mu^{\infty}) \, dx = \int_{\Omega} (\partial_t \rho^{\infty} + \divv (\rho^{\infty} \, v^{\infty})) \cdot (\tilde{\mu}- \mu^{\infty}) \, dx \, .
\end{align*}
On $B_{s_0,+}(t)$, the identity $\mathcal{P}(\tilde{\mu} - \mu) = 0$ is valid and, since $\partial_t\varrho^{\infty} + \divv(\varrho^{\infty} \, v^{\infty}) = 0$, hence
\begin{align*}
&  \eta^{\infty} \cdot (\tilde{\mu}- \mu^{\infty}) = \eta^{\infty}\cdot (\mu - \mu^{\infty}) = \eta^{\infty} \cdot (Dk(\rho) - Dk(\rho^{\infty}) + g^m(p)) - \divv v^{\infty}\, p^{\infty} \, .
\end{align*}
After a few steps (cp.\ \eqref{moreover})
\begin{align*}
&  \int_{\Omega} J^{\infty} \, : \, \nabla (\tilde{\mu} - \mu^{\infty}) - \partial_t\hat{\mu}(\rho^{\infty}) \cdot (\rho-\rho^{\infty}) \, dx = \int_{\Omega} v^{\infty} \cdot \nabla Dk(\rho^{\infty}) \cdot (\rho-\rho^{\infty}) \, dx\\
 + & \int_{B_{s_0,+}(t)} \eta^{\infty} \cdot (Dk(\rho) - Dk(\rho^{\infty}) - D^2k(\rho^{\infty}) \, (\rho-\rho^{\infty}))+  \eta^{\infty} \cdot g^m(p) - \divv v^{\infty} \, p^{\infty}  \, dx\\
 & + \int_{B_{s_0,+}^{\rm c}(t)} \eta^{\infty} \cdot  \big(\tilde{\mu} - \mu^{\infty} - D^2k(\rho^{\infty}) \cdot (\rho-\rho^{\infty})\big) \, dx\, .
\end{align*}
With $\mathcal{R}^{m,4}(t) = \int_{B_{s_0,+}(t)} \eta^{\infty}\cdot (Dk(\rho) - Dk(\rho^{\infty}) - D^2k(\rho^{\infty}) \, (\rho-\rho^{\infty})) \, dx$ and with
\begin{align*}
 \mathscr{E}^{m,3}(t) := & \int_{B_{s_0,+}^{\rm c}(t)} \eta^{\infty} \cdot  (\tilde{\mu} - \mu^{\infty}- D^2k(\rho^{\infty}) \cdot (\rho-\rho^{\infty})) + (p^{\infty}-p) \, \divv v^{\infty} \, dx \, ,
\end{align*}
the claim follows.
\end{proof}
Finally let us motivate why, relying on Prop.\ \ref{calculMS}, the proof of Theorem \ref{rigolo} can disclaim the positivity assumption \eqref{UNIFORMPOS}.
% 
% %It relies on the Prop.\ \ref{calculMS}. 
We first note that the assumption \eqref{TurevesII} yields $|\Omega \setminus \Omega_{p_1,p_2}(t)| = 0$. At second, we observe that the estimation of the remainder $\mathcal{R}^m$ can be essentially performed as in the simplified proof. So we have only to discuss the error functional $\mathscr{E}^m$.
Certain of the contributions in $\mathscr{E}^m$ can be controlled, even without assuming \eqref{TurevesII}. On $\Omega_{p_1,p_2}(t)$, we can employ Young's inequality and estimate
 \begin{align*}
 |J \, : \, \nabla \mu^{\infty} + M(\rho) \, \nabla \mu^{\infty} \cdot \nabla \mu^{\infty}| = |M(\rho) \, (2D^2f^m(w) \, W^{\frac{1}{2}} \, \nabla \sqrt{w}  - \nabla \mu^{\infty}) \, :\, \nabla \mu^{\infty}|\\
 \leq \epsilon \, \mathscr{D} + \frac{1}{4\epsilon} \,  M(\rho) \, \nabla \mu^{\infty} \cdot \nabla \mu^{\infty} \, .
 \end{align*}
Hence, recalling also \eqref{Erhobelow1bar},
 \begin{align*}
 & \Big|\int_{\Omega}\chi_{\Omega_{p_1,p_2}(t) \setminus B_{s_0,+}(t)} \, (J \, : \, \nabla \mu^{\infty} + M(\rho) \, \nabla \mu^{\infty} \cdot \nabla \mu^{\infty} ) \, dx\Big| \leq \epsilon \, \int_{\Omega_{p_1,p_2}(t)} \mathscr{D} \, dx \\
 & \quad + \frac{\bar{\lambda}}{4\epsilon} \, \|\nabla \mu^{\infty}\|_{L^{\infty}}^2 \,   \int_{\Omega_{p_1,p_2}(t) \setminus B_{s_0,+}(t)} |\rho| \, dx \leq \epsilon \, \int_{\Omega_{p_1,p_2}(t)} \mathscr{D} \, dx + \frac{C}{\epsilon} \, \mathcal{E}^m(t) \, .
\end{align*}
Similar observations and \eqref{ERSATZII} are used to show with $\epsilon_0$ from \eqref{wposit} that
\begin{align*}
 & \Big| \int_{B_{s_0,+}^{\rm c}(\tau)}  (\partial_t \rho^{\infty} + \divv (\rho^{\infty} \, v^{\infty})) \cdot  (\tilde{\mu} - \mu^{\infty}  - D^2k(\rho^{\infty}) (\rho-\rho^{\infty}))dxd\tau\Big|\\ 
 & \leq \|\partial_t \rho^{\infty} + \divv (\rho^{\infty} \, v^{\infty})\|_{L^{\infty}} \, \left(\|\tilde{\mu} - \mu^{\infty}\|_{L^{\infty}} \,  |B_{s_0,+}^{\rm c}(\tau)| + c \, \int_{B_{s_0,+}^{\rm c}(\tau)} \,  |\rho-\rho^{\infty}| \, dx \right) \, \\
 & \leq c \, \|\partial_t \rho^{\infty} + \divv (\rho^{\infty} \, v^{\infty})\|_{L^{\infty}} \, (\ln \epsilon_0^{-1} + \|\mu^{\infty}\|_{L^{\infty}}) \, \mathcal{E}^m(\tau) \, .
 \end{align*}
There are however terms in $\mathscr{E}^m$ that we cannot control otherwise than by the assumption that $p_1 < p^m < p_2$. At first we consider $\int_{B_{s_0,+}^{\rm c}(t)}  J^{\infty} \, : \, \nabla( \tilde{\mu} - \mu^{\infty}) \, dxd\tau$. In order to estimate this term, we recall that $B_{s_0,+}^{\rm c}(t) = \Omega_{p_1,p_2}(t) \setminus B_{s_0,+}(t)$. 
We also note that the construction \eqref{ERSATZII} implies that $\nabla \tilde{\mu}_i$ has support over the set $\omega^i_{\epsilon_0,+}(t) := \{x \, : \, \hat{x}_{i}(w)(x,t) > \epsilon_0\}$. We choose $i_1$ such that $\hat{x}_{i_1}(w) = \min_{i} \hat{x}_i(w)$. For all $i \neq i_1$, $\tilde{\mu}_{i_1} =  1/M_{i_1} \, \ln \epsilon_0$ in $\omega_{\epsilon_0,-}^i(t) := \Omega \setminus \omega_{\epsilon_0,+}^i(t)$. Hence $\nabla \tilde{\mu}_{i_1} = 0$ in $\omega_{\epsilon_0,-}^i(t)$. This allows to express
\begin{align*}
&  J^{\infty}\, : \, \nabla (\tilde{\mu} - \mu^{\infty}) =  J^{\infty}\, : \, (\nabla \tilde{\mu} - \nabla \tilde{\mu}_{i_1} \, 1^N  - \nabla \mu^{\infty} + \nabla \mu^{\infty}_{i_1} \, 1^N)\\
&  = \sum_{i = 1}^N J^{\infty,i} \cdot \Big( \chi_{\omega^i_{\epsilon_0,+}(t)} \, (\nabla \tilde{\mu}_i - \nabla \tilde{\mu}_{i_1}) - (\nabla \mu_i^{\infty} - \nabla \mu_{i_1}^{\infty})\Big) \, \\
&  = \sum_{i = 1}^N  \chi_{\omega^i_{\epsilon_0,+}(t)} \, J^{\infty,i} \otimes (e^i-e^{i_1})  \, : \,  (\nabla \tilde{\mu} - \nabla \mu^{\infty})- \sum_{i = 1}^N  \chi_{\omega^i_{\epsilon_0,-}(t)} \, J^{\infty,i} \cdot (\nabla \mu_i^{\infty} - \nabla \mu_{i_1}^{\infty}) \,  .
\end{align*}
In $\omega^i_{\epsilon_0,+}(t)$, we have $\tilde{x}_i\geq \epsilon_0$, hence $w_i \geq \epsilon_0 \, \min\{M\calv\}/\max M $, and \eqref{lesquotients} shows that there is a certain function $\calr_0$ such that $\rho_i \geq \calr_0(p_1,p_2,\epsilon_0) > 0$. Then
\begin{align*}
| J^{\infty,i}  \otimes (e^i-e^{i_1}) \, : \,  (\nabla (\tilde{\mu} - \mu^{\infty}))|  = \Big|\frac{J^{\infty,i} \otimes (e^i-e^{i_1}) }{\sqrt{\rho_i}} \, : \,  \sqrt{\rho}_i \, (\nabla (\tilde{\mu}-\mu^{\infty}))\Big|\\
% \leq \frac{1}{\sqrt{\rho_i}} \, |J^{\infty,i}| \, \sqrt{\rho_i} \, |(\nabla \tilde{\mu} - \nabla \mu^{\infty}) \cdot( e^{i}-e^{i_1})|\\
\leq \epsilon \, \rho_i \, |(\nabla \tilde{\mu} - \nabla \mu^{\infty})\cdot (e^{i}-e^{i_1})|^2 + \frac{|J^{\infty,i}|^2}{4 \, \rho_i \, \epsilon} \, .
\end{align*}
On $\Omega_{p_1,p_2}(t)$, this allows for the estimates
\begin{align*}
| J^{\infty,i} \cdot \nabla (\tilde{\mu}_i - \mu_i^{\infty})| \leq  \epsilon \,  \rho_i \, |(\nabla \tilde{\mu} - \nabla \mu^{\infty}) \cdot (e^i-e^{i_1})|^2 + \frac{C(p_1,p_2,\epsilon_0)}{\epsilon} \,  |J^{\infty,i}|^2 \, \, .
\end{align*}
Overall, after exploiting that $\mathscr{D} \geq d_0 \, \sum_{i=1}^N \rho_i \, |\mathdutchcal{P}_{i}(\nabla \tilde{\mu} - \nabla \mu^{\infty})|^2$, we see that
\begin{align*}
\left| \int_{B_{s_0,+}^{\rm c}(t)}  J^{\infty} \, : \, \nabla( \tilde{\mu} - \mu^{\infty}) \, dxd\tau\right| \leq \int_{ B_{s_0,+}^{\rm c}(t)} \epsilon \, \mathscr{D} + \frac{C}{\epsilon} \,|J^{\infty}|^2 \, dx \, .
\end{align*}
Finally it remains to consider
\begin{align*}
\Big|\int_{B_{s_0,+}^{\rm c}(t)}  (p^{\infty}-p) \, \divv v^{\infty} \, dx \Big| \leq (|p^{\infty}|_{L^{\infty}} + p_2-p_1) \, |\divv v^{\infty}|_{L^{\infty}} \, |B_{s_0,+}^{\rm c}(t)| \leq \psi^m(t) \, \mathcal{E}^m(t) \, .
\end{align*}
Hence, we can obtain the Gronwall inequality and the convergence argument needed for Theorem \ref{heuriheura}, without having assumed the strict positivity of the densities. It remains to show only the convergence of the pressure. Here we cannot rely on \eqref{MuM} over the whole domain $\Omega$, but only in $B_{s_0,+}(t)$.
Accordingly, we define $\bar{\zeta}^m :=  -(\eta\cdot\tilde{\mu}^m)_M$, $p^m_* :=  p^m +\bar{\zeta}^m \, \chi_{B_{s_0,+}(t)}$, and $\zeta^m := \eta\cdot\tilde{\mu}^m - (\eta\cdot\tilde{\mu}^m)_M$, and instead of \eqref{grouldei} we get
\begin{align}\label{grouldeinew}
p^m_* = \chi_{B_{s_0,+}(t)} \, \left(\zeta^m - \eta \cdot \Big(\frac{RT}{M} \, \ln \hat{x}(\rho^m) + h^m\Big) \right) + p^m \, \chi_{B_{s_0,+}(t)} \, .
\end{align}
As above, we show that $\nabla \zeta^m \rightarrow \eta\cdot\nabla \mu^{\infty}$ in $L^2(Q)$. Hence 
$\zeta^m \rightarrow \eta \cdot \mu^{\infty} - (\eta\cdot \mu^{\infty})_M$ in $L^2(0,\bar{\tau}; \, W^{1,2}(\Omega))$, and \eqref{grouldeinew} $p^m_* \rightarrow p^{\infty}$ in $L^2(Q)$. 
\end{document}